\documentclass[pdftex]{amsart}
\usepackage{graphicx, mathabx, amssymb,amsfonts,amsmath,amsthm,newlfont}
\usepackage{epsfig,url}
\usepackage[usenames,dvipsnames]{color}
\usepackage{enumerate}
\usepackage[colorlinks=true,linkcolor=red,citecolor=blue]{hyperref}
\usepackage{color}

\usepackage[all,2cell]{xy} \UseAllTwocells \SilentMatrices

\newtheorem{thm}{Theorem}[section]
\newtheorem{cor}[thm]{Corollary}
\newtheorem{lem}[thm]{Lemma}
\newtheorem{prop}[thm]{Proposition}

\theoremstyle{definition}
\newtheorem{defn}[thm]{Definition}
\newtheorem{notn}[thm]{Notation}

\newtheorem{ex}[thm]{Examples}
\newtheorem{example}[thm]{Example}

\theoremstyle{remark}
\newtheorem{rem}[thm]{Remark}

\numberwithin{equation}{section}
\newcommand{\To}{\longrightarrow}

\newcommand{\Z}{\mathbb Z}
\newcommand{\Q}{\mathbb Q}

\newcommand{\C}{\mathbb C}

\newcommand{\R}{\mathbb R}

\newcommand{\Pro}{\mathbb P}

\newcommand{\GC}{\mathcal{GC}}
\newcommand{\q}{/\!/}

\newcommand{\tr}{\mathrm{tr}}
\newcommand{\gr}{\mathrm{gr}}
\newcommand{\id}{\mathrm{id}}

\newcommand{\can}{\mathrm{can}}
\newcommand{\mer}{\mathrm{mer}}

\newcommand{\MM}{\mathsf{M}}

\newcommand{\mm}{\mathfrak{m}}

\newcommand{\Gmb}{\overline{\mathbb{G}}_m}

\newcommand{\SE}{\mathcal{S}}

\newcommand{\ql}{\setminus \! \setminus}

\newcommand{\Or}{\mathcal{O}}
\newcommand{\lf}{\mathrm{lf}}
\newcommand{\Gm}{\mathbb{G}_m}

\newcommand{\mot}{\mathrm{mot}}

\newcommand{\ess}{\mathrm{ess}}

\newcommand{\BB}{\mathcal{B}}

\newcommand{\BBB}{\mathfrak{B}}

\newcommand{\GL}{\mathrm{GL}}
\newcommand{\SL}{\mathrm{SL}}

\newcommand{\cone}{\mathsf{c}}
\newcommand{\trop}{\mathrm{trop}}
\newcommand{\vol}{\mathrm{vol}}

\newcommand{\rt}{\mathrm{rt}}
\newcommand{\perf}{\mathrm{perf}}
\newcommand{\opp}{\mathrm{opp}}

\newcommand{\Top}{\mathcal{T}op}

\newcommand{\Det}{\mathrm{Det}}
\newcommand{\red}{\mathrm{red}}
\newcommand{\Quad}{\mathcal{Q}}

\newcommand{\PC}{\mathcal{PC}}
\newcommand{\PF}{\mathsf{S}}
\newcommand{\PLC}{\mathrm{PLC}}
\newcommand{\BLC}{\mathrm{BLC}}

\newcommand{\LM}{L\mathcal{M}}
\newcommand{\LA}{L\mathcal{A}}

\newcommand{\stab}{\mathrm{stab}}

\usepackage{geometry}
\begin{document}

\author{Francis Brown}
\begin{title}[Bordifications of tropical moduli spaces]
{Bordifications of the moduli spaces of tropical curves and abelian varieties, and unstable  cohomology of  $\mathrm{GL}_g(\mathbb{Z})$ and $\mathrm{SL}_g(\mathbb{Z})$}\end{title}
\maketitle

\begin{abstract} We construct  bordifications of the moduli spaces of tropical curves and of tropical abelian varieties, and show that  the tropical Torelli map extends to their bordifications. We prove that the classical bi-invariant differential forms studied by Cartan and others extend to these bordifications  by studying their behaviour at infinity, and consequently deduce infinitely many new non-zero unstable cohomology classes in the cohomology of the general and special linear groups $\mathrm{GL}_g(\mathbb{Z})$ and $\mathrm{SL}_g(\mathbb{Z})$.  In particular, we obtain  a new and geometric  proof of Borel's theorem on the stable cohomology of these groups. 
 In addition, we   completely determine the cohomology of the link  of the moduli space of tropical abelian varieties within a certain range, and show that it contains the stable cohomology of the general linear group. 
 In the process, we define  new transcendental invariants associated to  the minimal vectors of quadratic forms, and  also show that  a certain   part of the cohomology of the general linear group $\mathrm{GL}_g(\mathbb{Z})$ admits the structure of a motive.  
 In an appendix, we give an algebraic construction of the Borel-Serre compactification by embedding it  in the real points of an iterated blow-up of a projective space along linear subspaces, which may have  independent applications. \end{abstract} 
 
\section{Introduction}
The main goal of this paper is an algebro-geometric construction of bordifications of the moduli spaces of tropical curves and abelian varieties. Before describing these in detail, we first present some  applications to the cohomology of the special and general linear groups.

\subsection{Unstable cohomology of  linear groups}

Let $\mathcal{P}_g$ denote the space of symmetric positive definite $g\times g$ matrices $X$ with real entries. It is a connected contractible space equipped with a right action  of $h\in \GL_g(\R)$ via the map $h(X) = h^TX h$.

Since the orbifold $\mathcal{P}_g/ \GL_g(\Z)$ is a $K(\pi,1)$ one has  $H^n(\GL_g(\Z);\R) = H^n(\mathcal{P}_g/\GL_g(\Z); \R)$.
Block direct sum of matrices $X \mapsto X \oplus 1$ defines a  map $\mathcal{P}_g \rightarrow \mathcal{P}_{g+1}$. The stable cohomology is defined to be the limit with respect to these maps:
\[ H^n (\GL_{\infty}(\Z);\R) = \varprojlim_{g} H^n (\GL_g(\Z);\R)\ .  \]
It  was famously computed by Borel \cite{Borel}, from which  he deduced  the  ranks of the rational algebraic $K$-theory of the integers (and, more generally, of  all number fields), which is   of  fundamental importance in the modern theory of motives.  Very little is known about the unstable cohomology of the  groups $\SL_g(\Z)$ and $\GL_g(\Z)$. See figures  1 and 3 for the   range in which their  cohomology groups  are completely known.

\begin{figure}[h]
\[
\begin{array}{c|cccccccccccccccccc}
g    & H^i(\GL_g(\Z)) \\ \hline
%1 & H_0 \\
2  & H^0       \\
 3 & H^0     \\
 4 &H^0    \\ 
 5 & H^0   && &H^5 &  \\ 
 6 &H^0   & & &H^5 & &\fbox{$H^8$}   \\ 
 7 &H^0    && &H^5 & & &H^9&&& & &  H^{14} & \fbox{$H^{15}$}  \\ 
\end{array}
\]
\caption{Known ranges for which the cohomology of $\GL_g(\Z)$ has been completely determined  \cite{SouleSL3, LeeSzczarba, ElbazVincentGanglSoule}.  An entry $H^i$ in the table  indicates that $H^i(\GL_g(\Z);\R)$ is non-zero and of rank $1$. All non-zero classes are explained by  theorem  \ref{thm: introcohomSLGL} except for the two boxed entries. }
\end{figure}

Even less is known about the 
  cohomology   with compact supports $ H_c^n(\mathcal{P}_g/\GL_g(\Z); \R)$,  which we shall denote by $H_c^n(\GL_g(\Z);\R)$. The notation is justified  by  duality: indeed, when $g$ is odd, the orbifold $\mathcal{P}_g/\GL_g(\Z)$ is orientable and its compactly supported cohomology is Poincar\'e dual to  ordinary cohomology. However,  in the case when $g$ is even, the cohomology with and without compact supports are \emph{a priori}  unrelated. There is no  stability property for cohomology with  compact supports,  as one may see from figure 2.

   The ordinary cohomology of the special linear group $\SL_g(\Z)$ coincides with that of $\GL_g(\Z)$ when $g$ is odd, but is built out of the cohomology of $\GL_g(\Z)$ with and without compact supports in the  case when $g$ is even. See figure 3 for a table of known results.

\subsection{Cohomological results}
  Consider the abstract graded exterior algebra 
 \[\Omega_{\can}^{\bullet} = \bigwedge \bigoplus_{k\geq 1} \Q \,\omega^{4k+1}\] with generators $\omega^{4k+1}$ in degree $4k+1$. 
 For  any  $X \in \mathcal{P}_g$  and $k>1$,  let
\begin{equation} \label{introomegadeftrace} \omega^{4k+1}_X =   \tr ((X^{-1} dX))^{4k+1} \ ,  
\end{equation} 
which is a well-defined closed differential form  on $\mathcal{P}_g$ with  a number of remarkable properties, including bi-invariance with respect to left and right multiplication by $\GL_g(\R)$. Consequently,   any  $\omega \in \Omega_{\can}^{n}$ defines  a closed differential  form of degree $n$ on $\mathcal{P}_g/\GL_g(\Z)$. 

Consequently there is a natural map, which is sometimes referred to as either the   `Borel map' or `Matsushima homomorphism':
\begin{equation} \label{into: jhomomorphism}  j :    \Omega_{\can}^{n}\otimes_{\Q} \R \To  H^{n} (\GL_g(\Z);\R)\ .
\end{equation}
Borel proved  
that $j$ is injective for $n\leq g/4$.  Results of Matsushima \cite{Matsushima}  and Garland  \cite{Garland} imply that the map $j$ is surjective in the same range.  Although this range is  narrow, it nonetheless tends to infinity, which allowed Borel to compute the stable limit:
\[  \Omega^{\bullet}_{\can}  \otimes_{\Q} \R  \cong    H^{\bullet} (\GL_{\infty}(\Z);\R)\ .    
\]

\begin{figure}
\[
\begin{array}{c|cccccccccccccccccc}
g    & %H_c^i(\GL_g(\Z))
 \\ \hline
2  &         \\
 3 & H^6_c \quad    \\
 4 &   \quad H^7_c    \\ 
 5 &   & & H^{10}_c  \quad   & & &   H^{15}_c \quad   \\ 
 6 &    & &  \quad H^{11}_c & \fbox{$H^{12}_c$} \quad  & &\quad  H^{16}_c   \\ 
 7 &   & &   &  \quad  \fbox{$H^{13}_c$} & H^{14}_c &&& &   H^{19}_c & & H^{23}_c & && H^{28}_c  \\ 
\end{array}
\]
\caption{Compactly-supported cohomology of $\mathcal{P}_g/\GL_g(\Z)$ in the known ranges  \cite{SouleSL3, LeeSzczarba, ElbazVincentGanglSoule}. An entry $H_c^i$ in the table  indicates that $H_c^i(\GL_g(\Z);\R)$ is non-zero and of rank $1$.    All non-trivial classes are explained by  theorem  \ref{thm: introcohomSLGL} except for the two boxed entries.    }
\end{figure}

Let $g>1$ be odd, and let  $\Omega(g) \subset \Omega_{\can}$ denote the graded subalgebra generated by $\omega^5,\ldots, \omega^{2g-1}$. All forms $\omega^{n}$  of degree $n \geq 2g$ vanish identically on $\mathcal{P}_g$.
In order to state our results, 
we divide  $\Omega(g)$ into two subspaces.
Let $\Omega_c(g) \subset \Omega(g)$ denote the ideal generated  by $\omega^{2g-1}$.  We shall call $\Omega_c(g)$ the space of forms    of  `compact type'  
and define $\Omega_{nc}(g) \subset \Omega(g)$ to be the graded subalgebra generated by $1, \omega^5, \ldots, \omega^{2g-5}$. The subscript $nc$ stands for `non-compact' type. One has $\Omega(g) = \Omega_c(g) \oplus \Omega_{nc}(g)$.

\begin{thm} \label{thm: introcohomSLGL} Let $g>1$ be odd.     There are injective maps  
\begin{eqnarray}
& (i). \! &\Omega^{\bullet}_{nc}(g)\otimes_{\Q}\R \hookrightarrow  H^{\bullet} ( \GL_h(\Z); \R)  \quad \hbox{ for all }  h\geq g  \ , \nonumber  \\
& (ii). & \Omega^{\bullet}_{c}(g) \otimes_{\Q}\R\hookrightarrow  H_c^{\bullet+1} ( \GL_g(\Z); \R) \ ,    \nonumber  \\
& (iii). & \Omega^{\bullet}_{c}(g)\otimes_{\Q}\R \hookrightarrow  H_c^{\bullet+2} ( \GL_{g+1}(\Z); \R)\ .  \nonumber  
\end{eqnarray} 
These facts  imply the following results about the special linear group: 
\begin{eqnarray}
& (iv). \! &\Omega^{\bullet}_{nc}(g)\otimes_{\Q}\R \hookrightarrow  H^{\bullet} ( \SL_h(\Z); \R)   \quad \hbox{ for all }  h\geq g\ ,   \nonumber  \\
& (v). & (\Omega^{\bullet}_{nc}(g) \oplus  \Omega^{d_{g+1} - \bullet -2 }_{c}(g))\otimes_{\Q}\R \hookrightarrow  H^{\bullet} ( \SL_{g+1}(\Z); \R)\ ,   \nonumber  
\end{eqnarray} 
where $d_g = \dim \mathcal{P}_g = \binom{g+1}{2}. $
\end{thm} 
Statement $(i)$ (and implicitly $(ii)$) is discussed in a research announcement of Ronnie Lee \cite{LeeUnstable}, but no proof seems ever to have  appeared.  
The  statement also appears in   Franke  \cite{Franke, Grobner}  using automorphic methods. Our proof bears  some similarity to the strategy suggested by Lee and implies that the map $(i)$ for $h=g$ and $(ii)$ are canonically split: the key point is the construction of representatives for elements in $\Omega_c(g)$ which have compact support\footnote{It is easy to show that the canonical forms of compact type vanish along the boundary of the Borel-Serre compactification, but this is insufficient: the difficulty  is to prove that they actually extend to the boundary in the first place, since once sees from the definition  \eqref{introomegadeftrace} that they blow up as one approaches the boundary where  $\det(X)\rightarrow 0$. To prove this result, and resolve this apparent contradiction,  requires a very careful analysis of the geometry of $L\mathcal{P}_g/\GL_g(\Z)$ at infinity.  }. Statement $(i)$ implies, but is much stronger than, Borel's result on the injectivity of  \eqref{into: jhomomorphism}  in small degrees.

The statement $(iii)$ about the unstable cohomology of linear groups of even rank is new. It follows from   a much stronger theorem (theorem \ref{thmintroOmegacginjectsDrLAgtrop} below)  on the existence of cohomology classes in the moduli space of tropical abelian varieties and uses  recent results on the acyclicity of the `inflation complex'  in  \cite{TopWeightAg}.

To illustrate the content of $(iv)$ and $(v)$, consider the following table of the known cohomology of $\SL_g(\Z)$, taken from \cite{ChurchFarbPutman} and based on computer calculations of \cite{ElbazVincentGanglSoule}.

\begin{figure}[h]
\[
\begin{array}{c|c|ccccccccccccccccc}
g    & \dim (\mathcal{P}_g) & n=  0   & 1 &2 &3 &4 &5 & 6 & 7 & 8   & 9 & 10 & 11 &12 & 13 &14 & 15 \\ \hline
2  & 3 &  1  &   \\
 3 &  6& 1 &   \\
 4 &10 &1 &  & & 1 \\ 
 5 &15 &1 &  &&&&1 \\ 
 6 &21 &1 & &&&&2 && & \fbox{1} & \fbox{1} & 1   \\ 
 7 &28 &1 &  &&&& 1 &&&& 1 &&& && 1 & \fbox{1} \\ 
\end{array}
\]
\caption{An entry in row $g$ and column $n$ equals $\dim H^n(\SL_g(\Z);\R)$;  blank entries are zero.
All non-zero entries in this table are explained by the previous theorem except for the three boxed entries. These are  discussed in \S\ref{sect: RecentProgress}.
Some partial computations are known in $g=8$ by \cite{GL8}.
} \label{fig:SLn}
\end{figure}

\subsection{Moduli of tropical abelian varieties}
Theorem  \ref{thm: introcohomSLGL} is a consequence of a stronger result concerning  cohomology classes on the link  $\left| \LA_g^{\trop}\right|$ of the moduli space of tropical abelian varieties, which we presently explain.   The reason for the vertical bar  notation is that $\left| \LA_g^{\trop}\right|$ is merely the topological incarnation of a richer kind of hybrid geometric object, denoted by  $\LA_g^{\trop}$,  which we shall discus later.  In any case, 
as a set,  $\left| \LA_g^{\trop}\right|$ is the quotient of the rational closure   $\mathcal{P}^{\rt}_g$ of $\mathcal{P}_g$ consisting of positive semi-definite matrices with rational kernel, modulo the action of $\GL_g(\Z)$, and by $\R^{\times}$ acting by  scalar multiplication.  We consider its decomposition into perfect cones,  due to Vorono\"i \cite{Voronoi, Martinet}, which are described as follows. 
Let $Q$ be a positive definite quadratic form, and let 
\[ M_Q = \{\lambda \in \Z^g\backslash \{0\}:  Q(\lambda) \leq Q(\mu) \quad \hbox{for all } \mu   \in \Z^g\backslash \{0\} \}\] 
denote the set of minimal vectors of $Q$. The set of $\lambda\lambda^T$, for $\lambda \in M_Q$,  span a convex polyhedral cone $\sigma_Q$ in the  link   $L \mathcal{P}^{\rt}_g  $  of $ \mathcal{P}^{\rt}_g$, whose points are projective classes of symmetric matrices. The   topological space $\left| \LA_g^{\trop}\right|$ is obtained  by gluing together $\GL_g(\Z)$-equivalence classes of quotients $\sigma_Q/ \mathrm{Aut}(\sigma_Q) $  of polyhedral cones by their finite groups of automorphisms. 

\begin{ex} (See figure \ref{figureVoronoiGL2}). 
 The quadratic form
$Q(x_1,x_2) = x_1^2 + x_1x_2 +x_2^2$  has minimal vectors $M_Q= \{(\pm 1, 0), (0, \pm 1), \pm (1,-1)\}$, and hence
$$ \sigma_Q =\left\{   \begin{pmatrix} \alpha_1 + \alpha_3  & - \alpha_3 \\    - \alpha_3 & \alpha_2+ \alpha_3\end{pmatrix}: \alpha_1,\alpha_2,\alpha_3 \geq 0\right\}  $$
is the convex hull of the rank 1 matrices  $\left(\begin{smallmatrix} 1 & 0 \\ 0 & 0 \end{smallmatrix}\right)$, $\left(\begin{smallmatrix} 0 & 0 \\ 0 & 1 \end{smallmatrix}\right)$, $\left(\begin{smallmatrix} 1 & -1 \\ -1 & 1 \end{smallmatrix}\right)$ in the space of projective classes of $2\times 2$ symmetric matrices.  There is a single $\GL_2(\Z)$-orbit of cells of maximal dimension in $|\LA_2^{\trop}|$ generated by $\sigma_Q$.   The stabiliser of $\sigma_Q$ is the symmetric group on three elements which permutes the three vertices of  $\sigma_Q$.
\end{ex}

 Denote the closed subspace   of  $\left|\LA_g^{\trop}\right|$ corresponding to  symmetric  matrices with vanishing determinant by    $| \partial \LA_g^{\trop}|$. Its  open complement:
 \[   \left| \LA_g^{\circ, \trop} \right|  =   | \LA_g^{\trop}| \ \setminus \    | \partial \LA_g^{\trop}| \]
 is nothing other than the space $L\mathcal{P}_g/   \GL_g(\Z) $ where $L \mathcal{P}_g= \mathcal{P}_g /  \R^{\times}  $.

 \begin{thm}  \label{thmintroOmegacginjectsDrLAgtrop}  Let $g>1$ be odd. Every  form of compact type $\omega \in \Omega_c(g)$ extends to a smooth differential  form on $| \LA_g^{\trop}|$. This  defines an injective map of graded algebras:
 \begin{equation} \label{introOmegacInjects}  \Omega^{\bullet}_c(g)\otimes_{\Q} \R  \To H_{dR}^{\bullet} (  | \LA_g^{\trop}|)\ .
 \end{equation} 
\end{thm} 
Theorem \ref{thm: introcohomSLGL} follows from theorem \ref{thmintroOmegacginjectsDrLAgtrop} using  a de Rham theorem for   certain kinds of 
 topological spaces, such as $ | \LA_g^{\trop}|$,  which  are defined  by gluing quotients of polyhedra together by finite group actions (see below) as well as results from \cite{TopWeightAg}.
 
 In fact, theorem \ref{thmintroOmegacginjectsDrLAgtrop} enables us to completely determine the cohomology of $  | \LA_{\bullet}^{\trop}|$ in a certain range.
If   $g>1$ is odd and $\kappa(g)$ denotes the stable range for the cohomology of the general linear group $\GL_g(\Z)$, we show in corollary \ref{cor: stablecohomLag} that 
\begin{align}
 H_{dR}^{n} (  | \LA_g^{\trop}|)& \cong  \Omega^{n}_c(g)_{\Q} \otimes \R     \\
   H_{dR}^{n-1} (  | \LA_{g-1}^{\trop}|) & =0  \nonumber 
   \end{align}
    for  $n\geq d_g - \kappa(g)$. Using  the recent results \cite{StableShuBinyong, KupersMillerPatzt}, we may take $\kappa(g)=g$.

\begin{figure}[h] \begin{center}
\quad {\includegraphics[width=6.5cm]{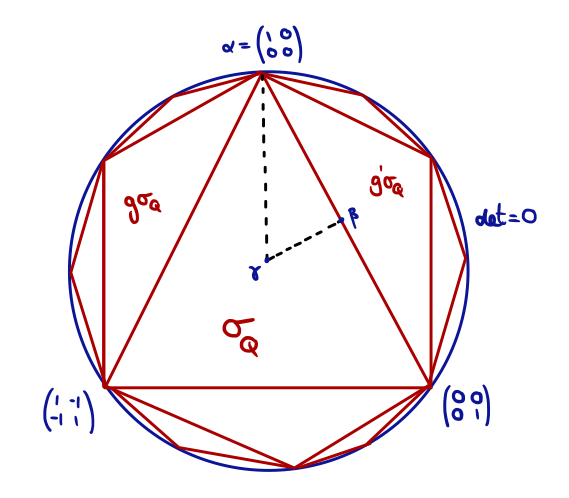}} 
\end{center}
\caption{The decomposition of the link $L \mathcal{P}^{\rt}$ into Vorono\"i cells viewed inside the projective space $\Pro^2(\R)$ of symmetric $2\times 2$ matrices.
The space $|\LA_2^{\trop}|$ is the quotient  of the closed simplex $ \sigma_{Q}$ by the action of $\Sigma_3$, and  is the one-point compactification of the interior  $|\LA^{\circ,\trop}|$ at the cusp.  The space $|\LA_2^{\trop,\BB}|$ has an additional half-interval at the  cusp.} \label{figureVoronoiGL2}

\end{figure}

\subsection{Bordification of the moduli space of  tropical abelian varieties.} 
In order to prove  theorem \ref{thmintroOmegacginjectsDrLAgtrop},  we construct an explicit bordification 
\[  \left |\LA_g^{\trop, \BB} \right| \To   \left |\LA_g^{\trop} \right| \ .\]
It has a boundary  $\left |\partial \LA_g^{\trop, \BB} \right| \subset \left |\LA_g^{\trop, \BB} \right|  $ with the property that 
\[     \left |\LA_g^{\trop, \BB} \right| \ \setminus \ \left |\partial \LA_g^{\trop,\BB} \right| =  \left |\LA_g^{\circ, \trop} \right| \ .\] 
The topological space $\left |\LA_g^{\trop, \BB} \right|$ is technically the space obtained by gluing together  `wonderful' compactifications  \cite{Wonderful} of perfect cones $\sigma_Q$ by   blowing up  specific boundary strata which lie at  infinity.  It may be constructed  informally as follows.

Consider the projective space $\Pro(\Quad(V))$ whose points are projective classes of quadratic forms $\Quad(V)$ in a  vector space $V$ of dimension $g$ over $\Q$. The vanishing of the  determinant defines a hypersurface  $\Det \subset \Pro(\Quad(V))$ whose complement satisfies:
\[ L\mathcal{P}_g \  \subset  \    \left( \Pro(\Quad(V))  \  \setminus  \ \Det \right) (\R) \  .  \]
Now consider the space obtained by blowing up the subspaces $\Pro(\Quad(V/K))$ whose points are quadratic forms with kernel $K$, for all rational subspaces $0 \neq K\subset V$, in increasing order of dimension. Since  $\Pro(\Quad(V/K))$ is contained in the determinant locus $\Det$,  every blow-up adds a new boundary component at infinity. Such a boundary component is  indexed by  a nested sequence $0< K_1 <K_2< \ldots < K_n < V$, and is isomorphic to a product of projective spaces. The closure of   $L\mathcal{P}_g$ in the iterated blow-up  admits an action by $\GL_g(\Z)$, whose quotient is the bordification 
 $\left |\LA_g^{\trop, \BB} \right|$. We show in an appendix that it is homeomorphic to  the Borel-Serre compactification \cite{BorelSerre}. 
However,  from this construction it is not obvious that the quotient  $\left |\LA_g^{\trop, \BB} \right|$  embeds into the real points of a finite diagram of algebraic varieties 
 obtained by blowing up projective spaces along \emph{finitely} many linear subspaces. This crucial point is proved in the main body of the text, and is discussed in   \S\ref{intro: PLCsection}

Theorem \ref{thmintroOmegacginjectsDrLAgtrop} is proven by studying the behaviour of the differential forms \eqref{introomegadeftrace} on the boundary of $  \left |\LA_g^{\trop, \BB} \right|$. The existence of this bordification, and the techniques used to construct it,  have a variety of other applications, which we discuss below.

\subsection{Bordification of the tropical Torelli map.} \label{subsect: introTorelli} 
The main thrust of this  paper  is to provide a  general technique  for constructing bordifications of spaces built out of quotients of  polyhedral cells. 
For example, we construct a bordification 
\[ \left| \LM_g^{\trop ,\BB} \right| \To    \left| \LM_g^{\trop} \right|\]
of the link of the moduli space of tropical curves whose existence was alluded to in \cite{BrSigma}. It is obtained by gluing together the `Feynman polytopes'  associated to stable  graphs and provides a single geometric object  whose cells are the spaces underling the  Feynman motives  of \cite{BEK, Cosmic}. This bordification  is presumably related to the bordification of Culler and Vogtmann's Outer space  which was constructed in  \cite{BestvinaFeighn, BordificationOuterSpace}. Note that our bordification concerns the quotient of unreduced Outer space by $\mathrm{Out}(F_n)$ and allows graphs with bridges; the  latter two papers  study reduced Outer space, which only  involves bridgeless (core) graphs, and work on the universal cover, rather than the quotient by $\mathrm{Out}(F_n)$.

The  bordifications of $\left|\LA^{\trop}_g\right|$ and $\left|\LM^{\trop}_g\right|$ are related as follows. The 
 tropical Torelli map $\lambda: \left| \LM_g^{\trop} \right|\rightarrow \left| \LA_g^{\trop} \right|$  was  studied  in  \cite{Nagnibeda, Baker, CaporasoViviani, MikhalinZharkov, BMV}.
It is non-degenerate on the subspace   $\left| \LM_g^{\red} \right| \subset \left| \LM_g^{\trop} \right|$ indexed by $3$-edge connected graphs. 

\begin{thm}The tropical Torelli map extends to a map  $\lambda^{\BB}$ of bordifications, giving a commutative diagram where the vertical maps are blow-downs: 
 \[
\begin{array}{ccc}
  \left| \LM_g^{\red ,\BB} \right| &   \overset{\lambda^{\BB}}{\To}   & \left| \LA_g^{\trop ,\BB} \right|    \\
  \downarrow &    &  \downarrow \\
 \left| \LM_g^{\red} \right|  &  \overset{\lambda}{\To}&    \left| \LA_g^{\trop } \right|  
\end{array}
\]
\end{thm} 
This  diagram gives relations between the cohomology of the  four spaces and the existence of $\lambda^{\BB}$ answers a question of Vogtmann. Note that
the  cohomology of $ \left| \LM_g^{\trop} \right| $  is  related  to the cohomology of the commutative graph complex $\GC_0$, and  to the top-weight cohomology of the moduli stack $\mathcal{M}_g$  \cite{CGP}. 
The cohomology of the bordification $ \left| \LM_g^{\trop ,\BB} \right| $  is described by a new graph complex  (see \S\ref{sect: NewGraphComplex})  involving nested sequences of graphs, and is   related to the Hopf algebra structure on graph homology. It would be  interesting to study its homology in relation to that of $\GC_0$.

\subsection{Canonical integrals of perfect cones} 
A further consequence of  the properties of canonical forms \eqref{introomegadeftrace} on the bordification of $\left| \LA_g^{\trop}\right|$ is that we may assign transcendental invariants to  polyhedral  cones in the Vorono\"i decomposition.

\begin{thm} \label{thm: introQvol} Let $Q$ be a positive definite quadratic form of rank $g$, and let $\sigma_Q$ be the associated cone.
Let $\omega \in \Omega^d_{\can}$ be any canonical form  \eqref{introomegadeftrace} of degree $d=\dim \sigma_Q$. If $\sigma_Q  $ has rank $g$  then the following  integral is finite: 
\begin{equation}  \label{intro: IQomega} I_Q(\omega)  = \int_{\sigma_Q}  \omega < \infty   \ .
\end{equation} 
The  integrals  $ I_Q(\omega) $ satisfy quadratic  relations arising from Stokes' formula (theorem \ref{thm: Stokes}).
\end{thm} 
The condition that $\sigma_Q$ has rank $g$ means that $\sigma_Q$ meets the interior $L\mathcal{P}_g$ of the link of the space of positive definite matrices (i.e., it does not completely lie at infinity).  In particular, every `perfect' quadratic form of maximal dimension may be assigned a volume, which is non-zero. 
Theorem \ref{thm: introQvol}  provides interesting transcendental invariants of quadratic forms which may provide a new perspective on the  extensive literature on perfect forms in relation to sphere packing problems, which was one of Vorono\"i's original goals.

In the case when the cone $\sigma_Q$ is cographical, i.e., the image under the tropical Torelli map of the cell associated to a graph in $\left|\LM_g^{\red}\right|$, the  integrals \eqref{intro: IQomega} reduce to the canonical integrals $I_G(\omega)$ studied in \cite{BrSigma}. It was proved in \emph{loc. cit.} that these are generalised Feynman integrals of the sort arising in perturbative quantum field theory.   An important slogan from this paper, therefore, is that  the \emph{volumes of cographical cells in the perfect cone decomposition of the space of symmetric matrices are Feynman integrals}. 
The  Borel regulator, in particular,  is a  linear combinations of integrals \eqref{intro: IQomega}.

The arithmetic nature of the period integrals  \eqref{intro: IQomega} is not  known.  However, a  computation  by Borinsky and Schnetz \cite{BorinskySchnetz} implies that the volume of the principle cone for $\GL_5(\Z)$ is a linear combination of  non-trivial multiple zeta values of weight $8$ which involves  $\zeta(3,5)$.

\subsection{Polyhedral cell complexes} \label{intro: PLCsection} 
We finally turn to the main geometric constructions.  Theorem \ref{thmintroOmegacginjectsDrLAgtrop} hinges upon  a de Rham theorem for  topological spaces  obtained by gluing together quotients of polyhedra by finite group actions.  In order to make sense of algebraic differential forms such as  \eqref{introomegadeftrace}  on  these spaces, we embed each polyhedron into an ambient algebraic variety, which are glued together according to the same pattern.

The most basic construction along these lines is a category, which we call $\PLC_k$,  of polyhedral linear configurations over a field $k\subset \R$. Its objects are triples:
$(\Pro(V), L_{\sigma}, \sigma)$,   where $V$ is a finite-dimensional  vector space over $k$,
 $\sigma \subset \Pro(V)(\R)$ is a closed convex polyhedron, and $L_{\sigma}\subset \Pro(V)$ is the union of linear subspaces defined by  the Zariski closure of the boundary $\partial \sigma$.   A morphism  $\phi: (\Pro(V), L_{\sigma}, \sigma)\rightarrow (\Pro(V'), L_{\sigma'}, \sigma')$ in this category is a  linear map  $\phi: \Pro(V) \rightarrow \Pro(V')$ such that $\phi(L_{\sigma}) \subset L_{\sigma'}$ and $\phi(\sigma) \subset \sigma'$. We demand, in addition, that  $\phi: \sigma \rightarrow \sigma'$ be either  the inclusion of a face, or an isomorphism.

We then define a   \emph{linear polyhedral complex}  to be  a space assembled out of such objects (and their quotients by finite linear group actions). 
More precisely, it is given by  
 a functor 
\[ F: \mathcal{D} \To  \PLC_{k} \] 
where $\mathcal{D}$ is equivalent to  a finite diagram category.  It has a topological realisation:
\[ |F| = \varinjlim_{x\in \mathcal{D}}  \sigma(F(x))\ ,\]
which is the topological space obtained by gluing the polyhedra $\sigma(F(x))$ along linear maps.
Examples of linear polyhedral complexes  include the moduli space of tropical curves $\LM^{\trop}_g$ of genus $g$, where $\mathcal{D}$ is a category of stable graphs,  and the moduli space of tropical abelian varieties   $\LA^{\trop}_g$, in which case $\mathcal{D}$ is a category of cones associated to quadratic forms. 
The topological realisation  of a polyhedral linear complex is closely related to similar notions in the literature, including stacky fans \cite{BMV}  and generalised CW-complexes \cite{ACP}.

However, a polyhedral linear complex   has  additional algebraic structure. In particular,  there is a 
functor $\PF F : \mathcal{D} \rightarrow \mathrm{Sch}_k$ to the category of schemes which picks out the first component in each triple
$x \mapsto \Pro(V_{F(x)})$. We can speak of a 
 subscheme of $F$: it is simply  a subfunctor  $\mathcal{X} : \mathcal{D} \rightarrow \mathrm{Sch}_k$ of $\PF F$. Concretely, it is given by the data of a compatible family of subschemes  $\mathcal{X}_x \subset \Pro(V_{F(x)})$ for every object $x$ in $\mathcal{D}$.  A global algebraic differential form on a polyhedral linear complex $F$ with poles along $\mathcal{X}$ is  simply an element of 
 \begin{equation} \label{intro: OmegaFminusX}   \Omega^{\bullet}(F \backslash \mathcal{X}) =  \varprojlim_{x\in \mathcal{D}} \,  \Omega^{\bullet} \left( \Pro(V_{F(x)})\backslash \mathcal{X}_x \right)    \ . 
 \end{equation}
It is a compatible system of differential forms on   $\Pro(V_{F(x)})$  with poles along $\mathcal{X}_x$ for each $x$.  
 Our main examples are: the subscheme of  $\LM_g^{\trop}$ defined  by the graph hypersurface locus, and the subscheme of $\LA_g^{\trop}$ defined by the determinant locus $\Det$. 

 \begin{figure}[h]\begin{center}
\quad {\includegraphics[width=9cm]{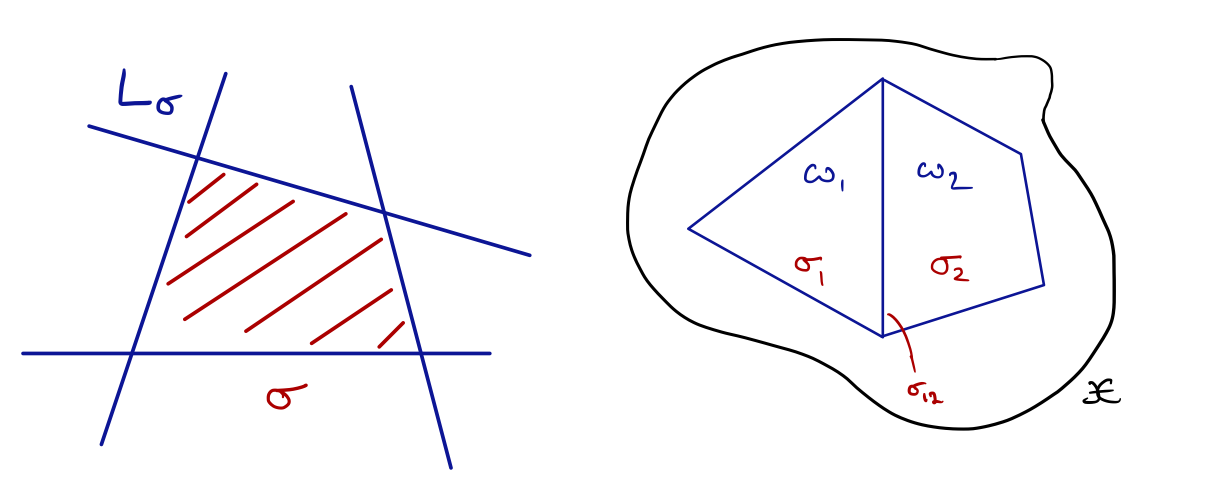}} 
\end{center}
\caption{Left:  a polyhedral linear configuration in $\Pro^2$. Right: two polyhedra $\sigma_1$,  $\sigma_2$ are glued together along the common face $\sigma_{12}=\sigma_1 \cap\sigma_2$.   An algebraic differential form $\omega$ defines two  forms $\omega_i= \omega\big|_{\sigma_i}$, for $i=1, 2$,   which coincide on $\sigma_{12}$. The form $\omega$ has poles along a subscheme $\mathcal{X}$, which, as depicted,  may not necessarily meet the topological realisation  $\sigma_1 \cup_{\sigma_{12}} \sigma_2$.   }
\end{figure}

The above definitions can be generalised. The most important construction, for our purposes, is a category $\BLC_k$ whose objects consist of iterated blow-ups of polyhedral linear configurations along linear subspaces. They are `wonderful' compactifications of linear polyhedra in the sense of \cite{Wonderful}.  
Using this concept,  we can efficiently  define the bordifications mentioned earlier. For instance,  we construct a  functor 
\[   \LA_g^{\trop,\BB}: \mathcal{D}_g^{\perf, \BB} \To \BLC_{\Q}   \] 
from a  suitable diagram category to $\BLC_{\Q}$.  It is defined by blowing up  the subspaces of quadratic forms with a non-trivial rational kernel which meet each  cone in the Vorono\"i decomposition. A key point is that that this space is a diagram of finitely many algebraic varieties with  extra structure; the informal definition of $\left| \LA_g^{\trop, \BB}\right|$ given earlier in this introduction \emph{a priori} involved infinitely many blow-ups.

The next theorem is the key geometric input in the proof of theorems  \ref{thm: introcohomSLGL} and \ref{thmintroOmegacginjectsDrLAgtrop}.

\begin{thm}  \label{thm: introLAgblowupdiagrams} There is a commutative diagram 
\[
\begin{array}{ccc}
\partial  \LA_g^{\trop,\BB}  &  \hookrightarrow   &  \LA_g^{\trop,\BB}  \\
\downarrow_{\pi_{\BB}}  &   &   \downarrow_{\pi_{\BB}} \\
 \partial  \LA_g^{\trop} &   \hookrightarrow &    \LA_g^{\trop}
\end{array}
\]
where the vertical maps $\pi_{\BB}$ are canonical blow-downs. The strict transform of the  determinant subscheme $\Det \subset \LA_g^{\trop}$ defines a subscheme 
\[ \widetilde{\Det} \subset \LA_g^{\trop, \BB}\]
which satisfies $\pi_{\BB} : \widetilde{\Det}  \rightarrow \Det$, and has  the  property 
\begin{equation} \label{intro: Dettildemovesaway} \widetilde{\Det} \cap  \left| \LA_g^{\trop, \BB} \right| = \emptyset\ .
\end{equation} 
On topological realisations, one has a commutative diagram of topological spaces
\[
\begin{array}{ccc}
  &    &   \left| \LA_g^{\trop,\BB}\right| \,  \setminus\,   \left| \partial \LA_g^{\trop,\BB}\right|  \\
   &  \nearrow &   \downarrow \\
\left|\LA_g^{\circ, \trop} \right| &   \rightarrow  &   \left| \LA_g^{\trop}\right| \,  \setminus\,   \left| \partial \LA_g^{\trop}\right| 
\end{array}
\]
where all arrows are isomorphisms.  The open    locus $\left|\LA_g^{\circ, \trop} \right|$ is  homeomorphic to the locally symmetric space $L \mathcal{P}_g/\GL_g(\Z) = \R_{>0}^{\times} \setminus \mathcal{P}_g/\GL_g(\Z) $. 
\end{thm} 
A similar theorem holds for the moduli space $\LM_g^{\trop}$ of tropical curves (\S \ref{sect: strictTransfGraphLocus}).

\subsection{`Motives' of quadratic forms} \label{sect: intromotivesQ} 
Hereafter, a motive refers to an object in a suitable Tannakian $\mathcal{H}_{\Q}$ category of realisations,  following Deligne \cite{DeligneP1}.  An object $M$ in 
$\mathcal{H}_{\Q}$ has Betti and de Rham realisations $M_B$, $M_{dR}$ and  a period pairing $M_B^{\vee}\otimes_{\Q} M_{dR} \rightarrow \C$. 
The property  \eqref{intro: Dettildemovesaway} of the determinant locus enables us to define a canonical  object $\mot_Q$ associated to a  positive definite quadratic form  $Q$  such that the integral \eqref{intro: IQomega}  is a period (definition \ref{def: motQ}).
In the special case when $\sigma_{Q_G}= \sigma_{G}$ is a cographical cone in the image of the tropical Torelli map, $\mot_{Q_G}$ is equivalent to the graph motive defined in  \cite{BEK}. 
 This construction provides, in particular, a motivic interpetation of   the Borel regulator.

\subsection{`Motives' associated to $\GL_g(\Z)$. } Similarly, we may find a motivic interpretation of the canonical cohomology of $\GL_g(\Z)$.  For every $g>1$ and $d\geq 0$ we define a  cohomology group $H^d_c(g)$   consisting of closed  compatible families of algebraic  differential forms with poles along the determinant locus \eqref{intro: OmegaFminusX}, modulo exact forms.
 There is a natural map:
\[  H^d_c(g) \To H_c^d(L \mathcal{P}_g/ \GL_g(\Z);\R) \]
which is injective on $\Omega^d_c(g) \subset H^d_c(g)$: in particular, $H^d_c(g)$ contains all the compactly-supported cohomology classes for $\GL_g(\Z)$ which are described in theorem \ref{thm: introcohomSLGL}.
Integration of differential forms   defines  a pairing 
\[  H^{\lf}_d(\GL_g(\Z);\Q)\otimes_{\Q} H^{d}_{ c}(g)  \To \C\]
where $H^{\lf}$ denotes locally finite (Borel-Moore) homology.
 \begin{thm} \label{introthm: MotiveGLg} There is an  object $\MM^d_g$ of $\mathcal{H}_{\Q}$ equipped with a pair of canonical  linear maps
\begin{eqnarray}
H^{\lf}_d(\GL_g(\Z);\Q)   & \To &   ( \MM^d_g)_B^{\vee} \nonumber \\ 
H^d_c(g)   & \To &   (\MM^d_g)_{dR} \nonumber 
\end{eqnarray} 
such that the integration pairing
factors through the period pairing:
$(\MM^d_g)_B^{\vee}  \otimes_{\Q}  (\MM^d_g)_{dR} \rightarrow \C$. 
\end{thm} 
This theorem  implies that the part  of the locally finite homology $H^{\lf}_d(\GL_g(\Z);\Q) $ which pairs non-trivially with $H^d_c(g)$ is motivic, since 
$\MM^d_g$ lies in the Tannakian subcategory of $\mathcal{H}_{\Q}$ generated by the cohomology of algebraic varieties over $\Q$.
 In particular, the theorem provides a motivic interpretation of the   integrals of canonical differential forms over homology cycles, including   the volume integrals  over fundamental domains considered by Minkowski. The objects $\MM$ are interesting: for example,  $\MM^5_3$ is a non-trivial extension of $\Q(-3)$ by $\Q(0)$ with period a rational multiple of $ \zeta(3)$.

%One can define  other motives associated to spaces such as  $L\mathcal{P}_g/\GL_g(\Z)$, for example, by allowing differential forms with poles along other loci.  For example, 
%using a theorem of Sullivan one may define a weight zero Tate motive which captures the entire cohomology of $\GL_g(\Z)$. Its periods are only rational numbers, however. 

\subsection{Borel-Serre compactification in algebraic geometry} 
In an appendix we study the Borel-Serre compactifiction of $L\mathcal{P}_g/\GL_n(\Z)$, which is defined topologically, and prove that it is homeomorphic to
 $|\LA^{\trop,\BB}_g|$, which is constructed algebraically.
 
\subsection{Recent progress and discussion of unstable cohomology} \label{sect: RecentProgress}
It was proven by Bismut and Lott  \cite{BismutLott} that for $g>1$ odd,  the class of  $\omega^{2g-1}$ in $H^{2g-1}(\GL_g(\Z))$ vanishes.   
Conjecture 5.3 in \cite{MoritaSakasaiSuzuki} states that the class of $\omega^{2g-1}$ also vanishes  in $H^{2g-1}(\GL_{g+1}(\Z))$. 

Since the first version of this paper became available, a flurry  of preprints  appeared which have enormously   advanced our understanding of the unstable cohomology of the general and special linear groups. 
Firstly, Ash  \cite{AshUnstable}  independently constructed two infinite families of unstable classes  in the cohomology of $\SL_n(\Z)$, for $n=3k+3$ and $n=3k+4$.  When $k$ is even, they lie in the same degrees as the  duals of the  classes we denote by  $[\omega_c^{2g-1}]  \in  H_c^{2g}(\GL_g(\Z))$ and their inflations in $ H_c^{2g+1}(\GL_{g+1}(\Z))$ (see theorem \ref{thm: introcohomSLGL}, (ii), (iii)). For $k$ odd, they  lie in different degrees from the classes constructed in this paper. Ash suggests that they align with classes announced by Lee \cite{LeeUnstable}  (more on which below).

At around the same time, \cite{AMP} and \cite{BCGP}  independently showed  that the compactly supported cohomology of the general linear group admits a Hopf algebra structure. By composing the classes constructed  in theorem    \ref{thm: introcohomSLGL}, one can not only prove  all the claims of Lee, but 
also deduce the existence of a plethora of new and unexpected classes. See \emph{loc. cit.} for further details. In particular, all boxed classes in our previous figures can now be explained in terms of canonical forms using these new results. 
For example,  the class in $H^9(\SL_6(\Z);\R)$ in figure \ref{fig:SLn}, which corresponds to a class in $H^{12}_c(\GL_6(\Z);\R)$, is denoted by $[\omega^5|\omega^5]$ in Lee's notation. The class in $H^{15}(\SL_7(\Z);\R)$ corresponds to a class in $H^{13}_c(\GL_7(\Z);\R)$ which is  related to it by the inflation map of \cite{TopWeightAg}.  The paper \cite{OddGraph}, which builds on \cite{AMP} and the methods developed here, establishes infinitely many new  classes on $\GL_{2n}(\Z)$. In particular, the last remaining boxed class in $H^8(\GL_6(\Z))$ (equivalently $H^8(\SL_6(\Z))$, see figure \ref{fig:SLn}) is part of a large infinite family.  We refer the interested reader to the  survey \cite{Bruck} for futher context.

In a different direction, cuspidal classes for $\GL_n(\Z)$ were constructed recently in \cite{boxer-calegari-gee-cuspidal}. 
%It is an open question whether  the  boxed class  in $H^8(\SL_6(\Z);\R)$,   which  comes from  a class in $H^8(\GL_6(\Z);\R)$, pairs non-trivially  
 %the image of the second Morita class in $H_8(\mathrm{Out}(F_6))$ under the Jacobian map
  %$H_8(\GL_6(\Z);\R)$, and   is possibly  proportional to.  It is not  known if this image is zero or not. 

Finally, a recent preprint of Vogtmann \cite{Vogtmann2024} revisits  the bordification of outer space and  asks the question of bordifying the tropical Torelli map, which is answered in \S\ref{section: Torelli}.

\subsection{Plan and further comments}
Section \S\ref{section: PolyhedralCellComplexes} introduces a general notion of  polyhedral cell complexes in algebraic varieties, for which we prove
a cellular homology and de Rham cohomology theorem in  \S\ref{section: HomologyCohomology}. In section \S\ref{sect: PLC} we study the particular subcategory of \emph{linear} polyhedral cell complexes $\PLC_k$, and their iterated blow-ups   in \S\ref{sect: Blowups}.

The first application of this theory is to the moduli space of tropical curves \S\ref{section: LMg}, and its bordification \S\ref{section: MgBB}-\ref{sect: strictTransfGraphLocus}. 
The discussion of the moduli space of tropical abelian varieties begins in \S\ref{section: PolyhedraQuadForms} with a study of polyhedra in spaces of quadratic forms and their blow-ups. In section \ref{section: PerfectConeCompact} we study  $\LA_g^{\trop}$ and define its perfect cone bordification  $\LA_g^{\trop,\BB}$. 
Section \ref{section:  DetLocus} studies the properties of the determinant locus and its strict transform. In section \ref{section: Torelli}, we construct the bordification of the tropical Torelli map.

From section \ref{section: CanForms} onwards, we study the properties of canonical forms and their integrals,  and  in \S \ref{section: CohomologyClasses} we   prove the main  results on the cohomology of  $\SL_g(\Z)$ and $\GL_g(\Z)$ mentioned in this introduction. Finally, \S\ref{section: Periods} discusses the periods and motives associated to canonical integrals and in an appendix we discuss the relation between the Borel-Serre compactification and the space $\left| \LA^{\trop,\BB}_g\right|$ defined algebraically using blow-ups.
We expect that the methods of this paper may be used to 
define and study  the geometric  spaces, differential forms, and  cohomology classes associated to other types of graph complexes (see, e.g., \cite{ModuliColoured}), as well as the  quotients of symmetric spaces by general linear groups over number fields.

\subsection{Acknowledgements} This project has received funding from the European Research Council (ERC) under the European Union's Horizon 2020 research and innovation programme (grant agreement no. 724638). The author  thanks Trinity College, Dublin for a Simons visiting Professorship during which much of this work was carried out and the University of Geneva, where it was completed. Very many thanks are owed to Chan, Galatius, Grushevsky and Payne for extensive discussions on $\mathcal{A}_g$ which motivated this work. 
The author benefited from discussions with Berghoff, Dupont, Grobner and Vogtmann, whose online notes on the Borel-Serre construction were most useful. 
Many thanks to Chris Eur and  Shiyue Li for comments and corrections. The author is especially grateful to Melody Chan and Raluca Vlad for valuable feedback and corrections.

\section{Algebraic polyhedral cell complexes} \label{section: PolyhedralCellComplexes} 
We describe a   formalism to construct algebraic models  of certain  topological spaces defined by gluing together polyhedral cells
  according to a diagram category. 
In the first instance, the cells will be objects of a very general  category $\PC_k$, but for the applications we shall  work with more restrictive subcategories  of convex linear  polyhedra in projective space ($\PLC_k$),  and of   their blow-ups along linear spaces ($\BLC_k$).
%There are many variants of the general idea, but in all cases, the  principles are  similar. 

\subsection{A category of polyhedral cells}
Let $k \subset \C$ be a field.  There are many situations in which one has a polyhedron embedded in the real or complex points of an algebraic variety. Since convexity does not make sense in this generality, one must define a  polyhedron in an algebraic  variety using different concepts. 
 To this end, let us denote by $\PC_k$ the category whose objects are triples
$ (P, L, \sigma) $
defined recursively as follows: 
\begin{itemize} \setlength\itemsep{0.03in}
\item  $P$ is a smooth scheme over $k$. 
\item $\sigma \subset P(\C)$  is  homeomorphic to a closed  real ball  $B_n\subset \R^n$ of dimension $n \geq 0$, i.e., there is a continuous map $g: B_n \hookrightarrow P(\C)$ such that $g: B_n \cong \sigma$. We call $n$ the dimension of $(P,L,\sigma)$. It may be strictly smaller than the dimension of $P$.

\item  The boundary of $\sigma $ satisfies  $\partial \sigma = \sigma \cap L(\C)$, 
where     $L \subset P$ is a subscheme with  finitely many distinct smooth irreducible components   $ L_i$, $i \in I$.
\item   Let $n\geq 1$. For every $i\in I$ such that  $L_i(\C) \cap\sigma$ is non-empty,  the triple
\begin{equation} \label{PC:face}  \left( L_i,   \bigcup_{L_j\neq L_i} L_i \cap  L_j , \sigma \cap L_i(\C)\right)
\end{equation}
is required to be an object in $\PC_k$ of dimension $<n$. 
\end{itemize} 
The objects in $\PC_k$ of dimension $0$ are triples $(P, L, \sigma)$ where $\sigma \subset L(\C)$ is  a point.   For any object $(P,L,\sigma)$, one has $\partial \sigma= \bigcup_{i \in I} \sigma_i$, where $\sigma_i = \sigma\cap L_i(\C)$ is of smaller dimension. Therefore by repeatedly taking boundaries, one obtains a stratification on  $\sigma$  giving a   structure of a regular CW-complex on the closed ball $B_n$.  Note that the   third condition $\partial \sigma = \sigma \cap L(\C)$ captures a notion of convexity (it fails  for non-convex Euclidean  polyhedra).  
 A morphism:
\[  \phi: (P, L, \sigma) \rightarrow (P',L',\sigma')\]
is given by a morphism $\phi: P \rightarrow P'$ such that 
$\phi(L) \subset L'$ and $\phi(\sigma) \subset \sigma'$.
  For any subset $J \subset I$, let $L_J = \bigcap_{j\in J} L_j$.  A \emph{face map} is  an inclusion  of a \emph{face}
\begin{equation} \label{iotafacemap}   \iota_{J} :    
\left(L_J,  \bigcup_{j\in I\setminus J} L_{J  \cup  \{j\}}, \sigma \cap L_J(\C)  \right)  \To (P, L,\sigma) 
\end{equation} 
where $\sigma \cap L_J(\C)   \neq \emptyset$. 
A \emph{facet} of $(P,L,\sigma)$ is a face
\eqref{PC:face} of dimension $\dim \sigma-1$.

\begin{example} \label{ex:simplex}  Let $P = \Pro^n_{\Q}$ denote projective space of dimension $n$ with homogeneous  coordinates $(x_0:\ldots : x_n)$. 
The \emph{algebraic simplex} is the triple $( \Pro^n_{\Q}, L, \sigma)$, where $L= V(x_0\cdots x_n)$ is the union of coordinate hyperplanes,   and $\sigma \subset   \Pro^n_{\Q}(\R)$ is  the standard coordinate simplex defined by  the region $x_i \geq 0$. Its faces are algebraic simplices of smaller dimension. 
\end{example}

\begin{notn} \label{notation: products} Define the product  $\prod_{i=1}^n (P_i, L_i, \sigma_i)$ of  objects $(P_i, L_i, \sigma_i)$ in $
\PC_k$ by 
\[\left( P_1 \times \ldots \times P_n \ , \  \bigcup_{i=1}^n P_1 \times \ldots \times P_{i-1} \times L_i \times P_{i+1} \times \ldots \times P_n \ , \  \sigma_1 \times \ldots \times \sigma_n\right)\ .\]
Denote a product of morphisms $f_1,\ldots, f_n$ in $\PC_k$  by $f_1\times \ldots \times f_n$.
\end{notn}

We shall also add to the collection of morphisms in $\PC_k$ the following  projection morphisms, defined for any ordered finite sets $I\supset J$:
\[   \prod_{i\in I} (P_i, L_i, \sigma_i) \rightarrow \prod_{j\in J} (P_j, L_j, \sigma_j) \ . \]

There are two natural functors, both of which preserve products: 
\begin{equation} \begin{array}{ccccccc} \label{defn: RealFunctors}
\sigma :   \PC_k  &\To & \Top & \qquad  \hbox{ and }   \qquad &   \PF: \PC_k  &\To& \mathrm{Sch}_k \\
 \sigma (P,L,\sigma) &  = &  \sigma  && \PF (P,L,\sigma) & =  & P
\end{array} \end{equation}
to the category of topological spaces, and schemes over $k$, respectively.

\subsection{Algebraic polyhedral cell complexes}
 Let $\mathcal{C}_k$ denote a  subcategory of $\PC_k$. 
The main cases of interest are $\mathcal{C}_{k} = \PLC_k$ (\S\ref{sect: PLC})  and 
 $\mathcal{C}_k= \BLC_k$ (\S \ref{sect: Blowups}).

\begin{defn} \label{defn: GenCkComplex}  A   $\mathcal{C}_k$-\emph{complex}   is a functor 
\begin{equation}  \label{PCcomplex} F: \mathcal{D} \To \mathcal{C}_k\end{equation} 
where $\mathcal{D}$ is equivalent to a finite  category. We write $F_x$ for $F(x)$ when $x$ is an object of $\mathcal{D}$.

 A \emph{morphism} $(\phi,\Phi)$ between two functors
$F: \mathcal{D} \rightarrow \mathcal{C}_k $  and  $ F': \mathcal{D'} \rightarrow \mathcal{C}_k$
   is the data of:
  
   \emph{(i)}      a functor $\phi: \mathcal{D} \rightarrow \mathcal{D'}$ and
   
   \emph{(ii)} 
    a natural transformation $\Phi:  F \rightarrow F' \circ \phi$. 
    \end{defn} 
    
   To spell this out, $\Phi$ is the data, for every object  $x$  of $\mathcal{D}$, of a morphism
   \[ \Phi_x :  F_x \To  (F' \circ \phi)_x \qquad  \hbox{ in } 
   \mathcal{C}_k\]
   such that, for every morphism $f: x \rightarrow y$ in $\mathcal{D}$, there is a commutative diagram 
   \begin{eqnarray} 
   F_x \quad & \overset{\Phi_x}{\To} &  (F' \circ \phi)_x \qquad \nonumber \\
   \downarrow_{F(f)}  & & \quad \,\,\, \downarrow_{(F' \circ \phi)(f)}  \nonumber \\
   F_y \quad & \overset{\Phi_y}{\To} & (F' \circ \phi)_y   \qquad \ .\nonumber 
   \end{eqnarray}  
   In particular,  any  functor   $\phi: \mathcal{D}\rightarrow \mathcal{D}'$ induces a morphism $(\phi, \id)$  between   $F=F'\circ \phi$ and $F'$. If $\phi$ is an equivalence, then $F$ and $F'$ are isomorphic. In particular, by replacing $\mathcal{D}$ with an equivalent category,  we may  assume that $\mathcal{D}$ is itself finite.

    \begin{defn}
    The \emph{topological realisation} of  \eqref{PCcomplex} is the topological space
    \begin{equation} \label{def: topreal}  |F| =  \varinjlim_{x\in \mathrm{Ob}(\mathcal{D})}  \sigma(F_x)  \  . \end{equation}
          By taking limits, a morphism $(\Phi,\phi)$ from $F$ to $F'$  induces a continuous map between their topological realisations. We denote it by 
$|\Phi|:  | F| \To |F'|.$
 \end{defn}

\subsection{Subschemes}  \label{sect: subschemes}   Let $F: \mathcal{D} \rightarrow \PC_k$ be a functor \eqref{PCcomplex} as above.
  \begin{defn} \label{defn: subscheme} 
 Define a closed (resp. open) \emph{subscheme} of $F$  to be a functor
 \[ \mathcal{X} : \mathcal{D} \To \mathrm{Sch}_k\] 
 such that $\mathcal{X}_x$ is a closed (resp. open) subscheme of $\PF F_x$, for all objects $x$ of $\mathcal{D}$, and such that the canonical embedding $i_x: \mathcal{X}_x \subset \PF F_x$
 is a natural transformation $i: \mathcal{X}\rightarrow \PF  F$.  If  $K= \R$ or $\C$, and  contains $k\subset K$, then the set of  $K$-points defines a topological space
 \[ |\mathcal{X}(K)| = \varinjlim_{x\in \mathcal{D}}  \mathcal{X}_x(K)\]
 with the analytic topology. 
   \end{defn}
 
 Definition \ref{defn: subscheme} 
 is analogous to that of a subfunctor. It 
  means that for all morphisms $f: x \rightarrow y$ in $\mathcal{D}$ there is a commutative diagram
  \[
  \begin{array}{ccc}
 \PF F_x  &  \overset{F(f)}{\To}   &\PF F_y \\
 \rotatebox[origin=c]{90}{$\subseteq$} &   &  \rotatebox[origin=c]{90}{$\subseteq$}   \\
 \mathcal{X}_x  &   \overset{X(f)}{\To}  &   \mathcal{X}_y
\end{array}
\]
and hence the morphisms between $\mathcal{X}_x$ are obtained by restricting those from $\PF F$. 
 
\begin{defn}  \label{defn: subschemeatinfinity} We say that a closed subscheme  $\mathcal{Z} $ of $F$ is \emph{at infinity}, 
which we denote by  $|F| \cap \mathcal{Z} = \emptyset$,  if its image does not meet any of the polyhedra  $\sigma(F_x)$:
\[\sigma(F_x) \cap i_x(\mathcal{Z}_x)(\C) = \emptyset  \ \hbox{ for all } \  x\in \mathrm{Ob}(\mathcal{D}) . \] 
 Similarly, for an open subscheme  $U $ of $F$  we write $|F| \subset U$ if
  \[\sigma(F_x) \ \subset \  U_x(\C)   \ \hbox{ for all } \  x\in  \mathrm{Ob}(\mathcal{D}) . \] 
\end{defn}

 \begin{defn} \label{defn: OpenSubComplex} 
 Given an open subscheme $\mathcal{U}$ of $ F$ such that $|F| \subset \mathcal{U}$, define \begin{eqnarray} F \cap  \mathcal{U}: \mathcal{D}  &\To & \PC_k    \\
  x & \mapsto &( F\cap \mathcal{U})_x =  (\mathcal{U}_x,  L_x \cap U , \sigma) \ , \nonumber
  \end{eqnarray} 
  which we may view as an (open) subfunctor of $F:\mathcal{D}  \rightarrow  \PC_k$.
  \end{defn} 
  
\subsection{Algebraic and meromorphic differential forms}  \label{sect: differentialforms}
We define various notions of differential forms on a $\mathcal{C}_k$-complex. The most important will be  the notion of  a smooth form \eqref{DGAdefnA}, which we can use  to compute cohomology. Since the cohomology classes we construct are in fact algebraic with specific polar locus, we also need to define notions  of algebraic forms with poles along a subscheme, and meromorphic forms with unspecified poles.

\begin{defn} \label{defn: GlobalForm} Consider a $\mathcal{C}_k$-complex $F: \mathcal{D} \rightarrow \mathcal{C}_k$ and a subscheme $\mathcal{X} \subset \PF F$.
A  global differential form of degree $d$ on $F$ with poles along $\mathcal{X}$  is an element  of the limit 
\[  \Omega^d(\PF F \backslash \mathcal{X}) = \varprojlim_{x\in \mathrm{Ob}(\mathcal{D})}  \Omega^d (\PF F_x \backslash \mathcal{X}_x )\ . \] 
Equivalently, it is a collection, for every $x\in \mathrm{Ob}(\mathcal{D})$, of  regular forms $\omega_x \in \Omega^d(\PF F_x \backslash \mathcal{X}_x)$  which are compatible in the sense  that 
\begin{equation} \label{CompatibilityForForms} 
F(f)^*( \omega_y )= \omega_x \quad \hbox{ for every }  \quad f \in \mathrm{Hom}_{\mathcal{D}}(x,y) \ .
\end{equation} 
\end{defn} 

The graded vector space $\Omega^{\bullet}( \PF F \backslash \mathcal{X} ) = \bigoplus_d \Omega^d(\PF F \backslash \mathcal{X})$ is a differential graded algebra.
We may also write it  $\Omega^{\bullet}( F \backslash \mathcal{X} )$, bearing in mind that it depends only on the functor $\PF F$. 

  Consider an object $(P, L, \sigma)$ of $\PC_k$. Let us denote by 
\[ \Omega_{\mer}^{d} \left( (P, L, \sigma) \right)  = \varinjlim_{U: \sigma \subset U(\C)}  \Omega^d(U;k)  \]
the space of meromorphic differential forms which are regular on an open affine subset   $U$ of $P$ whose complex points contain $\sigma$.
Such a form may be restricted to the faces of $\sigma$, and so $ \Omega_{\mer}^{\bullet} = \bigoplus_{d \geq 0} \Omega_{\mer}^d$ is a contravariant functor from $\PC_k$ to  the category of DGA's.

\begin{defn} \label{defn: MeroForm} Consider a $\mathcal{C}_k$-complex $F: \mathcal{D} \rightarrow \mathcal{C}_k$. 
A meromorphic differential form of degree $d$ on $F$ is an element  of the projective limit 
\[  \Omega_{\mer}^d(F) = \varprojlim_{x\in \mathrm{Ob}(\mathcal{D})}  \Omega_{\mer}^d (F_x)\ . \] 
It is a   compatible  collection  of  meromorphic forms $\omega_x \in \Omega^d_{\mer}(F_x)$ for $x\in \mathrm{Ob}(\mathcal{D})$.
The total space $\Omega_{\mer}^{\bullet} (F) = \bigoplus_{d\geq 0} \Omega_{\mer}^d(F)$ is a differential graded algebra. 
\end{defn} 

\begin{ex}
The DGA of meromorphic forms $\Omega_{\mer}^{d} \left( \Pro^n_{\Q}, L, \sigma\right)$  on an algebraic  simplex (Example \ref{ex:simplex})  contains the polynomial forms on $\sigma$ in the sense of Sullivan \cite{Sullivan}.  
\end{ex}

\begin{defn}
The $\C$-differential graded algebra of smooth forms on $|F|$ is defined to be
   \begin{equation} \label{DGAdefnA}  \mathcal{A}^{\bullet}(|F|) =\varprojlim_{x\in \mathrm{Ob}(\mathcal{D})}   \mathcal{A}^{\bullet}(\sigma(F_x))  \ ,
   \end{equation}  
   where $ \mathcal{A}^{\bullet}(\sigma(F_x))$ denotes the algebra of smooth differential forms over $\C$ which are defined in an open neighbourhood of 
   $\sigma(F_x)$ inside $\PF F_x(\C)$.  When   $k\subset \R$, and all  polyhedra $\sigma(F_x)$ are in fact contained in $\PF F_x(\R)$ (which will always be the case in our applications) then 
 $\mathcal{A}^{\bullet}(|F|)$  has a real structure   $ \mathcal{A}^{\bullet}(|F|;\R)$ consisting of the $\R$-subalgebra of real forms.
 \end{defn}
   
If $\mathcal{X}$ is at infinity, i.e.,  $\mathcal{X} \cap |F|= \emptyset$, then  there are natural maps of DGA's:   
\begin{equation}  \Omega^{\bullet}  (\PF F \backslash \mathcal{X} ) \ \subset  \  \Omega_{\mer}^{\bullet} (F) \To \mathcal{A}^{\bullet}(|F|) \ . \end{equation} 

\begin{defn} Define smooth (resp. global algebraic) cohomology groups:
\begin{equation}
H^n_{dR}(|F|;\C) = H^n ( \mathcal{A}^{\bullet}(|F|))  \quad (\hbox{resp. } \quad   H^n_{dR}(\PF F \backslash \mathcal{X}) = H_{dR}^n ( \Omega^{\bullet}(F \backslash \mathcal{X})))  \ . 
\end{equation} 
The former is a vector space over $\C$ isomorphic to the cohomology of $|F|$ as we shall show;  the latter a vector space over $k$, which  does not equal the cohomology of $|F|$ in general.  

If $\mathcal{X} \cap |F|= \emptyset$,  there is a natural map $H^n_{dR}(\PF F \backslash \mathcal{X}) \rightarrow H^n_{dR}(|F|) $. 
\end{defn}

\section{Homology and cohomology of  polyhedral complexes} \label{section: HomologyCohomology}

\subsection{Assumptions.} \label{sect: PolyhedralAssumptions} 
Let $\mathcal{C}_k$ be a subcategory of $\PC_k$ such  that:

\begin{enumerate}
\item   Every face of  every object $(P,L,\sigma)$ in  $\mathcal{C}_k$  is also an object of  $\mathcal{C}_k$, and the corresponding face map is a morphism in $\mathcal{C}_k$. 
\item All  morphisms $f: (P,L,\sigma) \rightarrow (P',L',\sigma')$ in  $\mathcal{C}_k$ are  either face maps, or induce homeomorphisms on topological realisations $f: \sigma \cong \sigma'.$   
\end{enumerate}

 In particular,  the topological realisation $f: \sigma \rightarrow \sigma'$ of any morphism in $\mathcal{C}_k$  is necessarily injective.  For every  $(P,L,\sigma)$ of dimension $n \geq 1$  one has 
\begin{equation}   \label{sigmaSphericalHomology} H_r(\sigma, \partial  \sigma;\Z) = \begin{cases} \Z  \quad  \hbox{ if }   r= n   \\
0 \quad  \hbox{ otherwise } \end{cases} \ . \end{equation}

An  \emph{orientation} on $\sigma$  is  a generator 
$\varpi \in H_n(\sigma, \partial  \sigma;\Z)$ if $n\geq 1$ or $\varpi \in H_0(\sigma;\Z)$ if $n=0$.
The two main categories of interest, $\PLC_k$ and $\BLC_k$,  satisfy $(1)$ and $(2)$.   An important  fact is that an isomorphism between objects of  $\mathcal{C}_k$ induces  isomorphisms of their faces.

\subsection{Cellular homology of a polyhedral $\mathcal{C}_k$-complex}  \label{sect: CellHomologyFaceComplex} 
Let $\mathcal{C}_k$ be a category satisfying the assumptions above and  let  $F: \mathcal{D} \rightarrow \mathcal{C}_k$ be a $\mathcal{C}_k$-complex.

\begin{defn} \label{defn: facecomplex} 
Define the  \emph{face complex} $ \left(\mathfrak{C}_F\right)_n$ of $F$ to be the graded $\Q$-vector space with generators $ [\sigma, \varpi]$,
where $\sigma$  is any face of dimension $n $ of the topological realisation $\sigma F_x$ of an  object $x$ of $\mathcal{D}$, and  $\varpi$ is an orientation on $\sigma$. These symbols are subject to relations:
\begin{eqnarray} 
&(i)& \qquad  [\sigma,  \lambda \varpi] = \lambda [\sigma, \varpi]  \qquad \hbox{ for any } \lambda \in \Z \ , \nonumber \\ 
& (ii)&\qquad \,\,\,\, [ \sigma, \varpi  ] = [\sigma', \varpi'] \nonumber
\end{eqnarray}
whenever $\sigma$, $\sigma'$ are $n$-dimensional faces of  $\sigma(F_x)$, $\sigma(F_y)$ respectively,  and  $g:x \rightarrow y$ is a morphism in  $\mathcal{D}$  which restricts to an isomorphism 
$ \sigma F(g) :    \sigma \overset{\sim}{\rightarrow}\sigma'$ which sends $\varpi$ to $\varpi'$.
\end{defn}

Define a  differential $d:   \left(\mathfrak{C}_F\right)_n \rightarrow  \left(\mathfrak{C}_F\right)_{n-1}$ by 
\begin{equation}  \label{FaceComplexDifferential} 
d [ \sigma  , \varpi] = \sum_{\tau}\,   [\tau, \varpi\big|_\tau]
\end{equation} 
where the sum is over all facets $ \tau$ of $\sigma$,  where $\dim \tau = \dim \sigma-1$,  and $\varpi\big|_\tau$ is the image of $\varpi$ under the boundary map $H_n(\sigma, \partial \sigma) \rightarrow H_{n-1} (\tau,\partial \tau)$.  One checks that $d^2=0$. 
\begin{thm} \label{thm: CellularHomology}   
 There is a natural isomorphism
 $H_n(  \mathfrak{C}_F ) \overset{\sim}{\To}  H_n( |F|;\Q)$.
 \end{thm}

\subsection{Differential forms and de Rham complex}
To show that the complex $\mathcal{A}^n$  of smooth forms (definition \ref{DGAdefnA}) computes cohomology, we only require that:
\begin{equation}   \label{Ancomputescohom} H^{\bullet}(\mathcal{A}^n(\sigma)) \cong H^n(\sigma;\C)  \quad \hbox{ and } \quad  H^{\bullet}(\mathcal{A}^n(\partial\sigma))= H^n( \partial \sigma;\C) 
\end{equation} 
and furthermore, that elements of $\mathcal{A}^n$ are  \emph{extendable}  in the sense of \cite[\S7]{Sullivan}:
\begin{equation} \label{Extendability} 
  i^* : \mathcal{A}^{\bullet} (\sigma) \To   \mathcal{A}^{\bullet}(\partial\sigma) 
 \quad \hbox{  is surjective, }  \end{equation}
where   $i: \partial \sigma \subset \sigma$ denotes the inclusion map.  The following theorem is in fact valid for any complex of differential forms which satisfies these properties.
   
   \begin{thm} \label{thm: CellularDeRhamCohomology}
Let $\mathcal{A}^n$ satisfy \eqref{Ancomputescohom}  and \eqref{Extendability} as above. Let $F: \mathcal{D} \rightarrow \mathcal{C}_k$ be a $\mathcal{C}_k$-complex, and define the de Rham complex of $F$ via \eqref{DGAdefnA}.
There is an isomorphism
\[ H^{n} \left( \mathcal{A}^\bullet(\left| F \right|) \right)  \overset{\sim}{\To} H^n\left(|F|;\C\right)\  , \]
where $H^n \left(|F|;\C\right)=H_n(|F|;\C)^{\vee}$,  
which is induced by a bilinear  pairing 
\begin{eqnarray}   \label{perfectpairing} H^{n} \left( \mathcal{A}^\bullet(| F |) \right) \otimes H_n(|F|)  &\To&  \C \\
\omega \otimes \gamma & \mapsto & \int_{\gamma} \omega\ . \nonumber    
\end{eqnarray} 
By  theorem \ref{thm: CellularHomology}, it may be interpreted as a pairing 
$H^{n} \left( \mathcal{A}^\bullet(|F|) \right) \otimes H_n(\mathfrak{C}(F))  \rightarrow \C.$ 
   \end{thm}  
   Before proceeding with the proof,  note that the integral \eqref{perfectpairing} makes sense because a differential form $\omega \in  \mathcal{A}^\bullet(F) $ gives rise to a well-defined smooth differential form on each simplex of $|F|$, because of the compatibility condition \eqref{CompatibilityForForms}. The integral converges because  the simplices  in $|F|$ 
  are compact, and have  finitely many isomorphism classes.

     \begin{rem}  \label{rem: AllFormsPolynomial} Sullivan defined a differential graded algebra of polynomial differential forms for  simplicial complexes \cite[\S7, (i), p. 297]{Sullivan},  and proved that  it satisfies the extendability condition \eqref{Extendability}. Since they  are special cases of meromorphic forms, his argument  implies that the induced map
 $H^n(\Omega^{\bullet}_{\mer}(F)) \rightarrow H^n(  \mathcal{A}^\bullet( | F| ))$
     is surjective. 
    \end{rem}

Even though the complex $ \Omega_{\mer}^{\bullet} (F)  $, and hence its cohomology, has a $k$-structure, the map $H^n(\Omega^{\bullet}_{\mer}(F))\otimes_k \C \rightarrow H^n(  \mathcal{A}^\bullet(  |F| ))$  is not an isomorphism (the former is  infinite-dimensional) and so cannot be used to define a rational structure of finite type on the de Rham cohomology $H^{n} \left( \mathcal{A}^\bullet(|F|) \right)$. In other words,  the periods one obtains by integration \eqref{perfectpairing} depend  on the  location of the poles of  $\omega$.

   \subsection{Proof of theorems \ref{thm: CellularHomology} and \ref{thm: CellularDeRhamCohomology}} 
   The proof of theorem \ref{thm: CellularHomology} is standard (compare with \cite[Theorem 2.35]{Hatcher}, \cite{TopWeightAg}). Consider the filtration of $X=|F|$ by subspaces
   \[ X_p =  \mathrm{Im} \left( \coprod_{\dim \sigma_x=p} \sigma_x \To X \right)\]
   where the disjoint union is over all faces  of polyhedral cells of dimension $p$. It induces a filtration $F_p C_{\bullet}(X)= C_{\bullet} (X_p)$ on the singular chain complex, giving rise to a spectral sequence 
   \[ E^1_{p,q} =  H_{p+q}(\gr^F_{p} C_{\bullet}(X)) \cong  H_{p+q} (  X_p, X_{p-1}) \] 
   which converges to $\gr^F_{\bullet} H_{p+q}(X)$.  The complex $(E^1_{p,q},d^1)$ takes the form: 
   \begin{equation} \label{E1complex} \cdots \To  H_{p+q}(X_p, X_{p-1}) \To H_{p+q-1}(X_{p-1}, X_{p-2} ) \To \cdots 
   \end{equation}
   where the maps are induced on the level of chains  by the boundary map.
   By definition of $X_p$ there is a  natural  morphism
     \begin{equation} \label{inproofdirectsumHpsigma}   \bigoplus_{\dim \sigma_x=p}  \left( H_{p}( \sigma_x, \partial \sigma_x) \right) / \mathrm{Aut}_{\mathcal{D}}(\sigma_x)  \To   H_{p}(X_p, X_{p-1}) 
     \end{equation} 
   where the     sum is over all equivalence classes  of faces,  where two faces $\sigma_x, \sigma_y$ are equivalent if they have the same image in $X$, and $\mathrm{Aut}_{\mathcal{D}}(\sigma_x)$ is the subgroup of automorphisms on the face $\sigma_x$ induced by morphisms in the category $\mathcal{D}$.
   The map \eqref{inproofdirectsumHpsigma}
 is surjective   since the intersections $\sigma_x \cap \sigma_y$ between all non-isomorphic faces $\sigma_x,  \sigma_y$ are unions of faces of dimension $\leq p-1$, and contained in $X_{p-1}$.
Since  $H_{p+q}(\sigma_x, \partial \sigma_x)$ is concentrated in degree $q=0$  \eqref{sigmaSphericalHomology} this proves that the  spectral sequence degenerates at $E^1$, and the cohomology of $X$ is isomorphic to the cohomology of the complex \eqref{E1complex} on setting $q=0$. In order to identify this complex with the face complex, 
 one may observe that  
 $H_{p+q}( \sigma_x, \partial \sigma_x) \cong H^{\mathrm{lf}}_{p+q}(  \overset{\circ}{\sigma}_x)$
   is isomorphic to the locally finite (Borel-Moore) homology  of the interior of $\sigma_x$. Since  by assumption on $\mathcal{C}_k$ the interiors of faces are either disjoint or isomorphic, we conclude  that \eqref{inproofdirectsumHpsigma}
 is also injective, and therefore defines an isomorphism 
\[ \mathfrak{C}(F)_p \cong   H_{p}(X_p, X_{p-1})\]
  by definition \ref{defn: facecomplex}. Furthermore,  since the morphisms in  \eqref{E1complex} are induced by the boundary map on relative cohomology, we may identify $(E^1_{p,0}, d^1)$  with  the face complex  \eqref{FaceComplexDifferential}. This completes the proof of Theorem  \ref{thm: CellularHomology}, and shows, in passing, that $d^2=0$.

  Theorem \ref{thm: CellularDeRhamCohomology} is  a de Rham version of theorem \ref{thm: CellularHomology}. We shall prove it in a similar way by replacing the singular chain complex with a complex of differential forms and replacing vector spaces and linear maps with their duals. First we need the following lemma.
     \begin{lem}
Assume \eqref{Ancomputescohom}  and \eqref{Extendability}. Let $K^{\bullet}(\sigma)= \ker (i^*: A^{\bullet}(\sigma) \rightarrow A^{\bullet}(\partial \sigma))$. Then  
\begin{equation}   \label{HnKtoRelCohom} H^n (K^{\bullet}(\sigma)) \overset{\sim}{\To} H^n(\sigma, \partial \sigma) \ . \end{equation} 
\end{lem}
\begin{proof} By assumption, the relative cohomology group $H^n(\sigma, \partial \sigma)$ is the cohomology of  the  mapping cone of $i^*$, which is the complex 
$ \mathcal{A}^n(\sigma) \oplus \mathcal{A}^{n-1} (\partial \sigma)$ with  differential $d (\alpha, \beta) = (d\alpha,   i^* \alpha- d\beta )$.
The map    \eqref{HnKtoRelCohom} is induced by the morphism of complexes: 
\[ 
  \ K^{\bullet}(\sigma)  \overset{\kappa}{\To}  \mathcal{A}^n(\sigma) \oplus \mathcal{A}^{n-1} (\partial \sigma) 
\]
where  $\kappa(\omega) = (\omega ,0)$. 
To see that \eqref{HnKtoRelCohom} is surjective, let  $(\omega, \eta) \in \mathcal{A}^n(\sigma) \oplus \mathcal{A}^{n-1} (\partial \sigma)$ be closed, which implies that  $d\omega=0$ and  $i^*\omega = d\eta$.  By \eqref{Extendability}, there exists $\alpha\in \mathcal{A}^{n-1}(\sigma)$ such that $i^*\alpha= \eta$.  The cohomology class of $(\omega, \eta)$ is also represented by 
$(\omega, \eta) - d(\alpha,0)= (\omega-d\alpha,0) $, which equals $\kappa(\omega - d\alpha)$. Note that  $\omega- d\alpha \in K^{\bullet}(\sigma)$ since $i^*(\omega- d\alpha) = i^*(\omega) - d \eta=0$.   To establish the injectivity of \eqref{HnKtoRelCohom},   suppose that $\kappa(\omega) = d(\alpha, \beta)$ is exact.  This implies that $d\alpha= \omega$, and $i^*\alpha = d\beta. $
By \eqref{Extendability}, there exists $\gamma \in \mathcal{A}^{\bullet}(\sigma)$  such that $i^*\gamma = \beta$, and hence $\omega = d(\alpha-   d\gamma)$ is exact in $K^{\bullet}(\sigma)$, since $i^*(\alpha - d \gamma) = i^*\alpha - d\beta =0$. 
\end{proof}

The proof of theorem \ref{thm: CellularDeRhamCohomology}  proceeds as for theorem \ref{thm: CellularHomology}. The filtration $X_p$ on $X$ induces a cofiltration on the  differential graded algebra $\mathcal{A}^n(|F|)$.  It produces a spectral sequence converging to the cohomology  of  $\mathcal{A}^{\bullet}(|F|)$.  It is enough to show that this spectral sequence is isomorphic, via integration,  to the dual of the homology  spectral sequence considered above. The integration pairing is well-defined on the level of chain complexes because of  Stokes' theorem. 
 The associated graded  of the cofiltration on $\mathcal{A}^n(|F|)$ is 
\[ K^{\bullet}_p (|F|)=  \ker \left(  \varprojlim_{\dim \sigma_x \leq  p} \mathcal{A}^{\bullet}(\sigma_x) \To  \varprojlim_{\dim \sigma_x \leq p-1} \mathcal{A}^{\bullet}(\sigma_x) \right) \]
consisting of compatible systems of differential forms on faces of dimension $p$ which vanish on faces of dimension $p-1$. As previously, one has an isomorphism 
\[  H^{p+q}(K^{\bullet}_p(|F| )) \cong  \bigoplus_{\sigma_x}  H^{p+q} \left(K^{\bullet}_p (\sigma_x)\right) / \mathrm{Aut}_{\mathcal{D}}(\sigma_x)  =   \bigoplus_{\sigma_x}  H^{p+q} \left(\sigma_x, \partial \sigma_x \right) / \mathrm{Aut}_{\mathcal{D}}(\sigma_x) \]
where the direct sum is over equivalence classes of $\sigma_x$, and the second equality follows from  \eqref{HnKtoRelCohom}. The
relative cohomology $ H^{p+q} \left(\sigma_x, \partial \sigma_x \right)$ is isomorphic to the compactly supported cohomology  $H_c^{p+q} ( \overset{\circ}{\sigma}_x)$ and is canonically dual to $H_{p+q}(\sigma_x, \partial \sigma_x)\otimes_{\Z} \C$. This implies that 
$ H^{p+q}(K^{\bullet}_p(|F| ))$ vanishes except in degree $q=0$ and, in this degree, is dual to $\mathfrak{C}(F)_p\otimes \C$ via the integration pairing, which proves
theorem \ref{thm: CellularDeRhamCohomology}.

\subsection{Relative  homology and compact supports}
Let $F, G$ be two $\mathcal{C}_k$-complexes and $i: G\rightarrow F$ a  morphism such that $i:|G| \hookrightarrow |F|$. We may identify $|G|$ with its image $i(|G|$).
Define the DGA of compactly supported forms on the complement by
\[ \mathcal{A}^{\bullet}_c \left( |F|  \ \setminus \  |G| \right)  =  \varprojlim_{x\in \mathcal{D}}  \mathcal{A}_c^{\bullet} \left(\sigma(F_x) \backslash (\sigma(F_x) \cap |G|)\right)\ \]
and  write  $H^{n}_{dR, c} \left( |F|  \ \setminus \  |G| \right)= H^{n} (\mathcal{A}^{\bullet}_c \left( |F|  \ \setminus \  |G| \right)    ). $ Define the relative de Rham cohomology $H^{n}_{dR} \left( |F| ,   |G| \right)$ to be the cohomology of the mapping cone $\mathcal{A}^{\bullet} \left( |F|  \right) \oplus \mathcal{A}^{\bullet-1} \left( |G|  \right)    $ with respect to the differential $d (\omega, \eta) = (d\omega , i^*\omega - d \eta)$.

Relative homology and cohomology satisfy the usual long exact sequences. 

\begin{thm} \label{thm: RelativeCellular} There is an isomorphism $ H_n(\mathfrak{C}_F/  \mathfrak{C}_G) \cong H_n(|F|, |G|;\Q)$. Integration defines a canonical  isomorphism of $\C$-vector spaces: 
\[   H^{n}_{dR} \left( |F| ,   |G| \right) \To   \left(H_n \left( |F| ,   |G|;\Q \right)\right)^{\vee} \otimes\C  \ . \]
The map defined on complexes by $\omega \mapsto (\omega, 0)$ passes to   a canonical isomorphism 
\begin{equation} \label{dRCompactSupportsToRelative} H^n_{dR,c} \left( |F|  \ \setminus \  |G| \right)\overset{\sim}{\To} H^{n}_{dR} \left( |F| ,   |G| \right)  \ .
\end{equation} 
\end{thm}  

\begin{proof} 
The  first part follows formally from theorems \ref{thm: CellularHomology} and  \ref{thm: CellularDeRhamCohomology}. 
To prove \eqref{dRCompactSupportsToRelative}, one may follow  the same strategy as theorem \ref{thm: CellularDeRhamCohomology}: the filtration $X_p$ on $|F|$ gives rise to a cofiltration on   $\mathcal{A}^{\bullet}_c \left( |F|  \ \setminus \  |G| \right) $     and a  spectral sequence whose $E_{pq}^1$ terms are 
\begin{equation} \label{inprooftheorem3.6} 
 \bigoplus_{\sigma_x} H^{p+q}_{dR,c}\left(\sigma_x \backslash (\sigma_x \cap |G|) \right)  /\mathrm{Aut}_{\mathcal{D}} (\sigma_x)  \overset{\sim}{\rightarrow} \bigoplus_{\sigma_x} H_{dR}^{p+q}\left(\sigma_x  , \partial \sigma_x \cup (\sigma_x \cap |G|) \right)  /\mathrm{Aut}_{\mathcal{D}} (\sigma_x)  
 \end{equation} 
where the direct sum is over equivalence classes  $\sigma_x$ of cells of dimension $p$.\footnote{One can directly compare this with  the associated spectral sequence for the cohomology of  the mapping cone  $\mathcal{A}^{\bullet} \left( |F|  \right) \oplus \mathcal{A}^{\bullet-1} \left( |G|  \right)    )$.   One replaces $K^{\bullet}_p(|F|)$ with a  complex consisting of forms which vanish on $X_{p-1} \cup |G|$. Its  cohomology is  $\bigoplus_{\sigma_x} H_{dR}^{p+q} (\sigma_x,\partial \sigma_x \cup (\sigma_x \cap |G|))/\mathrm{Aut}_{\mathcal{D}}(\sigma_x)$.   } 
 Since every morphism in $\mathcal{C}_k$ is either a face map or an isomorphism,  it follows that either $ \sigma_x \cap |G| $ is contained in the boundary of $\sigma_x$, in which case $H_{dR}^{p+q}\left(\sigma_x  , \partial \sigma_x \cup (\sigma_x \cap |G|) \right)= H_{dR}^{p+q}(\sigma_x, \partial \sigma_x)$, or $\sigma_x \subset |G|$, in which case this group  vanishes. It follows that the complex on the  right-hand side of \eqref{inprooftheorem3.6} is dual, via integration,  to  $\mathfrak{C}_F/  \mathfrak{C}_G\otimes_k \C$. 
\end{proof}

\section{Linear polyhedral complexes} \label{sect: PLC}
Let $V$ be a  vector space of dimension $n+1$  over a field $k\subset \R$, and let $\Pro(V)$ denote the associated projective space of dimension $n$. When $V$ has a preferred basis, we  write  $\Pro^n_k$ instead of  $\Pro(k^{n+1})$. 
We  call a   \emph{projective linear configuration}  $L \subset \Pro(V)$     any  finite union   $L = \bigcup_{i\in I}L_i$ of linear spaces $L_i \subset \Pro(V)$, all of which have  equal dimension.  Correspondingly, there are  linear subspaces $W_i\subset V$, such that 
$L_i = \Pro(W_i)$  for all $i\in I$.
 For any subset  $J\subset I$,  we shall write $L_J = \bigcap_{j\in J} L_j$ and $W_J = \bigcap_{j\in J}W_j$.

\subsection{Polyhedral linear configurations} Let us write $V_{\R} =V \otimes_k \R$. 

\begin{defn} 
A  \emph{real polyhedral cone} defined over $V$ is   the convex hull of a finite set of vectors $v_1,\ldots, v_m \in V$, where $m\geq 1$:
\begin{equation} \label{sigmahat}   \widehat{\sigma} = \R_{\geq 0} \langle v_1, \ldots, v_m \rangle  \quad \subset \quad V_{\R}  \ . 
\end{equation} 
Its cone point is the origin. 
  A polyhedral cone is called \emph{strongly convex} if it does not contain any  real line $\R w$, for a  non-zero vector $w\in V_{\R}$.

A   \emph{(projective) polyhedron in $\Pro(V)$} is   a pair $(\sigma, V)$, where $\sigma \subset \Pro(V_{\R})$ is  the  link of the cone point of 
a  strongly convex polyhedral cone \eqref{sigmahat}  defined over $V$:
\[ 
\sigma  =  \left(  \,  \widehat{\sigma}\,  \setminus \,  \{ 0\} \,   \right) / \R^{\times}_{>0}  \ . \] 
\end{defn} 
Given a polyhedron $(\sigma,V)$, we write $V_{\sigma} \subset V$ for the $k$-linear span of its defining vectors $v_1,\ldots, v_m$ \eqref{sigmahat}.  The space $V_{\sigma}$ only depends on $\sigma$, and  indeed, the associated projective space $\Pro(V_\sigma)$  is the Zariski-closure of the $k$-rational points $\sigma\cap \Pro(V)(k)$. In particular,  $(\sigma, V_{\sigma})$ is a polyhedron in $\Pro(V_{\sigma})$ and has maximal dimension. In general, the vector space $V_{\sigma}$  may be strictly contained in  $V$.  
 We may allow the case when all  $v_1,\ldots, v_m=0$ and $\sigma$ is empty. 
 
 By a well-known theorem of  Minkowski and Weyl,  a polyhedron may equivalently be described  by its facets. There is    a  unique, finite,   minimal  set  of hyperplanes $(H_j)_{j\in J}  \subset V_{\sigma} \subset V$  where $H_j$ is defined by the vanishing of a non-zero  linear form $f_j \in V_{\sigma}^{\vee}$ defined over $k$,  such that for all $j\in J$:
 \[  \sigma \cap (V_{\sigma} \otimes_k \R) =\{ x\in  V_{\sigma}\otimes_k\R:  f_j(x) \geq 0  \hbox{ for all } j \in J \}  \ .\] 
  Note that $\sigma \cap (V_{\sigma}\otimes_k \R)$ is canonically identified with $\sigma$ via the inclusion $V_{\sigma} \hookrightarrow V $. 
  A \emph{facet of $\sigma$} is a non-empty projective polyhedron  of the form $(\sigma \cap H_j , V)$, and has dimension one less than $\sigma$.      A \emph{face of $\sigma$} is any non-empty intersection of facets $(\sigma \cap H_{j_1} \cap \ldots \cap H_{j_k}, V)$, for $k\geq 0$  and a
 \emph{vertex of $\sigma$} is a face of dimension zero. Every vertex  of $\sigma$ thus defines a $k$-rational point in $\Pro(V)$   
and, one may show,  is  the image  $[v_i] \in \Pro_k(V)(k)$  of  some vector $v_i$, for $1\leq i\leq m$, where the $v_i$ are as in \eqref{sigmahat}.
  Not all  of the  vectors $v_i$ are  necessarily  vertices and may be  redundant in the definition of
  $\widehat{\sigma}$.

\begin{defn} A \emph{polyhedral linear configuration} over $k$ is a triple $(\Pro(V), L_{\sigma} , \sigma)$ where 
 $(\sigma ,V)$ is a   polyhedron, and  $L_{\sigma}$ is the linear configuration in $\Pro(V)$ whose components $L_i =V_{\sigma_i}$ are  the affine spans of every facet $\sigma_i$ of $\sigma$.
 \end{defn}
 
  In particular, each component $L_i$ of $L_{\sigma}$ satisfies  $\dim L_i = \dim \sigma -1$  and   the set of real points $L_{\sigma}(\R)$ is nothing other than the Zariski-closure of the set of points of  the boundary $\partial \sigma$ in $\Pro(V)(\R)$. 
In this manner,  a  polyhedron $(V,\sigma)$ uniquely determines a polyhedral linear configuration $(\Pro(V), L_{\sigma}, \sigma)$, and vice-versa.

 \begin{defn}
A \emph{map of polyhedral linear configurations}, which we denote  by 
\begin{equation} \label{phimapdef} \phi:  (\Pro(V) , L_{\sigma}, \sigma) \To   (\Pro(V'), L_{\sigma'},  \sigma')\ , \end{equation} 
is given by an injective linear  map $ \phi: V \hookrightarrow V'$ such that the induced map of projective spaces,  also denoted by $\phi: \Pro(V) \rightarrow \Pro(V')$, satisfies both  
\[ \phi(L) \subset L' \quad \hbox{ and  }  \quad \phi(\sigma) \subset \sigma' . \]
In particular, every face  of $\sigma$ maps to a face  of $\sigma'$. 
\end{defn}
 \begin{rem}  The above definitions  are  insufficient  to express the subdivision of polyhedra into smaller polyhedra. For this one must  consider a more general notion  where $L$  contains further linear subspaces in addition to  the Zariski closures of the facets of $\sigma$.  \end{rem} 
 
\begin{example} \label{ex: standardsimplex}  The standard simplex  in projective space $ \Pro^n_{\Q}$ with homogeneous  coordinates $(x_0:\ldots : x_n)$ is  the polyhedral linear configuration 
$( \Pro^n_{\Q}, L, \sigma)$, where $L= V(x_0\cdots x_n)$,   and $\sigma \subset   \Pro^n_{\Q}(\R)$ is  defined by the region $x_i \geq 0$. 
\end{example} 

\subsection{Faces and their normals}Let $ (\Pro(V) , L_{\sigma}, \sigma)$ be a polyhedral linear configuration, and 
consider any face $\sigma_F$ of $\sigma$.
 The polyhedron $(\sigma_F,V_{\sigma_F})$   defines a polyhedral linear configuration $( \Pro(V_{\sigma_F}), L_{\sigma_F}, \sigma_F) $ and a `face' map
\begin{equation} \label{facemap}
( \Pro(V_{\sigma_F}), L_{\sigma_F}, \sigma_F)    \To  (\Pro(V) , L, \sigma) \ , \end{equation}
which is the map   \eqref{phimapdef} induced by the inclusion $V_{\sigma_F}$ in $V$, and corresponds to the inclusion of $\sigma_F$ in $\sigma$. 
Note that  $(\sigma_F, V_{\sigma_F})$ has maximal dimension, i.e., $\dim(\sigma_F)= \dim(V_{\sigma_F})-1$, but the same is not necessarily true of $(\sigma,V)$. For the trivial face, when  $\sigma_F$ is equal to $\sigma$ itself,     \eqref{facemap}   gives a linear map
$ ( \Pro(V_{\sigma}), L_{\sigma}, \sigma)    \rightarrow  (\Pro(V) , L, \sigma)$.

\begin{defn}   Let $W\subset V$ be a vector subspace such  that 
$\sigma_W = \sigma \cap \Pro(W)$
is  a face of $\sigma$ (here and elsewhere, we write $\sigma \cap \Pro(W)$ for $\sigma \cap \Pro(W)(\R)$).  Consider the  polyhedral cone: 
\[ \widehat{\sigma_{/W}}  \subset   (V/W )_{\R} \] 
 which is  the image of $\widehat{\sigma}$ \eqref{sigmahat} under the natural map $V_{\R} \rightarrow   (V/W )_{\R}$. 
 It is defined over $V/W$ since it is the convex hull of the images of a set of defining  vectors $v_i$ \eqref{sigmahat} under the natural map $V \rightarrow V/W$. 
Denote by $\sigma_{/W}$ the link of its cone point. 
\end{defn} 

It is important to note that $\sigma_W$ is not necessarily assumed to be  Zariski-dense  in $\Pro(W)$, i.e., $W$ may strictly contain  the space $V_{\sigma_W}$.  

\begin{lem}  The pair $(\sigma_{/W}, V/W)$ is a polyhedron.
\end{lem} 
\begin{proof} It  suffices to show that  $\sigma_{/W}$  is strictly convex. We do this by showing that it is contained in a standard simplex  (example \ref{ex: standardsimplex}).
 Consider any  choice of irreducible components $L_i$ of $L_{\sigma}$ such that the face  $\sigma_W$ is  given by the    intersection 
$\sigma_W = \sigma \cap L_{1} \cap \ldots \cap L_{m}$. We may assume that $L_1,\ldots, L_m$ are normal crossing. Since $\sigma_W$ is Zariski-dense in $L_1\cap \ldots \cap L_m$ it follows that $L_1 \cap \ldots \cap L_m  \subset \Pro(W)$, and hence $\dim W\geq  \dim V- m$. 
Since the $L_i$ cross normally, it follows that any subset of $p =\dim (V/W) \leq m$ spaces $L_i$, and in particular $L_1,\ldots, L_{p}$,   defines a set of coordinate hyperplanes on $\Pro(V/W)$.  Let  $x_i$ be a system of coordinates on $\Pro(V/W)$   whose zero loci are the  $L_i$ for $i=1,\ldots,p$. 
By replacing $x_i$ with $-x_i$ we may assume that the $x_i$ are non-negative on $\sigma$.
 By construction, the link    $\sigma_{/W}$ is contained in the strictly convex region $\{(x_1:\ldots :x_p) :  x_i\geq 0\}$.  \end{proof} 

\begin{defn}  \label{defn: normal} Define the \emph{normal} of $\sigma$ relative to $W$ to be  
$(\Pro(V/W), L_{\sigma_{/W}}, \sigma_{/W})$. 
\end{defn}

A   map of polyhedral linear configurations induces maps   simultaneously on   faces and  their normals. More precisely, let $\phi$ be as in \eqref{phimapdef},   and  suppose that $W\subset V$ meets $\sigma$ in a face $\sigma_W =\sigma \cap W$ of $\sigma$. Then $\phi(W) = W'$ also meets $\sigma'$ in the face $\sigma'_{W'}= \sigma'\cap W'$, and  we deduce a pair of maps of polyhedral linear configurations:
\begin{eqnarray}
\phi\big|_W :   (\Pro(W), L_{\sigma_W}, \sigma_W)  & \rightarrow & (\Pro(W'), L_{\sigma'_{W'}}, \sigma'_{W'})   \\
\phi_{/W} :   (\Pro(V/W), L_{\sigma_{/W}}, \sigma_{/W})  & \rightarrow & (\Pro(V'/W'), L_{\sigma'_{/W'}} , \sigma'_{/W'})  \ . \nonumber 
\end{eqnarray}

\begin{rem}  \label{rem: Products}
The projectivised normal bundle   of the  linear subspace $\Pro(W) \subset \Pro(V)$  is trivial, and is  canonically isomorphic  to a product of projective spaces:
\[ \Pro (N_{\Pro(W)|\Pro(V)}) = \Pro(W) \times \Pro(V/W)\ .\]
For any subspace $W\subset V$ meeting $\sigma$ in a face $\sigma_W$, the  product of polyhedra
\[  \sigma_W  \times \sigma_{/W}  \quad \subset \quad   \Pro(W)(\R) \times \Pro(V/W)(\R) \]
is contained  within it.  A map of polyhedral linear configurations $\phi$ as above induces a map  on the products $ \sigma_W \times \sigma_{/W} \rightarrow \sigma'_{W'} \times \sigma'_{/W'}$.  These products  encode the infinitesimal structure of $\sigma$ in the neighbourhood of  $\sigma_W$. This will be discussed  in \S\ref{sect: FacesAndMultBoundary}. 
\end{rem}

\begin{lem} \label{lem: SlackFaces} Let $(\sigma, V)$ be a  polyhedron, and let $W_1\subsetneq W_2 \subset V$ such that $\sigma\cap \Pro(W_1)(\R)  = \sigma \cap \Pro(W_2)(\R)$ is a face of $\sigma$.
Then $\Pro(W_2/W_1) (\R) $ does not meet $(\sigma_{/W_1},V/W_1)$.
\end{lem} 
\begin{proof} If  $\Pro(W_2/W_1)(\R)$ were to meet $\sigma_{/W_1}$, then     there  would exist a  $0 \neq v \in W_2/W_1$  whose image in $\Pro(W_2/W_1)(\Q)$ is a vertex of $\sigma_{/W_1}$.  This would imply that  $\sigma \cap  \Pro(W_2)(\R)$ strictly contains $\sigma\cap \Pro(W_1)(\R)$, a contradiction.
\end{proof}

\subsection{Category of polyhedral linear configurations}

\begin{defn}   \label{defn: PLC} Define a category $\PLC_k$  whose objects are polyhedral linear configurations over $k$, and whose  morphisms are generated by: 
\begin{enumerate}  \setlength\itemsep{0.03in} 
\item (Linear embeddings)  Maps  of the form $f: (\Pro(V) , L, \sigma) \rightarrow (\Pro(V'), L', \sigma')$, where  $f$ is a linear embedding  which satisfies $f(L) = L'$ and $f(\sigma) = \sigma'$, 
\item   (Inclusions of faces)  For any face $\sigma_F$ of $\sigma$, the  face maps  \eqref{facemap}:
\[
 ( \Pro(V_{\sigma_F}), L_{\sigma_F}, \sigma_F)    \To  (\Pro(V) , L, \sigma)\ . 
  \]
\end{enumerate} 
\end{defn}

The category $\PLC_k$ is a sub-category of $\PC_k$ which satisfies the assumptions of \S\ref{sect: PolyhedralAssumptions}.

\begin{rem} The above categories are adapted to studying the links of cones. If one is interested in  the cones  \emph{per se},  one may consider a  version in which one replaces projective space $\Pro^n_k$ with affine space $\mathbb{A}^{n+1}_k$, and $\sigma$ with $\widehat{\sigma}$, \emph{etc}. \end{rem}

  \subsection{Linear polyhedral  complexes}  \label{sect: LinearPolyhedralComplexes} 
 Definition \ref{defn: GenCkComplex} leads to the following:

 \begin{defn} A \emph{linear polyhedral  complex} is a $\PLC_k$-complex, i.e., a functor from a finite diagram category to the category of polyhedral linear configurations.
 \end{defn}

 \begin{rem} The topological realisation
  of a polyhedral   linear complex
is  obtained by gluing together finitely many quotients of strictly convex polyhedra by finite groups of automorphisms and defines a  symmetric CW-complex (see \cite{ACP}).  
 \end{rem} 
 
 We may define subschemes of polyhedral linear complexes, and differential forms upon them in the manner of \S\ref{sect: subschemes}, \ref{sect: differentialforms}.

\section{Wonderful compactications of linear polyhedral complexes} \label{sect: Blowups}

\subsection{Blowing-up linear subspaces} 
Let $\BB$ denote a  finite set of linear subspaces $\Pro(W) \subsetneq \Pro(V)$ with the property that $\BB$ is closed under intersections.

\begin{defn} The (wonderful)  compactification   \cite{Wonderful} of  $\Pro(V)$ along $\BB$ is denoted 
\[ 
\pi_{\BB} :  P^{\BB}(V) \To \Pro(V)\]
and is defined to be   the iterated  blow-up of $\Pro(V)$ along the strict transforms of the  strata $\Pro(W_j)  \in \mathcal{B}$, in increasing order of dimension. It is shown to be independent of the order of blowings-up, and  well-defined. Let  $D^{\BB} \subset P^{\BB}(V)$  denote the exceptional divisor.
 \end{defn}

\begin{prop}  \label{prop: UniversalANDIdeal} The iterated blow-up $P^{\BB}$ has the following  properties.

\vspace{0.02in} 
(i)  Let $f:U \rightarrow V$  be an injective linear map, and $\BB$ a set of subspaces of $\Pro(V)$  as above,  such that  $\Pro(fU)$ is not contained in any element of $\BB$.   If  $f^{-1} \BB $ denotes the set of preimages $f^{-1}\Pro(W)$ of spaces $\Pro(W)$ in $\BB$, there is a canonical map
$  f^{\BB}:  P^{f^{-1}\BB}(U) \rightarrow  P^{\BB}(V) $ such that 
\[ 
\begin{array}{ccc}
  P^{f^{-1}\BB}(U) & \overset{f^{\BB}}{\To}&   P^{\BB}(V)   \\
 \downarrow   &   &   \downarrow \\
\Pro(U)   &  \overset{f}{\To}   &   \Pro(V)
\end{array}
\]
commutes, where the vertical maps are the blow-downs $\pi_{f^{-1} \BB}$, $\pi_{\BB}$ respectively. 
\vspace{0.02in} 

(ii) Suppose that $\BB , \BB'$ are two sets of linear subspaces as above which are closed under intersections. Suppose that $\BB' \subset \BB$. Then there is a canonical morphism
\[ \pi_{\BB/\BB'} : P^{\BB}(V) \To P^{\BB'}(V)\ .\]
\end{prop} 

\begin{proof}
Part (i) follows from repeated application of the universal property of strict transforms  \cite{Hartshorne}[II, Corollary 7.15].  For (ii),  the proof is by induction. 
Choose any linear ordering on  $\BB$ compatible with the order of blowing-up (i.e., in increasing dimension). Let   $\BB_n\subset \BB $ denote  the first $n$ elements of $\BB$.  Suppose by induction on $n$ that we have constructed  morphisms $\pi_m: P^{\BB_m}(V) \rightarrow P^{\BB' \cap \BB_m}(V)$ for all $m\leq n$ such that the following diagrams commute, where the vertical maps are the natural (blow-down) maps: 
 \begin{equation} \label{inproof:recursiveblowups}
\begin{array}{ccc}
P^{\BB_{m}}(V) &  \overset{\pi_m}{\To}  & P^{\BB' \cap \BB_{m}}(V)  \\
 \downarrow  &   & \downarrow    \\
 P^{\BB_{m-1}}(V) &   \overset{\pi_{m-1}}{\To}  &    P^{\BB'\cap \BB_{m-1}} (V)
\end{array}
\end{equation} 
and where $\pi_0: \Pro(V) \rightarrow \Pro(V)$ is the identity. To define $\pi_{n+1}$, there are two cases to consider.  Let    $\BB_{n+1} \setminus \BB_n= \{ \Pro(W)\}$ where $\Pro(W) \subset\Pro(V)$ is  a blow-up locus for $P^{\BB}(V)$.      Either $\Pro(W)$  is an element of  $\BB'$,  or it is not.  
If it is not, then $\BB' \cap \BB_{n+1} =  \BB' \cap \BB_{n}$ and the map $ \pi_{n+1}: P^{\BB_{n+1}}(V) \rightarrow P^{\BB' \cap \BB_{n+1}}(V)$ is simply the composition of the blow-up $P^{\BB_{n+1}}(V) \rightarrow P^{\BB_n}(V) $ with $\pi_n$ (in this case, the vertical map on the right in \eqref{inproof:recursiveblowups}  is the identity  when $m=n+1$). In the contrary case, $\BB_{n+1} = \BB_n \cup \{\Pro(W)\}$ where $\Pro(W) \in \BB'$.  Let $Y$ denote the strict transform of $\Pro(W)$ in $P^{\BB' \cap \BB_n}(V)$.  By commutativity of  the diagrams \eqref{inproof:recursiveblowups} for all $0\leq m\leq n$, its inverse image $\pi_n^{-1}(Y)$ in $P^{\BB_n}(V)$ is a Cartier divisor since it is the union of the strict transform  $Y'$ of $\Pro(W)$ in   $P^{\BB_n}(V)$, which is Cartier, with exceptional divisors (likewise). By the universal property of blowing up
 \cite{Hartshorne}[II, Proposition 7.14], there is a unique morphism $\pi: \mathrm{Bl}_{Y'} P^{\BB_n}    (V)\rightarrow    \mathrm{Bl}_{Y} P^{\BB'\cap \BB_n}(V)$ such that the following diagram commutes:
 \[ 
\begin{array}{ccc}
\mathrm{Bl}_{Y'} P^{\BB_n}  (V) & \rightarrow   & \mathrm{Bl}_{Y} P^{\BB' \cap \BB_n}(V)  \\
 \downarrow  &   & \downarrow    \\
 P^{\BB_n}(V) &   \overset{\pi_n}{\To}  &    P^{\BB'\cap \BB_n} (V)
\end{array}
\]
Since $\mathrm{Bl}_{Y'} P^{\BB_n}(V)= P^{\BB_{n+1}}(V)$ and $ \mathrm{Bl}_{Y} P^{\BB'\cap \BB_n}(V) = P^{\BB' \cap \BB_{n+1}}(V)$ we may define $\pi_{n+1}$ to be the morphism $\pi$, which completes the induction step. 
\end{proof}

\begin{prop} \label{prop: BlowUpStructure}\cite{Wonderful}.
The  irreducible components $\mathcal{E}^{\BB}_W$ of $D^{\BB}$ are in one-to-one  correspondence  with subspaces $\Pro(W) \in \BB$ of codimension $\geq 1$, where 
 $\pi_{\BB}: \mathcal{E}^{\BB}_W \rightarrow  \Pro(W)$, and $\mathcal{E}^{\BB}_W$ is the Zariski closure of the inverse image with respect to $\pi_{\BB}$ of the generic point of $\Pro(W)$. 
 
 If we define  the following sets (which are  closed under intersections): 
 \begin{eqnarray} \label{BBrestrictions} 
  \BB_W & = &   \{  \Pro(T) \subsetneq \Pro(W)  \ \hbox{ for }  \ \Pro(T)  \in \BB  \hbox{ such that }  T \subseteq W \}   \\
 \BB_{/W} & = &  \{  \Pro(T/W)  \subset \Pro(V/W) \  \hbox{ for }  \     \Pro(T)\in \BB \hbox{ such that }    T \supseteq W\}   \nonumber 
 \end{eqnarray} 
 (recall that $\Pro(0) =\emptyset$) then there is a commutative diagram
 \begin{equation} \label{blowupsquare}
\begin{array}{ccc}
 P^{\BB_W}  \times P^{\BB_{/W}}     &  \overset{\sim}{\To}    &   \mathcal{E}^{\BB}_W   \\
  \downarrow_{\pi_{\BB_W} \times \pi_{\BB_{/W}}}    &   &    \downarrow_{\pi_{\BB}}\\
\Pro(W) \times \Pro(V/W)
     &   \To   &    \Pro(W) \end{array}
\end{equation}
where the horizontal  map along the top is a canonical isomorphism, and the one along the bottom is projection onto the first factor $\Pro(W)$. 
 The divisor $D^{\BB}$ is simple normal crossing, and   two components $\mathcal{E}^{\BB}_W$ and $\mathcal{E}^{\BB}_{W'}$ have non-empty intersection if and only if one of the  two spaces
 $\Pro(W) , \Pro(W')$ is contained in the other. 
\end{prop} 

\begin{rem}  \label{rem: singleblowupfirst} The following observation will be  useful.
By \eqref{blowupsquare}, $\mathcal{E}^{\BB}_W$ is canonically isomorphic to the  iterated blow-up of  
$\Pro(W) \times \Pro(V/W)$
relative to $\pi_{\BB_W}$ and $\pi_{\BB_{/W}}$ on each factor.  The product $\Pro(W) \times \Pro(V/W)$ is 
  isomorphic to  the exceptional divisor  of a single blow-up of $\Pro(V)$ along $\Pro(W)$. 
Thus  the exceptional divisor $\mathcal{E}^{\BB}_W$  may be computed by first blowing-up a single linear space $\Pro(W)$ inside $\Pro(V)$, and then computing the  iterated blow-ups relative to $\BB_W$ and $\BB_{/W}$ on each factor  $\Pro(W)$ and $\Pro(V/W)$ of its exceptional divisor. \end{rem} 

By repeated application of  proposition \ref{prop: BlowUpStructure}, one shows that intersections of  irreducible components of $D^{\BB}$ are in one-to-one correspondence with 
sequences of subspaces
\begin{equation}  \label{proWflag} \Pro(W_k) \subset \Pro(W_{k-1}) \subset \ldots \subset \Pro(W_1) \subset  \Pro(V)  \end{equation} 
where   $\Pro(W_i) \in \BB$ for $i=1,\ldots, k$ and all inclusions are strict. The corresponding subscheme of  $D^{\BB}$ is isomorphic to
$\prod_{i=1}^{k+1} P^{\BB_i}(W_{i-1}/W_i)$  where  $W_{k+1}=0$ and $W_0= V$ and where
 $\BB_i = \{ \Pro(T/W_i) \hbox{ where } \Pro(T) \in \BB \hbox{ such that }W_i \subseteq T \subsetneq W_{i-1}\}$. 
The previous  proposition can be proved using explicit local coordinates, which we describe presently.

\subsection{Local coordinates for linear blow ups}  \label{sect: LocalBlowUpCoordinates}
Given a sequence \eqref{proWflag}, we may choose projective coordinates $\alpha_1,\ldots, \alpha_n$  on $\Pro(V)$ such that 
 \[ \Pro(W_m) = V (\alpha_{1}, \alpha_2, \ldots, \alpha_{i_{m}} )   \quad \hbox{ for }  \ 1\leq m \leq k \ ,
 \] 
 for some increasing sequence $0< i_1 < \ldots <i_k < n$. 
A choice of  local affine coordinates on   $P^{\BB}$ lying over the  open chart $\alpha_n=1$  of $\Pro(V)$ is given by $\beta_1, \ldots, \beta_{n-1}$  where:
\begin{multline*} 
\beta_1= \frac{\alpha_1}{\alpha_{i_1}}, \  \ldots,  \   \beta_{i_1-1} = \frac{\alpha_{i_1-1}}{\alpha_{i_1}} \ , \    \beta_{i_1}= \frac{\alpha_{i_1}}{\alpha_{i_2}}, \ \ldots,  \beta_{i_2-1}=\frac{\alpha_{i_2-1}}{\alpha_{i_2}} \ , \  \ldots \\
   \beta_{i_{k-1}} =\frac{\alpha_{i_{k-1}}}{\alpha_{i_k}}, \ldots ,   \   \beta_{i_{k}-1} = \frac{\alpha_{i_{k}-1}}{\alpha_{i_k}} ,   \  \beta_{i_{k}} = \alpha_{i_k} , \ldots \ , \    \beta_{n-1} =\alpha_{n-1} 
\end{multline*}
 The equation of the exceptional divisor $\mathcal{E}_{W_m}$ which  lies over $\Pro(W_m)$ is  given by 
$\beta_{i_{m}} = 0$.

The isomorphism $\mathcal{E}_{W} \cong P^{\BB_{W}} \times P^{\BB_{/W}}$, in the case $W=W_m$,   is represented by  the partition of $\beta_1,\ldots, \beta_{n-1}$ into two sets of coordinates $\beta_1,\ldots,  \beta_{i_m-1}$ and $\beta_{i_{m}+1}, \ldots , \beta_{n-1}$ corresponding to the iterated blow-ups of  two  nested sequences: 
\[   \Pro(W_k) \subset  \ldots \subset \Pro(W_{m+2})   \subset \Pro(W_{m+1}) \subset  \Pro(W_m)\]
\[   \Pro(W_{m-1}/W_m) \subset \ldots \subset \Pro(W_{1}/W_m) \subset  \Pro(V/W_m)\]
lying over the  affine charts with coordinates $(\alpha_1: \ldots : \alpha_{i_{m}-1}:1)$ for $\Pro(V/W_m)$, and $(\alpha_{i_{m}+1}: \ldots : \alpha_{n-1}:1)$
for $\Pro(W_m)$. 
Propositions \ref{prop: UniversalANDIdeal},  \ref{prop: BlowUpStructure} may be proven by computing directly in  these coordinates.  

 \begin{lem} \label{lem: strtransformhyperplane} Let $H\subset \Pro(V)$  be a hyperplane and let $\Pro(W)\in \BB$.  Denote by $\widetilde{H} \subset P^{\BB}$ the strict transform of $H$  under $\pi_{\BB}$. Then 
 \[ 
  \widetilde{H} \cap \mathcal{E}_W  =  
 \begin{cases}   P^{\BB_W}  \times  \widetilde{H}_{/W} \quad \quad  \hbox{ if }  \quad  \Pro(W) \subseteq H \\ 
  \widetilde{H}_W \times P^{\BB_{/W}} \qquad \hbox{ if }  \quad  \Pro(W) \not\subseteq H\ , 
 \end{cases} 
 \]
where $H_W = H \cap \Pro(W)$, and,  when $H$ corresponds to a subspace $W \subset H_0 \subset V$, 
we write $H_{/W} =  \Pro(H_0/W)$. Their  versions with tildes denote  their strict transforms under the iterated blow-ups $\pi_{\BB_W}$ and $\pi_{\BB_{/W}}$, respectively.  Note that if $H\subseteq  \Pro(W)$ then $\widetilde{H}= \emptyset$.
\end{lem} 

\begin{proof} It follows from remark \ref{rem: singleblowupfirst}  that  it is enough  to compute the case when $\BB= \{ \Pro(W)\}$ reduces to a single blow-up.
 The strict transform of $H$ is 
either
\[  H_W \times \Pro(V/W)  \quad  \hbox{ or } \quad    \Pro(W) \times H_{/W}    \quad \hbox{ inside } \quad \Pro(W) \times \Pro(V/W)\] 
depending on which of the two cases are satisfied, since  $H$ meets the normal bundle of $\Pro(W)$ in the product of a hyperplane and a projective space. 
 Alternatively,  
 one can also verify the lemma by direct computation. Suppose that $H$ has the   equation 
\begin{equation}  \label{inproof: Hequation}  \lambda_1\alpha_1 + \cdots + \lambda_n \alpha_n=0 \ , \end{equation} 
and suppose  that  $W= W_m$. 
One has $\Pro(W_m) \subseteq H$ if and only if $\lambda_{i_m+1} = \ldots = \lambda_n=0$. 
After performing the change of variables on the affine chart described above, it becomes
\begin{multline*} \beta_{i_m}\ldots \beta_{i_{k}} \left( \lambda_1 \beta_1 \beta_{i_1} \ldots \beta_{i_{m-1}} +  \lambda_2 \beta_2 \beta_{i_1} \ldots \beta_{i_{m-1}}  + \cdots + \lambda_{i_{m}-1} \beta_{i_{m}-1} +  \lambda_{i_m} \right)  \\
+ \left(  \lambda_{i_{m}+1}\,  \beta_{i_{m}+1} \beta_{i_{m+1}} \ldots \beta_{i_k}   + \cdots + \lambda_{n-1}\, \beta_{n-1}   + \lambda_n \right ) \ .
\end{multline*} 
In the case when $\Pro(W_m) \not\subseteq H$, then setting $\beta_{i_m}=0$ annihilates the first line of the previous expression, leaving only the  second term in parentheses.   It is precisely the equation of the strict transform of $H_W$ in $P^{\BB_{W}}$. 
In the case when $\Pro(W_m) \subset H$,  the second term in parentheses is identically zero, and therefore the strict transform of $H$ has the equation:
\[ \lambda_1 \beta_1 \beta_{i_1} \ldots \beta_{i_{m-1}} +\lambda_2 \beta_2 \beta_{i_1} \ldots \beta_{i_{m-1}}   + \cdots + \lambda_{i_{m}-1} \beta_{i_{m}-1} +  \lambda_{i_m} \] 
which is the equation of the strict transform of $H_{/W}$ in $P^{\BB_{/W}}$. 
 \end{proof}

\subsection{Blow-ups of polyhedral linear configurations} Let  $(\Pro(V), L, \sigma)$ be a  polyhedral linear configuration, and let  $\BB$  be a finite set of linear subspaces $\Pro(W) \subsetneq \Pro(V)$  such that:
\vspace{0.03in}

\begin{enumerate}[(B1).]

\item   \label{B1} $\BB$ is stable under intersections, and  
\vspace{0.03in}

\item  \label{B2} for every  $\Pro(W) \in \BB$, the set $\sigma_W= \sigma \cap \Pro(W)(\R)$ is a face of $\sigma$, or is empty.
\end{enumerate}

Consider the iterated blow up $\pi_{\BB}: P^{\BB}\rightarrow \Pro(V)$ defined above.
In the usual case when $\sigma$ is not contained in $\Pro(W)(\R)$ for some $\Pro(W) \in \BB$, we let 
\begin{equation} \label{sigmaBBdefn} \sigma^{\BB} = \overline{ \pi_{\BB}^{-1} (\overset{\circ}{\sigma}) }   \ \subset \  P^{\BB}(\R)
\end{equation}
denote the closure, in the analytic topology, of the inverse image of the interior of $\sigma$. If $\sigma$ is contained in a $\Pro(W) \in \BB$, then $\sigma^{\BB}$ is defined to be the  empty set (alternatively, we may define $\sigma^{\BB}$ to be the closure of the inverse image of  $\sigma \cap (\Pro(V)\setminus \bigcup \BB)(\R)$, 
the intersection of $\sigma$ with the complement in $\Pro(V)(\R)$ of all blow-up loci  in $\BB$, 
 which covers both cases).  In local coordinates \S\ref{sect: LocalBlowUpCoordinates},  the interior of $\sigma^{\BB}$   is defined by the positivity of a finite number of polynomials $f_1,\ldots, f_k>0$ where the $f_i$ are the local equations of strict transforms of bounding hyperplanes, and  exceptional divisors. The semi-algebraic set $\sigma^{\BB}$ is then locally defined by $f_1,\ldots, f_k\geq 0$.

 When $\sigma^{\BB}$ is non-empty, 
we define a \emph{face} of $\sigma^{\BB}$ to be a non-empty intersection 
\[ \sigma^{\BB} \cap D(\R) \]
where $D$ is any  intersection of irreducible components  of the total transform $\pi_{\BB}^{-1}(L)$. A  \emph{facet} is a face of dimension $\dim \sigma -1$. Finally, define 
$L^{\BB} \subset P^{\BB}$ to be the union of the Zariski closures of the facets of $\sigma^{\BB}$. It depends on $\sigma$, but we usually write $L^{\BB}$ instead of $L^{\BB}_{\sigma}$. 

\begin{defn} The blow-up of $(\Pro(V), L, \sigma)$ along $\BB$ is  the triple $(P^{\BB}, L^{\BB}, \sigma^{\BB})$. 
\end{defn}

By definition, the facets of $\sigma^{\BB}$ are Zariski-dense in the irreducible components $L$ of $L^{\BB}$. From the description of the local coordinates on $P^{\BB}$ one sees that  $\sigma^{\BB}$ is a topological polyhedron inside $P^{\BB}(\R)$, whose boundary satisfies $\partial \sigma^{\BB} = \sigma^{\BB}\cap L^{\BB}(\R)$.   Note that $L^{\BB}$ contains the strict transform  of $L$, and is contained within its  total transform.  It will follow from the description of faces in \S\ref{sect: FacesAndMultBoundary}  that  the triple $(P^{\BB}, L^{\BB}, \sigma^{\BB})$ defines an object of $\PC_k$.
As in \S\ref{section: PolyhedralCellComplexes}, a map of triples $\phi: (P_1^{\BB}, L_1^{\BB}, \sigma_1^{\BB})\rightarrow (P_2^{\BB}, L_2^{\BB}, \sigma_2^{\BB})$  is a morphism of schemes $\phi: P_1^{\BB}\rightarrow P_2^{\BB}$ such that $\phi(L_1^{\BB}) \subset L_2^{\BB}$ and whose restriction to real points induces $\phi(\sigma_1^{\BB}) \subset \sigma_2^{\BB}$. In practice, we shall only consider maps of  very specific types.

\begin{example}  (i). The blow down map  
 \begin{equation} \label{FullBlowdown}   \pi_{\BB} :  (P^{\BB}, L^{\BB}, \sigma^{\BB})\rightarrow (\Pro(V), L, \sigma)
\end{equation}
is a morphism in $\PC_k.$ The   map $\pi_{\BB}: \sigma^{\BB} \rightarrow \sigma$ is not in general a homeomorphism since it may collapse faces.

(ii).  More generally,  consider two sets $\BB, \BB'$ which satisfy  B\ref{B1}, B\ref{B2} such that $\BB' \subset \BB$ (the previous example is the case $\BB'=\emptyset$). Then   Proposition \ref{prop: UniversalANDIdeal} (ii) provides a map 
\[ \pi_{\BB/\BB'} : P^{\BB} \To P^{\BB'} \ . \]
It restricts to a continuous map $\sigma^{\BB} \rightarrow \sigma^{\BB'}$ and defines a morphism
  \begin{equation}   \pi_{\BB/\BB'} :  (P^{\BB}, L^{\BB}, \sigma^{\BB})\rightarrow (P^{\BB'}, L^{\BB'}, \sigma^{\BB'}) \ .
\end{equation}
\end{example}

\begin{defn} \label{def: extraneousmodification} 
Recall from the proof of proposition \ref{prop: UniversalANDIdeal} (ii) that 
 $\pi_{\BB/\BB'}$ is constructed inductively by blowing up elements of $\BB$. Let  $\BB_n$ denote the first $n$ elements of $\BB$. Suppose that 
$\BB'\cap \BB_{n+1} = \BB' \cap \BB_{n}$, in which case we have 
a  commutative diagram \eqref{inproof:recursiveblowups}:
\[
\begin{array}{ccc}
P^{\BB_{n+1}}(V)   &  \overset{\pi_{n+1}}{\To} &  P^{\BB'\cap \BB_{n+1}}(V)  \\
  \downarrow &   &  ||  \\
 P^{\BB_{n}}(V)  &   \overset{\pi_{n}}{\To}     &   P^{\BB'\cap \BB_{n}}(V) 
\end{array}
\]
where the vertical map on the right is the identity. Suppose that $\BB_{n+1} = \BB_n \cup\{ \Pro(W)\}$, and let $Y$ denote the strict transform of $\Pro(W)$ in 
$P^{\BB_{n}}(V)$. Then $P^{\BB_{n+1}}(V) $ is the blow-up of $ P^{\BB_{n}}(V)$ along $Y$. We shall call this blow-up \emph{extraneous} (to $\sigma$) if 
\begin{equation} \label{Yavoidssigma} Y \cap \sigma^{\BB_n} = \emptyset \ . \end{equation}
 Then, since $P^{\BB_{n+1}}(V) $ is isomorphic to $P^{\BB_{n}}(V)$ away from $Y$, and hence in a Zariski-open neighbourhood of the compact set $\sigma^{\BB}_n$, it follows that $\sigma^{\BB_{n+1}}\cong  \sigma^{\BB_n}$. In other words, the extraneous blow up only changes the geometry 
 of   $P^{\BB_{n}}(V)$  away from the region $\sigma^{\BB_n}$. 

More generally, we shall call the  map $ \pi_{\BB/\BB'}$ an  \emph{extraneous  modification}  if, every time this situation occurs in the formation of $P^{\BB}(V)$,  i.e., for all $n$ such that  $\BB'\cap \BB_{n+1} = \BB'\cap \BB_n$, then the blow-up locus is extraneous (i.e., \eqref{Yavoidssigma} holds).  Since every such extraneous blow-up is  trivial on a Zariski-open neighbourhood of $\sigma^{\BB_n}$, it follows that every $\pi_n$ is an isomorphism  locally near $\sigma^{\BB_n}$. We conclude that an extraneous modification satisfies
\[  \pi_{\BB/\BB'}:    \sigma^{\BB} \overset{\sim}{\To} \sigma^{\BB'}\]
since the additional blow-ups in $\BB\setminus \BB'$ are away from  the polyhedron $\sigma^{\BB'}$. 

\end{defn}

\begin{example} (Linear embeddings).  
 Let $h:V_1\rightarrow V_2$ be an injective linear map and consider the corresponding  morphism  of polyhedral linear configurations:
\[ h: (\Pro(V_1), L_1, \sigma_1)\rightarrow    (\Pro(V_2), L_2, \sigma_2)\]
as in definition \ref{defn: PLC} (1), where $h(L_1) = L_2$ and $h(\sigma_1) = \sigma_2$.   If $\BB_2$ is  a set of linear subspaces of $\Pro(V_2)$  satisfying $(B1)$ and  $(B2)$ relative to $\sigma_2$, such that $\Pro(V_1)$ is not contained in $\Pro(W)$ for some $\Pro(W) \in \BB_2$ then let  
\[ \BB_1 =h^{-1} \BB_2  = \left\{  \Pro(h^{-1}W )  \subsetneq \Pro(V_1) \ , \hbox{ for all }\   \Pro(W) \in  \BB_2 \right\} \ . \] 
 The set $\BB_1$  satisfies $(B1), (B2)$   relative to $\sigma_1$ and proposition \ref{prop: UniversalANDIdeal} (i) gives a morphism
 \begin{equation} \label{LinearEmbeddingsBlowUps} (P^{\BB_1}(V_1), L_1^{\BB_1} , \sigma_1^{\BB_1} )  \To  (P^{\BB_2}(V_2), L_2^{\BB_2} , \sigma_2^{\BB_2} )  
 \  . \end{equation}
There is a commutative diagram where the vertical maps are $\pi_{\BB_1}$ and $\pi_{\BB_2}$:
 \begin{equation} 
\begin{array}{ccc}
(P^{\BB_1}(V_1), L_1^{\BB_1} , \sigma_1^{\BB_1} )  &  \overset{\sim}{\To}  & (P^{\BB_2}(V_2), L_2^{\BB_2} , \sigma_2^{\BB_2} )   \\
 \downarrow  &   & \downarrow   \\
 (\Pro(V_1), L_1, \sigma_1)  &  \overset{\sim}{\To}    &    (\Pro(V_2), L_2,  \sigma_2)
\end{array}
\end{equation}

\end{example}

\subsection{Faces and their product structure}  \label{sect: FacesAndMultBoundary}

\begin{prop} \label{prop: multiplicativestructureonBLC} Consider a polyhedral linear configuration  $(\Pro(V), L, \sigma)$ and  $\BB$ as  above. Let $\Pro(W) \in \BB$ which  meets $\sigma$, and let  $\mathcal{E}^{\BB}_W \subset P^{\BB}$ denote the exceptional divisor lying above $\Pro(W)$. 
Then, via the isomorphism $\mathcal{E}^{\BB}_W\cong   P^{\BB_W} \times P^{\BB_{/W}}$ (see  \eqref{blowupsquare}),  one has
\begin{equation} \label{ProductStructureonsigmaB} 
 \sigma^{\BB} \cap \mathcal{E}^{\BB}_W (\R)  =     \sigma^{\BB_W}_{W}  \times \sigma^{\BB_{/W}}_{/W}   
 \end{equation} 
where $\sigma_W = \sigma \cap \Pro(W)(\R)$ is the face of $\sigma$ cut out by $\Pro(W)$, which exists by assumption $(B2)$, and $\sigma_{/W} \subset  \Pro(V/W)(\R)$ is  its normal. If $\sigma_W$ is contained in some $\Pro(T) \in \BB$, where $T\subsetneq W$ (i.e.,  $\sigma_W=\sigma \cap \Pro(T)(\R)$), then $\sigma^{\BB_W}_W$ is the empty set. If not, 
 we have:
\[ L^{\BB} \cap \mathcal{E}^{\BB}_W =  \left( L^{\BB_W}_{\sigma_W}  \times P^{\BB_{/W}} \right)   \, \bigcup \,    \left( P^{\BB_W} \times  L^{\BB_{/W}}_{\sigma_{/W}}\right)  \ .  \]
\end{prop}

\begin{proof}   Let $\BB'= \{\Pro(W)\}$.   Proposition \ref{prop: UniversalANDIdeal} (ii)  implies that the blow-down  $P^{\BB}\rightarrow \Pro(V)$ 
factorizes through $P^{\BB} \rightarrow P^{\BB'} \rightarrow \Pro(V)$.  Consider its restriction to the complement of all exceptional divisors except for $\mathcal{E}_W$. It 
factors through the  maps \[  P^{\BB}\,  \setminus \,  \left( \bigcup_{\Pro(T) \in \BB, T\neq W}  \mathcal{E}_T \right)  
\To  P^{\BB'} \, \setminus D \overset{\pi_{\BB'}}{\To} \Pro(V) \]
where $D$ is the union of the strict transforms of $\Pro(T) \in \BB$   for every $T$ not contained in $W$, and the total transforms $\pi_{\BB'}^{-1}(T)$ for $\Pro(T) \subsetneq \Pro(W)$ in $\BB$.  Since the first map is an isomorphism, the interior of $ \sigma^{\BB} \cap \mathcal{E}^{\BB}_W (\R) $ can be computed in the middle space $P^{\BB'} \, \setminus D$, and so it suffices to consider a single-blow up  of $\Pro(V)$ along $\Pro(W)$. To compute it, note that the interior of the polyhedron $\sigma$ is  the intersection of a finite number of  regions $f_i > 0$, where $f_i$ is a linear form such as \eqref{inproof: Hequation} whose zero locus $H_i=V(f_i)$ is a bounding hyperplane of $\sigma$.
The inverse image of its interior, intersected with the exceptional divisor $\mathcal{E}' \subset    P^{\BB'}$ of $\pi_{\BB'}$,  is therefore cut out by the strict transforms of the  $H_i$   in $\mathcal{E}' \cong \Pro(W) \times \Pro(V/W)$. By lemma \ref{lem: strtransformhyperplane}, the strict transforms are  of the form $H \times \Pro(V/W)$ or $\Pro(W) \times H$. The former define bounding hyperplanes for  $\sigma_W  \subset \Pro(W)(\R)$; the latter define bounding hyperplanes for $\sigma_{/W}$, by definition of the normal polyhedron (definition \ref{defn: normal}). Thus, after a single blow-up of $\Pro(W)$ in $\Pro(V)$, the interior of the intersection
$  \sigma^{\BB'}\cap \mathcal{E}'(\R)$ is 
\[   \overset{\circ}{\sigma}_W \times \overset{\circ}{\sigma}_{/W}  \ \subset \ \Pro(W)(\R) \times \Pro(V/W)(\R) \ , \]
where a superscript $\circ$ denotes the interior.  In the case when there exists $\Pro(T) \subsetneq \Pro(W)$ in $\BB$ such that  $\Pro(W) \cap \sigma = \Pro(T) \cap \sigma$, we have $\pi_{\BB'}^{-1}(\Pro(T)) =  \Pro(T) \times \Pro(V/W) $,  and it follows that $\overset{\circ}{\sigma}_W \times \overset{\circ}{\sigma}_{/W} $ is contained in $D$. We conclude that the interior of     $ \sigma^{\BB} \cap \mathcal{E}^{\BB}_W (\R) $  is empty.
By the earlier description of $\sigma^{\BB}$ as a connected semi-algebraic set, the face $\sigma^{\BB} \cap \mathcal{E}^{\BB}_W (\R) $ is equal to the closure of its interior, which in this case is empty. By definition,  $\sigma_W^{\BB_W}$ is also empty and so  formula \eqref{ProductStructureonsigmaB} holds.

In the case when there does not exist such a $\Pro(T)$,   every element of  $\BB_{W}$ meets $\sigma_W \subset \Pro(W)(\R)$ along a face (and does not strictly contain it). The same holds for $\BB_{/W}$ relative to $\sigma_{/W} \subset \Pro(V/W)(\R)$. By 
  remark \ref{rem: singleblowupfirst}, the restriction of $\pi_{\BB/\BB'}$ to $\mathcal{E}_W$ is the product $\pi_{\BB_W} \times \pi_{\BB_{/W}}$.  The interior 
  of $\sigma^{\BB} \cap \mathcal{E}_W(\R)$ is therefore the product 
  \begin{equation} \label{inproof: pibsigmaproduct} \pi_{\BB_W}^{-1}   (\overset{\circ}{\sigma}_W) \times   \pi_{\BB_{/W}}^{-1}   (\overset{\circ}{\sigma}_{/W})  \ .
  \end{equation}
  As before, $\sigma^{\BB} \cap \mathcal{E}_W(\R)$  is equal to  the closure of its interior. 
  By \eqref{sigmaBBdefn}, the topological closure of   \eqref{inproof: pibsigmaproduct} is $\sigma^{\BB_W}_W \times \sigma^{\BB_{/W}}_{/W}$, which proves     \eqref{ProductStructureonsigmaB}.
    The final  statement concerning $L^{\BB} \cap \mathcal{E}^{\BB}_W$ follows from the definition of $L^{\BB}$ as the union of the Zariski closures of the facets of $\sigma^{\BB}$. 
\end{proof} 

In the particular case when $\BB=\{\Pro(W)\}$ is a singleton and hence $P^{\BB}$ is simply the blow-up of $\Pro(V)$ along $\Pro(W)$, formula \eqref{ProductStructureonsigmaB} states that 
\[  \sigma^{\BB} \cap \mathcal{E}_W (\R)  =   \sigma_{W}  \times \sigma_{/W} \  , \]
which provides an infinitesimal interpretation of the normal polyhedron (remark \ref{rem: Products}).

The situation when $\sigma^{\BB_W}_W =\emptyset$ arises when there exists $\Pro(T) \subsetneq \Pro(W)$ in $\BB$ such that $\sigma_W = \sigma \cap \Pro(T)(\R) = \sigma \cap \Pro(W)(\R)$, i.e., $W$ is not minimal amongst the set of spaces in  $\BB$ which meet $\sigma$ along $\sigma_W$. This phenomenon is  discussed in further detail in \S\ref{sect: MinimalBlowup}. 
Another way to see that  $\sigma^{\BB_W}_W =\emptyset$ in this case is as follows. After blowing up the subspace $\Pro(T)$, it follows from   proposition \ref{prop: BlowUpStructure} applied to $\mathcal{E}_T$ that the strict transform of $\Pro(W)$ no longer meets the closure of the inverse image of $\sigma$, since by lemma \ref{lem: SlackFaces}, it does not meet $\sigma_{/T}$ (see figure \ref{figureExtraneous}).

\begin{notn}
Let  $(P^{\BB}, L^{\BB}, \sigma^{\BB})$ be as above, and let $D$ be an intersection of irreducible components of $L^{\BB}= \bigcup  L $. Denote
 by \[ D\cap    (P^{\BB}, L^{\BB}, \sigma^{\BB}) = ( D \cap P^{\BB}  ,   \bigcup_{L\neq D} (D\cap L ) , D(\R) \cap \sigma^{\BB} ) \ .  \] 
 It is a face if $ D(\R) \cap \sigma^{\BB}\neq \emptyset$. 
 \end{notn}
\begin{cor}
A face of the blow-up of a polyhedral linear configuration is isomorphic to  a product of blow-ups of polyhedral linear configurations.

More precisely,   for any intersection of   irreducible components $D$   of $L^{\BB}$ we have 
\begin{equation} \label{faceidentification}  D \cap (P^{\BB}, L^{\BB}, \sigma^{\BB})  \cong   (P^{\BB_1}, L_1^{\BB_1}, \sigma_1^{\BB_1}) \times \ldots \times 
(P^{\BB_n}, L_n^{\BB_n}, \sigma_n^{\BB_n}) \end{equation}  
for suitable $\sigma_1,\ldots, \sigma_n$  polyhedra in $\Pro(V_1),\ldots, \Pro(V_n)$,  and  $\BB_i$ a  subset of linear subspaces in $\Pro(V_i)$, satisfying (B1) and (B2) relative to $\sigma_i$, for all $1\leq i \leq n$.
\end{cor} 
\begin{proof}Let $(P^{\BB}, L^{\BB}, \sigma^{\BB})$ be the iterated blow up of 
$(\Pro(V), L, \sigma)$. Assume  for the time being  that 
the  face in question is given by intersecting with $D$, a single irreducible component of $L^{\BB}$. 
Suppose first of all that $D\subset \mathcal{E}_W$ is contained in an exceptional divisor for some $\Pro(W) \in \BB$. Then by proposition \ref{prop: multiplicativestructureonBLC},
\[ \mathcal{E}_W \cap (P^{\BB}, L^{\BB}, \sigma^{\BB}) \cong (P^{\BB_W}, L_{\sigma_W}^{\BB_W}, \sigma_W^{\BB_W}) \times (P^{\BB_{/W}}, L_{\sigma_{/W}} ^{\BB_{/W}}, \sigma_{/W}^{\BB_{/W}})\] 
 and by the final part of the same proposition, $D \subset   \mathcal{E}_W \cong P^{\BB_W}\times P^{\BB_{/W}}$ is of the form $D=D_W \times  P^{\BB_{/W}}$ or $D=P^{\BB_W} \times D_{/W}$.

Now suppose that $D$ is not contained in any exceptional divisor, and is therefore the  strict transform of a linear subspace   $\Pro(U)=\pi_{\BB}(D)$ for some $U\subset V$, such that  $\Pro(U)\subset L$ is not contained in any $\Pro(W) \in \BB$.  It follows from  proposition \ref{prop: UniversalANDIdeal} (i)  that $D$ is the iterated blow up of $\Pro(U)$ along the set of linear subspaces $\BB_U  = \{ \Pro(W \cap U) : \Pro(W) \in \BB\}$. It follows that  
 \[ D\cap    (P^{\BB}, L^{\BB}, \sigma^{\BB}) \cong  (P^{\BB_U}, L_U^{\BB_U}, \sigma_U^{\BB_U})     \]
   is the iterated blow up  of the face 
$(\Pro(U), L_U, \sigma_U)$ of $(\Pro(V), L, \sigma)$.

In   general when $D$ has codimension $>1$, we may proceed by induction by repeatedly taking the intersection with irreducible components of $D$ as above.
\end{proof} 

\subsection{Minimal blow-up} \label{sect: MinimalBlowup} 
The minimal blow up is obtained from $\BB$ by stripping out  extraneous modifications (definition \ref{def: extraneousmodification}). 

\begin{defn} \label{defn: Bmin} Let $(\Pro(V), L, \sigma)$ be as above, and let $\BB$ be a finite set of linear subspaces $\Pro(W)$ of $\Pro(V)$ such that  (B1) and (B2) hold.  Define $\BB^{\min, \sigma} \subset \BB$ as follows.

First consider the subset $S(\sigma) \subset \BB$  of spaces $\Pro(W) $ such that  
\vspace{0.02in}

(i). $\Pro(W) (\R) \cap \sigma \neq \emptyset$
\vspace{0.02in}

(ii).  if $\Pro(W) (\R) \cap \sigma  =  \Pro(W') (\R) \cap \sigma$  for some $\Pro(W') \in \BB$ then $W \subset W'$. 
\vspace{0.02in}

\noindent
The set of $U$ such that $\Pro(U) (\R) \cap \sigma$ equals the face $\Pro(W) (\R) \cap \sigma$ is closed under intersections: (ii) requires that $W$ be minimal for this property.

 Define $\BB^{\min, \sigma} \subseteq \BB$ to be the subspace generated by $S(\sigma)$   by taking intersections. In other words, it is the set of intersections of minimal spaces which meet $\sigma$ along a face. 
\end{defn} 

Note that taking intersections may violate $(i)$: in the situation where  $\sigma$ is not a simplex, there can be spaces $\Pro(W)$ in $\BB^{\min, \sigma}$ which do not meet $\sigma$. The definition of $\BB^{\min,\sigma}$ makes sense even if $\BB$  is infinite: the minimal  set $\BB^{\min,\sigma}$ is necessarily finite.

\begin{figure}[h] \begin{center}
\quad {\includegraphics[width=10.5cm]{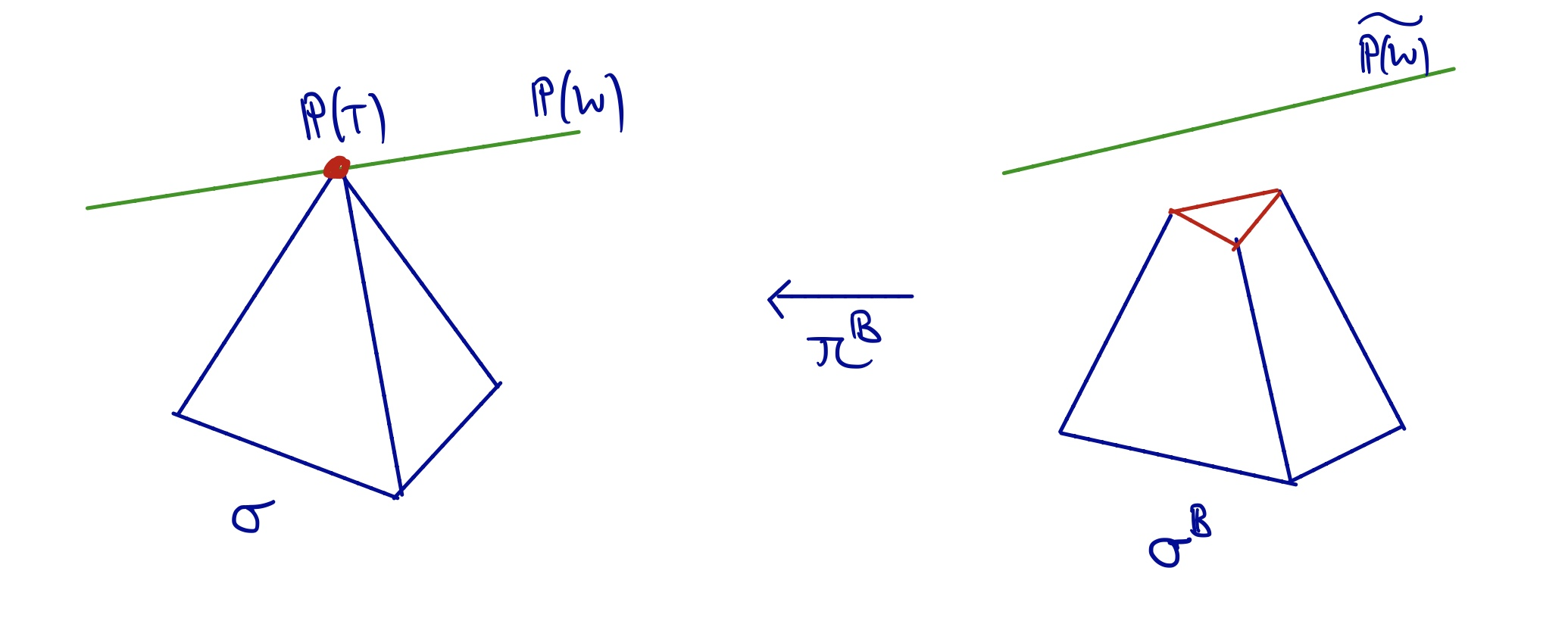}} 
\end{center}
\caption{Two loci $\Pro(T)\subset \Pro(W)$ meet $\sigma$ in the same face, and so $\Pro(W)$ is not minimal. After performing the blow-up of $\BB=\{\Pro(T)\}$, the strict transform of   $\Pro(W)$ has moved away from the polyhedron $\sigma^{\BB}$. The subsequent blow-up of $\Pro(W)$ is thus extraneous and does  not affect  $\sigma^{\BB}$.    } \label{figureExtraneous}
\end{figure}

\begin{prop}  \label{prop: minimalblow-down} The morphism $ \pi  : P^{\BB} \rightarrow  P^{\BB^{\min, \sigma }}$ constructed in proposition \ref{prop: UniversalANDIdeal} (ii)   is an extraneous modification (definition \ref{def: extraneousmodification}).  It restricts to an isomorphism:
\[ \pi: \sigma^{\BB}\overset{\sim}{\rightarrow} \sigma^{\BB^{\min, \sigma}}\ .\] 
\end{prop}
\begin{proof} One must show by definition   \ref{def: extraneousmodification}   that every  blow-up locus in $\BB \setminus    \BB^{\min, \sigma } $  is extraneous. For those loci in $\BB$ which do not meet $\sigma$ (excluded from $ \BB^{\min, \sigma } $ in definition \ref{defn: Bmin} (i)), this clear. 
For non-minimal loci (in the sense of definition \ref{defn: Bmin} (ii)), we argue as follows.
 Let $\Pro(T) \subset \Pro(W)$ be elements of $\BB$ such that 
$\sigma \cap \Pro(T)(\R) = \sigma \cap \Pro(W)(\R)$.  By  proposition \ref{prop: BlowUpStructure} the exceptional divisor $\mathcal{E}_W\subset P^{\BB}$ does not meet $\sigma^{\BB}$. Since the iterated blow up is an isomorphism away from exceptional divisors, it follows that, immediately prior to blowing up $\Pro(W)$ in the formation of $P^{\BB}$,  the strict transform of $\Pro(W)$ does not meet the closure of the inverse image of the interior of $\sigma$, and hence $\Pro(W)$ is extraneous.
 \end{proof}

We shall call the $P^{\BB^{\min,\sigma}}$  the minimal blow-up of $\Pro(V)$ relative to $\BB$, $\sigma$. 

\subsection{Category of blow-ups of  polyhedral linear configurations} 
We may now define a category $\BLC_k$ of blow-ups of linear configurations.

Its objects are products (using notation  \ref{notation: products}):
\[ 
(P^{\BB_1}, L^{\BB_1}, \sigma^{\BB_1}) \times  \ldots \times    (P^{\BB_k}, L^{\BB_k}, \sigma^{\BB_k}) 
\]
of blow-ups of polyhedral linear configurations.
\vspace{0.04in}

Morphisms between objects are  generated by  products  of the following:

\begin{enumerate}

\item (Linear Embeddings). The maps   induced by linear embeddings \eqref{LinearEmbeddingsBlowUps}: 
\[  (P^{\BB_1}(V_1), L^{\BB_1} , \sigma^{\BB_1} )  \To  (P^{\BB_2}(V_2), L^{\BB_2} , \sigma^{\BB_2} ) \] 
They induce isomorphisms $\sigma^{\BB_1} \cong \sigma^{\BB_2}$. 
 \vspace{0.05in}

\item (Extraneous modifications). With notation as in definition \ref{def: extraneousmodification},  the morphisms 
\[  \pi_{\BB/\BB'}: (P^{\BB}, L^{\BB}, \sigma^{\BB})   \To     (P^{\BB'}, L^{\BB'}, \sigma^{\BB'}) \ . \]
where $\BB' \subseteq \BB$ and $ \pi_{\BB/\BB'}$ is an extraneous modification. Here too, $\sigma^{\BB} \cong \sigma^{\BB'}$. 
 \vspace{0.05in}

\item (Face maps).  A face map is defined to be a  composition of   inclusions of facets. There are two types of inclusions of facets depending on whether the   facet in question is contained in an exceptional divisor or not.
\vspace{0.03in}

(i). Consider the inclusion  of an exceptional divisor 
$\mathcal{E}_W$ into $P^{\BB}$, where $\Pro(W) \in \BB$ such that $\sigma_W = \sigma \cap \Pro(W)(\R)$ is non-empty and hence by $(B2)$ is a face of $\sigma$.  By identifying
$\mathcal{E}^{\BB}_W \cong   P^{\BB_W} \times P^{\BB_{/W}}$
we obtain a map via proposition \ref{prop: multiplicativestructureonBLC}:
\[  (P^{\BB_W}, L^{\BB_W}_{\sigma_W}, \sigma_W^{\BB_W})   \times    (P^{\BB_{/W}}, L^{\BB_{/W}}_{\sigma_{/W}}, \sigma_{/W}^{\BB_{/W}})      \To    (P^{\BB}, L^{\BB}_{\sigma}, \sigma^{\BB})    
\]
 which  defines a morphism in $\BLC_k$.
 \vspace{0.03in}

(ii).  Consider the case of an irreducible component   $D$  of $L^{\BB}$ which is the strict transform of $\Pro(U) \subset L$ for some $U\subset V$, where $\Pro(U)$ is not contained in any element of  $\BB$. By proposition \ref{prop: UniversalANDIdeal} (i)  we may identify
\[D\cap    (P^{\BB}, L^{\BB}, \sigma^{\BB}) \cong  (P^{\BB_U}, L_U^{\BB_U}, \sigma_U^{\BB_U}) \]
to obtain a map:
$
(P^{\BB_U}, L_U^{\BB_U}, \sigma_U^{\BB_U})   \To    (P^{\BB}, L^{\BB}_{\sigma}, \sigma^{\BB}). $  
\vspace{0.03in}

\end{enumerate} 
 \vspace{0.03in}

The category $\PLC_k$   is a  subcategory of $\BLC_k$, which
 in  turn is a subcategory of $\PC_k$. Assumptions (1) and (2) of \S\ref{sect: PolyhedralAssumptions}
 follow from 
 the description of iterated blow-ups  in local coordinates and the description of the faces of the $\sigma^{\BB}$ given in  \S\ref{sect: FacesAndMultBoundary}.

 \begin{rem} \label{rem: blowdowncontractsnormal}
  Consider   a  polyhedral linear configuration $(\Pro(V), L, \sigma)$ and a set $\BB$ satisfying $(B1),(B2)$.  Let $\Pro(W) \in \BB$. The inclusion of the exceptional divisor $\mathcal{E}_W$ gives rise to the horizontal map along the top of the following commutative diagram in $\PC_k$:
   \[ 
\begin{array}{ccccc}
 (P^{\BB_W}, L^{\BB_W}_{\sigma_W}, \sigma_W^{\BB_W})  & \times  &  (P^{\BB_{/W}}, L^{\BB_{/W}}_{\sigma_{/W}}, \sigma_{/W}^{\BB_{/W}})    &  \To   & (P^{\BB}, L^{\BB}_{\sigma}, \sigma^{\BB})    \\
  &  \downarrow  &  &   & \downarrow   \\
  & (\Pro(W), L_{\sigma_W}, \sigma_W) &  &  \To &     (\Pro(V), L, \sigma) 
\end{array}
\]
The top row is in the subcategory $\BLC_k$, the bottom row is in $\PLC_k$. 
The horizontal map along the bottom is the inclusion of the face $\sigma_W$ of $\sigma$. The  right-most vertical map is the blow-down $\pi_{\BB}$  and the left-hand vertical map is the projection onto the first factor, followed by  the blow-down $\pi_{\BB_W}$ (on topological realisations,   $\sigma_{/W}$  is collapsed  to a point.)
\end{rem}

 \subsection{Complexes of blow-ups of polyhedra}

 \begin{defn} \label{defn: BLCcomplex} A complex of blow-ups of polyhedra is a  $\BLC_k$-complex, i.e., a functor  
 $ F: \mathcal{D} \rightarrow \BLC_k  $
 where $\mathcal{D}$ is equivalent to a  finite diagram category. 
  \end{defn} 
 
 The definition of a morphism of $\BLC_k$-complexes  proceeds in an identical way to definition \ref{defn: GenCkComplex}.
 The  definition of  subschemes,  and differential forms follows  \S\ref{sect: subschemes}, \ref{sect: differentialforms}.

   Although blow-downs and projection morphisms exist in the category $\PC_k$, we do not  include them in the category $\BLC_k$ in order that 
 assumptions \S\ref{sect: PolyhedralAssumptions} hold.

\section{The moduli space $\mathcal{M}_{g}^{\trop}$ of tropical curves} \label{section: LMg}
We   recast the theory of Feynman polytopes  in the context of polyhedral linear complexes, before turning to the moduli space of tropical curves. 

\subsection{Graphs and polyhedral linear complexes}
Let $G$ be a finite  graph with vertices $V_G$ and edges $E_G$. 
We shall call a self-edge (or tadpole) an edge of $G$ with a single endpoint. The loop number (or genus) $h_G$ of $G$ is equal to the number of independent cycles in $G$, i.e.,  the rank of the free $\Z$-module $H_1(G;\Z)$.

Let $\Pro^{E_G} = \Pro( \Q^{E_G}) $ denote the  projective space with projective coordinates $\alpha_e$ for  every edge $ e\in E_G$. For the time being its vertices are unweighted (equivalently, have  weight zero).
For any subgraph $\gamma$ of $G$, defined by a subset of edges of $E_G$, we denote by $G/\gamma$ the graph obtained by contracting all edges $e$ of $E_{\gamma}$. It does not depend on the ordering of the contractions.  The edges of $G$ are labelled in this section.

\begin{defn}  \label{defn: Ffunctor} For any  $G$, consider the object in the category $\PLC_{\Q}$ defined by
\[  \mathcal{F}(G) =   \left(\Pro^{E_G},  L,  \sigma_G\right) \quad \hbox{where} \quad L = \bigcup_{e \in E_G}  L_{e},\]
where $\sigma_G$ is the region in projective space where $\alpha_e \geq 0$, and $L_{e}$ is the  coordinate hyperplane $\alpha_e=0$.  The object $\mathcal{F}(G)$  is nothing other than a standard simplex (Example \ref{ex: standardsimplex}).
\end{defn}

\begin{rem}  \label{rem: SimplexPropertyForGraphs} Since $\sigma_G$ is a simplex, the associated polyhedral configuration has the special property that any non-empty intersection of components of $L$ is the Zariski closure of a face of $\sigma_G$. This is not true for general polyhedra. 
\end{rem} 
For every edge $e\in E_G$, there is a  canonical  face map (definition \ref{defn: PLC}):
\[ \mathcal{F}(G/e) \To \mathcal{F}(G)\]
corresponding to the inclusion of the face $\sigma_{G/e} \subset \sigma_G$, which identifies $\Pro^{E_{G/e}}$  with the locus $\alpha_e=0$ in the projective space $\Pro^{E_G}$. 

The simplex $\mathcal{F}(G)$, together with the data of all its faces, admits a categorical description.
For this, all graphs in the following have labelled edges and hence have no non-trivial automorphisms. 
Consider  the category $\mathcal{C}_G$ whose objects are all quotients $G/\gamma$  of  $G$, for all strict subgraphs $\gamma \subsetneq E_G$, including the case $\gamma = \emptyset$. 
The morphisms in this category are  generated by edge contractions $\Gamma \rightarrow \Gamma/e$ for any edge $e \in E_{\Gamma}$.  The map $\mathcal{F}$ is a functor
\[  G  \mapsto  \mathcal{F}(G)  \quad :  \quad \mathcal{C}^{\mathrm{opp}}_G   \To  \PLC_{\Q}  \nonumber \\
\]
which sends   morphisms in $ \mathcal{C}^{\mathrm{opp}}_G$ to face maps. Indeed, the faces of the standard simplex $\mathcal{F}(G)$ are in one-to-one correspondence with the elements of the category $\mathcal{C}^{\mathrm{opp}}_G$.

\subsection{The moduli space of tropical curves $\mathcal{M}_g^{\trop}$}
The moduli space of tropical curves is  constructed in a completely analogous manner  to the above by gluing  together simplices associated to  isomorphism classes of stable graphs  of a fixed genus.

\subsubsection{Weighted and stable graphs} \label{subsect: WeightedGraphs}
Let $G$ be a finite graph.  A \emph{weighting} is a map $w: V(G) \rightarrow \mathbb{N}_{\geq0}$ which assigns a non-negative integer to every vertex. 

 The \emph{genus} of $G$ is then defined to be:
\begin{equation} \label{defn: GenusandWeight} g(G) =  h_G + w(G) \quad \hbox{ where  } \quad  w(G) = \sum_{v\in V_G} w(v)\ . 
\end{equation} 
A weighted graph $(G,w)$ is called \emph{stable} if every vertex of weight $0$ has degree at least 3, and if every vertex of weight $1$ has degree $\geq 1$. 
An isomorphism $ (G,w) \overset{\sim}{\rightarrow} (G',w')$ of weighted graphs is an isomorphism $G \overset{\sim}{\rightarrow} G'$ which respects the weightings.  
For every edge $e \in E_G$, the contraction of $e$ is defined as follows. If $e$ is not a self-edge, then 
$(G, w)/e$ is the  pair $(G/e, w')$ where $G/e$ is the graph in which $e$ is removed and the endpoints $v_1,v_2$ of $G$ are identified to produce a new vertex $v$  of weight $w'(v)= w(v_1)+ w(v_2)$. The weights of all other vertices are unchanged.  In the case when $e$ is a self-edge,  $G/e$ is defined by simply removing the edge $e$. In this case, the common endpoint  $v$ of $e$ has weight $w'(v) = w(v)+1$ in the graph $G/e$;  the weights of all other vertices are unchanged.

\begin{defn} \label{defnIg} Let $g \geq 2$ and let $I_g$ denote the category of connected,  stable,  weighted  graphs $(G,w)$  of genus $g$  and $|E_G|\geq 1$,   whose morphisms are generated by isomorphisms   and edge contractions $(G,w) \mapsto (G,w)/e$. Let $I^{\mathrm{opp}}_g$ denote the opposite category. 
\end{defn}

 \begin{figure}[h]\begin{center}
\quad {\includegraphics[width=10cm]{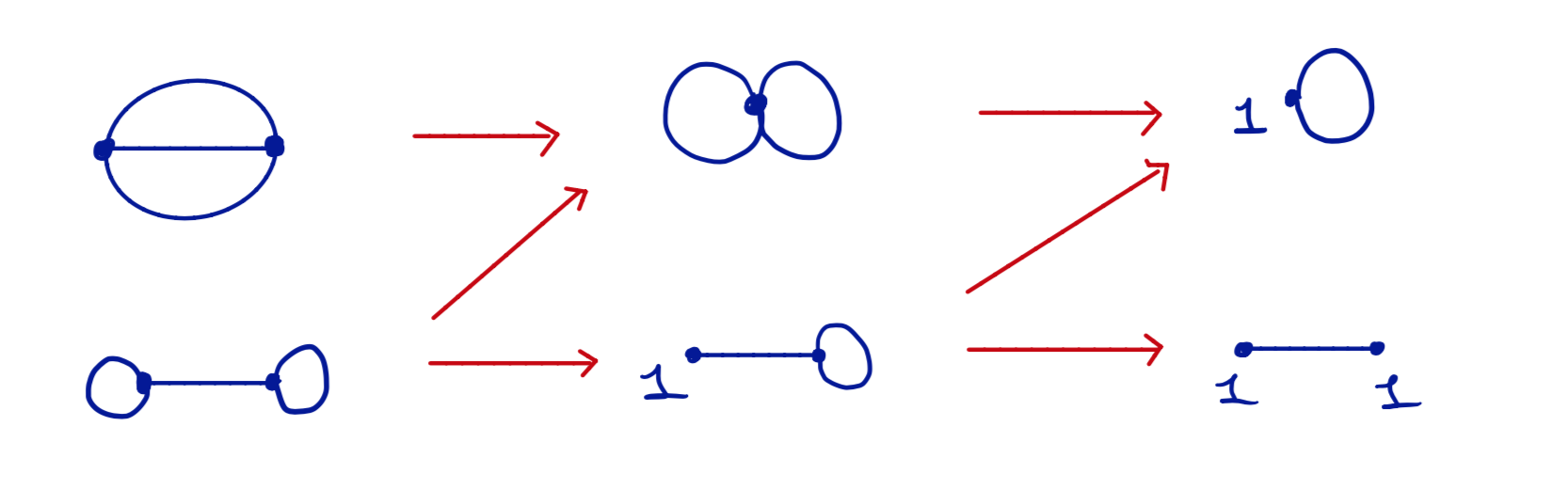}} 
\end{center}
\caption{Illustration of the category $I_2$. Unlabelled vertices have weight $0$. Automorphisms of graphs are not depicted, only edge contractions.}
\end{figure}

An isomorphism of weighted graphs $(G,w) \cong (G',w')$ is in particular an isomorphism of graphs $G \cong G'$ and induces a bijection between the corresponding sets of edges. This gives rise to a  linear isomorphism $\Pro^{E_G} \cong \Pro^{E_{G'}}$ and hence a   canonical isomorphism $\mathcal{F}(G) \cong \mathcal{F}(G')$ in $\PLC_{\Q}$. Consequently,    the restriction of $\mathcal{F}$ to stable, connected graphs defines a functor.

\begin{defn} 
\label{defnModuliTropCurves}The link of the moduli space of tropical curves (in the category $\PLC_{\Q}$)   is  the polyhedral linear complex defined by the functor 
\[  \LM_{g}^{\trop} : I_{g}^{\mathrm{opp}} \To  \PLC_{\Q} \] 
which to any weighted graph $(G,w)$ assigns the standard simplex (example \ref{ex: standardsimplex}):
\[  \LM_{g}^{\trop} (G,w)  =  \mathcal{F}(G)\ .\]
 Note that $\mathcal{F}(G)$ does not depend on the weighting of $G$.
\end{defn} 
 Its topological realisation $| \LM_g^{\trop}|$  \eqref{def: topreal}  is what is usually called  the `link of the moduli space of tropical curves'. In this paper, we use this phrase to describe   $\LM_{g}^{\trop}$.

 \begin{figure}[h]\begin{center}
\quad {\includegraphics[width=8cm]{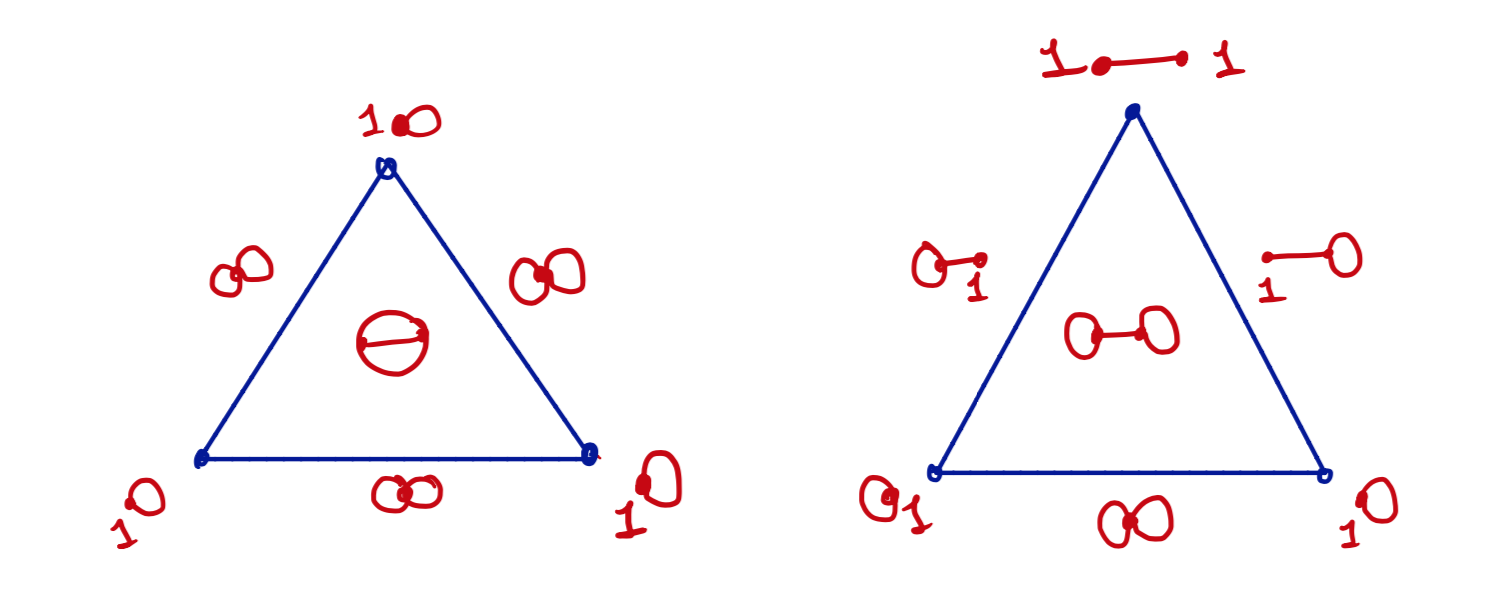}} 
\end{center}
\caption{A picture of  $|\LM_2^{\trop}|$. Faces with common labels are identified. The simplex on the left is to be quotiented by the action of its symmetry group $\Sigma_3$; the one on the right by $\Sigma_2$.
}
\end{figure}

\begin{rem} \label{rem: Mgtrop} The definitions are easily adapted to the moduli space of tropical curves, rather than the link of its cone point.  Let  $\widetilde{I}_{g}$ denote the category    $I_g$, 
with an additional, final, object given by    a single vertex of weight $g$. Consider the functor $\mathcal{M}_g^{\trop}$ from $\widetilde{I}_{g}^{\mathrm{opp}}$ to $\PC_{\Q}$:
\[ (G,w) \mapsto   \left( \mathbb{A}^{E_G}, V\left(\prod_{e\in E_G} \alpha_e\right)   , \widehat{\sigma}_G  \right)\ ,  \]
where $\widehat{\sigma}_G$  is the cone $\alpha_e\geq 0$.  The final object in $\widetilde{I}^{\opp}_{g}$ maps to $ \left(\mathrm{Spec}(\Q),   \mathrm{Spec}(\Q),    \{pt\}\right)$, whose image  in 
 $|\mathcal{M}_g^{\trop}|$ is the cone point.
The  functor $\mathcal{M}_g^{\trop}$ lands in a subcategory of $\PC_{\Q}$ which is an analogue of $\PLC_{\Q}$ defined using affine spaces. Since 
$\left|\mathcal{M}_g^{\trop}\right|$ is topologically trivial, these considerations will not be pursued further in this paper.
\end{rem} 

\subsection{The open moduli space and Outer Space} 
Let $\partial I_g$ denote the full subcategory of $I_g$ consisting of graphs of total weight $w(G)>0$. 
Consider the functor
\[   \partial  \LM^{\trop}_g  :  \partial I_g^{\opp}   \To    \PLC_{\Q} 
 \]   
obtained by restricting $\LM_g^{\trop}$ to  $\partial I_{g}^{\opp}$. It sends $G$  to $\mathcal{F}(G)$. 

\begin{defn} \label{defn: OpenModuliSpace} 
Let us denote its topological realisation by 
\[    \left|  \partial \LM_g^{\trop} \right|  =  \varinjlim_{ G\in  \partial I^{\opp}_g}  \sigma_G 
 \] which canonically embeds in $  \left|  \LM_g^{\trop}  \right| $. Let
 $ \left|   \LM_g^{\circ, \trop} \right| =  \left|  \LM_g^{\trop}\right| \setminus \left|  \partial \LM_g^{\trop} \right|$ denote its open complement.
 \end{defn}
 
 The open  moduli space 
is obtained by   gluing together the open interiors of simplices $\sigma_G$ for  graphs  $G$ which have  total weight $0$:
\[  \left|  \LM_g^{\circ, \trop} \right|  = \mathrm{Im} \left(\bigcup_{ G\in I^{\opp}_g, w(G) =0} \overset{\circ}{\sigma}_G \To    \left|  \LM_g^{\trop}\right|\right) \ ,\]
and is equipped with the subspace topology. 
Note that the open moduli space $ \left|  \LM_g^{\circ, \trop} \right| $ is not the topological realisation of a polyhedral complex since it is not  closed.
It is  the quotient of Outer space $\mathcal{O}_g$  \cite{CullerVogtmann} by the action of $\mathrm{Out}(F_g)$.

\subsection{The graph locus}
Let $G$ be a graph.
The graph polynomial for a connected graph $G$ (with all vertices of weight zero) is  defined to be the polynomial 
\[ \Psi_G = \sum_{T\subset G} \prod_{e\notin T} x_e   \quad \in \quad  \Z [ x_e, e\in E_G] \]
where the sum is over all spanning trees  $T$ of $G$.  It is homogeneous of degree $h_G$. 
For a graph $G$ with connected components $G_1,\ldots, G_n$, we define 
\[ \Psi_G = \prod_{i=1}^n \Psi_{G_i} \ .\]
If a graph $G$ has a vertex  $v$ with $w(v)>0$, i.e., the total weight  $w(G)$ is positive, then we define $\Psi_G =0$.
The graph locus $X_G \subset \Pro^{E_G}$ is defined to be the zero locus of  $\Psi_G$.
It is   a hypersurface if $G$ has total vertex weight $0$, but equals $\Pro^{E_G}$ otherwise.

\begin{prop}  \label{prop: GraphLocusFunctor}
The map $G \mapsto X_G$  defines a functor 
\[  \mathcal{X}:  \left(I_g\right)^{\opp}  \To  \mathrm{Sch}_{\Q}    \ . 
\]
It is a subscheme functor of $\LM_g^{\trop}$, i.e., of  $\PF\mathcal{F} : \left(I_g\right)^{\opp}   \rightarrow \mathrm{Sch}_{\Q}$ which sends $G\mapsto \Pro^{E_G}$,   (definition \ref{defn: Ffunctor}).   Its topological complement is the open moduli space:
 \[ \left|  \LM_g^{\circ, \trop}  \right|    =    \left| \LM_g^{\trop}\right| \setminus  \left(   \left| \LM_g^{\trop}\right| \cap  \left| \mathcal{X}(\R) \right|   \right)  
=
 \varinjlim_{x\in I_g^{\opp}} \sigma_G \backslash  (\sigma_G \cap  X_G(\R)  )  \ .  \nonumber  
 \]
 Equivalently, one has
$ \left| \mathcal{X}(\R) \right|  \cap | \LM_g^{\trop} |= |  \partial \LM_g^{\trop} |.$
\end{prop} 
\begin{proof} 
An isomorphism $(G,w) \cong (G',w')$ of weighted graphs induces a bijection on edge sets, and gives  a linear isomorphism $\Pro^{E_G} \cong \Pro^{E_{G'}}$  and an isomorphism $X_G \cong X_{G'}$. 
The image of an edge contraction  $G\rightarrow G/e$ under the functor $\PF \mathcal{F}$ is the inclusion $\Pro^{E_{G/e}} \rightarrow \Pro^{E_G}$ whose image is the coordinate hyperplane $x_e=0$.
If $e$ is not a self-edge, there is a one-to-one correspondence between the set of spanning trees of $G$ containing $e$, and the spanning trees of $G/e$, which implies that 
$\Psi_G|_{x_e=0} =  \Psi_{G/e}$,
and therefore $X_{G/e}= X_G \cap V(x_e)$.  This formula holds also for $e$ a self-edge ($G$ has no spanning trees containing $e$), in which case $G/e$ has positive weight, and both $\Psi_G|_{x_e=0}$ and  $\Psi_{G/e}$ vanish.  
This proves that $\mathcal{X}$ indeed defines a subscheme of $\PF\mathcal{F}$. 
For the last part,    observe that $\sigma_G \subset  X_G(\R) $ if $w(G) > 0$ and that 
\[
\overset{\circ}{\sigma}_G \cap X_G(\R) = \emptyset \quad \hbox{if} \quad w(G) = 0 \ . \] 
This  follows from the fact  that if $w(G)=0$ then $\Psi_G$ is  a  non-trivial sum of positive monomials, and so
$\Psi_G>0$  for $\alpha_e>0$.  The open faces of $\sigma_G$ are in one-to-one correspondence with the $\overset{\circ}{\sigma}_{G/\gamma}$ for all $\gamma \subset G$.  It follows that $\sigma_G \cap X_G(\R)$ 
is the union of the faces $\sigma_{G/\gamma}$  where $h_{\gamma}>0$, or equivalently, such that $w(G/\gamma)$ is  positive.  \end{proof}

\subsection{The graph complex $\GC^{\stab}_0$} 
Consider the  version of the graph complex \cite{Kontsevich93} defined to be the $\Q$-vector space generated by oriented, stable  unweighted graphs $[G,\omega]$. 
It is bigraded by genus $g$, and the number of edges $e_G$, which is the homological degree.  An orientation $\omega$ on $G$ is an element of $\bigwedge (\Z^{E_G})^{\times}$, i.e., an order of the edges up to even permutations. One has relations $[G,-\omega]=-[G,\omega]$ and $[G, \omega] = [ G', \omega']$ if $G \cong G'$ is an isomorphism which  sends  $\omega$ to $\omega'$. The differential is:
\[ d [G, e_1\wedge \ldots \wedge e_n ] = \sum_{e} (-1)^i [G\q e_i, e_1 \wedge \ldots \widehat{e_i} \wedge \ldots \wedge e_n ] \]
where $G\q e$ denotes the graph in which the edge $e$ is contracted, and  is the empty graph if $e$ is a self-edge.  One shows that $\GC^{\stab}_0$ is quasi-isomorphic to the complex usually denoted by $\GC_0$.  
The following proposition  is proven in \cite{CGP} and also follows from  theorem \ref{thm: RelativeCellular}.
\begin{prop} \label{prop: relativefacecomplexGC0}
The relative face complex $\mathfrak{C}(\LM_g^{\trop})/\mathfrak{C}(\partial \LM_g^{\trop}) $ is the graph complex $\GC^{\stab}_0[-1]$ where the homological degree is given by numbers of edges $-1$.  Thus:
\[ H_n (\GC_0 ) \cong \bigoplus_{g\geq 1}  H_{n-1}( |\LM_g^{\trop}|, \left| \partial \LM_g^{\trop} \right|;\Q)\ .\]
\end{prop} 
A more profound statement is the fact, due to \cite{CGP},  that 
$\left| \partial \LM_g^{\trop} \right|$ is contractible. Therefore 
the graph complex also computes the reduced homology of $|\LM_g^{\trop}|$.

\section{The bordification    $\mathcal{M}_{g}^{\trop,\BB}$ of $\mathcal{M}_{g}^{\trop}$ } \label{section: MgBB}
We review the blow-ups associated to graphs, and define  $\mathcal{M}_{g}^{\trop,\BB}$     by gluing together the blow-ups of  polyhedral configurations of  stable, weighted graphs.

\subsection{Blow-ups and Feynman polytopes}A core graph is one which is bridgeless. 
\begin{defn} Let $G$ be a  graph. Consider the set of subspaces
\[ \BB^G=\BB^{G, \mathrm{core}} = \{ \Pro(\Q^{E_{G/\gamma}}) :   \gamma \subset G \quad \hbox{core subgraph} \}\] 
which, one verifies (see e.g. \cite[\S7]{BEK},  \cite[\S5-6]{Cosmic}), satisfies $(B1), (B2)$.\footnote{ 
If we set $\BB^{G,\max}= \{  \Pro(\Q^{E_{\gamma}}): \hbox{ all } \gamma \subset G\}$, then one may verify that $(\BB^{G,\max})^{\min} = \BB^{G, \mathrm{core}}$. Therefore the core blow-up is  minimal, and proposition \ref{prop: minimalblow-down}  defines  a morphism  $P^{\BB^{G,\max}} \rightarrow P^{\BB^{G,\mathrm{core}}}.$
}
 Define 
\[ \mathcal{F}^{\BB}(G) =    \left(P^{\BB^G},  L^{\BB^G}, \sigma^{\BB^G} \right)\]
to be the object in $\BLC_{\Q}$ obtained by blowing up the linear subspaces corresponding to core subgraphs.  It is equipped with a  canonical blow-down morphism in $\PC_{\Q}$: 
\[ \pi^{\BB} :  \mathcal{F}^{\BB}(G)  \To  \mathcal{F}(G)\ .\]
\end{defn}

A key point is that the set $\BB^G$ is intrinsic to $G$, which is why we shall drop the superscript $G$ with impunity.  By this we mean the following. Let $\gamma \subset G$ be a core subgraph. Then, employing the notation \eqref{BBrestrictions}, we have canonical identifications 
\[ \left(\BB^G\right)_{ /\Q^{E_{G/\gamma}}}    =   \BB^{\gamma} \qquad  \hbox{ and } \qquad 
   \left(\BB^G\right)_{\Q^{E_{G/\gamma}}} =  \BB^{G/\gamma} \]
since, for the first equation, there is a bijection between the set of  core subgraphs $g\subset G$ which are contained in $\gamma$,  and  the set of core subgraphs of $\gamma$  (the property of being a core subgraph is intrinsic, and does not depend on the ambient graph).  
For the second equation, there is a bijection between core subgraphs of $G$ which contain $\gamma$, and core subgraphs of $G/\gamma$. 
 For all  $e\in E_G$ which is not a self-edge,  there is a face morphism 
\begin{equation} \label{F1}   \mathcal{F}^{\BB}(G/e)  \To  \mathcal{F}^{\BB}(G)\end{equation}
in the category $\BLC_{\Q}$. It  is compatible, via blow-down, with the corresponding map   $ \mathcal{F}(G/e)  \rightarrow  \mathcal{F}(G)$. 
 Exceptional divisors are  indexed by core subgraphs $\emptyset \neq \gamma \subset G$. For each such subgraph, there is a canonical face morphism:
\begin{equation} \label{F2}  \mathcal{F}^{\BB}(\gamma) \times  \mathcal{F}^{\BB}(G/\gamma)   \To  \mathcal{F}^{\BB}(G) \ . 
\end{equation}
Compatibility with the face map    $ \mathcal{F}(G/\gamma)  \rightarrow  \mathcal{F}(G)$ is expressed by the diagram in $\PC_{\Q}$:
\[\begin{array}{ccccc}
\mathcal{F}^{\BB}(\gamma)&  \times &  \mathcal{F}^{\BB}(G/\gamma) &   \To &  \mathcal{F}^{\BB}(G)  \nonumber \\
  &  \downarrow  & & &  \downarrow  \nonumber \\
      &  \mathcal{F}(G/\gamma)&   &   \To & \mathcal{F}(G)  \nonumber 
\end{array} 
\]
which commutes (see remark \ref{rem: blowdowncontractsnormal}). The facets of $\mathcal{F}^{\BB}(G)$ are in one-to-one correspondence with either non-self-edge edges \eqref{F1}, or core subgraphs \eqref{F2}.

\subsection{Sequences of graphs and categorical formulation of the blow-up}

\begin{defn} \label{defnNestedGraphs} Consider a  sequence of  strictly  nested (non-empty) subgraphs 
\begin{equation}  \label{NestedGraphs}   (\gamma_1, \gamma_2, \ldots, \gamma_{n-1}, \gamma_n)  \ \hbox{ such that }  \ \gamma_i \subsetneq \gamma_{i+1}  \hbox{ for } 1\leq i \leq n-1\ , 
\end{equation}
which are not necessarily connected (even if $\gamma_n$ is).  Such a sequence may equivalently be viewed as the data of a strict  filtration $F_{\bullet}$ on   $G= \gamma_n$, where $F_i G = \gamma_i$. 
The  \emph{graded sequence} associated to this filtration  is  the sequence of graphs:
\[   (\gamma_1, \gamma_2/\gamma_1, \ldots, \gamma_n/\gamma_{n-1})\ . \]
Let $e \in E_{\gamma_n}$ be an edge. There is a unique index  $k$ such that $e$ is an edge of $\gamma_k/\gamma_{k-1}$. We call  $e$ \emph{admissible} if it is neither a self-edge nor the only edge in the quotient $\gamma_k/\gamma_{k-1}$, in which case we define the \emph{edge contraction}   with respect to $e$ to be the sequence:
\begin{equation} \label{EdgeContractSequence} (\gamma_1,\ldots, \gamma_n)/e \ =  \  \left( (\gamma_1 \cup e)/e \ , \   (\gamma_2\cup e)/e \ , \  \ldots \  , \  (\gamma_{n-1}\cup e) /e \  , \  \gamma_n/e \right)  
\end{equation} 
Define a \emph{refinement} of $(\gamma_1, \gamma_2, \ldots, \gamma_{n-1}, \gamma_n) $  to be: $(\gamma, \gamma_1,\ldots, \gamma_n)$ where $\gamma \subsetneq \gamma_1$ or
\[ (\gamma_1,\ldots, \gamma_i, \gamma, \gamma_{i+1} ,\ldots, \gamma_n) \hbox{ where } \gamma_i \subsetneq \gamma \subsetneq \gamma_{i+1}
\  \] for any $i$.
Define the 
  \emph{edge degree}  to be  $|E_{\gamma_n}|-n+1$, and the \emph{genus} to be $h_{\gamma_n}$. 
  \end{defn} 

Given a graph $G$ with labelled edges, define a category $\mathcal{C}^{\BB}_G$ whose objects are  sequences \eqref{NestedGraphs}  of edge degree $>0$ such that $\gamma_n =G/\gamma$ is a quotient of $G$ satisfying $h_{\gamma_n} = h_G$ (i.e., $\gamma$ does not contain any loops), and such that each  graph $\gamma_i$, $i<n$ is  core. The morphisms are given by contraction of admissible edges, and refinements (insertion of an additional graph in the nested sequence as above): indeed, $\mathcal{C}^{\BB}_G$ is simply the category of objects of edge degree $\geq 1$ generated by the sequence  $(G)$ under these two operations.  See figure \ref{fig: 8}.

\begin{thm} \label{thm: FeynmanBlowUpsFunctor}
There is a canonical functor $
\left(\mathcal{C}^{\BB}_G\right)^{\opp}  \rightarrow  \BLC_{\Q} $ which sends
\[ (\gamma_1,\gamma_2,\ldots, \gamma_{n-1}, \gamma_n)  \mapsto  \mathcal{F}^{\BB}(\gamma_1) \times   \mathcal{F}^{\BB}(\gamma_2/\gamma_1) \times \ldots 
\times \mathcal{F}^{\BB}(\gamma_n/\gamma_{n-1})  \] 
and all morphisms to  face morphisms.  The objects of $\mathcal{C}^{\BB}_G$ are in one-to-one correspondence with the faces of $\mathcal{F}^{\BB}(G)$; the morphisms of $\mathcal{C}^{\BB}_G$ are in one-to-one correspondence with inclusions of faces.
The  face corresponding to a sequence  $(\gamma_1,\ldots ,\gamma_n)$ has codimension $e_G - e_{\gamma_n} +n-1$.
There is a canonical blow-down  in $\PC_{\Q}$ from the functor 
\[ \mathcal{F}^{\BB}: \mathcal{C}^{\BB}_G \rightarrow \BLC_{\Q} \subset \PC_{\Q}   \qquad \hbox{ to } \qquad  \mathcal{F}: \mathcal{C}_G \rightarrow \PLC_{\Q}   \subset \PC_{\Q} \]
which is induced by  the pair $(\phi, \Phi)$, where the functor $\phi: \mathcal{C}^{\BB}_G  \rightarrow  \mathcal{C}_G$ is defined by 
\[\phi: (\gamma_1, \gamma_2,\ldots, \gamma_n)  \mapsto  \gamma_n/\gamma_{n-1} \]  
 and  $\Phi$ is the natural transformation defined by the family of morphisms in $\PC_{\Q}$:
\[   \mathcal{F}^{\BB}(\gamma_1) \times   \mathcal{F}^{\BB}(\gamma_2/\gamma_1) \times \ldots 
\times \mathcal{F}^{\BB}(\gamma_n/\gamma_{n-1})   \To   \mathcal{F}(\gamma_n/\gamma_{n-1}) \]
 given by projection onto the final component  followed by blow-down for $  \mathcal{F}^{\BB}(\gamma_n/\gamma_{n-1}) $.
\end{thm}
 
\begin{proof}  
To verify that the map in the statement is indeed a functor, we must check that the following diagram commutes, where $e$ is an edge of $\gamma_i$ but not of $\gamma_{i-1}$ for $i\geq 2$:
\[
\begin{array}{ccc}
 (\gamma_1,\gamma_2,\ldots, \gamma_{n-1}, \gamma_n) &   \To   &  \mathcal{F}^{\BB}(\gamma_1) \times   \mathcal{F}^{\BB}(\gamma_2/\gamma_1) \times \ldots 
\times \mathcal{F}^{\BB}(\gamma_n/\gamma_{n-1}) \\
 \downarrow  &   & \uparrow   \\
  ((\gamma_1 \cup e)/e, \ldots  (\gamma_n\cup e)/e) & \To    &   \mathcal{F}^{\BB}(\gamma_1) \times \ldots   \times \mathcal{F}^{\BB}(\gamma_i/(\gamma_{i-1} \cup e)) \times \ldots 
\times \mathcal{F}^{\BB}(\gamma_n/\gamma_{n-1})
\end{array}
\]
This follows from the fact that $ \mathcal{F}^{\BB}(G/e) \rightarrow   \mathcal{F}^{\BB}(G)$ is a face  map \eqref{F1}, applied to $G= \gamma_i/\gamma_{i-1}$. 
The case $i=1$ is similar.  For refinements, we must check that 
\[
\begin{array}{ccc}
 (\gamma_1, \ldots,  \gamma_n) &   \To   &  \mathcal{F}^{\BB}(\gamma_1) \times   \mathcal{F}^{\BB}(\gamma_2/\gamma_1) \times \ldots 
\times \mathcal{F}^{\BB}(\gamma_n/\gamma_{n-1}) \\
 \downarrow  &   & \uparrow   \\
  (\gamma_1, \ldots, \gamma_{i-1}, \gamma, \gamma_i, \ldots, \gamma_n ) & \To    &   \mathcal{F}^{\BB}(\gamma_1) \times \ldots    \mathcal{F}^{\BB}(\gamma/\gamma_{i-1}) \times   \mathcal{F}^{\BB}(\gamma_i/\gamma)    \ldots 
\times \mathcal{F}^{\BB}(\gamma_n/\gamma_{n-1})
\end{array}
\]
commutes also. This follows from the fact that $ \mathcal{F}^{\BB}(\gamma/\gamma_{i-1}) \times   \mathcal{F}^{\BB}(\gamma_i/\gamma)   \rightarrow    \mathcal{F}^{\BB}(\gamma_i/\gamma_{i-1})  $ is the face map corresponding to the inclusion of an exceptional divisor \eqref{F2}. 
The statement follows from proposition \ref{prop: BlowUpStructure}: the exceptional divisors of $\mathcal{F}^{\BB}(\gamma_n)$  are indexed by strict core subgraphs of $\gamma_i$, and so an intersection of exceptional divisors  is indexed by nested sequences of core subgraphs. The description of the blow-down map  follows from remark \ref{rem: blowdowncontractsnormal}.
\end{proof} 

On topological realisations, $|\Phi|$ induces the blow-down morphism 
$ \left| \mathcal{F}^{\BB}(G) \right| \rightarrow \left| \mathcal{F}(G) \right| \    $
from $ \left| \mathcal{F}^{\BB}(G) \right|$ to the closed simplex. The space $ \left| \mathcal{F}^{\BB}(G) \right|$ is called    the Feynman polytope in  \cite[\S6]{Cosmic}, which provides a similar description of its faces in greater generality.

\subsection{Canonical blow-up of $\LM_g^{\trop}$.} \label{sect: canblowupG}
Consider the 
category  whose objects are the faces of all $\mathcal{F}^{\BB}(G)$, where $G$ ranges over stable connected graphs of total  weight $0$, and whose morphisms are generated by isomorphisms and  face morphisms. By construction, it admits a  canonical functor  to $\BLC_{\Q}$.  The complex 
 $\LM_g^{\trop, \BB}$ is defined to be this functor.   It may be described more explicitly using nested sequences  of graphs.

\begin{defn} \label{defn: IgBB} Define a category $I_{g}^{\BB}$ whose objects are nested sequences \eqref{NestedGraphs} of graphs  of genus $h_{\gamma_n}=g$, with the property  
that each $\gamma_i$, for $i<n$, is a core graph (but not necessarily connected), and $\gamma_n$ is a stable connected  graph with all vertices of weight $0$. The morphisms in this category  are isomorphisms, admissible edge contractions, and refinements. 
\end{defn}

 \begin{figure}[h] \begin{center}
\quad {\includegraphics[width=10cm]{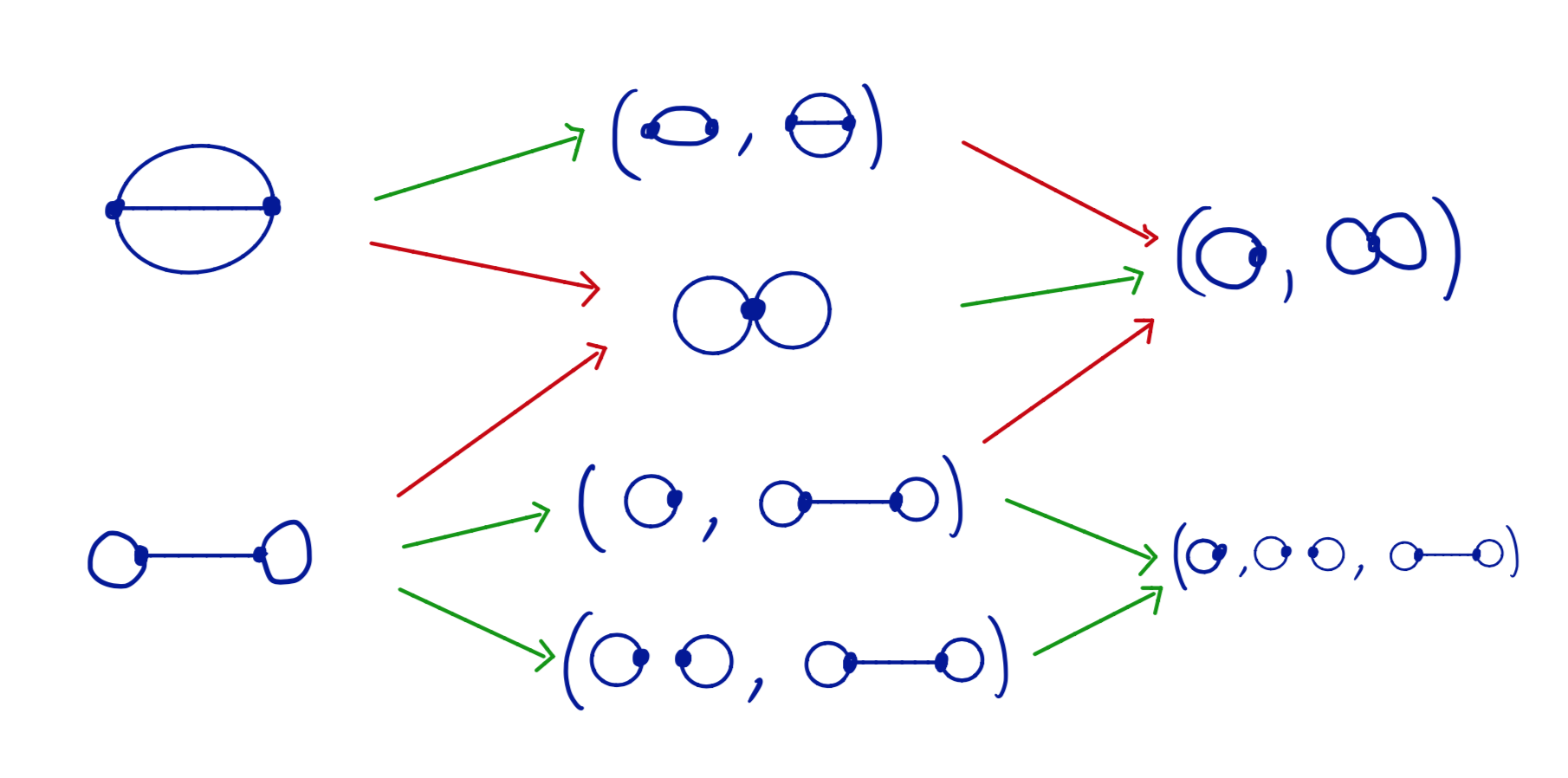}} 
\end{center} 
\caption{A picture of  $I_2^{\BB}$. All vertices have weight $0$.  Edge contractions are depicted in red,  refinements in green, and automorphisms are not shown. }\label{fig: 8}
\end{figure}

There is a `collapsing' functor:
\begin{eqnarray} \label{collapsing} 
I_{g}^{\BB}  & \To &  I_g \\
(\gamma_1,\ldots,\gamma_{n-1},  \gamma_n) & \mapsto & \gamma_n/\gamma_{n-1} \nonumber 
\end{eqnarray} 
which sends a filtered graph to its highest graded component, where
the quotient 
$ \gamma_n/\gamma_{n-1}$
is to be viewed as a weighted graph as follows. First assign weight $0$ to every vertex of $\gamma_n$, and  contract each edge of $\gamma_{n-1}$, in any order,
keeping track of the induced weights by the process described in \S\ref{subsect: WeightedGraphs}.
  To verify that this is indeed a functor, one only needs to consider  refinements where an additional graph $\gamma$ satisfying $\gamma_{n-1} \subset \gamma \subset \gamma_n$ is inserted between $\gamma_{n-1}$ and $\gamma_n$. It maps to the morphism in  $I_g$ given by  contracting the edges in $E_{\gamma} \backslash E_{\gamma_{n-1}}$. 
We emphasize that the collapsing functor sends \emph{unweighted} nested sequences  to \emph{weighted} graphs.
 It induces a functor between the opposite categories in the usual manner. 
The  functor  \eqref{collapsing} is essentially surjective. Furthermore, $I_g^{\BB}$ is generated  by singletons $(G)$ for $G$  an  unweighted, connected, stable graph  of genus $g$ 
 by edge contractions and refinements.   Since there are finitely many isomorphism classes,  $I_g^{\BB}$ is equivalent to a finite category.

\begin{defn} Consider the map
\begin{eqnarray}  \LM_g^{\trop , \BB}: \left( I_{g}^{\BB} \right)^{\opp} & \To&  \BLC_{\Q} \\ 
 (\gamma_1,\ldots, \gamma_n)  &\mapsto & \mathcal{F}^{\BB}(\gamma_1) \times \mathcal{F}^{\BB}(\gamma_2/\gamma_1) \times  \ldots \times \mathcal{F}^{\BB}(\gamma_n/\gamma_{n-1})  \nonumber 
 \end{eqnarray} 
 sending a nested sequence to the product of the blow-ups of its graded sequence.
\end{defn} 

\begin{thm} The map $\LM_g^{\trop , \BB}$ is a functor, and sends morphisms in $(I_g^{\BB})^{\opp}$  to isomorphisms and face maps. 
There is a  canonical  morphism of functors in $\PC_{\Q}$: 
\[ \LM_g^{\trop , \BB} \To \LM_g^{\trop }\]
which is induced by the pair $(\phi^{\opp},\Phi)$  where $\phi:I_{g}^{\BB}   \rightarrow    I_g$   is the collapsing functor, and the natural transformation $\Phi$ is defined on the image of $(\gamma_1,\ldots, \gamma_n)$ by 
\[  \Phi:  \mathcal{F}^{\BB}(\gamma_1) \times \mathcal{F}^{\BB}(\gamma_2/\gamma_1) \times  \ldots \times \mathcal{F}^{\BB}(\gamma_n/\gamma_{n-1})   \To \mathcal{F}(\gamma_n/\gamma_{n-1}) \
 , \]
namely, projection onto the last component followed by  blow down  $ \mathcal{F}^{\BB}(\gamma_n/\gamma_{n-1})   \rightarrow  \mathcal{F}(\gamma_n/\gamma_{n-1})$.
\end{thm}

\begin{proof}
Any isomorphism of nested sequences
$ (\gamma_1, \gamma_2, \ldots, \gamma_n ) \cong  (\gamma'_1, \gamma'_2, \ldots, \gamma'_n )$
induces an isomorphism between the associated graded sequences, and hence a canonical isomorphism in the category $\BLC_{\Q}$ between 
\[\mathcal{F}^{\BB}(\gamma_1) \times \mathcal{F}^{\BB}(\gamma_2/\gamma_1) \times  \ldots \times \mathcal{F}^{\BB}(\gamma_n/\gamma_{n-1}) \]
and its version with each $\gamma_i$ replaced with $\gamma'_i$.  Furthermore, this isomorphism is compatible, via blow-downs, with the isomorphism
$\mathcal{F}(\gamma_n/\gamma_{n_1}) \cong  \mathcal{F}(\gamma'_n/\gamma'_{n_1}).$
The fact that $\LM_{g}^{\trop, \BB}$ is a functor then follows  the proof of theorem \ref{thm: FeynmanBlowUpsFunctor}, as does the rest of the statement.
\end{proof} 
The topological realisation
$ \left|   \LM_g^{\trop , \BB} \right|$
is defined differently, but  equivalent, we expect,  to the  bordification discussed in  \cite{BordificationOuterSpace}. There is a canonical continuous morphism:
\begin{equation} \label{BlowDownOnLMgTop}   \left|   \LM_g^{\trop , \BB} \right|  \To  \left|   \LM_g^{\trop } \right|  
\end{equation} 
which collapses all exceptional components.

 \begin{figure}[h]\begin{center}
\quad {\includegraphics[width=8cm]{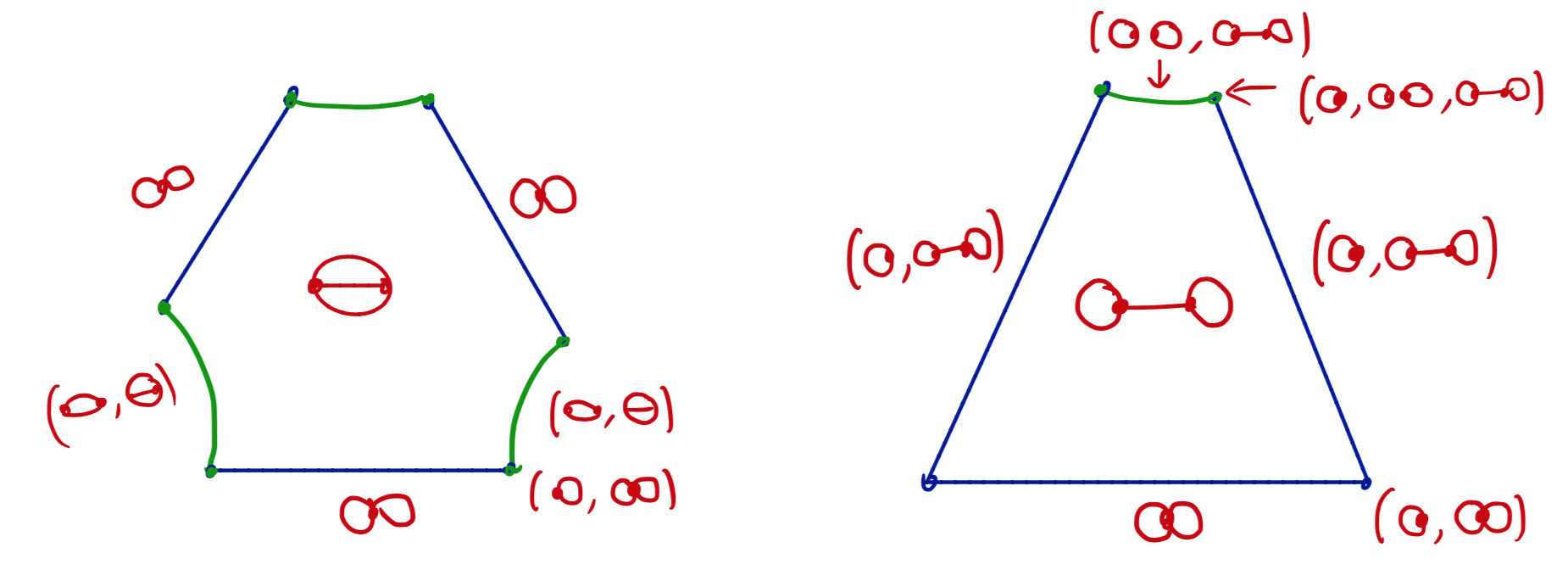}} 
\end{center}
\caption{A picture of  $|\LM_2^{\trop,\BB}|$. Faces with isomorphic labels are to be identified.   We have not labelled all faces for reasons of space.}
\end{figure}

\section{The strict transform of the graph locus in $\LM_g^{\trop,\BB}$} \label{sect: strictTransfGraphLocus}
Let  $G$ be  a graph  with zero weights. Recall that  $\pi_{\BB}: P^{\BB_G}\rightarrow \Pro^{E_G}$ is the corresponding blow-up and 
 $X_G    \subset \Pro^{E_G}$ is the graph hypersurface. Denote its strict transform by  
\[ \widetilde{X}_G \subset  P^{\BB_G}  \] 
and 
let  
$ U_G \subset  P^{\BB_G}$ 
denote its open complement $P^{\BB_G} \backslash \widetilde{X}_G$. 
Recall from theorem \ref{thm: FeynmanBlowUpsFunctor} that $ \left( \mathcal{C}^{\BB}_G \right)^{\opp}$ is equivalent to the category of faces of the blown-up Feynman polytope $\mathcal{F}^{\BB}_G$. 

\begin{defn} Define a functor
\[ \widetilde{\mathcal{X}}:  \left( \mathcal{C}^{\BB}_G \right)^{\opp} \To \mathrm{Sch}_{\Q} \]
which sends  the singleton $(G)$ to the strict transform $\widetilde{X}_G \subset  P^{\BB_G}$. It is  uniquely defined on all other objects \eqref{NestedGraphs} by restriction to faces (i.e., by intersecting with the corresponding divisors in $P^{\BB_G}$), and defines a subscheme functor of $\PF \mathcal{F}^{\BB}$. 
\end{defn} 

An isomorphism of graphs (or nested sequence of graphs) induces an isomorphism of graph hypersurfaces and hence of their strict transforms by 
 proposition \ref{prop: UniversalANDIdeal} (i).    Since $I^{\BB}_g$ is generated by stable unweighted graphs of genus $g$, we deduce the existence of a  functor 
 \[ \widetilde{\mathcal{X}}:  \left( I^{\BB}_g \right)^{\opp} \To \mathrm{Sch}_{\Q} \]
which is a subscheme functor of $\LM_g^{\trop,\BB}$.

\begin{prop}  \label{prop: BlowUpGraphLocusFunctor}
The functor 
$   \widetilde{\mathcal{X}}:  \left( I^{\BB}_g \right)^{\opp} \rightarrow \mathrm{Sch}_{\Q} $
is given on sequences \eqref{NestedGraphs} by
\begin{equation} \label{nestedtoGraphLocus}   (\gamma_1,\ldots, \gamma_n)    \ \mapsto \   \bigcup_{i=1}^n P^{\gamma_1} \times \ldots  \times P^{\gamma_{i-1}/\gamma_{i-2}}\times \widetilde{X}_{\gamma_i/\gamma_{i-1} } \times P^{\gamma_{i+1}/\gamma_{i}}\times  \ldots \times P^{\gamma_n/\gamma_{n-1}}        \ .
\end{equation}
It is a closed subscheme of $\LM_g^{\trop,\BB}$ at infinity, i.e., a subscheme of the functor $\PF\LM_g^{\trop,\BB}$, whose   real points do  not meet the topological realisation: 
\begin{equation} \label{XtildeavoidsLMgtropBB}   \widetilde{\mathcal{X}}(\R)   \cap   |\LM_g^{\trop,\BB}| = \emptyset\ .  
\end{equation}
The blow-down $(\phi,\Phi) : \LM_g^{\trop,\BB} \rightarrow \LM_g^{\trop}$ induces  
  a natural transformation
$\widetilde{\mathcal{X}} \rightarrow \mathcal{X}\circ \phi$, where $\phi$ is the collapsing functor.  
\end{prop} 
\begin{proof} 
If $e \in E_G$ is not a self-edge, there is a canonical inclusion $\widetilde{X}_{G/e} \rightarrow \widetilde{X}_G$ induced by the inclusion $P_{G/e}^{\BB}\rightarrow P_G^{\BB}$, since this is true for graph hypersurfaces and 
strict transforms are functorial (proposition \ref{prop: UniversalANDIdeal} (i)).  This proves the formula \eqref{nestedtoGraphLocus} for sequences of length $n=1$. 

The general case follows by taking refinements.  Let $\gamma \subset G$ be a core subgraph of $G$. One has the following identity of graph polynomials   \cite[(2.2)]{Cosmic}  
\begin{equation} \label{Psifactorizes} \Psi_G = \Psi_{\gamma} \Psi_{G/\gamma} + R_{G/\gamma}
\end{equation} 
where $R_{G/\gamma}$ is a polynomial of homogeneous degree strictly greater than $\deg \Psi_{\gamma}= h_{\gamma}$ in the edge variables $\alpha_e$ for $e \in E_{\gamma}$. The identity  implies that 
 if $D \subset P^G$ is the exceptional divisor corresponding to the blow-up of $\Pro^{E_\gamma}\subset \Pro^{E_G}$, one has a canonical isomorphism:
\[  \widetilde{X}_G \cap  D  \cong  \left( \widetilde{X}_{\gamma} \times P^{\BB}_{G/\gamma} \right) \   \cup    \ \left( P^{\BB}_{\gamma} \times \widetilde{X}_{G/\gamma} \right) \ . \]
See also \cite[(3.7)]{BEK}.  The right-hand side of   \eqref{nestedtoGraphLocus} follows by iterating this isomorphism.
The fact that $\widetilde{\mathcal{X}}$ is at infinity follows from the fact that the strict transform locus $\widetilde{X}_G(\R)$ does not meet the region $\sigma^{\BB}_G$ \cite[Proposition 7.3 (iii)]{BEK}, \cite[\S6.5]{Cosmic}.
 The final statement follows by definition since $X_G$ is the blow-down of $\widetilde{X}_G$.
\end{proof} 
 The properties of graph hypersurfaces required in the proof will actually be rederived in \S\ref{section:  DetLocus} from a  more conceptual viewpoint via properties of the determinant. 

The  open complement  of  the subscheme $\widetilde{\mathcal{X}}$  is the functor
\begin{eqnarray} \mathcal{U}: \left(I^{\BB}_g\right)^{\opp}& \To&  \mathrm{Sch}_{\Q}   \nonumber \\
 (\gamma_1,\ldots, \gamma_n)   \  &\mapsto&  \   U_{\gamma_1} \times U_{\gamma_2/\gamma_1} \times \ldots \times U_{\gamma_n/\gamma_{n-1}}  \nonumber 
\end{eqnarray} 
It is an open subscheme of $\PF \LM_g^{\trop,\BB}$ and  satisfies $ |\LM_g^{\trop,\BB}|  \subset \mathcal{U}$ (definition \ref{defn: subschemeatinfinity}).

\begin{cor} Via definition \ref{defn: OpenSubComplex} we deduce    a functor 
\begin{eqnarray} 
\LM_g^{\trop,\BB} \backslash \widetilde{\mathcal{X}} \  : \   \left( I_{g}^{\BB} \right)^{\opp}  &\To&  \PC_{\Q} \nonumber  \\
x & \mapsto & (\mathcal{U}_x,  L^{\BB} \cap \mathcal{U}_x, \sigma_x) \nonumber
\end{eqnarray} 
\end{cor} 

\begin{rem} A functorial system of \emph{linear} subspaces of  $ \PF \mathcal{F}(G)$ was defined in \cite[\S5.4]{Cosmic} whose complement in $P^G$ is affine.  It defines a subscheme functor  $\mathcal{A} \subset \LM_g^{\trop, \BB}$  such that $\left| \LM_g^{\trop, \BB}\right| \subset  P^G \backslash \mathcal{A}$. It would be very interesting to  construct natural differential forms on $P^G \backslash \mathcal{A}$. They would have linear  poles (by contrast with  the canonical forms studied here.)
\end{rem} 

\subsection{The boundary and open locus}
\begin{defn}
Let $\partial I_{g}^{\BB}$ denote the full subcategory of $I_g^{\BB}$ whose objects are sequences of graphs $(\gamma_1,\ldots, \gamma_n)$ where $n\geq 2$.
Denote  the restriction of  $\LM_g^{\BB}$ to this category by
\[ \partial \LM_g^{\BB} : \partial  I_{g}^{\BB} \To \BLC_{\Q}\ . \]
  We shall call it the  boundary locus, or exceptional locus,  of $\LM_{g}^{\trop, \BB}$. 
\end{defn}

The open locus $|\LM_g^{\circ, \trop}|$ embeds canonically into both $|\LM_g^{\trop}|$ and $|\LM_g^{\trop,\BB}|$. 

\begin{prop}  \label{prop: EmbedOpentoBlowup} There is a morphism of $\PC_{\Q}$-complexes
\[ \partial \LM_g^{\trop,\BB}   \To \partial \LM_g^{\trop}   \]
given by the pair $(\phi^{\opp}, \pi^{\BB})$, where $\phi :  \partial I_{g}^{\BB} \rightarrow   \partial I_{g}$ is the restriction of the collapsing functor, and $\pi^{\BB}$ is the canonical blow-down map.
In addition,  there is a canonical embedding  
\begin{equation} \label{EmbedOpentoBlowup}   \left|  \LM_g^{\circ, \trop}\right| \overset{\sim}{\To} \left|   \LM_g^{\trop,\BB} \right|  \setminus  \left| \partial \LM_g^{\trop,\BB}\right|  
\end{equation} 
whose inverse is   the blow-down  
$  \left|   \LM_g^{\trop,\BB} \right|  \setminus  \left|  \partial \LM_g^{\trop,\BB} \right| \overset{\sim}{\To}  \left|   \LM_g^{\trop} \right|  \setminus  \left|   \partial \LM_g^{\trop}  \right|$.
\end{prop} 

\begin{proof}
The functor  $\phi$ is  given by restricting the collapsing functor  \eqref{collapsing} to nested sequences of length $\geq 2$. To see that it   lands in  the subcategory $\partial I_{g}$, note that collapsing a nested sequence of length $\geq 2$ results in a graph $\gamma_n/\gamma_{n-1}$ of positive weight,  since it involves the contraction of a core sugraph $\gamma_{n-1}$, which  has positive loop number $h_{\gamma_{n-1}}>0$.  The exceptional locus of $\mathcal{F}^{\BB}(G) \rightarrow \mathcal{F}(G)$ corresponds to precisely to sequences \eqref{NestedGraphs} of length $\geq 2$. 
 To prove \eqref{EmbedOpentoBlowup}, first define  the full subcategory $I_{g,w=0} $  of $I_g$ consisting of graphs of weight $0$, and consider the functor:
 \begin{eqnarray} 
\sigma\cap \mathcal{U} :  I^{\opp}_{g,w=0} & \To & \Top \label{inproof:sigmaU} \\
G & \mapsto &  ( \sigma\cap \mathcal{U})_G=  \sigma_G \backslash   \left( \sigma_G \cap X_G(\R) \right) \ .\nonumber 
\end{eqnarray} 
Note that $(\sigma\cap \mathcal{U})_G$ is contained in $\sigma_G$,  but  contains $\overset{\circ}{\sigma}_G$ since $(\sigma\cap \mathcal{U})_G= \bigcup_{h_{\gamma}=0} \overset{\circ}{\sigma}_{G/\gamma}$.
It follows from proposition \ref{prop: GraphLocusFunctor}  and the fact that $(\sigma\cap \mathcal{U})_G$ is the empty set if $w(G)>0$ that:
\[  \left|  \sigma \cap \mathcal{U} \right| =     \varinjlim_{G \in  I^{\opp}_{g,w=0}}\sigma_G \backslash   \left( \sigma_G \cap X_G(\R) \right) =          \left| \LM_g^{\circ, \trop} \right|  \ .\]  
 Now consider the functor  $\sigma\backslash \mathcal{E}: I_g^{\BB} \rightarrow \Top$  which sends sequences \eqref{NestedGraphs} of length $>1$ to the empty set, and maps singletons $(G)$ to
  \[ (G)  \ \mapsto  \  (\sigma\backslash \mathcal{E})_G = \sigma^{\BB}_G  \  \backslash \  \left( \sigma^{\BB}_G \cap  \mathcal{E}_G(\R)  \right) \]
 where $\mathcal{E}_G\subset P^{\BB_G}$ is the exceptional divisor of $\pi_{\BB_G}: P^{\BB_G} \rightarrow \Pro^{E_G}$. 
 The blow-down morphism  defines a functorial  isomorphism of algebraic varieties, and of topological spaces:
 \begin{eqnarray}   \pi_{\BB_G}  \  : \    P^{\BB_G} \backslash \left(\mathcal{E}_G \cup \widetilde{\mathcal{X}}_G \right) & \overset{\sim}{\To}     &  
 \Pro^{E_G} \backslash X_G  \nonumber \\
   \sigma^{\BB}_G  \  \backslash \  \left( \sigma^{\BB}_G \cap  \mathcal{E}_G(\R)  \right)      & \overset{\sim}{\To}&  \sigma_G \backslash (\sigma_G \cap X_G)  = (\sigma \cap \mathcal{U})_G
   \nonumber
 \end{eqnarray} 
 The second line follows from restriction of the first, since  $\widetilde{\mathcal{X}}_G(\R)$ does not meet $\sigma^{\BB}_G$ by \eqref{XtildeavoidsLMgtropBB}. 
  Let $j:   \sigma_G \backslash (\sigma_G \cap X_G) \overset{\sim}{\rightarrow}  \sigma^{\BB}_G  \  \backslash \  \left( \sigma^{\BB}_G \cap  \mathcal{E}_G(\R)  \right) $ denote its inverse. To define the   continuous  map \eqref{EmbedOpentoBlowup} we define a pair consisting of a functor and natural transformation, and take its limit.
 The natural transformation is given by $j$; the functor is given by
\[ 
\iota:  I_{g,w=0} \To   I_{g}^{\BB}  \qquad \hbox{where} \qquad  \iota:G \mapsto (G)  \ , 
 \] 
 and $(G)$ denotes a sequence of graphs of length one. The isomorphism $j$, which is functorial,  thus defines a  natural transformation from the functor $\sigma\cap \mathcal{U}$ defined by \eqref{inproof:sigmaU} to 
 \[  (\sigma \backslash \mathcal{E}) \circ \iota: I_{g,w=0} \overset{\iota}{\To} I_g^{\BB} \overset{\sigma\backslash \mathcal{E}}{\To}  \Top\]
 It induces a continuous map  
 \[  \left|  \sigma \cap \mathcal{U} \right| \To   \left| \sigma \backslash \mathcal{E}\right|=  \varinjlim_{G \in  (I_{g}^{\BB})^{\opp}} \sigma^{\BB}_G  \  \backslash \  \left( \sigma^{\BB}_G \cap  \mathcal{E}_G(\R) \right)   \] 
 which is exactly \eqref{EmbedOpentoBlowup}.
The last statement follows from the fact that the inverse of the continuous map $j$ is, by definition,   the restriction of the blow-down map $\pi^{\BB}$.
\end{proof}
The key point in the previous proposition is that simplices  which are glued together in the tropical moduli space, and which are not contained in  the graph hypersurface locus,  continue to be  glued together after  blowing-up   (i.e.,  if $\mathcal{F}(\gamma)$ is a common face of both $\mathcal{F}(G_1)$ and $\mathcal{F}(G_2)$ and $w(\gamma)=0$,  then  $\mathcal{F}^{\BB}(\gamma)$ is a common face of both $\mathcal{F}^{\BB}(G_1)$ and $\mathcal{F}^{\BB}(G_2)$). This is expressed by the fact that $\iota$ is a functor, using  the fact that  core graphs are intrinsic.

\subsection{The face complex associated to $\LM_g^{\trop,\BB}$}  \label{sect: NewGraphComplex} 
Using the general definition \ref{defn: facecomplex}, we may write down the face complex associated to $\LM_g^{\trop,\BB}$.
\begin{defn} Let $\GC_0^{\BB}$ denote the  $\Q$-vector space generated  by symbols $\big[ \underline{\Gamma}, \varpi \big] $
\[   \hbox{where } \quad \underline{\Gamma}=(\gamma_1,\ldots, \gamma_n)  \ , \quad \hbox{ and }  \quad  \varpi =   \varpi_{\gamma_1} \wedge \varpi_{\gamma_2/\gamma_1} \wedge  \ldots \wedge \varpi_{\gamma_n/\gamma_{n-1}}\ ,  \] 
where $\gamma_1 \subsetneq \ldots \subsetneq \gamma_n$ is a  strict nested sequence of  graphs where $\gamma_i$ is core for $i<n$ and $\gamma_n$ is stable and connected (definition \ref{defnNestedGraphs}), and 
 $\varpi_{\gamma_{i+1}/\gamma_i} $ is an orientation  on the  quotients $\gamma_{i+1}/\gamma_i$, with relations:
\[ 
 {[} \underline{\Gamma}, - \varpi]   =  -  [\underline{\Gamma}, \varpi] \qquad \hbox{ and } \qquad 
 {[} \underline{\Gamma},  \varpi]  =  [ \underline{\Gamma}', \varpi']  \]
whenever $\underline{\Gamma} \cong \underline{\Gamma}'$ is an isomorphism  of nested sequences $f: \gamma_n \cong \gamma_n'$ such that  $f(\gamma_i) = \gamma'_i$, and $\varpi'= f(\varpi)$.  It is bigraded by  genus, and edge number.  
The differential 
\[ d : \GC_0^{\BB} \To \GC_0^{\BB}\]
has two components:  $d= d^i+d^e$, which we call the internal and exceptional differentials, respectively. 
The internal differential is defined by edge contraction \eqref{EdgeContractSequence}
\[ d^i \big[(\gamma_1,\ldots, \gamma_n), \varpi  \big] = \sum_{e \in E_{\gamma_j/\gamma_{j-1}}} (-1)^j \big[(\gamma_1,\ldots, \gamma_{n})/e,  \varpi'
\big]\]
where  the sum is over all admissible edges,    $j$ is the unique index such that $e \in E_{\gamma_j/\gamma_{j-1}}$ and 
\[ \varpi'=  \varepsilon\,  \varpi_{\gamma_1} \wedge \ldots   \wedge \varpi_{\gamma_{i-1}/\gamma_{i-2}}  \wedge  \varpi'_{\gamma_i/\gamma_{i-1}} \wedge  \ldots \wedge \varpi_{\gamma_n/\gamma_{n-1}} \]
where $\varepsilon = \pm 1$  is defined by $\varpi'_{\gamma_{i}/\gamma_{i-1}} = \varepsilon \, e \wedge  \varpi_{\gamma_i /(\gamma_{i-1}\cup e)}$.
The exceptional differential is:
\[ d^e \big[(\gamma_1,\ldots, \gamma_n), \varpi  \big] = \sum_{\gamma} (-1)^i \big[(\gamma_1,\ldots, \gamma_{i-1},  \gamma, \gamma_i, \ldots,  \gamma_n),  \varpi'
\big]\]
where the sum is over all refinements, i.e., all possible choices of core graphs $\gamma$ such that  $\gamma_i \subsetneq \gamma \subsetneq \gamma_{i+1}$ 
for some $1 \leq i \leq n$, or $\gamma \subsetneq \gamma_1$.  The orientation is given by 
\[ \varpi'=  \varepsilon\,  \varpi_{\gamma_1} \wedge \ldots \wedge \varpi_{\gamma_{i-1}/\gamma_{i-2}} \wedge  \varpi'_{\gamma/\gamma_{i-1}} \wedge  \varpi'_{\gamma_i/\gamma} \wedge \ldots \wedge \varpi_{\gamma_n/\gamma_{n-1}} \]
where $\varepsilon = \pm 1$ is defined by 
$\varpi_{\gamma_{i}/\gamma_{i-1}} = \varepsilon     \varpi'_{\gamma/\gamma_{i-1}} \wedge  \varpi'_{\gamma_i/\gamma}  $.
\end{defn} 

The complex $\GC_0^{\BB}$ is filtered by the length  $n$ of sequences. Let $F_k \GC_0^{\BB}$ denote  the subcomplex whose generators $[(
\gamma_1,\ldots, \gamma_n), \varpi]$ have $n> k$ components. 

\begin{prop}  If we write $\GC_0^{\partial \BB} = F_1 \GC_0^{\BB}$ then   
\begin{equation} \label{homologynestedsequencecomplex}  H_n(\GC_0^{\BB}) \cong  \bigoplus_{g} H_{n}( \left| \LM_g^{\trop, \BB}\right|) \quad \hbox{ and } \quad  H_n(\GC_0^{\partial \BB}) \cong \bigoplus_g H_{n}( \left|\partial \LM_g^{\trop, \BB}\right|) \ .
\end{equation}
Furthermore, the homology of the  usual commutative even graph complex satisfies $ H_n(\GC_0) = H_{n}(\GC_0^{\BB} , \GC_0^{\partial\BB})$ and fits into a long exact sequence
\[ \ldots \To  H_{n}( \GC_0^{\partial \BB})  \To H_n( \GC_0^{\BB} ) \To  H_n(\GC_0) \overset{d^e}{\To} H_{n-1}( \GC_0^{\partial \BB})  \To \ldots \ . \]
This sequence may be identified with the direct sum over all $g$, of the relative homology sequence of the pair $(\left| \LM_g^{\trop, \BB}\right|,   \left| \partial \LM_g^{\trop, \BB}\right|)$. 
\end{prop} 
\begin{proof}  
The face complex of $\left|\LM_g^{\trop, \BB}\right|$ is  $(\GC_0^{\BB})_g$, the subcomplex of $\GC_0^{\BB}$ in genus $g$. This follows from the categorical description \S\ref{sect: canblowupG} of $\left|\partial \LM_g^{\trop, \BB}\right|$ and theorem \ref{thm: FeynmanBlowUpsFunctor}: facets of $\sigma^{\BB}_G$ are identified either with $\sigma_{\gamma}^{\BB} \times \sigma_{G/\gamma}^{\BB}$ for $\gamma \subset G$ a core subgraph corresponding to an exceptional divisor, or $\sigma^{\BB}_{G/e}$, where $e$ is a non-tadpole edge corresponding to the intersection of $\sigma_G^{\BB}$ with strict transforms of coordinate hyperplanes $x_e=0$. These two types of facets correspond to the two edge differentials (respectively, exceptional and internal). By iterating one arrives at the desciption of the face complex given above. 
Similarly, the face complex of the subspace $\left|\partial \LM_g^{\trop, \BB}\right|$ is given by the subcomplex  $(\GC_0^{\partial \BB})_g$ of $(\GC_0^{\BB})_g$ consisting of nested sequences of length $\geq 2$, which correspond precisely to the exceptional faces.
These facts, together with theorem \ref{thm: CellularHomology} imply the two equations in   \eqref{homologynestedsequencecomplex}.

Finally, use the relative long exact homology sequence
\[  \ldots \To H_n( \GC_0^{\partial \BB}) \To H_n( \GC_0^{\BB}) \To H_n( \GC_0^{\BB},\GC_0^{\partial \BB}) \To   H_{n-1}( \GC_0^{\partial \BB})\To \ldots  \]
and use  the definition of $\GC_0$ to  identify the relative homology group $H_n( \GC_0^{\BB},\GC_0^{\partial \BB})$ with $H_n( \GC_0)$. Taking the boundary along $\GC_0^{\partial \BB}$  is precisely the exceptional differential $d^e$.  

Note that by exicision, the relative homology $H_n(\left| \LM_g^{\trop, \BB}\right|,   \left| \partial \LM_g^{\trop, \BB}\right|)$ is isomorphic to 
$   H_n(\left| \LM_g^{\trop}\right|,   \left| \partial \LM_g^{\trop}\right|)$, and so these sequences are compatible with Proposition  \ref{prop: relativefacecomplexGC0}. 
\end{proof}

\section{Polyhedra in spaces of quadratic forms}  \label{section: PolyhedraQuadForms}
We now turn to the  study of polyhedra  whose vertices are  positive semi-definite quadratic forms with rational null spaces.

\subsection{Positive semi-definite matrices }  \label{sect: Pg}
Let $\mathcal{P}_g$ denote the space of positive definite $g\times g$ real symmetric matrices. It may be identified, via the map $X \mapsto X^T X$, with  $      
 O_g(\R) \setminus \GL_g(\R) $, and admits a right  action  $M\mapsto h^TMh$ for  $h \in \mathrm{GL}_g(\R)$.
 We write $L\mathcal{P}_g = \mathcal{P}_g/\R^{\times}_{>0}$.
   
  The space of real non-zero symmetric matrices up to scalar multiplication may be identified with the 
  real points of the projective space $\Pro^{d_g-1}$,  where $d_g = \binom{g+1}{2}$, and hence $L \mathcal{P}_g \subset \Pro^{d_g-1}(\R)$.
  A choice of homogeneous coordinates on $\Pro^{d_g-1}$ are given by the set of matrix entries  $X_{ij}$ for $i\leq j$, for $X$ a symmetric matrix.
     Since the determinant  is a homogeneous function of these coordinates, its vanishing locus defines a hypersurface
\[ \Det \  \subset \  \Pro^{d_g-1}  \   \]
which satisfies  $  L\mathcal{P}_g  \cap  \Det (\R)=\emptyset$. 
The  linear  action of $\GL_g$ on 
$\Pro^{d_g-1}$,
 which corresponds to    the action $M \mapsto h^T M h$ on symmetric matrices,  preserves  $\Det$  and  $ L \mathcal{P}_g \subset \Pro^{d_g-1}(\R)$. 

\subsection{Positive polyhedra in the space of symmetric matrices} It is convenient to reformulate the above in a coordinate-free manner. 
Consider a vector space $V$ of dimension $g$ defined over a field $k\subset \R$.
Let us denote by 
\[ \Quad(V) = \left(\mathrm{Sym}^2\,  V\right)^{\vee}  \]
the $k$-vector space  of symmetric bilinear (i.e., quadratic) forms 
\[Q: V\otimes_k V \To k\ . \] 
Such a quadratic form may be viewed as a linear map $Q: V \rightarrow V^{\vee}$ which  is self-dual, i.e., $Q= Q^T$.  Consequently, 
$\Quad$ defines  a contravariant functor from the category of vector spaces to itself:  for any linear map $V \rightarrow W$, there is a natural map $\Quad(W) \rightarrow \Quad(V)$.  The functor $\Quad$ sends surjective maps to injective maps, and vice-versa. Given an isomorphism $V\cong k^g$  one may  identify  elements of $\Quad(V)$  with symmetric matrices with entries in $k$.

Recall that the null space  $\ker(Q)$ of a quadratic form  is defined to be the kernel of $Q: V \rightarrow V^{\vee}$.  If $Q$ is a  positive semi-definite quadratic form, one has 
\begin{equation} \label{NullQcriterion}  Q(v,v)=0 \quad \Longleftrightarrow \quad v\in \mathrm{ker}(Q) \ .
\end{equation}

Consider the projective space $\Pro(\Quad(V))$. It has distinguished linear subspaces:  
\[\Pro(\Quad(V/K)) \lhook\joinrel\rightarrow \Pro(\Quad(V)) \qquad \hbox{ for every subspace } K \subset V \ . \] 
Such a subspace is  contained in the  determinant locus $\Det_V \subset \Pro(\Quad(V))$ if and only if $K \neq 0$.  
Viewing $\Quad(V/K)$ as the  subspace of $\Quad(V)$ of quadratic forms $Q$  satisfying  $Q(k,v)=0$ for all $k\in K$, and $v\in V$, one deduces the formula:
\begin{equation} \label{IntersectionofQspaces}
\Quad \left(V/K_1\right) \cap \Quad \left(V/K_2\right) = \Quad \left(V/ (K_1+K_2)\right) 
\end{equation} 
which implies that $ \Pro\left(\Quad \left(V/K_1\right) \right)\cap \Pro\left(\Quad \left(V/K_2\right)\right) = \Pro\left(\Quad \left(V/ (K_1+K_2)\right) \right)$.

   Consider the convex subsets 
 \[ \Quad^{> 0} (V) \   \subset \  \Quad^{\geq 0} (V) \   \subset \ \Quad(V)\]
 consisting of  positive definite, and positive semi-definite, quadratic forms on $V$.  
 
 \begin{defn}  \label{defn: sigmapositive}
We shall say that a polyhedron $(\sigma, \Quad(V))$ in the space of quadratic forms is \emph{positive},  which we shall denote by $\sigma\geq 0$,   if its vertices  lie in $\Quad^{\geq 0} (V)$, and hence  $\sigma \subset  \Quad^{\geq 0} (V)(\R)$ is contained in the set of positive semi-definite quadratic forms. 

We shall call a polyhedron \emph{strictly positive}, denoted by $\sigma>0$,  if $\sigma$  positive, and  if in addition  it meets the interior $\sigma \cap  \Quad^{> 0} (V)  \neq \emptyset$.  In other words, $\sigma$ is strictly  positive if it contains at least one positive definite quadratic form. \end{defn}

Positive polyhedra meet the determinant locus in a  specific manner.
\begin{lem} \label{lem: sigmameetsDet} Let $\sigma$ be a   positive polyhedron in $\Quad(V)$  and let $\sigma_F$ be a face of $\sigma$ of dimension $>0$.  The following are equivalent:

(i) The face $\sigma_F$  is contained in the determinant locus $\Det(\R)$.

(ii) There is a point in the interior of $\sigma_F$  which lies in $\Det(\R)$, i.e., 
\[ \overset{\circ}{\sigma}_F  \cap \Det(\R) \neq \emptyset\]

(iii)  There is a non-zero linear subspace $0 \neq K\subset V$  such that
\[ \sigma_F \  \subset \   \Pro(\Quad(V/K))(\R)  \ .  \]
\end{lem}
\begin{proof} Clearly $(i)$ implies $(ii)$ and $(iii)$ implies $(i)$ since $\Pro(\Quad(V/K)) \subset \Det$. It suffices to prove that $(ii)$ implies $(iii)$.   Let us write 
$\widehat{\sigma}_F = \{ \sum_{i=1}^n \lambda_i Q_i :  \lambda_i \in \R_{\geq 0} \} $
where $Q_i$ are  non-zero positive semi-definite quadratic forms on $V$.
If we assume $(ii)$ then there exist  $\lambda_i >0$ such that 
$ \det \left(\sum_{i=1}^n \lambda_i Q_i \right) = 0$
since every interior point of $\sigma$ admits a (non-unique) representation as a linear combination of $Q_i$ with strictly positive coefficients.  
Therefore there is a non-zero vector $x\in V$ such that  $x\in \mathrm{ker} \left( \sum_{i=1}^n \lambda_i Q_i \right)$. This implies in particular that  
$ \sum_{i=1}^n \lambda_i Q_i (x,x)  =  0$.
Since $Q_i(x,x) \geq 0$ we deduce that  $Q_i(x,x)=0$ for all $i=1,\ldots, n$.  By \eqref{NullQcriterion},  $x\in \ker Q_i$ for all $i$.  If we set  $K_{\sigma_F} = \bigcap_{i=1}^n \ker(Q_i)$, then   $ x\in  K_{\sigma_F}$, and in particular $K_{\sigma_F} \neq 0$.
 Property $(iii)$ holds on setting $K= K_{\sigma_F}$.  
 \end{proof}
 
 It follows that if $\sigma\geq 0$ is a positive  polyhedron, then it is  strictly positive  if and only if its interior does not meet the determinant locus: 
 \begin{equation} \sigma>0 \quad \Longleftrightarrow \quad  \overset{\circ}{\sigma} \cap \Det(\R) = \emptyset  \ . \end{equation}

\begin{defn} \label{defn: Ksigma}  For any positive polyhedron $\sigma\geq 0 $ in $\Pro(\Quad(V))$, we denote by $K_{\sigma} \subset V$ the intersection of all the null spaces of its vertices. 
 \end{defn}

It follows from \eqref{IntersectionofQspaces} that $K_{\sigma}$ is the unique largest subspace of $V$ such that 
\[\sigma  \ \subset \  \Pro(\Quad(V/K_{\sigma}))(\R)\ .\]   
  In particular, $\sigma$ meets the  strictly positive locus in $V/K_{\sigma}$, i.e., $\sigma \cap \Quad^{>0}(V/K_{\sigma}) \neq \emptyset$ and thus the polyhedron  $(\sigma, V/K_{\sigma})$ is strictly positive.
 
 \begin{lem}  \label{lem: sigmameetsVmodKinface}
 Let $\sigma\geq 0$ be a  positive polyhedron.
 For any subspace $K\subset V$, 
 \[ \sigma  \cap \Pro(\Quad(V/K))(\R) \]
 is either empty, or is a face of $\sigma$.
 \end{lem} 
 
 \begin{proof} For every $v \in K$, consider the linear map $f_v: \Quad(V) \rightarrow k$ defined by 
 $  Q  \mapsto  Q(v,v).$
  Its kernel is $\Quad(V/kv)$, which contains  $\Quad(V/K)$  by definition. The linear form  
  $f_{v}$ is non-negative on the subspace $\Quad^{\geq 0}(V) \subset \Quad(V)$. 
 It follows from \eqref{IntersectionofQspaces} that, for any  vectors $v_1,\ldots, v_d$ which span $K$, one has 
 $ \Quad(V/K)   = \bigcap_{i=1}^d   \Quad(V/kv_i)= \bigcap_{i=1}^d   H_i$, where $H_i=\ker ( f_{v_i})$. Since $f_{v_i}$  is non-negative on $\sigma$,  it follows from the interpretation due to Minkowski and Weyl  of a polyhedron as a region bounded by  a finite number of hyperplanes, that the set 
 $\sigma \cap \Pro(H_i) $ is either empty (in the case  when $f_{v_i}>0$  on $\sigma$), or a face of $\sigma$, and
 \[ \sigma \cap \Pro(V/K) (\R)  = \bigcap_{i=1}^d  \sigma\cap H_i (\R) \]
 is either empty, or a non-empty intersection of faces of $\sigma$ and hence also a face.   \end{proof}
 
The previous lemma has the important consequence that  in defining $\LA_g^{\trop,\BB}$ we shall only need to blow up linear subspaces of $\Pro(\Quad(V))$  of the special form $\Pro(\Quad(V/K))$.

\begin{defn} For any face $\sigma_F$ of $\sigma$,   define its \emph{essential  envelope}  to be  the face
\[   \sigma_F^{\ess}=    \sigma \cap \Pro(\Quad(V/K_{\sigma_F}))(\R) \ .\] 
It satisfies $\sigma_F \subset \sigma_F^{\ess}$.  We call a face  $\sigma_F$ \emph{essential} if $\sigma_F = \sigma^{\ess}_F$. 
\end{defn}

If $\sigma_F>0$ is strictly positive,  then $K_{\sigma_F}=0$ and  one has $\sigma^{\ess}_F= \sigma$. 
 In general, a  face  $\sigma_F$ of $\sigma$ (whether essential or not) will not be Zariski dense in  $\Pro(\Quad(V/K_{\sigma_F}))$. 
 See \S\ref{sec: Example} for some examples of essential faces (we shall show that for the image of Feynman polytopes under the tropical Torelli map, 
 the essential envelope of a face $\sigma_{G/\gamma}$ indexed by a subgraph $\gamma \subset G$ is the face $ \sigma_{G/\gamma^{\mathrm{core}}}$
 where $\gamma^{\mathrm{core}} \subset \gamma$ is the maximal core subgraph of $\gamma$).
 
\begin{lem}  \label{lem: EssentialStableIntersection}
The intersection of an essential face of $\sigma$ with any other face $\sigma_{Q'}$ of $\sigma$,   is either empty or   an  essential face of $\sigma_{Q'}$. 
\end{lem} 

\begin{proof}The set of essential faces of $\sigma$ are precisely the sets 
 $\sigma \cap \Pro(\Quad(V/K))(\R)$ where $K\subset V$ is a linear subspace. The intersection of such a face with $\sigma_{Q'}$ is
$  \sigma_{Q'} \cap  \Pro(\Quad(V/K))(\R)$, which 
is  therefore also essential or empty. 
\end{proof}

\subsection{Collapsing the complement of the determinant in the normal bundle} \label{sect: CollapseNormal} 
Restriction of quadratic  forms to any  subspace $K\subset V$ defines a  linear map:
 \begin{equation} \label{RestrictQtoK}   \Quad(V) / \Quad(V/K) \To \Quad(K)\ .
 \end{equation}
 It does not induce a map on the corresponding projective spaces since it is not injective in general. To see this, 
choose a  complementary space  $V\cong K \oplus C$ where $C \cong V/K$. The space  $\Quad(V)/\Quad(V/K)$ 
may be represented by  symmetric matrices in block matrix  form,
\begin{equation} \label{QBlockform} Q =   \left( 
\begin{array}{c|c}
Q_0   & Q_1 \\ \hline 
Q_1^T & 0 \end{array} \right) \ ,
\end{equation} 
 which implies that $\Quad(V)/\Quad(V/K)$ is isomorphic to the product $\Quad(K) \times  \mathrm{Hom}(V/K,K)$.
The map $ \Quad(V) / \Quad(V/K) \rightarrow \Quad(K)$ sends $Q\mapsto Q_0$.  
 The determinant function $\det: \Quad(K) \rightarrow k$ defines a homogeneous polynomial  
\[ Q \mapsto \det(Q\big|_K)  \ :  \ \Quad(V) / \Quad(V/K)\To  k\ . \]  
In the  coordinates $\eqref{QBlockform}$, $\det(Q\big|_K) $ is simply $\det(Q_0)$. 
 Denote its zero locus  by 
\begin{equation} \label{defn: detslashK}   \Det\big|_K \  \subset \  \Pro \left( \Quad(V) / \Quad(V/K)\right)\ . 
\end{equation} 
\begin{lem} \label{lem: PiKdef} Restriction of quadratic forms to  $K$  gives a well-defined  projection
\begin{equation} \label{piKdef}    \pi_K \ : \  \Pro \left( \Quad(V) / \Quad(V/K)\right) \  \setminus \  \Det\big|_K \To  \Pro\left( \Quad(K) \right) \  \setminus \   \Det_K    
\end{equation}
which is   a fibration in affine spaces $\mathrm{Hom}(V/K,K ) \cong \mathbb{A}^{\dim K \cdot \dim V/K}  $. 
It is functorial, i.e., for  any  $h\in \GL(V)$,   there is a commutative diagram:
\[ 
\begin{array}{ccc}
 \Pro \left( \Quad(V) / \Quad(V/K)\right) \  \setminus \  \Det\big|_K  & \overset{\pi_K}{\To}    & \Pro\left( \Quad(K) \right) \  \setminus \  \Det_K  \\
  \uparrow &   &  \uparrow  \\
  \Pro \left( \Quad(V) / \Quad(V/K')\right) \  \setminus \  \Det\big|_{K'}  &   \overset{\pi_{K'}}{\To}    & \Pro\left( \Quad(K') \right)  \ \setminus \   \Det_{K'}  
\end{array}
\]
where the vertical maps are isomorphisms  $Q\mapsto h^TQh$  and  $K'= h K $.
\end{lem}

\begin{proof} 
A  decomposition $V\cong K \oplus C$  defines projective coordinates on $\Pro \left( \Quad(V) / \Quad(V/K)\right)$. The   complement of $  \Det\big|_K$ consists of projective classes  of  matrices \eqref{QBlockform} with $\det(Q_0)\neq 0$. In particular, $Q_0 \neq 0$ and so the map $Q \mapsto Q_0$ is  well-defined on projective spaces. Its fiber is isomorphic to the space
of matrices $Q_1$, which is nothing other than the vector space $ \mathrm{Hom}(V/K,K)$, viewed as the $k$ points of an affine space of equal dimension.
\end{proof}

  Note that the ambiguity in the  choice of decomposition $V\cong  K \oplus V/K$, is   the set of splittings of the exact sequence $0 \rightarrow K \rightarrow V \rightarrow V/K \rightarrow 0$, which is a torsor over   $ \mathrm{Hom}(V/K,K)$.

\subsection{Strict positivity  of normal faces} 
For simplicity of notation, let us write
\begin{equation} \label{notation:sigmamodF} \sigma/\sigma_{F} = \sigma_{/K_{\sigma_F}} 
\end{equation} 
for the normal polyhedron (definition \ref{defn: normal}) of $\sigma$ relative to the null space $K_{\sigma_F}$ of a face $\sigma_F$. 
\begin{lem}  \label{lem: positivityoffaceandnormal}  Let $\sigma$ be strictly positive and let $\sigma_F$ be a face of $\sigma$.  In the product 
\[ \sigma_{F} \times \left( \sigma/  \sigma_{F} \right)  \ \subset \ \Pro\left(\Quad(V/K_{\sigma_F})\right) \  \times \  \Pro\left(\Quad(V)/\Quad(V/K_{\sigma_F})\right)\ ,  \] 
both $\sigma_F$ and $\sigma/\sigma_F$ are strictly positive (definition \ref{defn: sigmapositive}). In  the latter case, this means that  the interior of the polyhedron $\sigma/\sigma_F$ does not meet the determinant locus: 
\[  \overset{\circ}{\left(\sigma/\sigma_F  \right)}  \subset   \Pro\left(\Quad(V)/\Quad(V/K_{\sigma_F}) \right)   \ \setminus \  \Det\big|_{K_{\sigma_F}} \]
or equivalently, that the image  $\pi_K ( \sigma/\sigma_F)$ is a strictly positive polyhedron in $\Pro(\Quad(K_{\sigma_F}))$.  
\end{lem}

\begin{proof} 
The vertices of $\sigma$ are projective classes of positive semi-definite quadratic forms $Q_1,\ldots, Q_n \in \Quad(V)$.   The vertices of the face 
$\sigma_F$ correspond to a  subset  of them:  $\{ Q_i , i \in I\}$, where   $I \subset \{1,\ldots, n\}$.
By definition, 
\begin{equation}  \label{inproof: verticescondition} K_{\sigma_F}  =   \bigcap_{i\in I}  \ker Q_i \ , 
\end{equation} 
which implies that in any choice of decomposition $V = K_{\sigma_F} \oplus C$, where $C\cong V/K_{\sigma_F}$, the matrices $Q_i$ have  the block matrix form 
\[ Q_i =     \left(
\begin{array}{c|c}
 0  & 0      \\ \hline
  0 &   \overline{Q}_i    
\end{array}
\right)   \quad  \hbox{ if  }  \quad  i \in I \ , 
\]
where the $\overline{Q}_i$ are the restrictions of  $Q_i$ to $C$, and 
 are positive semi-definite. 
 Suppose  that $\sigma_F$ is not strictly positive. Then $\sigma_F $ is contained in  $\Det_{V/K_{\sigma_F}}(\R) \subset \Pro(\Quad(V/K_{\sigma_F}))(\R)$. By lemma  \ref{lem: sigmameetsDet},  there is a  subspace $0 \neq K' \subset V/K_{\sigma_F}$ which is contained in $\bigcap_i \ker \, \overline{Q}_i$. 
 But this implies   $\bigcap_i \ker \, Q_i$ is strictly larger than $K_{\sigma_F}$, contradicting  \eqref{inproof: verticescondition}.

 The polyhedron $\sigma/\sigma_F$ has vertices given by the projective classes of the $Q_j$, $j\notin I$ which are non-zero in the quotient $\Quad(V)/\Quad(V/K_{\sigma_F})$ and have  block matrix representatives 
 \[ Q_j =  \left(
\begin{array}{c|c}
 R_j   & P_j    \\ \hline
  P_j^T &  0  
\end{array}
\right)   \quad   \hbox{ for all } j \notin I \ .
\]
If $\sigma/\sigma_F$  were contained in  $\Det|_{K_{\sigma_F}}(\R)$,  then  by the argument of 
lemma \ref{lem: sigmameetsDet} applied to the matrices $R_j$, for $j\notin I$,  there would exist $0\neq x\in K_{\sigma_F}$ such that $x\in \ker(R_j) \subset K_{\sigma_F}$ for all $j\notin I$.   In this case, the image $y$ of $x$ in $V$ satisfies $Q_j(y,y)=R_j(x,x)=0$ and hence   $y\in \ker(Q_j)$ for all $j\notin I$. Since
 $y\in \ker(Q_i)$ for all $i \in I$ by        \eqref{inproof: verticescondition},  we deduce that $y$ lies in  $\ker\, Q_i$, for all $i=1,\ldots, n$,  contradicting the strict positivity of $\sigma$  by lemma  \ref{lem: sigmameetsDet}. Note that 
 the image $\pi_K(\sigma/\sigma_F)$ is the link of the convex hull of the quadratic forms $R_j$, for $j\notin I$.
 \end{proof}

 \begin{rem}  \label{rem:  ExtendedLemmaOnStrictPositivity} This lemma easily extends to the case when $\sigma$ is a strictly positive polyhedron in the space $\Pro(\Quad(V)/\Quad(V/K))$. The statement is the following: in the  product of polyhedra 
 \[ \sigma_F \times ( \sigma/\sigma_F)  \subset  \Pro \left(\Quad(V/K_{\sigma_F})/\Quad(V/K) \right)  \times  \Pro \left(\Quad(V)/\Quad(V/K_{\sigma_F}) \right) \] 
  both $\sigma_F$ and  $\sigma/\sigma_F$ are strictly positive. 
  The proof is identical: one only needs to  consider symmetric  matrices with a number of initial rows and columns which are identically zero. 
  \end{rem} 
 
 \subsection{Blow-ups and the determinant locus}  \label{sect: BlowUpandDet}
 The   determinant locus is well-behaved with respect to  blowing up linear subspaces of the form $\Pro(\Quad(V/K))$.  
  \begin{prop} \label{lem: stricttransformofDet}
  Let $0\neq K \subset V$ be a subspace and let 
  \[ \pi : P^{\BB}  \To  \Pro(\Quad(V))\]
  denote the blow-up of $ \Pro(\Quad(V))$ along $\BB=\{\Pro(\Quad(V/K))\}$.  The exceptional divisor  $\mathcal{E}$ may be canonically identified with the projectivised normal bundle of $\Pro(\Quad(V/K))$: 
    \[ \mathcal{E} \cong  \Pro(\Quad(V/K)) \times \Pro(\Quad(V)/ \Quad(V/K))\ . \]
        If $\widetilde{\Det}_V$ denotes  the strict transform of the determinant locus, we have
  \[   \widetilde{\Det}_V  \cap \mathcal{E} \  \cong  \   \left( \Det_{V/K} \times \Pro(\Quad(V)/\Quad(V/K))  \right) \  \cup  \ \left(   \Pro(\Quad(V/K)) \times \Det\big|_{K} \right) \ , \]
  where $\Det\big|_K$ was defined in \eqref{defn: detslashK}.
   In particular one has  the product formula:
 \begin{equation}  \nonumber  \mathcal{E}  \ \setminus \  \left(  \mathcal{E} \cap \widetilde{\Det}_{V}  \right) \ \cong  \  \left( \Pro(\Quad(V/K))  \ \setminus \  \Det_{V/K} \right) \ \times \ \left( \Pro(\Quad(V)/\Quad(V/K)) 
  \ \setminus \  \Det\big|_{K} \right) \ . 
 \end{equation} 
  \end{prop} 
  
  \begin{proof}
 Choose a complementary space  to $K$, i.e., 
$ V  \cong   K  \oplus C$
and  choose bases of $K,C$. We can compute the blow up in local affine coordinates (\S\ref{sect: LocalBlowUpCoordinates})  as follows.  
A  quadratic form in $Q\in \Quad(V)$ may be written  in the form
\begin{equation} \label{inproofQblockform}  Q = \begin{pmatrix} 
Q_K   &   P \\ 
P^T  &  Q_C 
\end{pmatrix} 
\end{equation}
 Let $x_{ij}$ denote the $(i,j)^{\mathrm{th}}$ matrix coefficient. 
A choice of   affine coordinates on $\Pro(\Quad(V))$ is given by the  $x_{ij}$  for $i\leq j$, one  of which, say  $x_{i_0j_0}$, is set to $1$.  The locus $\Pro(\Quad(V/K))$ is represented by matrices such that $Q_K=P=0$, i.e., the vanishing of all other $x_{ij}$ for $i\leq j$ such that $i \leq \dim(K)$. 
Therefore, suppose that $i_0 \leq \dim(K)$. The blow up may be represented in any such local affine coordinates (see \S\ref{sect: LocalBlowUpCoordinates})   by the map: 
\begin{eqnarray} s^*_K: k[x_{ij}, i\leq j ]  &\To &  k[z, x_{ij} , i\leq j]    \nonumber \\ 
s^*_K (x_{ij} )  &= & z  x_{ij}  \quad \hbox{ if } i  \leq   \dim(K)  \nonumber \\
s^*_K (x_{ij} )  &= &  x_{ij}   \quad \hbox{ otherwise } \nonumber
\end{eqnarray} 
The exceptional divisor is  defined by the equation  $z=0$. 
Applying the  map $s^*_K$ to the entries of the matrix  of a quadratic form \eqref{inproofQblockform}, we obtain 
\begin{equation}  \label{LambdaBlockForm}
s_K^* Q  = \begin{pmatrix}     zQ_K  &  z P    \\    z P^{T}  &    Q_C   \end{pmatrix}  \ . 
\end{equation} 
Taking determinants of the  following matrix identity:
\[ 
\begin{pmatrix}    z Q_K  &  z P    \\    z P^{T}  &    Q_C   \end{pmatrix}  = \begin{pmatrix}     I  &  0    \\    P^TQ_K^{-1}   &  I  \end{pmatrix} \begin{pmatrix}     z Q_K  &  0    \\    0  &    Q_C  - z P^TQ_K^{-1} P  \end{pmatrix} \begin{pmatrix}    I   &  Q_K^{-1}P   \\    0  &   I  \ .\end{pmatrix} 
\] 
 gives 
$\det( s_K^* Q ) = \det(  z Q_K)  \det ( Q_C   -  z P^TQ_K^{-1} P  )  $.    This implies that 
\begin{equation} \label{detfact} \det( s_K^* Q ) =  z^{\dim K} \det(  Q_K)  \det ( Q_C )  \mod( z^{\dim K +1}) \ .  
\end{equation}
It follows that the  intersection of the  strict transform of the locus $\det(Q)=0$ with the exceptional divisor $z=0$  is given by the equation $\det(Q_K) \det(Q_C)=0$. 
The zero loci of $\det(Q_K)$ and $\det(Q_C)$ are  $ \Det\big|_{K}$, and $\Det_{V/K}$  respectively.   \end{proof} 

 This result generalises several  asymptotic factorisation theorems for graph polynomials \cite[\S2]{Cosmic} simply by applying the proposition to a suitable graph Laplacian matrix. In particular,  \eqref{Psifactorizes} and other similar formulae  can be deduced  from the universal fact that the determinant admits an asymptotic factorisation formula 
 \eqref{detfact}. 
  
\section{Perfect cone compactification}  \label{section: PerfectConeCompact}
Now consider the case when $k= \Q$ and the vector space $V = V_{\Z} \otimes\Q$ has a lattice $V_{\Z}$. 
It is convenient to identify $V_{\Z}  \cong \Z^g$ for the purposes of the ensuing discussion.

\subsection{Rational closure of $\mathcal{P}_g$}
 The rational closure   $\mathcal{P}^{\rt}_g$ of the space of positive definite matrices $\mathcal{P}_g$  is defined to be the space of positive semi-definite matrices $M$  whose null space $\ker(M)$ is defined over $\Q$. A positive semi-definite  matrix $M$ has rational kernel if and only if  there exists an element $h\in \GL_g(\Z)$ such that 
  \[ h^T M h = \begin{pmatrix}  M_0 & 0 \\ 0 & 0   \end{pmatrix}  \]
  where $M_0$ is positive definite, i.e., $M_0 \in \mathcal{P}_{g'}$ for some $g'\leq g$.

Recall  from \S\ref{sect: Pg} that we identify the space of $g\times g$ symmetric matrices with $ \R^{d_g}$, where $d_g = \binom{g+1}{2}$.
We have inclusions $\mathcal{P}_g \subset \mathcal{P}^{\rt}_g \subset \R^{d_g}$ and 
\[   L \mathcal{P}_g  \subset L \mathcal{P}^{\rt}_g \subset\Pro^{d_g-1}(\R) \]
where
the link of $\mathcal{P}^{\rt}_g$ is defined by 
 $L\mathcal{P}^{\rt}_g = \left( \mathcal{P}^{\rt}_g  \setminus \{0\}\right)/\R^{\times}_{>0}$.  
 The group $\GL_g(\Z)$ preserves all three spaces.
 One has 
$L \mathcal{P}^{\rt}_g  \  \setminus \  L\mathcal{P}_g \  \subset \  \Det(\R)$.

 \begin{rem} Lemma \ref{lem: sigmameetsDet} and the discussion which follows implies that a positive polyhedron $(\sigma, \Quad(\Q^g))$ has the property that $\sigma$ is contained in the space $\mathcal{P}_g^{\rt}$, i.e., every point of $\sigma$ defines a real symmetric matrix with rational kernel. 
  \end{rem} 

\subsection{Minimal vectors and polyhedral linear configurations} 

For any positive definite real quadratic form  $Q$  on $\R^g$, denote its set of   minimal vectors by
\[ M_{Q} = \{ \lambda \in \Z^g \setminus \{0\}:  Q(\lambda) \leq  Q(\mu)  \hbox{ for all } \mu \in \Z^g \setminus \{0\} \}\ . \]
The associated polyhedral cone is  defined to be 
\[ \widehat{\sigma}_Q =  \R_{\geq 0} \langle \lambda \lambda^T\rangle_{\lambda\in M_Q} \subset \mathcal{P}^{\rt}_g \]
where the    $\lambda \lambda^T$ are  rank one quadratic forms in $\mathcal{P}^{\rt}_g$. 
The cone  $\widehat{\sigma}_Q$ is strongly convex since this is already the case for $\mathcal{P}^{\rt}_g$ (although $\mathcal{P}^{\rt}_g$
is not  itself a polyhedral cone). A positive definite quadratic form $Q$ is   called  \emph{perfect} if $\widehat{\sigma}_{Q}$  has maximal dimension $d_g$. Equivalently, $Q$ is uniquely determined up to a scalar  multiple by its set of minimal vectors \cite{Martinet, Voronoi}.

\begin{defn} \label{defn: coneQ} For any positive definite real quadratic form $Q$ denote by  
\begin{equation} \label{eqndef: coneQ}
 \cone_{Q} = ( \Pro^Q  ,  L_Q ,  \sigma_Q  ) 
\end{equation} 
 the polyhedral linear configuration where $\sigma_Q =L\widehat{\sigma}_Q  = \left( \widehat{\sigma}_Q \backslash \{0\}\right)/\R^{\times}_{>0}  $ is the link of $\widehat{\sigma}_Q$ in $\Pro^{d_g-1}(\R)$,  $L_Q = \bigcup_i L_i$ is the union of the Zariski closures of its facets, and  $\Pro^Q \subset \Pro^{d_g-1} \cong \Pro (\Quad (V))$ is the linear subspace defined by the Zariski-closure of $\sigma_Q$.
\end{defn} 

The configuration $\cone_{Q}$ is thus an  object of the category $\PLC_{\Q}$.  The form $Q$ is perfect if and only if $\Pro^Q = \Pro^{d_g-1}$, i.e., $\sigma_Q$ is Zariski dense in the whole space of quadratic forms. 

\subsection{Admissible decompositions}  \label{sect: AdmissDecomp} See  \cite{TopWeightAg,Namikawa, FaltingsChai, Voronoi}.
\begin{defn} A  set  $\Sigma = \{ \widehat{\sigma}_i\}$  of polyhedral cones $\widehat{\sigma}_i  \subset \mathcal{P}_g^{\rt}$ is called \emph{admissible} if it is stable under the action of $\GL_g(\Z)$, and such that:
\begin{itemize} 
\item it covers $\mathcal{P}^{\rt}_g$, i.e., 
$ 
\bigcup_{i}  \widehat{\sigma}_i = \mathcal{P}^{\rt}_g, 
$
\item if $\widehat{\sigma} \in \Sigma$, then all the faces of $\widehat{\sigma}$ are elements of  $\Sigma$,
\item the intersection of two cones in $\Sigma$ is a face of both cones,
 \item the set of $\GL_g(\Z)$-orbits of $\Sigma$ is finite.
\end{itemize} 
\end{defn} 

\begin{thm} (Vorono\"i).  The set of  cones $\{\widehat{\sigma}_Q, Q \in \mathcal{P}_g\}$ is an admissible decomposition of $\mathcal{P}^{\rt}_g$. In particular, every face of $\widehat{\sigma}_Q$ is a cone $\widehat{\sigma}_{Q'}$ for some other quadratic form $Q'$.
Furthermore, every cone $\widehat{\sigma}_Q$ is a face of a perfect cone. 
\end{thm} 
Let us say that two quadratic forms $Q, Q'$ are equivalent if  their sets of minimal vectors (and hence their associated cones) coincide: $M_{Q}= M_{Q'}$. 
Denote by $[Q]$ the equivalence classes. The group  $\GL_g(\Z)$ acts upon the classes $[Q]$. 
Write \[ [Q'] \leq [Q] \]
 if  the cone $\widehat{\sigma}_{Q'}$ is a face of $\widehat{\sigma}_{Q}$.  Let us write $\dim\, [Q] = \dim \sigma_{Q}$.
 
\begin{defn} \label{defn:Dperfg} Let  $\mathcal{D}^{\perf}_g$ denote the category whose  objects are equivalence classes $[Q]$ for   $Q\in \mathcal{P}_g$,  and  whose  morphisms are  generated by the following two kinds of maps:
\begin{enumerate}[(i)]
\item  face maps 
$[Q'] \rightarrow [Q]$ whenever $[Q']\leq [Q]$, 
\item   isomorphisms  $h: [Q ] \rightarrow [h^TQh]$ for any $h\in \GL_g(\Z)$ \ .
\end{enumerate}
  \end{defn}

The category $\mathcal{D}^{\perf}_g$   is equivalent to a finite category \cite{TopWeightAg}, since there are 
 only finitely many isomorphism classes of objects, and automorphism groups of cones are finite.

Every $h\in \GL_g(\Z)$ defines a linear isomorphism in $\PLC_{\Q}$  we denote by:
\[ \cone_Q  \overset{h}{\To} \cone_{h^T Qh}   \ . \] 
 For every $Q' \leq Q$ one has a face map  
$ \cone_{Q'}  \rightarrow  \cone_{Q}$ in the category $\PLC_{\Q}$.

 \begin{defn} Consider the functor 
 \begin{equation} 
 \LA_g^{\trop}: \mathcal{D}^{\perf}_g   \To   \PLC_{\Q}  
  \end{equation}
 which sends $Q$ to $\cone_Q$.  
Its   topological realisation  $ | \LA_g^{\trop}| $ 
is the link of the moduli space of tropical abelian varieties.  By construction, it is  a generalised cone complex as in \cite{TropModuli}. 
 \end{defn}

\subsection{Blow-ups of cones associated to positive definite quadratic forms}

\begin{defn} Suppose that  $Q$ is a positive definite quadratic form on $V$, such that  $\sigma_Q>0$ is strictly positive (definition \ref{defn: sigmapositive}).    Its Zariski closure  $\Pro^Q$ admits a canonical embedding into $\Pro(\Quad(V))$.    Consider the finite set  of subspaces of $\Pro^Q$ defined as follows:
\begin{equation} \label{defn: BlowupSetforQ} \BB_Q  =\bigcup_{n\geq 1} \left\{  \bigcap_{i=1}^n  \Pro^{Q} \cap \Pro(\Quad(V/K_{\sigma_{i}}))   \ : \ \hbox{ where } \sigma_i \hbox{ are faces of } \sigma_Q   \right\}  \ .  \end{equation}

By definition and \eqref{IntersectionofQspaces}, the spaces which are to be blown up are of the form $\Pro(\Quad(V/K))$,  which  meet $\sigma_Q$ precisely along its essential faces.

Alternatively, consider  for all  $0\neq K\subset V$,  the (infinite) set $\BB^{\mathrm{all}}$ of  all subspaces
\[  \Pro(\Quad(V/K))   \subset \Pro(\Quad(V))\]
Then  $\BB_Q =  (\BB^{\mathrm{all}})^{\min, \sigma_Q} \cap \Pro^{Q}$,  by definition \ref{defn: Bmin} and lemma \ref{lem: sigmameetsVmodKinface}.

Denote  the iterated blow-up of  $\cone_Q$ along $\BB_Q$ to be:
\[     
\cone^{\BB}_Q
  \quad \in \quad \mathrm{Ob}(\BLC_{\Q}) \ .
\]
Let us denote its topological realisation  by $\sigma^{\BB}_Q = \sigma(\cone^{\BB}_Q)$.
There is a natural blow-down morphism $\cone^{\BB}_Q \rightarrow \cone_Q$ in the category $\PC_{\Q}$, which induces a continuous map 
$\sigma^{\BB}_Q \rightarrow \sigma_Q$.

 \end{defn} 
\begin{rem} Because $\sigma_Q$ is not a simplex in general,  the set of loci $\Pro^{Q} \cap \Pro(\Quad(V/K)) $  where   $ \Pro(\Quad(V/K)) $ meets $\sigma_Q$   is not  closed under intersections (the intersection of two essential faces can be empty). Thus the intersection of two blow-up loci may not intersect $\sigma_Q$ at all, giving rise to an extraneous blow-up. This does not occur in  $\LM_g^{\trop,\BB}$ because $\sigma_G$ is a simplex. 
This is a fundamental difference between the constructions for $\LM_g^{\trop,\BB}$ and $\LA_g^{\trop,\BB}$. 
Another difference between these constructions arises for a different reason. Unlike the case of graphs, $\BB_Q$  does \emph{not} only consist  of the Zariski closures of the faces of $\sigma_Q$: we must blow up loci which can be strictly larger.    See \S \ref{sec: Example} for some concrete examples. 
\end{rem}

Consequently, by contrast to the situation with polytopes associated to graphs, the blow-ups are not intrinsic, which means that we do not strictly have face maps.
In other words, the blow-up locus is not necessarily stable under passing to faces: if $Q'$ is a face of $Q$, one has $\BB_{Q'} \subseteq \Pro^{Q'} \cap \BB_Q$ but we do not \emph{a priori} have equality. 
 Nevertheless, the following lemma implies that face maps are intrinsic,  up to extraneous modifications. 

\begin{lem} \label{lem: extraneousblowdowns}
Let    $Q' \leq Q$ be positive definite.  
Let  us denote by $F_{Q'} \cone^{\BB}_{Q}$ the face of $\cone^{\BB}_{Q}$ corresponding to the face $\sigma_{Q'} $. Then there is a canonical  morphism
\begin{equation}\label{extraneousblowdown}  F_{Q'} \cone^{\BB}_{Q} \To \cone^{\BB}_{Q'}
\end{equation}
in the category $\BLC_{\Q}$ 
which is a composition of extraneous modifications. Consequently, it induces an isomorphism on topological realisations:
\[ F_{Q'} \sigma_Q^{\BB} \cong \sigma_{Q'}^{\BB}\ .\]
\end{lem} 
\begin{proof}
By definition  and \eqref{IntersectionofQspaces},   $\BB_Q$  consists of spaces of the form $\Pro^Q \cap \Pro(\Quad(V/K))$ for certain $0\neq K\subset V$. By lemma \ref{lem: sigmameetsVmodKinface}, $\sigma_{Q'} \cap \Pro(\Quad(V/K))(\R) $ is either a face of $\sigma_{Q'}$ or the empty set. Since $\Pro^{Q'}$ is the Zariski-closure of $\sigma_{Q'}$,  we deduce that  $\BB_{Q'}=  (  \BB_{Q} \cap \Pro^{Q'})^{\min,\sigma_{Q'}}$ and  \eqref{extraneousblowdown} is a consequence of
proposition \ref{prop: minimalblow-down} and proposition \ref{prop: UniversalANDIdeal} (i) applied to $\Pro^{Q'} \subset \Pro^Q$.
\end{proof} 
Face morphisms are thus replaced in this context by \emph{face diagrams} in $\BLC_{\Q}$:
\begin{equation}  \label{facediagram}
\cone^{\BB}_{Q'} \longleftarrow   F_{Q'} \cone^{\BB}_{Q}  \longrightarrow  \cone^{\BB}_{Q} \\
\end{equation} 
where the map on the right is the  inclusion of a face, and the map on the left is an extraneous modification \eqref{extraneousblowdown}. The blow-down of \eqref{facediagram} in  $\PLC_{\Q}$ is the diagram 
\begin{equation} \label{coneQfacediag} \cone_{Q'}  \ \overset{\sim}{\longleftarrow} \   \cone_{Q'}   \longrightarrow  \cone_{Q}  \ .
\end{equation}

\subsection{Definition of $\LA_{g}^{\trop, \BB}$} 
Every  $h \in \GL_g(\Z)$
 gives rise to  a bijection
 $ \BB_{Q} \overset{h}{\rightarrow}  \BB_{h^TQh} $
 and hence induces a linear isomorphism and commutative diagram in $\PC_{\Q}$: 
\begin{equation} 
\begin{array}{ccc}\label{blowupsfunctorial} 
 \cone^{\BB}_Q & \overset{h}{\To}   &   \cone^{\BB}_{h^TQh} \\
   \downarrow &   &   \downarrow  \\
   \cone_Q &  \overset{h}{\To}    &  \cone_{h^T Qh}  \ 
\end{array}
\end{equation}
where the horizontal maps are in $\BLC_{\Q}$. 
In other words,  blow-ups along linear strata are functorial for  the action of  $\GL_g(\Z)$.  
Consequently, the  action of $\GL_g(\Z)$ is compatible with inclusions of faces:  
for every face morphism $F \rightarrow \cone^{\BB}_Q$ and $h \in \GL_g(\Z)$ there is a unique face  $h(F)$  of $\cone^{\BB}_{h^TQh}$  such that 
 the following diagram in $\BLC_{\Q}$ commutes:
\[
\begin{array}{ccc}
   F    & \overset{h}{\To}  & h(F)   \nonumber \\ 
  \downarrow   &   & \downarrow    \\
  \cone^{\BB}_Q     &  \overset{h}{\To} &   \cone^{\BB}_{h^TQh} \\
\end{array} \ ,
\]
where the vertical maps are inclusions of faces.  Furthermore, this diagram is compatible with face diagrams \eqref{facediagram} and their blow-downs \eqref{coneQfacediag}.
The reason  is that  an element $h \in \GL_g(\Z)$ induces an isomorphism not only on faces but also on their normals (remark \ref{rem: Products}). 

\begin{defn}
Define a category  $\mathcal{D}_g^{\perf, \BB}$ whose objects are the sets of faces of $\cone^{\BB}_Q$, for all equivalence classes of positive definite quadratic forms $[Q]$,  and whose morphisms are generated by face maps,  linear isomorphisms $h$, for all $h \in \GL_g(\Z)$, and extraneous modifications of the form  \eqref{extraneousblowdown}.
\end{defn} 
Since   $\cone^{\BB}_Q$ has finitely many faces, each of which blows down to a face of $\cone_{Q}$, it follows that $\mathcal{D}_g^{\perf, \BB}$, like $\mathcal{D}_g^{\perf}$, is equivalent to a finite category. 
Blow-down defines a canonical functor 
\[ \phi :  \mathcal{D}_g^{\perf,\BB} \To  \mathcal{D}^{\perf}_g\ .\]

\begin{defn} Consider the $\BLC_{\Q}$-complex   
 \begin{equation}  
\LA_g^{\trop,\BB}: \mathcal{D}^{\perf,\BB}_g   \To    \BLC_{\Q}  
\end{equation}
which to any object of $ \mathcal{D}^{\perf,\BB}_g$ associates the corresponding face of  $ \cone^{\BB}_Q$,  for suitable $Q$. 
\end{defn} 

There is a canonical blow-down morphism (in $\PC_{\Q}$)
\[ \LA_g^{\trop,\BB} \To  \LA_g^{\trop}\]
given by the pair $(\phi, \Phi)$, where $\phi:  \mathcal{D}_g^{\perf,\BB} \rightarrow  \mathcal{D}^{\perf}_g$ is the functor considered above, and 
$\Phi$ is the natural transformation obtained by restricting the canonical blow-down map $\cone^{\BB}_Q \rightarrow \cone_Q$ to faces.
It induces a continuous map on topological realisations:
\[   \left| \LA_g^{\trop,\BB}  \right|   \To  \left| \LA_g^{\trop}  \right| \ .  \]

\begin{rem}  \label{rem: glueconesfacediagrams} It is instructive to explain how, in the absence of face maps \emph{per se}, the blow ups of cones are glued together along faces.
Consider two positive definite quadratic forms $Q_1, Q_2$ which share  a common face $[Q] \leq [Q_1], [Q_2]$. 
In the perfect cone complex, this is reflected by face maps 
$  \cone_{Q} \rightarrow \cone_{Q_1}$  and  $\cone_{Q} \rightarrow \cone_{Q_2}$
which we think of as gluing $\cone_{Q_1}, \cone_{Q_2}$ along the common face $\cone_Q$; indeed this is precisely what happens in the topological realisation: the images of the polyhedra  $\sigma_{Q_1}$ and $\sigma_{Q_2}$ are identified along $\sigma_Q$ in the limit $|\LA_g^{\trop}|$. 

In the complex $\LA_g^{\trop,\BB}$, we have instead a slightly more complicated picture obtained by joining two face diagrams \eqref{facediagram} in $\BLC_{\Q}$ pictured below on the left-hand side: 
\[
\begin{array}{ccccccc}
F_{Q} \cone^{\BB}_{Q_1}  &  \To   &   \cone^{\BB}_{Q_1}  &   \qquad  \qquad   &  \sigma^{\BB}_{Q}  &  \To   &   \sigma^{\BB}_{Q_1}     \\ 
\downarrow   &   &   &   \qquad  \qquad  & \downarrow   &   &   \\
  \cone^{\BB}_{Q} &   &   &   \qquad  \overset{\sigma}{\mapsto}  \qquad&  \sigma^{\BB}_{Q} &   & \\
  \uparrow   &   &  &   \qquad    \qquad&   \uparrow   &   &   \\
  F_{Q} \cone^{\BB}_{Q_2}  &  \To   &   \cone^{\BB}_{Q_2}   &   \qquad \qquad   &    \sigma^{\BB}_{Q}  &  \To   &   \sigma^{\BB}_{Q_2}  
\end{array}
\]
where the horizontal maps are inclusions of  faces, and the vertical maps are extraneous modifications. This diagram means that the blow-ups $ \cone^{\BB}_{Q_1} $ and  $\cone^{\BB}_{Q_2}$  are indeed glued along the common face $\cone^{\BB}_{Q}$, but only after collapsing extraneous exceptional divisors. 
The diagram on the right is the realisation, in the category of toplogical spaces, of  the diagram on the left.
The vertical maps are isomorphisms, since they are induced by extraneous modifications. Thus, in the topological realisation $\left| \LA_g^{\trop,\BB}  \right|$,  this subtlety falls away and one does in effect have canonical face maps, as in  the case of the perfect cone complex.
\end{rem}

\section{The determinant locus in $\LA_g^{\trop}$ and its blow-up} \label{section:  DetLocus}

\subsection{The determinant locus in $\LA_g^{\trop}$}

\begin{lem} \label{lem: DetSubFunctor} 
The determinant locus defines a subscheme functor of $\LA^{\trop}_g$: 
\[ \Det: \mathcal{D}^{\perf}_g \To \mathrm{Sch}_{\Q} \ .  \]
The topological realisation  of its complement is 
$ \left|  \LA_g^{\trop} \right| \backslash \left| \Det(\R) \right|  \cong  L \mathcal{P}_g / \GL_g(\Z) $.
\end{lem} 
\begin{proof}The functor $\PF\,  \LA^{\trop}_g$ is given on   $[Q]$  by the first component $\Pro^{Q}$ of $ \cone_{Q} = ( \Pro^{Q}, L_{Q}, \sigma_{Q})$ which is  canonically embedded  in $\Pro(\Quad(V))$. Define 
$\Det([Q]) =  \Det \cap \Pro^{Q} \subseteq \Pro^Q $. 
It is compatible with inclusions of  faces since they  correspond to  inclusions  $\Pro^{Q'} \rightarrow \Pro^Q$ of linear subspaces of $\Pro(\Quad(V))$. The compatibility with isomorphisms of definition \ref{defn:Dperfg} (ii) follows since the determinant locus  is invariant under the action of $\GL_g(\Z)$ (i.e., $\det(P) =0  \Leftrightarrow \det (hPh^T)=0$ for all $h \in \GL_g(\Z)$.)
 The second statement  follows from lemma \ref{lem: sigmameetsDet}.
\end{proof} 
The subscheme $\Det$ is not at infinity (definition \ref{defn: subschemeatinfinity}) since, in general: 
\[ \sigma_{Q} \cap \Det(\R)  \neq \emptyset\ .\]
Indeed, any face $\sigma_{Q'}$ where $[Q'] \leq [Q]$ lies at infinity  is necessarily contained in  $\Det(\R)$.
Nevertheless, we show below how  performing  iterated  blow-ups has the effect of separating  the strict transform of the determinant locus away  from the perfect cones.

\subsection{The strict transform of the determinant locus in  $\LA_g^{\trop,\BB}$}

\begin{lem}
The   strict transform of the determinant  defines a subfunctor of $\PF\LA_g^{\trop, \BB}$:
\[ \widetilde{\Det} :  \mathcal{D}^{\perf,\BB}_g \To \mathrm{Sch}_{\Q}\ . \]
\end{lem} 

\begin{proof}  Recall that $\PF\,  \LA^{\trop}_g$ on $[Q]$ is the iterated blow-up $P^{\BB_Q}$ of $\Pro^{Q}$ along $\BB_{Q}$. Let
$\widetilde{\Det}$ denote the strict transform of $\Det \cap \Pro^Q$.  The functoriality with respect to isomorphisms follows from  
the invariance of the determinant locus under $\GL_g(\Z)$. The functoriality with respect to face diagrams follows from proposition \ref{prop: UniversalANDIdeal}
(i) in the case of face morphisms, and for extraneous modifications by definition since $\widetilde{\Det}$ is a strict transform.
 \end{proof} 

The next section is devoted to proving the following theorem.

\begin{thm} \label{thm: Detatinfinity}
The functor $\widetilde{\Det}: \mathcal{D}_g^{\perf, \BB} \rightarrow \mathrm{Sch}_{\Q}$ of $\LA^{\trop, \BB}_g$ is at infinity, i.e. 
\[  \left|    \LA_g^{\perf, \BB} \right|  \cap \widetilde{\Det} = \emptyset  \ . \]
\end{thm} 
The proof  follows
from corollary \ref{cor: sigmadoesnotmeetDet}.
 Consequently we define a  functor (definition \ref{defn: OpenSubComplex})
\begin{equation} \label{LAminusDet}  \LA_g^{\trop, \BB} \  \backslash  \ \widetilde{\Det} :   \mathcal{D}_g^{\perf, \BB} \To \BLC_{\Q} \ .
\end{equation} 

\subsection{Structure of faces and proof of theorem \ref{thm: Detatinfinity}}
Let $Q$ be a positive definite quadratic form such that $\sigma_Q>0$.  Let 
$\BB $   denote  the set of subspaces $\Pro(\Quad(V/K))$ of $\Pro(\Quad(V))$   occuring in equation \eqref{defn: BlowupSetforQ}, such that 
$ \BB_Q = \BB \cap \Pro^Q$. By proposition \ref{prop: UniversalANDIdeal} (i), we may  form the blow-up $P^{\BB}$ either by blowing up $\Pro^Q$ along $\BB_Q$, or by  first blowing up $\Pro(\Quad(V))$ along $\BB$, and then restricting to the strict transform of $\Pro^Q$. Here we shall do the latter.

\begin{prop}  \label{prop: structureoffacesinblownupcone} An intersection $\mathcal{E}$ of irreducible components of the  exceptional divisor  of the iterated  blow-up $P^{\BB}$ of $\Pro(\Quad(V))$ along $\BB$  is indexed by nested sequences of spaces
\begin{equation} \label{Kflag} 0=K_0 \subset  K_1 \subset K_2 \subset \ldots \subset K_{n+1} = V\ . 
\end{equation}
There is a canonical isomorphism $\mathcal{E} \cong  P^{\BB_1}_1 \times \ldots \times P^{\BB_n}_n$
where $P^{\BB_i}_i$ is the iterated blow-up  of 
\[    \Pro \left( \frac{\Quad(V/K_i)}{\Quad(V/K_{i+1}) }\right)    \] 
along a  certain set $\BB_i$ of   spaces $\Pro(\Quad(V/K)/\Quad(V/K_{i+1}))$, where $K_{i} \subset K \subset K_{i+1}$.  The space $\BB_i$ is closed under taking intersections.

The   intersection  of $\mathcal{E}$ with the strict transform $\widetilde{\Det}$ is  canonically isomorphic to 
\[ \mathcal{E} \cap \widetilde{\Det}  \ \cong  \  \bigcup_{i=1}^{n}  P^{\BB_1}_1 \times \ldots \times P^{\BB_{i-1}}_{i-1} \times \widetilde{\Det}_i \times P^{\BB_{i+1}}_{i+1} \times \ldots \times P^{\BB_n}_n\]
where $\widetilde{\Det}_i \subset P^{\BB_i}_i$  is the strict transform of the  zero locus of the homogeneous polynomial  map   obtained by composing the restriction  map \eqref{RestrictQtoK} with the determinant:
\[    \frac{\Quad(V/K_i)}{\Quad(V/K_{i+1})} \To \Quad\left( K_{i+1}/K_i \right)  \overset{\det}\To \Q\ .\]
Let $\sigma^{\BB}$ be the closure of the inverse image of the interior  of $\sigma_{Q}$, viewed inside $\Pro(\Quad(V))(\R)$. The  set    $ \sigma^{\BB}\cap \mathcal{E}(\R)$ is either empty or a face which is canonically homeomeorphic to  a product
\begin{equation} \label{Quadpolytopefaces} \sigma^{\BB_1}_{1} \times \sigma^{\BB_2}_{2} \times \ldots \times \sigma^{\BB_n}_{n}\end{equation} 
where  $\sigma^{\BB_i}_i \subset P_i(\R)$ is the blow-up of a polyhedral cone 
$ \sigma_i   \subset   \Pro \left(  \Quad(V/K_i) / \Quad(V/K_{i+1}) \right) (\R) $ which is strictly positive. 
\end{prop} 

\begin{proof} The first part follows from the  description  \eqref{proWflag} of exceptional divisors and their intersections in iterated blow-ups in terms of nested sequences of spaces in  $\BB$. The rest follows from repeated application of lemma  
\ref{lem: positivityoffaceandnormal}, remark \ref{rem:  ExtendedLemmaOnStrictPositivity} and  proposition \ref{lem: stricttransformofDet}. 
The fact that $\sigma^{\BB} \cap \mathcal{E}(\R)$ may be empty follows since some  blow-ups may be extraneous to $\sigma$.  \end{proof}

\begin{cor}  \label{cor: sigmadoesnotmeetDet} 
Let  $\cone^{\BB_Q}_{Q}=(P^{\BB_Q},L^{\BB_Q}, \sigma^{\BB}_Q)$ be the  blow-up of $\cone_Q$ (definition \ref{defn: coneQ}).  Then 
\[ \sigma_Q^{\BB} \cap\widetilde{\Det}(\C) = \emptyset \ . \]  

\end{cor} 
\begin{proof}
One has $\sigma_Q^{\BB} \cong \sigma^{\BB}$ since  $\Pro^Q$ is the Zariski closure of $\sigma_Q$.  
Since the  boundary of  $\sigma^{\BB}$ is the union of the interiors of its faces,    it suffices to show that 
the interior of each face does not meet the strict transform of the determinant locus. Since each face of $\sigma_Q$  either lies at infinity (when it is contained in an exceptional divisor), or is a strictly positive polyhedron of the form 
$\sigma_{Q'}$ for some $[Q']\leq [Q]$,  every face of $\sigma^{\BB}$ is  a product of blow-ups of strictly positive polyhedra
by \eqref{Quadpolytopefaces} (on applying proposition \ref{prop: structureoffacesinblownupcone} to a polyhedron $\sigma_{Q'}$). The statement follows from lemma \ref{lem: sigmameetsDet} which implies that their interiors do not meet the strict transform of the determinant locus. 
\end{proof} 

\begin{defn} With the notations of proposition  \ref{prop: structureoffacesinblownupcone}, define a morphism
\begin{equation} \label{defn: pired}
\pi_{\red}:  \mathcal{E} \setminus (\mathcal{E}\cap \widetilde{\Det} )  \To  \prod_{i=0}^n  \left(  \Pro(\Quad(K_{i+1}/K_i))  \ \setminus \ \Det_{K_{i+1}/K_i} \right) \ .
\end{equation} 
It is obtained from the blow-down map $\mathcal{E} \cong P_1^{\BB_1} \times \ldots \times P_n^{\BB_n}\rightarrow \prod_{i=1}^n \Pro \left( \frac{\Quad(V/K_i)}{\Quad(V/K_{i+1}) }\right) $,  by restricting it to the complement of the strict transform of the determinant locus, and composing with the  restriction  map $\eqref{piKdef}$ on each component. 
\end{defn}

\subsection{Inverting extraneous modifications} \label{sect: InvertExtraneous} Consider the  category 
$\mathcal{I}_g^{\perf,\BB}$ which has the same objects as $D_g^{\perf,\BB}$ but in which all morphisms corresponding to extraneous modifications are inverted.
A face diagram in the category $\mathcal{D}_g^{\perf,\BB}$  which is represented by 
\[  [Q'] \longleftarrow F_{Q'} [Q] \longrightarrow [Q]\]
is  therefore replaced, in the category $\mathcal{I}_g^{\perf,\BB}$,  by morphisms
\[   [Q'] \overset{\sim}{\longrightarrow} F_{Q'} [Q] \longrightarrow [Q]\]
giving rise to a genuine face map $[Q'] \rightarrow [Q]$ in $\mathcal{I}_g^{\perf,\BB}$. The functor $\LA_g^{\trop,\BB}$ does not extend to  $\mathcal{I}_g^{\perf,\BB}$, but does if the target category $\BLC_{\Q}$ is replaced by its localisation 
 with respect to  extraneous modifications. We shall not pursue this any further in this paper.

However, by  the final comments of remark \ref{rem: glueconesfacediagrams}, extraneous modifications are already isomorphisms in the topological realisation. As a result,   the topological realisation  functor $ \sigma \LA^{\trop,\BB}_g: \mathcal{D}_g^{\perf,\BB} \rightarrow \Top$ canonically  extends to a functor 
\[  \sigma \LA^{\trop,\BB}_g \  :  \   \mathcal{I}_g^{\perf,\BB} \To \Top
\] 
whose associated topological  space is precisely 
\begin{equation} \label{LAgtropBBasFunctoronInvertedBlowups} \left|  \LA^{\trop,\BB}_g  \right|  = \varinjlim_{x\in \mathcal{I}_g^{\perf,\BB}}  \sigma_x  \ .
\end{equation} 

\begin{rem} The category $ \mathcal{I}_g^{\perf,\BB}$ can be described in terms of nested sequences of objects 
$([Q_1], \ldots, [Q_r])$ where $[Q_i]$ are objects of the category $\mathcal{D}_g^{\perf}$, in much the same way as  definition \ref{defnNestedGraphs}  in the case of graphs. This will not be discussed further here.  
\end{rem}

\subsection{The boundary and interior} 
Let $\partial \mathcal{D}_g^{\perf}$ denote the full subcategory of $\mathcal{D}_g^{\perf}$ generated by classes of positive definite quadratic forms $Q$ such that  $\sigma_Q \subset \Det(\R)$. 
\begin{defn}
Denote the  restriction of the functor $\LA_g^{\trop}$ to  $\partial \mathcal{D}_g^{\perf} \subset \mathcal{D}_g^{\perf}$ by
\[\partial \LA_g^{\trop}: \partial\mathcal{D}_g^{\perf}\rightarrow \PLC_{\Q}\ .\] \end{defn} 
Its topological realisation  $\left|\partial \LA_g^{\trop}  \right| $ is the union  of all faces  of cones in the Vorono\"i decomposition which lie at infinity, i.e., in $P_g^{\rt} \backslash P_g$.   Denote its complement by 
\begin{equation} 
\left|{\LA}_g^{\circ, \trop}  \right| = \left| \LA_g^{\trop}  \right| \setminus  \left|\partial \LA_g^{\trop}  \right| \ .
\end{equation} 
By  lemma \ref{lem: DetSubFunctor},
$\left| \LA_g^{\trop}  \right| \cap \left|\Det(\R) \right| =   \left| \partial \LA_g^{\trop} \right|$
or equivalently,   
\[ \left|{\LA}_g^{\circ, \trop}  \right|  =   \left| \LA_g^{\trop}  \right| \backslash \left(  \left| \LA_g^{\trop}  \right| \cap \left|\Det(\R) \right| \right)  \cong    L \mathcal{P}_g /\GL_g(\Z) \ . \]

Now let us define $ \partial \mathcal{D}_g^{\perf, \BB}$ to be the full subcategory of $\mathcal{D}_g^{\perf,\BB}$ consisting of all faces whose image under canonical blow-downs are objects of $\partial \mathcal{D}_g$. 

\begin{defn}
Denote the  restriction of $\LA_g^{\trop, \BB}$ to $ \partial \mathcal{D}_g^{\perf,\BB}$ by 
\[\partial \LA_g^{\trop, \BB}:  \partial \mathcal{D}_g^{\perf, \BB} \rightarrow \BLC_{\Q}\ . \] 
\end{defn}

We now show the  open  $ \left|{\LA}_g^{\circ, \trop}  \right| $ embeds canonically into  $|\LA_g^{\trop,\BB}|$. 

\begin{prop} There is a morphism of $\PC_{\Q}$ complexes  \[ \partial \LA_g^{\trop, \BB} \rightarrow \partial \LA_g^{\trop}\] given by the pair 
 $(f, \Phi)$ where the functor  $f: \partial \mathcal{D}_g^{\perf, \BB} \rightarrow \partial \mathcal{D}_g^{\perf}$  and the natural transformation $\Phi$ 
is induced by the canonical blow-downs $\cone^{\BB}_Q \rightarrow \cone_Q$. 
There is a canonical embedding  
\begin{equation} \label{EmbedOpenofAgtoBlowup}   \left|  \LA_g^{\circ, \trop}  \right| \overset{\sim}{\To} \left|   \LA_g^{\trop,\BB} \right|  \setminus  \left|  \partial  \LA_g^{\trop,\BB} \right|  
\end{equation} 
whose inverse is  the blow-down  
$\left|   \LA_g^{\trop,\BB} \right|  \setminus  \left|  \partial \LA^{\trop,\BB}_g \right| \overset{\sim}{\To}  \left|   \LA_g^{\trop} \right|  \setminus  \left|   \partial \LA_g^{\trop}  \right|.$
\end{prop} 

\begin{proof}
Consider the full  subcategory $\mathcal{D}^{\circ,\perf}_g$  of $\mathcal{D}_g^{\perf}$ whose objects are  classes $[Q]$, where $Q$ is positive definite, such that $\sigma_Q>0$.  Consider the functor 
$\mathcal{D}^{\circ,\perf}_g  \rightarrow  \Top $ defined by  
\begin{equation}  \label{inproofQtoopensigmaQfunctor}
{[}Q]  \mapsto  \sigma_Q \backslash \left( \sigma_Q \cap \Det(\R) \right) \ .  \end{equation}
The associated topological space is 
\[  \left| \LA^{\circ, \trop}_g \right| = \varinjlim_{Q \in  \mathcal{D}^{\circ,\perf}_g}  \sigma_Q \backslash \left( \sigma_Q \cap \Det(\R) \right)  \ .  \]
Because $ \mathcal{D}_g^{\perf,\BB}$ does not have face maps, there is no  natural functor from $\mathcal{D}^{\circ,\perf}_g$ to $\mathcal{D}_g^{\perf,\BB}$, but  since we are working with topological spaces, for which face maps do exist, we can work in the category $ \mathcal{I}_g^{\perf,\BB}$ (see \S\ref{sect: InvertExtraneous}) for which $[Q']\cong F_{Q'}[Q]$.  For this category, there is a functor  
$\iota: \mathcal{D}^{\circ,\perf}_g \rightarrow   \mathcal{I}_g^{\perf,\BB}$. 
From this point onwards, the argument is similar to that of  proposition \ref{prop: EmbedOpentoBlowup} on replacing 
``$\LM_g$" with ``$\LA_g$" and the graph hypersurface locus $\mathcal{X}$ with $\Det$. 
 Indeed, the blow-down induces a canonical isomorphism
\[  \sigma_Q^{\BB}  \backslash \left(  \sigma_Q^{\BB} \cap \partial \LA^{\trop, \BB}_g  \right)   \  \overset{\sim}{\To} \  \sigma_Q \backslash (   \sigma_Q \cap \Det(\R) ) \]
whose inverse we shall denote by $j$. 
The pair $(\iota, j)$ defines a morphism of functors from the functor $\mathcal{D}^{\circ,\perf}_g  \rightarrow  \Top $ defined by \eqref{inproofQtoopensigmaQfunctor} to the functor $I_g^{\perf,\BB} \rightarrow \Top$ which sends $[Q]$ to $\sigma_Q^{\BB}  \backslash \left(  \sigma_Q^{\BB} \cap \partial \LA^{\trop, \BB}_g  \right)$.  By taking the limit over objects, 
 we deduce  from  \eqref{LAgtropBBasFunctoronInvertedBlowups}  that it induces a continuous map 
 \[  \left| \LA^{\circ, \trop}_g  \right|  \To  \left| \LA^{\trop, \BB}_g  \right| \setminus  \left|  \partial\LA^{\trop, \BB}_g  \right| \ . \]
 The last statement follows since  the composition  of the previous map with the blow-down map 
$  \left| \LA^{\trop, \BB}_g  \right| \setminus  \left|  \partial\LA^{\trop, \BB}_g  \right| \rightarrow   \left| \LA^{\trop}_g  \right| \setminus  \left|  \partial\LA^{\trop}_g  \right| =    
 \left| \LA^{\circ, \trop}_g  \right|$
 is the identity. 
\end{proof}

\section{The blow-up of the tropical Torelli map}  \label{section: Torelli}

\subsection{The tropical Torelli map}
Let $G$ be a connected  graph with zero weights. To each edge $e\in E_G$ we assign a variable $x_e$.   
Define a bilinear  inner product on  $\Z^{E_G}$ by setting $\langle e, e'\rangle = \delta_{e,e'} x_e$ for all edges $e,e'$. 
The graph Laplacian  
is defined to be  its restriction to $H_1(G;\Z) \subset \Z^{E_G}$
which defines  a quadratic form on $ H_1(G;\Z)$  taking values in $\Z[x_e, e\in E_G]$.  See \S \ref{sec: Example} for an example. 
 Equivalently, it is given by  a linear map
\[ \Lambda_G :  H_1(G;\Z) \To  \mathrm{Hom} \left(H_1(G;\Z) , \Z[x_e, e\in e_G] \right) \ .\]  
A graph Laplacian matrix is a matrix representative with respect to a  basis of $H_1(G;\Z)$. It is an $h_G \times h_G$ symmetric matrix whose entries lie in 
$\Z[x_e, e\in E_G]$.  Its determinant, which is well-defined, is equal to the graph polynomial:
\begin{equation} \label{detLambdaPsi} \det \Lambda_G = \Psi_G\ .
\end{equation}

The definition makes sense for any such graph $G$, even in the absence of a metric, since  the edge parameters $x_e$ are interpreted as abstract variables. In the case when $G$ is  a metric graph, with edge lengths $\ell_e \in \R_{\geq 0}$, then the variables $x_e$ are assigned the values $\ell_e$ and the corresponding graph Laplacian is the real quadratic form
\[     \Lambda_G = \Lambda_G \Big|_{x_e=\ell_e}   \quad \in \quad \Quad(H_1(G;\R))   \] 
which is denoted, by abuse of notation, by the same symbol. One shows that the graph Laplacian (for a metric graph) is positive semi-definite, with rational kernel.

The tropical Torelli map, defined and studied in \cite{Nagnibeda, Baker, CaporasoViviani, MikhalinZharkov}, is the map 
\begin{eqnarray}  \label{tgtropicalTorelli} 
t_g: \mathcal{M}_g^{\trop}    & \To&   \mathcal{A}_g^{\trop}   \\
G & \mapsto & [ \Lambda_G  \oplus 0^{w(G)}]\nonumber
\end{eqnarray}  
which sends a weighted metric graph $G$ to the class of the Laplacian of the associated  unweighted metric graph, where $0^{w(G)}$ denotes the $w(G)\times w(G)$ matrix  whose entries are all zero.    A matrix representative $\Lambda_G \oplus 0^{w(G)}$  has $g$ rows and columns, where $g$ is the genus of $G$.  It was proven in \cite{BMV} that the tropical Torelli map is a morphism in a category of stacky fans.
In this section, we extend it to a  map on bordifications. 
Before doing so, we observe that the tropical Torelli map is degenerate on certain kinds of graphs.

\begin{figure}[h]\begin{center} 
\quad {\includegraphics[width=10cm]{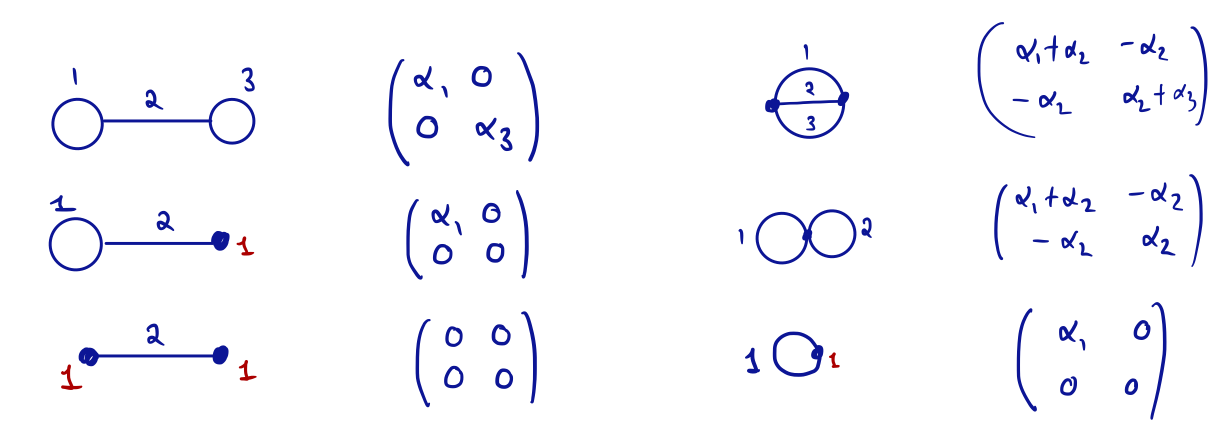}} 
\end{center}
\caption{The graph Laplacians of the stable graphs of genus 2}\label{fig: genus2pictures}
\end{figure}
\begin{example} Consider  the sunrise diagram (figure \ref{fig: genus2pictures}, top right), with edges oriented to the right, and let $c_1= e_1-e_2$, $c_2=e_2-e_3$ denote  cycles in $\Z^{E_G}$ whose image in $H_1(G;\Z)$ is a basis.  The graph Laplacian matrix  with respect to this basis is 
\[  \Lambda_G = \begin{pmatrix} 
x_1+ x_2 & -x_2 \\
-x_2 & x_2+x_3
\end{pmatrix} \]
with determinant $\Psi_G = x_1x_2+x_1x_3+x_2x_3$.  The vanishing of the matrix $\Lambda_G$ implies that $x_1=x_2=x_3=0$.
The Laplacians of the remaining stable graphs of genus 2 are depicted in the same figure.  
The tropical Torelli map is not injective on the dumbbell graph (top left), and is even identically zero on the graph depicted in the bottom left. 
Thus the tropical Torelli map collapses the  dumbbell cell into  the boundary of the  image of the sunrise cell. 
\end{example}

Let $G$ be a connected graph. The Laplacian $\Lambda_G$ defines a  linear map
\begin{equation} \label{LinearMapLambda}  \lambda_G:  \Q^{E_G} \To  \Quad(H_1(G;\Q)) 
\end{equation} 
which is not injective in general, as the previous example shows. When this is the case, the tropical Torelli map does not  extend to a morphism in the category $\PLC_{\Q}$.

\begin{prop} \label{Prop: lambdaInjective} The linear map \eqref{LinearMapLambda} is injective if and only if $G$ is 3-edge connected (the removal of any edge of $G$ is connected and bridgeless). 
\end{prop} 

\begin{proof} Suppose that $G$ is 3-edge connected. Let $e\in E_G$ be any edge of $e$ which is not a self-edge.  The graph $G\backslash e$ is 2-edge connected. By Menger's theorem  there exist two edge-disjoint paths $p_1, p_2$ between the endpoints of $e$ which lie in $G\backslash e$. 
It follows that, by choosing an orientation on the edges of $G$ and adjoining the edge $e$ to the paths $p_1, p_2$, we may find two cycles $c_1,c_2 \in \Z^{E_G}$  
which are independent homology classes in $H_1(G;\Z)$, and only overlap in the single edge $e$.  We may complete these to a family of cycles $c_1,c_2, \ldots, c_h$ which represent a basis of $H_1(G;\Z)$. It follows from the definition of $\Lambda_G$ that the entry 
$( \Lambda_G)_{1,2}  = \pm x_e$, 
where the sign depends on the choice of orientation of edges, and is immaterial. Consequently, the kernel of $\lambda_G$ is contained in the subspace $x_e=0$ of $\Q^{E_G}$. 
Now suppose that $e$ is a self-edge. We may choose the first cycle $c_1$ in a choice of representatives $c_1, \ldots, c_h$ for $H_1(G;\Z)$ to consist of the single edge $e$. In this case, $( \Lambda_G)_{1,1}  =  x_e $
and again we conclude that the kernel of \eqref{LinearMapLambda} is contained in $x_e=0$. Since this holds  for all edges $e\in E_G$, $\lambda_G$ is injective. 

Conversely,  if $G$ is not $3$-edge connected,  there exist distinct edges $e_1,e_2 \in E_G$ such that $G'=G\backslash \{e_1,e_2\}$ is disconnected, and  
 $G/G'$ is the graph with  two vertices $v_1,v_2$ and two edges connecting $v_1$ and $v_2$.   Its Laplacian is represented by the $1\times 1$ matrix with a single entry $x_{e_1}+x_{e_2}$.  It follows that $\lambda_{G/G'}$, and a fortiori $\lambda_G$ (which restricts to $\lambda_{G/G'}$ on the subspace where all $x_e=0$ for $e\in G'$), are  not injective since the linear subspace of $\Q^{E_G}$ given by  $x_{e_1} = -x_{e_2}$ and all remaining $x_e=0$, is contained  in their kernel. 
\end{proof} 

The answer to a more precise question, namely when two tropical curves have equivalent Laplacians, is provided in \cite{CaporasoViviani}. 
Note that  for any bridge $e \in E_G$,  the variable $x_e$ does not appear at all in the Laplacian $\Lambda_G$ and one always has $\Q e \subset \ker (\lambda_G)$. 

\begin{rem} \label{rem: corevanishingLambda}
For any core (bridgeless) graph $G$, the intersection of the kernel of $\lambda_G$ with the region $\widetilde{\sigma}_G = \{x_e\geq 0, e\in E_G\}$  (which is the affine cone over  $\sigma_G$) is zero. This is because the vanishing locus of $\lambda_G$ contains the vanishing locus of $\Psi_G$, which is positive on the interior of $\widetilde{\sigma}_G$ and hence only vanishes along the boundary faces $\widetilde{\sigma}_{G/e} = \widetilde{\sigma}_G \cap V(x_e)$.  One concludes by repeating the argument with $G/e$, which is also core. 
\end{rem} 

\subsection{Reduced moduli space of tropical curves}
The set of $3$-edge connected weighted graphs is stable  under edge contractions. 

\begin{defn} Let $I_g^{\red}$ be the full subcategory of $I_g$ (definition \ref{defnIg}) whose objects are isomorphism classes of  weighted connected graphs with edge-connectivity three. 
Let 
\[ \LM_g^{\red, \trop} :  (I_g^{\red})^{\opp} \To \PLC_{\Q}\]
denote the restriction of the functor $\LM_g^{\trop}$ to $(I_g^{\red})^{\opp}$. Similarly, let 
\[ \mathcal{X}^{\red} :  (I_g^{\red})^{\opp} \To \mathrm{Sch}_{\Q}\]
denote the restriction of the graph locus functor to  $(I_g^{\red})^{\opp}$.
\end{defn}

There is a natural map $\LM_g^{\red, \trop} \rightarrow \LM_g^{\trop}$ and an inclusion 
$ \left| \LM_g^{\red, \trop}\right|  \  \hookrightarrow   \   \left| \LM_g^{ \trop}\right|$. 
In genus 2, for example, the reduced moduli space of tropical curves has only one top-dimensional cell, which is indexed by the sunrise graph (figure \ref{fig: genus2pictures}).

\subsection{The projective tropical  Torelli map}
Let  $G$ be a 3-edge connected weighted graph. By proposition \ref{Prop: lambdaInjective},   $\lambda_G$ is injective and  its projectivisation defines   a linear morphism: 
\begin{eqnarray} 
\Pro^{E_G}  & \overset{[\lambda_G]}{\To}&  \Pro(\Quad(H_1(G;\Q)))  \\ 
(x_e)_{e\in E_G} & \mapsto & \left[  \sum_{e} x_e Q_e\right]  \ ,\nonumber 
\end{eqnarray} 
where $Q_e$ is the quadratic form $Q_e(e_1,e_2) = \delta_{e,e_1} \delta_{e,e_2}$.
A key theorem \cite[Defn. 3.4 and Theorem 6.7]{AlexeevBrunyate} (see also  
\cite[Theorem 4.2.1]{MeloViviani}), states that for any such $G$, there exists a quadratic form $Q_G$ whose minimal vectors are exactly the  set of $Q_e$, for $e\in E_G$. Thus, for  every 3-edge connected graph $G$ of weight zero, the map   $[\lambda_G]: \sigma_G \cong \sigma_{Q_G}$ is an isomorphism. Therefore $[\lambda_G]$ defines a morphism 
\begin{equation} \label{lambdaGiso} 
 \cone_{G} = (\Pro^{E_G}, L_G, \sigma_G)  \quad  \To \quad  \cone_{Q_G} = (  \Pro^{Q_G}, L_{\sigma_{Q_G}}, \sigma_{Q_G})
\end{equation} 
in the category $\PLC_{\Q}$. 
The second part of the  following proposition is a paraphrase of the results \cite{BMV} for stacky fans, transposed to the setting of polyhedral linear complexes.
\begin{prop}
There is a canonical  morphism of  $\PLC_{\Q}$-complexes:
\begin{equation} \label{TropicalTorellionPLC}  \lambda:  \LM^{\red, \trop}_g \rightarrow \LA^{\trop}_g 
\end{equation}
which maps the subscheme $\mathcal{X}^{\red} $ to $\Det$.  Its topological incarnation
\[  \left|   \LM^{\red, \trop}_g  \right| \To \left|\LA^{\trop}_g \right|\]
is  the restriction of the tropical Torelli map \eqref{tgtropicalTorelli} to the reduced moduli space. 
\end{prop} 
\begin{proof}  By the references quoted above, there exists a functor 
$ t_g:  \left(I_g^{\red}\right)^{\opp} \To \mathcal{D}_g^{\perf}$
which sends  a $3$-edge connected graph $G$  to $[Q_G]$.
The pair $(t_g, [\lambda])$ defines the required morphism of $\PLC_{\Q}$-complexes, since $[\lambda]$ is functorial and defines a natural transformation. The compatibility of $\mathcal{X}^{\red}$ and $\Det$ follows since $\Psi_G = \det \Lambda_G$, and hence $[\lambda](X_G) = \Det$.  
\end{proof} 

\subsection{The blow-up of the tropical Torelli map}
The reader may wish to refer to the example studied in  \S\ref{sec: Example}. 

\begin{prop} \label{prop: lambdaBBG} 
Let $G$ be a 3-edge connected   graph. Then the tropical Torelli map induces  a continuous map
$ \lambda^{\BB}_G:  \sigma^{\BB}_G  \rightarrow \sigma^{\BB}_{Q_G}$ on blown-up polyhedra. 
\end{prop} 

\begin{proof} Because  $G$ is 3-edge connected, the map $\lambda_G$ is injective. Since $\lambda_G: \sigma_G \cong \sigma_{Q_G}$ is an isomorphism and since $\Pro^{Q_G}$ is  the Zariski-closure of $\sigma_{Q_G}$ it follows that \eqref{lambdaGiso} is an isomorphism in  $\PLC_{\Q}$. 
It suffices to show that for every face $\sigma_F$ of   $\sigma_{Q_G}$, one has  $\Pro(\Quad(V/K_{\sigma_F}))(\R)\cap \sigma_{Q_G} = \lambda_G \sigma_{G/\gamma}$ for some core subgraph $\gamma \subset G$.  To see this, 
 note that 
since $\sigma_G$ is a simplex, its faces  $\sigma_G \cap L_{\gamma} \cong \sigma_{G/\gamma}$ are in one-to-one correspondence with subgraphs $\gamma \subset G$ where $L_{\gamma} $ is the linear coordinate  space defined by the vanishing of $x_e$, for $e\in E_{\gamma}$. Therefore let $\gamma \subset G$ such that $\lambda: \sigma_{G/\gamma} \cong \sigma_F$. We first show that $K_{\sigma_{F}} = H_1(\gamma;\Q)$. 
For this, choose an isomorphism $H_1(G;\Q) \cong H_1(\gamma; \Q) \oplus H_1(G/\gamma ; \Q)$. The restriction of the graph Laplacian $\Lambda_G$ to  $L_{\gamma}$  takes the block matrix form 
\[  \Lambda_G \Big|_{L_{\gamma}} = \begin{pmatrix}  0 & 0 \\
0 & \Lambda_{G/\gamma} 
\end{pmatrix}\ .  \]
Since $\det(\Lambda_{G/\gamma})=\Psi_{G/\gamma}$  is positive on the interior of $\sigma_G$, the kernel of this quadratic form is precisely $H_1(\gamma;\Q)$ and hence    $K_{\sigma_{F}} = H_1(\gamma;\Q)$.   Let $\gamma^{\mathrm{core}}$ denote the largest core subgraph of $\gamma$. Since $H_1(\gamma;\Q) =  H_1(\gamma^{\mathrm{core}};\Q)$, we deduce that 
 $\Pro(\Quad(V /H_1(\gamma;\Q))(\R)  \cap 
 \lambda  \sigma_G$ contains
  $\lambda \sigma_{\gamma^{\mathrm{core}}}$. To show the reverse inclusion, note that the 
 restriction of $\Lambda_G$ to $H^1(\gamma;\Q)$ is $\Lambda_{\gamma^{\mathrm{core}}}$.  By remark \ref{rem: corevanishingLambda}, 
it vanishes on $\widetilde{\sigma}_{\gamma^{\mathrm{core}}}$ only when $x_e=0$ for all $e\in \gamma^{\mathrm{core}}$ and 
hence  $\Pro(\Quad(V /H_1(\gamma;\Q))(\R)  \cap \lambda \sigma_G =\lambda \sigma_{\gamma^{\mathrm{core}}}$.

Consider   the smallest set of subspaces  $\widetilde{\BB}$ of $\Pro(\Q^{E_G})$ which contains both $\BB^G$ and $\lambda^{-1} \BB^{Q_G}$ and is closed under intersections.
It admits morphisms, by proposition \ref{prop: UniversalANDIdeal} (ii),  to both 
\[ P^{\widetilde{\BB}} \To    P^{\BB^G}\  \hbox{ and }   \   P^{\widetilde{\BB}} \To    P^{\lambda^{-1} \BB^{Q_G}}\ . \]
We have established that the first is an extraneous modification since the subspaces of $\lambda^{-1} \BB^{Q_G}$ meet $\sigma_G$ along $\sigma_{G/\gamma}$ for core $\gamma$ and therefore contain the Zariski closures of $\sigma_{G/\gamma}$, which are in $\BB^G$ by definition.  It gives a homeomorphism $\sigma_G^{\BB} \cong \sigma^{\widetilde{\BB}}_G$. The second can be composed with the map $\lambda: P^{\lambda^{-1} \BB^{Q_G}}\rightarrow P^{\BB^{Q_G}}$ from proposition  \ref{prop: UniversalANDIdeal} (i) and leads to a map $P^{\widetilde{\BB}} \rightarrow P^{\BB^{Q_G}}$. It induces a continuous map $\sigma^{\widetilde{\BB}}_G\rightarrow \sigma_{Q_G}^{\BB}$.  We deduce the existence of a continous map $\sigma_G^{\BB}\rightarrow   \sigma_{Q_G}^{\BB}$ as claimed. 
  \end{proof}

Define $\LM^{\red,\BB}_g$ to be the restriction of the functor $\LM^{\trop,\BB}_g$ to the (opposite of the) subcategory $I_g^{\red,\BB}$ of $I_g^{\BB}$ 
generated by the images of three-edge connected graphs under admissible edge contractions and refinements (definition  \ref{defn: IgBB}).
\begin{thm}
The  tropical Torelli map induces a continuous map 
\[ \lambda^{\BB} :   | \LM^{\red,\BB}_g |  \To | \LA^{\trop,\BB}_g| \] 
which maps $|\partial \LM^{\red,\BB}_g|$ to   $|\partial \LA^{\trop,\BB}_g|$ and is compatible, via  the canonical blow-down maps, with the Torelli map \eqref{TropicalTorellionPLC}. 
\end{thm}

 \begin{proof} This   follows from  propositions  \ref{prop: lambdaBBG} and \eqref{TropicalTorellionPLC}, which makes essential use of the   theorem of \cite{AlexeevBrunyate} on the existence and functoriality of $[Q_G]$, for $G$ a graph. 
 \end{proof}

\section{An example} \label{sec: Example}

The following example serves to illustrate the difference in the blow-up constructions for $\LM_3^{\trop}$ and $\LA_3^{\trop}$.
Let $G=W_3$ be the wheel with 3 spokes, with inner edges oriented inwards  from the center and outer edges oriented counter-clockwise (below). 
First we consider the   quadratic form in three variables $z_1,z_2,z_3$ defined by   $Q_3=z_1^2+z_2^2+z_3^2+ z_1z_2+z_2z_3+z_1z_3$. Its minimal vectors are \[
M=\{ \pm (1,0,0), \pm (0,1,0), \pm (0,0,1), \pm (1,-1,0), \pm (1,0,-1), \pm (0,1,-1)\} \ .\]  
The cell $\sigma_{Q_3}$ is the convex hull of the six matrices 
\[  \begin{pmatrix} 1 & 0 & 0 \\ 
0 & 0 & 0 \\ 
0 & 0 & 0 \end{pmatrix} \ , \  \begin{pmatrix} 0 & 0 & 0 \\ 
0 & 1 & 0 \\ 
0 & 0 & 0 \end{pmatrix} \ , \
\begin{pmatrix} 0 & 0 & 0 \\ 
0 & 0 & 0 \\ 
0 & 0 & 1 \end{pmatrix} \ , \
\begin{pmatrix} 1 & -1 & 0 \\ 
-1 & 1 & 0 \\ 
0 & 0 & 0 \end{pmatrix}\ ,  \ \begin{pmatrix} 1 & 0 & -1 \\ 
0 & 0 & 0 \\ 
-1 & 0 & 1 \end{pmatrix}
\ ,  \ \begin{pmatrix} 0 & 0 & 0 \\ 
0 & 1 & -1 \\ 
0 & -1 & 1 \end{pmatrix}
  \]
and is given explicitly by the set of projective equivalence classes of matrices
\[    \begin{pmatrix} \alpha_1 + \alpha_4 +\alpha_5 & -\alpha_4& -\alpha_5 \\ 
-\alpha_4 &  \alpha_2 + \alpha_4 +\alpha_6  & -\alpha_6 \\ 
-\alpha_5 & -\alpha_6 & \alpha_3 + \alpha_5 +\alpha_6 \end{pmatrix}         \] 
for $\alpha_1,\ldots, \alpha_6 \in \R_{\geq 0}$.  The essential faces of $\sigma_{Q_3}$ are   those where this matrix has null space exactly $K$ for some subspace $K \subset \Q^3$. For example, if $K$ is the one-dimensional subspace corresponding to $z_1$, the intersection $\Quad(\Q^3/K) \cap \sigma_{Q_3}$ corresponds to when the  matrix has vanishing first row and column: $\alpha_1=\alpha_4=\alpha_5=0$. 
Such a face will be blown-up. Examples of non-essential faces are given by $\alpha_i=0$, or even $\alpha_i=\alpha_j=0$  for any $i,j$, since the matrix above continues to have full rank along these loci.  They will not be blown-up.

This matrix can be  realised as the graph Laplacian for a particular choice of homology basis for the wheel with 3 spokes. However, to illustrate our point,  we shall consider a different  basis of cycles for $H_1(G;\Z)$ consisting of edges $\{1,2,4,5\}$, $\{2,4,6\}$, $\{3,5,4\}$:
\[ h_1 = e_1 + e_2 +e_4 -e_5 \ , \  h_2 = e_2 + e_4 - e_6\ , \ h_3 = e_3-e_4 + e_5\]    \begin{figure}[h]\begin{center}
\quad {\includegraphics[width=3cm]{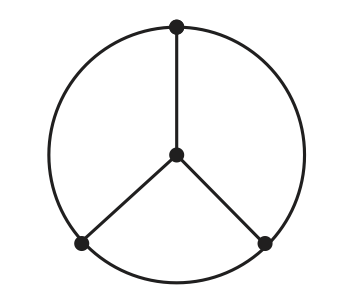}} 
\put(-80,50){$1$}\put(-12,50){$3$}\put(-50,-5){$2$}
\put(-50,45){$5$}\put(-62,27){$6$}
\put(-30,27){$4$}
\end{center}
\end{figure}
With respect to this basis and orientation,  the graph Laplacian is represented by the $3\times 3$ matrix
\[ \Lambda_G=  \mathcal{H}^T_G \, D_G \,\mathcal{H}_G = 
   \begin{pmatrix} 
x_1 +x_2+x_4 +x_5 &  x_2 + x_4 & -x_4 -x_5 \\
x_2+x_4 & x_2+ x_4+x_6 & - x_4 \\
-x_4-x_5 & -x_4 & x_3+x_4+x_5 
\end{pmatrix}\ . \]
The $(i,j)$th entry is given by $\langle h_i,h_j\rangle$ where $\langle , \rangle$ is the bilinear product  on $\Z^{E_G}$ such that $\langle e_i, e_j \rangle = \delta_{ij} x_i$. The matrix $\Lambda_G$ gives a  particular choice of isomorphism from the coordinate simplex $\sigma_{W_3}$ in $\Pro^5(\R)$ to $\sigma_{Q_3}$.  The essential faces of $\sigma_{Q_3}$, via this isomorphism, correspond precisely to the faces defined where $x_e=0, e \in E(\gamma)$, where $\gamma$ ranges over core (bridgeless) graphs; all other faces are non-essential.

Consider the subgraph $\gamma$ spanned by the edges $1,2,4,5$. The quotient graph $G/\gamma$ has a single vertex, and two edges $3, 6$. Setting $x_e=0$ for $e\in E_{\gamma}$ one obtains
\[   \Lambda_G\big|_{L_\gamma} =   
   \begin{pmatrix} 
0  & 0 & 0 \\
0 &  x_6 & 0 \\
0 & 0  & x_3 
\end{pmatrix}\  , 
\qquad \hbox{ with Zariski closure}  \qquad 
V_{\sigma_{G/\gamma}}= \left\{   
   \begin{pmatrix} 
0  & 0 & 0 \\
0 &  * & 0 \\
0 & 0  & * 
\end{pmatrix} \right\}   \ . 
\]
It defines a quadratic form on $V= \bigoplus_{i=1}^3  h_i \Q \cong H_1(G;\Q)$ whose kernel is the one-dimensional subspace $K=h_1 \Q$. The space of quadratic forms $\Quad(V/K)$ with null space $K$  may be identified with the space of symmetric matrices of the form 
\[ \Quad(V/K) \cong \left\{   \begin{pmatrix} 
0  & 0 & 0 \\
0 &  * & * \\
0 & * & *
\end{pmatrix} \right\}\ . \]

The image of $\Lambda_G|_{L_\gamma}$ has codimension $1$ inside it, and therefore is not Zariski-dense.  This example illustrates the difference between $\LM_3^{\trop,\BB}$ and $\LA_3^{\trop,\BB}$.  The former involves blowing  up the subspace $L_{\gamma}= \{x_e=0, e\in E_{\gamma}\}$, corresponding to the core subgraph $\gamma$, whose image under the Torelli map $\lambda_G$ is the projective space 
 $\Pro(V_{\sigma_{G/\gamma}} )\cong \Pro^1$. However, in $\LA_3^{\trop,\BB}$ it is the space $\Pro(\Quad(V/K))\cong \Pro^2$ which is blown up,  which strictly contains $\Pro(V_{\sigma_{G/\gamma}} )$.  
  Indeed, its preimage under $\lambda_G$ equals
 \[ V(x_1+x_2+x_4+x_5,  x_2+ x_4 , x_4+x_5) \subset \Pro(\Q^{E_G}) \]
 which strictly contains the locus $V(x_1,x_2,x_4,x_5)$ corresponding to $\gamma$. It follows that  $\lambda_G$ does not extend to a morphism of blow-ups from $P^{\BB_G}$ to $P^{\BB^{\min,\sigma_{Q_3}}}$. 
Nevertheless  $\Pro(V_{\sigma})(\R)$ and $\Pro(\Quad(V/K))(\R)$ both meet  $\sigma_G$  along the same face $\sigma_{G/\gamma}$, and so
$\lambda$ does extend to a morphism between the  corresponding blown-up polytopes $\sigma^{\BB}_{W_3}$ and $\sigma^{\BB}_{Q_3}$.

\section{Canonical bi-invariant forms and their integrals}  \label{section: CanForms}
 
\subsection{Canonical forms}  \label{subsect: CanForms} 
Employing the same notation as in \cite[\S4.1]{BrSigma},
we define \begin{equation} \Omega_{\can}  = \bigwedge \bigoplus_{k\geq 1} \Q\, \omega^{4k+1}  \end{equation}
to be the graded-commutative Hopf algebra generated by elements $\omega^{4k+1}$ in degree $4k+1$, for $k\geq 1$, where the coproduct 
$\Delta^{\can}:  \Omega^{\can}  \rightarrow  \Omega^{\can}  \otimes_{\Q}  \Omega^{\can}  $ has the property that the $\omega^{4k+1}$ are primitive, i.e., 
$\Delta^{\can} (\omega^{4k+1}) =  \omega^{4k+1} \otimes 1 + 1 \otimes \omega^{4k+1}$. 

Let $R= \bigoplus_{n\geq0 } R^n$ be a differential graded algebra with differential $d$ of degree $+1$,  and let $X \in \GL_g(R^0)$ be an invertible  $g\times g$ symmetric matrix. Let  $dX$ denote the matrix whose  entries $(dX)_{ij}= dX_{ij}$ are in $R^1$. We define
\[ \omega_X^{4k+1} = \tr\left( ( X^{-1} dX)^{4k+1} \right) \quad \in \quad R^{4k+1} \   .   \]
More generally,  any $\omega\in \Omega_{\can}$  is a polynomial in the $\omega^{4k+1}$. We write $\omega_X$  for the corresponding polynomial in the $\omega_X^{4k+1}$.  One can show that  the  $\omega_X$  are:
\begin{enumerate}  \setlength\itemsep{0.03in}
\item  \emph{Closed}:  $d\omega_X =0$ for any $X$ as above.
\item \emph{Bi-invariant}: $\omega_{AXB} = \omega_X$ for all $A, B \in \GL_g(R^0)$ such that $dA=dB=0$.  
\item \emph{Additive}:  if we write $\Delta^{\can} \omega = \sum \omega' \otimes \omega''$  and $X_1,X_2 \in   \GL_g(R^0)$, then 
\[ \omega_{X_1\oplus X_2} = \sum \omega'_{X_1} \wedge  \omega''_{X_2} \ . \]  
\item \emph{Projective}:  $\omega_{\lambda X} = \omega_X$ for all $\lambda \in (R^0)^{\times}$,
\item  \emph{Vanish in low rank}:  $\omega^{4k+1}_X =0 $ \hbox{ if }  $4 k+1>  2 g.$
\end{enumerate}
See \cite[Lemmas 4.3, 4.4 and Proposition 4.5]{BrSigma} for proofs. In \emph{loc. cit.} these forms were studied on the moduli space $|\LM_g^{\trop}|$ and were used to study  the cohomology of the  even commutative graph complex. In this paper we shall study them on the space $|\LA_g^{\trop}|$, and in particular, reprove and strengthen the results of Borel \cite{Borel}, who showed that  
\[\Omega_{\can} \overset{\sim}{\To} \varprojlim_g H^{\bullet}(\GL_g(\Z);\R)\ .\]
%As we shall see, the forms $\omega^{4k+1}$ play rather different roles on $|\LM_g^{\trop}|$ and $|\LA_g^{\trop}|$.
\subsection{Forms of `compact type'} 
Let $g>1$ be odd. By the vanishing property (5), only a finite subspace of $\Omega_{\can}$ is non-zero on  $g\times g$ symmetric matrices.

\begin{defn} \label{defn: Omegagnotions}
 Let  $\Omega(g) \subset \Omega_{\can}$ denote the graded exterior algebra generated by $1$,  $\omega^5$ , \ldots, $\omega^{2g-1}$.  It is naturally a quotient of $\Omega_{\can}$ via the map $\Omega_{\can} \rightarrow \Omega(g)$ which  sends $\omega^{4k+1}$ to zero for all $2 k > g$. We shall call a form \emph{simple} if it is a  monomial in the generators of $\Omega(g)$,  i.e., a product of primitive elements.  An element of $\Omega(g)$ is a finite linear combination of simple forms.
  Call a form $\omega \in \Omega(g)$ of \emph{compact type} in genus  $g$ if it is divisible by $\omega^{2g-1}$
 \begin{equation} \label{defnCompactType} \omega = \eta \wedge \omega^{2g-1} \quad \hbox{ for some }\  \eta \in \Omega(g) \ .  \end{equation}
  Otherwise,  a simple form $\omega$ will be called  of 
  \emph{non-compact type}. Denote the space spanned by  forms of compact  (resp. non-compact) type by  $\Omega_c(g)$  (resp. $\Omega_{nc}(g)$).  We have
 \[ \Omega(g) = \Omega_c(g) \oplus \Omega_{nc}(g)\]
 The subspace  $\Omega_c(g) \subset \Omega(g)$ is an ideal. 
 \end{defn}

 Consider the volume form 
 \begin{equation} \label{volg}  \vol_g = \omega^5 \wedge \ldots \wedge \omega^{2g-1} \quad \in \quad \Omega_c(g) \end{equation}
 which is of degree $d_g-1$, where $d_g=\frac{g(g+1)}{2}$. It is  simple, and of compact type.
 
 \begin{defn}  Let $\omega \in \Omega(g)$ be simple. There is a unique  $\star \omega \in \Omega(g)$ such that
 \[ \omega \wedge \star \omega = \vol_g \ .\]
 The   map $\star$ extends by linearity to an endomorphism (the Hodge star operator) on $\Omega(g)$ which interchanges the subspaces of forms of compact and non-compact types:
 \[ \star: \Omega_c(g) \overset{\sim}{\To} \Omega_{nc}(g)\ .\] 
 It is an involution up to sign: $\star \star \omega = \pm \omega$. 
  \end{defn} 
 \begin{example} Let $g=7$. Then $\vol_7 = \omega^5 \wedge \omega^9 \wedge \omega^{13}$.  The following table lists the simple elements in $\Omega_{nc}(7)$, and below them, their images     in $\Omega_{c}(7)$ under the map $\star$: 
 \[
\begin{array}{cccc}
  \quad 1 \quad &  \quad \omega^5 \quad  & \quad \omega^9\quad  &\quad  \omega^5\wedge \omega^9  \\
 \quad  \downarrow\quad  & \quad \downarrow\quad   &  \quad \downarrow\quad & \quad \downarrow \\
 \omega^5 \wedge \omega^9 \wedge \omega^{13}  \quad   &  \quad    \omega^9 \wedge \omega^{13}  \quad   & \quad   - \,\omega^5 \wedge \omega^{13}    \quad  & \quad  \omega^{13} 
\end{array}
\]
  \end{example} 
\subsection{Canonical forms on $\LM_g^{\trop}$ and $\LA_g^{\trop}$. } 

\begin{lem} Let $\omega \in \Omega^d_{\can}$ be a canonical form of degree $d$, and let $V$ be a vector space over $k\subset \R$.  Then $\omega$ defines a closed projective differential form 
\[ \omega\quad  \in \quad \Omega^d(\Pro(\Quad(V)) \backslash \Det) \]
which is invariant under the action of $\GL(V)$. 
\end{lem} 
\begin{proof}
The projectivity follows from  \S\ref{subsect: CanForms} (4), the invariance from (2). 
\end{proof} 
\begin{prop}\label{prop: CanFormsOnLA}
A canonical form $\omega \in \Omega^d(g)$   defines an algebraic differential form (\S \ref{sect: differentialforms}): 
\begin{equation} \label{omegainOmegadLag}  \omega  \quad \in \quad \Omega^d ( \LA_g^{\trop} \backslash \Det)  \ . 
\end{equation}
It restricts to a smooth differential form, also denoted by $\omega$, on 
\[ \omega \quad \in \quad \mathcal{A}^d \left( \left|    \LA^{\circ, \trop}_g  \right| \right) = \mathcal{A}^d \left(  L \mathcal{P}_g / \GL_g(\Z) \right) \ .   \]
 \end{prop}
 \begin{proof}  
By definition \ref{defn: MeroForm}  a differential form  must verify the compatibility relations \eqref{CompatibilityForForms}. A canonical  form $\omega$ satisfies the compatibilities for the two types of morphisms in the category $\mathcal{D}_g^{\perf}$ of definition \ref{defn:Dperfg}: for face morphisms, the compatibility is automatic, and for isomorphisms, this follows from the bi-invariance \S\ref{subsect: CanForms} (2).  The fact that $\omega$ only has poles along the determinant locus follows immediately from its definition since for any square matrix $X$, the one-form $X^{-1} dX$ is regular outside the locus  $\det(X)=0$. 
 \end{proof} 
 
 In  \cite{BrSigma}, we studied the canonical forms  $\omega_G = \omega_{\Lambda_G}$ on  the polyhedral linear configuration  associated to  a graph $G$.    \begin{prop}  A form  $\omega \in \Omega^d(g)$  defines an algebraic  form  
$ \omega  \in\Omega^d ( \LM_g^{\trop} \backslash \mathcal{X})$.
It restricts to a smooth differential form, also denoted by $\omega$, on 
\[ \omega \quad \in \quad \mathcal{A}^d \left( \left|    \LM^{\trop}_g \right| \backslash \left|   \partial \LM^{\trop}_g \right|  \right) \ .   \]
The restriction of  $\omega$ to the reduced moduli space $\LM_g^{\red, \trop}$ is the pull-back $\lambda_G^*\omega$ of the corresponding canonical forms \eqref{omegainOmegadLag} on $\LA_g^{\trop}$ by the tropical Torelli map.
 \end{prop} 
 This follows from Theorem 2.1 in \cite{BrSigma}:  the first statement is equivalent to the `compatibility' and `equivariance' property $(2.2)$ \emph{loc. cit.}; the second on applying proposition \ref{prop: GraphLocusFunctor}. That they are pulled back from the  tropical Torelli map is automatic   from their definition. 
 \subsection{Behaviour at infinity}
 A key property of  canonical forms  is their behaviour on  blow-ups and  in particular, a factorisation property  at infinity.

 \begin{prop} \label{prop: blowupforms} 
 Let $0 \leq K \subset V$ be a subspace and let 
$\pi: P^{\BB} \rightarrow   \Pro(\Quad(V))$
 denote the blow-up along $\BB=\{\Pro(\Quad(V/K))\}$.   Let $\omega \in \Omega^d_{\can}$. Denote its coproduct by 
 \[ \Delta^{\can} \omega = \sum_i \omega_i' \otimes \omega_i'' \ .\] 
 The pull-back $\pi^*\omega$ to $P^{\BB}$ has no poles along the exceptional divisor, and hence defines a regular  differential form on the complement of the strict transform $\widetilde{\Det}$ of the determinant:
 \[ \pi^*(\omega) \quad \in \quad \Omega^{d}  \left( \Pro \, \setminus \,  \widetilde{\Det} \right) \ . \]  
Its restriction to the exceptional divisor  $\mathcal{E}$, which we recall by proposition \ref{lem: stricttransformofDet} is canonically isomorphic to the product
$ \mathcal{E}  = \Pro\left(\Quad(V/K)\right ) \times \Pro \left( \Quad(V) / \Quad(V/K) \right)$,
 satisfies 
 \begin{equation} \label{omegafactorisationalongE} \pi^*(\omega) \big|_{\mathcal{E}} = \sum_i \omega_i' \wedge \pi_K^*(\omega_i'') \end{equation} 
 where $\omega_i' \in  \Omega^{\bullet}  \left( \Pro(\Quad(V/K)) \backslash \Det_{V/K} \right)$,
 $\omega_i'' \in\Omega^{\bullet}  \left(\Pro \left( \Quad(K)\right)  \backslash \Det_{K} \right)$, and 
 $\pi_K$ is the map \eqref{piKdef}.
 \end{prop} 
 \begin{proof} The proof is almost identical to that of theorem 7.4 in \cite{BrSigma} to which we refer for further details.  The key point is to 
 choose a complementary subspace $V\cong  K \oplus C$ and write $Q$ in block form as in \eqref{inproofQblockform}. Then we may 
  write formula \eqref{LambdaBlockForm} in the form:
 \[ s_K^* Q = \begin{pmatrix} zQ_K & zP \\ zP^T & Q_C\end{pmatrix}  = \Lambda U  \quad \hbox{ where we set }   \quad  \Lambda =  \begin{pmatrix} zQ_K &   0 \\ 0 & Q_C\end{pmatrix}\ ,  \]
and  where 
 \[  U = \Lambda^{-1} s_K^* Q \equiv   \begin{pmatrix} 1 &   Q_K^{-1}P \\ 0 & 1 \end{pmatrix} \pmod{z}\ .  \]
 For simplicity, assume that $\omega = \omega^{4k+1}$ is primitive. Then 
  \[ \omega^{4k+1}_{\Lambda U}= \tr \left( (  U^{-1} \Lambda^{-1} (d\Lambda) U +  U^{-1}dU )^{4k+1} \right) =\tr \left( \left( \Lambda^{-1} d\Lambda  +  dU. U^{-1}\right)^{4k+1} \right)   \ .   \]
 The argument follows exactly as in theorem 7.4 of \emph{loc. cit.} to conclude that 
 \[ \omega_{ s_K^* Q}   \Big|_{z=0} = \omega^{4k+1}_{Q_C} + \omega^{4k+1}_{ Q_K} \]
 and hence $\pi^*(\omega) \big|_{\mathcal{E}} =    \omega^{4k+1}_{Q_C} \wedge 1  + 1\wedge \omega^{4k+1}_{ Q_K} $. The case when $\omega$ is a simple form follows by multiplying primitive forms together; the general case follows   by linearity. 
  \end{proof} 
   \begin{rem} \label{rem: canformsarereduced}
 This implies that the restriction of $\pi^*(\omega)$  to the exceptional divisor is actually just the pre-image of canonical forms 
  under the collapsing map \S\ref{sect: CollapseNormal} of the normal bundle. 
 \end{rem}

 It may often happen that the restriction $\pi^*(\omega)$ to the exceptional divisor $\mathcal{E}$ vanishes. There is one important  situation in which this always occurs. 
 \begin{cor} \label{cor: compacttypevanishesonexceptional}
 Let $\omega \in \Omega^d_c(g)$ be of compact type, where $g=\dim(V)>1$  is odd. 
 Then, in the situation of   proposition \ref{prop: blowupforms}, $\pi^*\omega$ vanishes along the exceptional divisor $\mathcal{E}$.  
 \end{cor} 
 \begin{proof} We  may assume that $\omega$ is simple. Since $\omega$ is of compact type, it follows that in the expression
$\Delta^{\can} \omega = \sum_i \omega_i' \otimes \omega_i''$  either $\omega_i'$ or $\omega_i''$ is of compact type for each $i$, since one of these two forms must  be divisible by $\omega^{2g-1}$. Note that since $g=2k+1$ is odd, $\omega^{2g-1} = \omega^{4k+1}$. Property  \S\ref{subsect: CanForms} (5) implies that $\omega^{2g-1}$ vanishes on all matrices of rank  $h$ satisfying $2h<2g-1$, or equivalently, $h<g$. 
Let us choose a complement $V\cong K \oplus C$ and use the notations of proposition \ref{prop: blowupforms}.  Since $K\neq 0$,  
 both $Q_K$ and $Q_C$ have rank strictly smaller than $g$, and hence 
either    $(\omega_i')_{Q_C}  = 0 $  or 
$(\omega_i'')_{Q_K}  = 0$  
for every $i$, since one or other is divisible by $\omega^{2g-1}$.  It follows from the formula  \eqref{omegafactorisationalongE} that
   $\pi^*(\omega)$ vanishes along $\mathcal{E}$.   \end{proof} 

The essential part of the proof is the fact that $\Omega_c(g)$ is a Hopf ideal in $\Omega(g)$: 
\[ \Delta^{\can} (\Omega_c(g))  \ \subset \  \left(\Omega_c(g) \otimes_{\Q} \Omega_{nc}(g)\right)  +    \left(\Omega_{nc}(g) \otimes_{\Q} \Omega_{c}(g)\right)\ . \]

\subsection{Canonical forms on the blow-ups $\LA_g^{\trop,\BB}$ and $\LM_g^{\trop, \BB}$}

\begin{thm}  \label{thm:  CanFormsOnAgBlowup} Let $\omega \in \Omega^d_{\can}$. It defines a closed algebraic  differential form 
\begin{equation} \label{omegaMeroOnAg}   \widetilde{\omega} \in  \Omega^d \left( \LA_g^{\trop,\BB} \backslash \widetilde{\Det} \right) \end{equation} 
which restricts to a  smooth form, also denoted by the same notation, 
\begin{equation} \label{omegasmoothAgTop} \widetilde{\omega} \in  \mathcal{A}^d \left( \left| \LA_g^{\trop,\BB} \right| \right) \end{equation}
on the topological realisation.  In the case when $g>1$ is odd and $\omega \in \Omega^d_c(g)$ is of compact type, then   the restriction of \eqref{omegaMeroOnAg} to the boundary 
$\partial \LA_g^{\trop,\BB} \  \backslash \  ( \partial \LA_g^{\trop,\BB}  \cap  \widetilde{\Det}   ) $ vanishes.  In particular, it follows that the form  
$ \eqref{omegasmoothAgTop}$ on $ \left| \LA_g^{\trop,\BB} \right|$ vanishes on  $\left| \partial \LA_g^{\trop,\BB} \right|$. 
\end{thm} 

\begin{proof}
Let $Q$ be a positive definite quadratic form such that $\sigma_Q>0$.  Let  $\widetilde{\omega} = \pi^* \omega$ denote the pull-back of the form $\omega$ along the iterated blow-up 
$\pi: P^{\BB_Q} \rightarrow \Pro^Q.$
By repeated application of proposition \ref{prop: blowupforms} for each blow-up in the formation of   $P^{\BB_Q}$, we deduce that $\widetilde{\omega}$ is a meromorphic differential form with no poles along the exceptional divisor. 
It follows that 
$\widetilde{\omega} \in \Omega^d ( P^{\BB_Q} \  \backslash  \    \widetilde{\Det})$.
By restricting it to faces, we obtain  a family of  forms  in $\Omega^d  (( \PF   \LA^{\trop, \BB}_g    \backslash \widetilde{\Det} )(x) )$ for every object $x$ of   $\mathcal{D}_g^{\perf}$.  By the universal property of strict transforms (proposition \ref{prop: UniversalANDIdeal} (i)) and the fact that restriction of differential  forms is compatible with pull-back along blow-ups, the compatibility conditions   \eqref{CompatibilityForForms} for $\widetilde{\omega}$ follow from those of $\omega$,  established in proposition \ref{prop: CanFormsOnLA}.  This is  because the objects in a  face diagram \eqref{facediagram}  are the blow-ups of a face map in $\mathcal{D}_g^{\perf}$, and the form $\widetilde{\omega}$ on each component is the pull-back of $\omega$ along a suitable blow-up.

The statement \eqref{omegasmoothAgTop} is a consequence of the fact that $ \left| \LA_g^{\trop,\BB} \right|$ does not meet the strict transform of the determinant  locus  (theorem \ref{thm: Detatinfinity}). When $\omega$ is of compact type, $\widetilde{\omega}$  vanishes along exceptional divisors by 
corollary \ref{cor: compacttypevanishesonexceptional}.
\end{proof}

The following  statement for $\LM_g^{\trop, \BB}$ may be proved using \cite[Theorem 7.4]{BrSigma}. 
\begin{thm} Let $\omega \in \Omega^d_{\can}$.  It defines both a holomorphic differential form on $\LM_g^{\trop,\BB} \backslash \widetilde{\mathcal{X}}$, and  restricts to a smooth form on     $\left| \LM_g^{\trop,\BB} \right|$   which we may  denote by:  
\begin{equation} \label{omegaMeroOnMg}   \widetilde{\omega} \in  \Omega^d \left( \LM_g^{\trop,\BB} \backslash \widetilde{\mathcal{X}} \right) \qquad \hbox{resp. }  \qquad 
\widetilde{\omega} \in  \mathcal{A}^d \left( \left| \LM_g^{\trop,\BB} \right| \right)\ .
 \end{equation}
\end{thm} 
Note that in the case of $\LM_g^{\trop,\BB}$ the existence of the dual graph Laplacian  implies that most forms of compact type in genus $g$ actually vanish  (see  \cite[Lemma 8.4]{BrSigma}), except for the primitive ones. 
 See the discussion in \S \ref{sect: VanishUnderTorelli}.

\subsection{Integrals of canonical forms} \label{sect: CanonicalIntegrals}
Canonical integrals  define new invariants associated to the minimal vectors of quadratic forms. 

\begin{cor} \label{cor: IntQcoverges}
Let $Q$ be a positive definite quadratic form such that $\sigma_Q$ is strictly positive.
Then for any canonical form $\omega \in \Omega^d_{\can}$ where $d=\dim \sigma_Q$, the integral 
\[ I_Q(\omega) = \int_{\sigma_{Q}} \omega \]
is finite.  In the case where $Q= Q_G$ is associated to a graph, the  integral $I_{Q_G}(\omega)$  reduces  to the  canonical integrals $I_G(\omega)$ studied in \cite[\S8]{BrSigma}. 
\end{cor} 
\begin{proof} By passing to the blow-up $P^{\BB_Q} \rightarrow \Pro^Q$, we may write  
\[  I_Q(\omega) = \int_{\sigma^{\BB}_Q} \widetilde{\omega}\]
since $\overset{\circ}{\sigma}_Q$ is contained in $\sigma^{\BB}_Q$, and has complement  of Lebesgue measure zero. 
By theorem  \ref{thm: Detatinfinity}, $\widetilde{\omega}$ has no poles along the compact region $\sigma^{\BB}_Q$ and therefore the integral is finite. 
\end{proof} 

These integrals satisfy many interesting relations. The analogous relations for graphs were studied in \cite[\S 8.2]{BrSigma} and related to the `waterfall' spectral sequence of \cite{SpectralSequenceGC2}.

\begin{thm} \label{thm: Stokes} (Stokes relation).  Let $Q$ be a positive definite quadratic form such that $\sigma_{Q}>0$, and let  $\omega \in \Omega_{\can}$ have  degree $\dim \sigma_{Q}-1$. Let $\Delta^{\can}\omega = \sum_i \omega'_i \otimes \omega''_i$. Then
\begin{equation} \label{StokesForQ}    0= \sum_{Q'} \int_{\sigma_{Q'}} \omega +   \sum_i \sum_F  \int_{\sigma_F} \omega'_i \int_{\sigma_Q/\sigma_F}  \pi^*_{K_{\sigma_F}} (\omega''_i)  
\end{equation} 
where $Q'\leq Q$ denote the facets of $\sigma_Q$ which are strictly positive (not contained in $\Det(\R)$), and $\sigma^{\BB}_F\times (\sigma_Q/\sigma_F)^{\BB}$ denote the set of facets of 
its blow-up $\sigma^{\BB}_Q$ which are at infinity (in other words, $\sigma_F$ is an essential face of $\sigma_Q$ at infinity, and $\sigma_Q/\sigma_F$ is
defined in \eqref{notation:sigmamodF}). The map $\pi_{K_{\sigma_F}}$ is defined in lemma \ref{lem: PiKdef}.
\end{thm} 

\begin{proof} This follows from an application of Stokes' formula:
\[ 0 = \int_{\sigma^{\BB}_Q} d \widetilde{\omega} = \int_{\partial  \sigma^{\BB}_Q} \widetilde{\omega}\ , \]
 the description of the boundary facets of $\sigma^{\BB}_{Q}$, and the factorisation formula \eqref{omegafactorisationalongE}.
\end{proof}

\subsection{Compactly-supported representatives}
Every form of compact type has a representative which vanishes identically  in a neighbourhood of the boundary. 
\begin{cor} \label{cor: CompactSupportedomega} 
Let $g>1$ be odd. The cohomology class of the differential form $\omega^{2g-1}$ in $ H_{dR}^{2g-1}( \left|\LA_g^{\trop,\BB} \right|, \left| \partial \LA_g^{\trop,\BB} \right|)$ has a compactly supported representative \[  [\omega^{2g-1}_c]  \  \in \  H_{dR,c}^{2g-1} ( \left| \LA_g^{\circ, \trop} \right| )\ . \] 
 Having chosen  $\omega^{2g-1}_c$,  any other  form $\omega = \omega^{2g-1}\wedge \eta \in \Omega_c(g)$ of compact type  has a canonical  compactly supported representative $\omega_c = \omega^{2g-1}_c \wedge \eta$. 
\end{cor} 

\begin{proof} By  \eqref{dRCompactSupportsToRelative}, and \eqref{EmbedOpenofAgtoBlowup} we have
$ H^{\bullet}_{dR,c} ( \left|\LA_g^{\circ,\trop} \right|) \cong   H_{dR}^{\bullet}( \left|\LA_g^{\trop,\BB} \right|, \left| \partial \LA_g^{\trop,\BB} \right|)$.
The form $\omega^{2g-1}$, being of compact type, vanishes on $ \left| \partial \LA_g^{\trop,\BB} \right|$ by theorem \ref{thm:  CanFormsOnAgBlowup} and therefore defines a class in the latter cohomology group.
\end{proof} 
The fact that $\omega^{2g-1}$ has a compactly supported representative on $L\mathcal{P}_g/\GL_g(\Z)$  was claimed without proof in \cite{LeeUnstable}.
Since  one also has $\left|\LA_g^{\circ,\trop} \right| =  \left|\LA_g^{\trop} \right| \setminus \left| \partial \LA_g^{\trop} \right| $, 
the blow-down map $\pi:|\LA_g^{\trop,\BB}| \rightarrow \left|\LA_g^{\trop}\right|$  induces an isomorphism (excision):
\[   \pi^*:  H_{dR}^{\bullet}( \left|\LA_g^{\trop} \right|, \left| \partial \LA_g^{\trop} \right|) \overset{\sim}{\To}  H_{dR}^{\bullet}( \left|\LA_g^{\trop,\BB} \right|, \left| \partial \LA_g^{\trop,\BB} \right|)  \]
since both sides may be identified with the compactly supported cohomology of the interior $\left|\LA_g^{\circ,\trop} \right| $.  In particular, any compactly supported form 
 $\omega_c$   as in the corollary may  be viewed as a closed differential form on $ \left|\LA_g^{\trop} \right|$. Its  cohomology class 
 \begin{equation} \label{omegacasformonLagdownstairs}  [ \omega_c]   \in   H_{dR}^{\bullet}( \left|\LA_g^{\trop} \right|) 
 \end{equation}
 does not depend on the choice of representative for $\omega_c^{2g-1}$ since 
  it is the image of $(\pi^*)^{-1} [ \widetilde{\omega}]$ under the natural  map $ H_{dR}^{\bullet}( \left|\LA_g^{\trop} \right|, \left| \partial \LA_g^{\trop} \right|)\rightarrow   H_{dR}^{\bullet}( \left|\LA_g^{\trop} \right|). $

\section{Cohomology classes in $\left|\LA_g^{\trop}\right|$.  } \label{section: CohomologyClasses}
In this section we prove the theorems announced in the introduction. 

\subsection{Canonical Classes in $\left|\LA_g^{\trop}\right|$} Let $g>1$ be odd. We  use the notation
\[  d_g =  \frac{g(g+1)}{2} \quad , \quad  \ell_g = \frac{g(g+1)}{2}-1 \]
where $d_g = \dim \mathcal{P}_g$, and $\ell_g = \dim L \mathcal{P}_g$.

\begin{lem} \label{lem: volnonzeroinrelative} 
The  form $\vol_g$ defines a non-trivial relative cohomology class
\[ 0 \neq  [ \vol_g ]  \ \in \    H_{dR}^{\ell_g} \left( | \LA_g^{\trop, \BB} |,  | \partial \LA_g^{\trop, \BB} | \right) . \]   
\end{lem} 
\begin{proof}

 Since $g$ is odd, the space $L\mathcal{P}_g/\GL_g(\Z) $ is orientable and has a relative fundamental class $[\mathcal{F}^{\lf}_{g}] \in H^{\lf}_{\ell_g} \left( \left|\LA_g^{\circ, \trop}\right|\right)$ in locally finite homology, or equivalently,
\[ [\mathcal{F}^{\lf}_{g}] \in H_{\ell_g} \left(          \left|\LA_g^{ \trop,\BB}\right|,   \left|\partial\LA_g^{ \trop,\BB}\right|        \right)\ . \]
A volume form  $\eta_g$ on $L\mathcal{P}_g/\GL_g(\Z)$ is defined on symmetric matrices $M= (M_{ij})_{ij}$ by
\[  \eta_g(M) = \partial \left(\frac{1}{\det(M)^{\frac{g+1}{2}}}  \bigwedge_{1\leq i \leq j \leq g}  d M_{ij}  \right) \]   
  where $\partial = \sum_{i,j} M_{ij} \frac{\partial}{\partial M_{ij}}$ is the Euler vector field. 
 By H. Cartan \cite[Exp. 17, \S12]{Cartan}  and \cite[figure on page 265]{Borel},  the algebra of invariant differential forms on $\SL_g(\R)$ is isomorphic to a graded exterior algebra  on elements  in degree $4k+1$, for $k\geq 0$ (and in degree $4k+1$ for $k\geq 1$ for $\GL_g(\R)$). Since the forms $\omega^{4k+1}$ are invariant and non-zero  they give an explicit choice of algebra generators. 
 It follows that  there exists  an $\alpha \in \Q^{\times}$ such that 
$ \eta_g = \alpha \, \vol_g = \alpha\,  \omega^5 \wedge \omega^{9} \wedge \ldots \wedge \omega^{2g-1}$, since $\eta_g$ is invariant and $\vol_g$ (defined by  \eqref{volg})  is the unique generator  in $\Omega^{\bullet}(g)$ of degree  $\ell_g$. 
 We have
 \[ \vol( L\mathcal{P}_g/\GL_g(\Z)) =  \int_{ L\mathcal{P}_g/\GL_g(\Z) } \eta_g=   \alpha  \int_{  [ \mathcal{F}^{\lf}_{g}]   }   \vol_g  \ .   \]
Since $ \vol( L\mathcal{P}_g/\GL_g(\Z)) \neq 0 $  
(for a more precise statement see  \eqref{MinkowskiVolume}), we deduce that the cohomology class of 
$ [ \vol_g ] $  in   $H_{dR}^{\ell_g} \left( | \LA_g^{\trop, \BB} |,  | \partial \LA_g^{\trop, \BB} | \right) $ is non-zero.
\end{proof} 

\begin{rem}  It follows from a theorem of  Minkowski (see \cite{Siegel}) that 
\begin{equation} \label{MinkowskiVolume}  \int_{L\mathcal{P}_g/\GL_g(\Z)  } \eta_g  \ \in  \    \zeta(3) \zeta(5) \ldots \zeta(2g-1)\, \Q^{\times} \ .
\end{equation}
The point is that the volume of the symmetric space $\SL_g(\R) /\SL_g(\Z)$ is proportional, by a rational number,  to the product of consecutive zeta values $\zeta(2) \ldots \zeta(2g-1)$. One proves \cite{Voskresenski} that the volume  of $L\mathcal{P}_g/\GL_g(\Z)$  is obtained by dividing by   the volume of  $\mathrm{SO}_g(\R)$, which is proportional to the product of the even zeta values $\zeta(2)\ldots \zeta(2g-2)$, giving \eqref{MinkowskiVolume}.
\end{rem}

We shall first prove an injectivity theorem for the cohomology of the general linear group, before proving a stronger but more technical result for $|\LA_g^{\trop}|$.  

\begin{thm} \label{thm: weakerembedding} 
The  map $ \omega \mapsto [\omega_c]$ defined in corollary \ref{cor: CompactSupportedomega} induces an injective map of  graded algebras: 
\[  \Omega^{\bullet}_c(g) \hookrightarrow H_{dR}^{\bullet} \left(L \mathcal{P}_g/\GL_g(\Z)\right)\ .\]
There are canonical injective maps of graded algebras:
\begin{eqnarray}
\Omega^{\bullet}_{nc}(g)  & \hookrightarrow  &   H_{dR}^{\bullet} \left( |\LA^{\trop, \BB}_g| \right) \  ,  \nonumber  \\
\Omega^{\bullet}_{nc}(g)  & \hookrightarrow  &   H_{dR}^{\bullet} (L\mathcal{P}_g/\GL_g(\Z))  \ .  \nonumber 
\end{eqnarray} 
There are natural splittings  of vector spaces  $H_{dR}^{n} \left(L \mathcal{P}_g/\GL_g(\Z)\right)\rightarrow  \Omega^{n}_c(g)\otimes_{\Q} \R$ for each $n$, and likewise  $H_{dR}^{n} (L\mathcal{P}_g/\GL_g(\Z)) \rightarrow \Omega^{n}_{nc}(g) _{\Q} \otimes \R$.
\end{thm} 
\begin{proof} Recall that $\ell_g = \dim L \mathcal{P}_g$. Let $n\geq 0$ and
consider the  cup product:
\begin{eqnarray} 
H^{n}(|\LA_g^{\trop,\BB}|, |\partial \LA_g^{\trop,\BB}| ) \otimes H^{\ell_g-n}(|\LA_g^{\trop,\BB}|)     & \To &   H^{\ell_g}(|\LA_g^{\trop,\BB}|, |\partial \LA_g^{\trop,\BB}| ) \nonumber \\
{[} \omega ] \otimes [\eta] & \mapsto & [\omega \wedge \eta]  \ .   \nonumber
\end{eqnarray} 
Apply this map to $\omega \in \Omega^n_{c}(g)$ of compact type, and $\eta= \star \omega $ which is of non-compact type. They define classes in the above cohomology groups by theorem \ref{thm:  CanFormsOnAgBlowup}. Since the Hodge star operator is non-degenerate,  $[\omega \wedge \eta] = [\omega \wedge \star \omega] $ is a non-zero rational multiple of $[\vol_g]$, which is itself  non-zero  by lemma  \ref{lem: volnonzeroinrelative}. We conclude that both $[\omega]$ and $[\star \omega]$ are non-vanishing in the respective cohomology groups.   This proves the first two equations, since every element of $\Omega^{\bullet}_{nc}$ is in the image of the star operator.

Recall that $|\LA_g^{\trop,\BB}| \setminus   |\partial \LA_g^{\trop,\BB}| =  |\LA_g^{\circ,\trop}| \cong  L \mathcal{P}_g/\GL_g(\Z)  $. By theorem \ref{thm: RelativeCellular}, relative cohomology may be interpreted as compactly supported cohomology.
A variant of the  above cup product may thus be written  as follows:
\begin{equation}   \label{CupProdLPG} H_{c, dR}^{n} \left(L \mathcal{P}_g/\GL_g(\Z)  \right) \otimes   H_{ dR}^{\ell_g-n} \left( L\mathcal{P}_g/\GL_g(\Z)  \right)  \To H_{ c,dR}^{\ell_g} \left(L \mathcal{P}_g/\GL_g(\Z)  \right)  \ .  \end{equation} 
It  sends $[\omega_c] \otimes [\star \omega] \mapsto [ \omega_c \wedge \star \omega] = [(\vol_{g})_c] \neq 0$.   It follows that both $[\omega_c]$ and $[\star \omega]$ are non-zero classes in the respective cohomology groups. This proves that $\Omega_{nc}^{\bullet}(g) $ injects into the ordinary (non compactly-supported) cohomology 
of $L\mathcal{P}_g/\GL_g(\Z)  $.

A splitting   $H^{\bullet}(L \mathcal{P}_g/\GL_g(\Z))  \rightarrow  \Omega^{\bullet}_{nc}(g)\otimes_{\Q} \R $ may be defined as follows. For any 
$[\alpha]  \in  H^{n}(L\mathcal{P}_g/ \GL_g(\Z))$,  it follows from Stokes' theorem that the  map 
\[\omega_c \mapsto \int_{L\mathcal{P}_g/\GL_g(\Z) } [\alpha] \wedge \omega_c\] is a well-defined  element of 
$\mathrm{Hom}(\Omega^{\ell_g-n}_c(g) , \R)$, which is isomorphic to $\Omega_{nc}^n(g) \otimes_{\Q} \R$ via the star operator since 
  $ \star \omega\wedge \omega_c$ pairs non-trivially with  the fundamental class $[\mathcal{F}^{\lf}_g]$.  This defines a linear map $H^{\bullet}(L\mathcal{P}_g/ \GL_g(\Z);\R)  \rightarrow  \Omega^{\bullet}_{nc}(g)\otimes_{\Q} \R $  which, after rescaling, defines the required splitting.  A splitting $H_{dR}^{\bullet} \left(L \mathcal{P}_g/\GL_g(\Z)\right)\rightarrow  \Omega^{\bullet}_c(g)$  may be defined  similarly. 
\end{proof} 

We wish to repeat a similar argument in order to prove the more subtle fact that $\Omega_{c}^{\bullet}(g)$ injects into the cohomology of $|\LA_g^{\trop}|$. 
However, replacing the cup product for the  \emph{relative} cohomology of $|\LA_g^{\trop,\BB}|$ with \emph{ordinary} cohomology will not succeed, since the volume form actually vanishes in the  cohomology of   $|\LA_g^{\trop,\BB}|$. Replacing $|\LA_g^{\trop,\BB}|$ with $|\LA_g^{\trop}|$ is also of no avail, since the non-compact canonical forms  are not defined on the latter space in the first place. The solution is to work on a space which lies in between $ | \LA_g^{\trop,\BB} | \rightarrow |\LA_g^{\trop}|$.  If we interpret the former as the Borel-Serre compactification of $L\mathcal{P}_g/\GL_g(\Z)$ (see Appendix), then an appropriate space is the reductive Borel-Serre compactification. It has the property that its boundary is of sufficiently high codimension that the volume class is non-zero (unlike $| \LA_g^{\trop,\BB} |$), but  in addition is sufficiently regular that the non-compact canonical forms extend to it (unlike $|\LA_g^{\trop}|$). Instead of constructing this space explicitly, we shall imitate its cohomology with an ad-hoc complex whose cohomology will be denoted by $H^{\mathrm{red}}$.

\begin{thm}  \label{thm: OmegagembedsLAg}
The  map $ \omega \mapsto [\omega_c]$ \eqref{omegacasformonLagdownstairs} induces an injective map of graded algebras
\[  \Omega^{\bullet}_c(g) \hookrightarrow H_{dR}^{\bullet} \left(\left|\LA_g^{\trop}\right|\right)\ .\]
It factors through the  injective  map 
$ \Omega^{\bullet}_c(g) \hookrightarrow H_{dR}^{\bullet} \left(\left|\LA_g^{\trop}\right| ,  \left|\partial\LA_g^{\trop}\right| \right)\cong H_{c,dR}^{\bullet} \left(\left|\LA_g^{\circ, \trop}\right|  \right)   $ constructed in theorem \ref{thm: weakerembedding}.
\end{thm} 

\begin{proof}  
We prove a stronger result, namely that the classes $\Omega^{\bullet}_c(g)$ are non-zero in a larger complex of  differential forms 
with `reduced boundary'.\footnote{To motivate this in the case $n=2$, 
consider the following picture of the projectivised normal bundles associated to a single  blow-up  of $\Pro(\Quad(V/K))$ in $\Pro(\Quad(V))$ where  $K \subset V$. We choose, as  usual, a complement $V= K \oplus C$. We have the maps
\begin{equation} \label{footnoteexact}  \Pro(\Quad(V/K)) \times  \Pro(\Quad(V)/\Quad(V/K)) \setminus \Det\big|_K  \To   \Pro(\Quad(V/K)) \times  \Pro(\Quad(K)) \To  \Pro(\Quad(V/K))  
\end{equation}  
where the first map is $\id \times \pi_K$ (see \S\ref{sect: CollapseNormal}, \eqref{piKdef}) and the second map is projection on to the first factor. Their composition is  the blow-down map (see remark \ref{rem: blowdowncontractsnormal}).  The corresponding map on affine spaces, after identifying the middle space as block diagonal matrices, may be viewed on symmetric matrices as the map 
\[ 
\begin{pmatrix}
Q_K & P \\
P^T & Q_C
\end{pmatrix}
\To \begin{pmatrix}
Q_K & 0 \\
0  & Q_C
\end{pmatrix}
\To  
\begin{pmatrix}  0  &  0 \\
0 &  Q_C \end{pmatrix}
\]
The far left  space in \eqref{footnoteexact} describes the local structure of an exceptional divisor on $\LA_g^{\trop, \BB}$, the far right space describes the structure of a  face at infinity in $\LA_g^{\trop}$. The space in the middle corresponds to the structure at infinity of the reductive Borel-Serre compactification. 
We consider forms whose restriction to the boundary are constant in the $P$-direction, i.e., pull-backs from the middle space to the left-most one.}   
\[   \Omega^{\bullet,\red} ( \LA_g^{\trop, \BB})    \quad \subset  \quad   \Omega^{\bullet} ( \LA_g^{\trop, \BB}  \ \setminus \  \widetilde{\Det} )  \ \oplus \  \bigoplus_{\mathcal{E}}  \Omega^{\bullet} \left( \prod_{i=0}^n  \Pro \left(\Quad(K_{i+1}/K_{i})  \backslash \Det_{K_{i+1}/K_i}   \right) \right)    \] 
where the direct sum is over all components of the boundary $\partial \LA_g^{\BB}$ indexed by  nested sequences \eqref{Kflag}. 
The strategy will be similar  to the proof of theorem \ref{thm: weakerembedding} once we have  shown that the volume class is non-zero.
 The complex $ \Omega^{\bullet,\red} ( \LA_g^{\trop, \BB}) $
consists of pairs $(\omega, (\eta_{\mathcal{E}})_{\mathcal{E}})$ such that the restriction of $\omega$ to a  face $ \mathcal{E} \cap \partial \LA_g^{\BB}$ satisfies
\begin{equation} \label{omegapiKetacompat}   \omega \Big|_{\mathcal{E}} =  \pi_{\red}^* \, \eta_{\mathcal{E}}
\end{equation} 
where $\pi_{\red}$ was defined in \eqref{defn: pired}.   There is a corresponding complex  $ \mathcal{A}^{\bullet,\red} (  | \LA_g^{\trop, \BB} |   )$  of smooth differential forms  defined in the same way: it consists  of pairs  $(\omega, (\eta_{\mathcal{E}})_{\mathcal{E}})$ which satisfy \eqref{omegapiKetacompat},   where
$\omega \in  \mathcal{A}^{\bullet} (  | \LA_g^{\trop, \BB} |   )$  and 
 $\eta_{\mathcal{E}}$ are  smooth forms on  $ \pi_{\red}(|\partial \LA_g^{\trop,\BB}| \cap \mathcal{E}(\R))$.
 Let us denote its cohomology by $H^{\bullet, \red}( | \LA_g^{\trop, \BB} | ) $.    There is a natural map of complexes $  \Omega^{\bullet,\red} ( \LA_g^{\trop, \BB})  \rightarrow     \mathcal{A}^{\bullet,\red} (  | \LA_g^{\trop, \BB} |   )$.

 Since the restriction of the blow down  to the boundary  $ |\partial \LA_g^{\trop,\BB}| \rightarrow |\partial \LA_g^{\trop}|$  factors through $\pi_{\red}$ (it factors through the map  which  projects  a product  \eqref{Quadpolytopefaces} onto its final component), it follows that 
the map on cohomology induced by   $\pi_{\BB} : | \LA_g^{\trop,\BB}| \rightarrow | \LA_g^{\trop}|$  
factors through the reduced cohomology, i.e., $\pi_{\BB}^*$ is the composition of two maps:
\begin{equation} \label{blowupfactorsthroughcollapse}
  H^{\bullet}_{dR} (   |\LA_g^{\trop}| )  \To   H^{\bullet,\red}_{dR} (   |\LA_g^{\trop,\BB}| )  \To H^{\bullet}_{dR} (   |\LA_g^{\trop,\BB} | )   
  \end{equation}  
where the first map is  $\omega \mapsto (\pi_{\BB}^* \omega,0)$  the second is $(\omega, (\eta)_{\mathcal{E}}) \mapsto \omega$.

We will show that there is an injective map 
\[     \  \Omega^{\bullet}(g)\To   H^{\bullet,\red}_{dR} (   |\LA_g^{\trop,\BB}| )\ . \]
First of all, observe that there is a well-defined map  
$ \Omega^{\bullet}(g) \rightarrow \Omega^{\bullet,\red} ( \LA_g^{\trop, \BB})   $
by  equation \eqref{omegafactorisationalongE}, which implies that the restriction of   a canonical form to an exceptional divisor 
$\mathcal{E}$  is always of the form $\pi^*_{\red}(\eta_{\mathcal{E}})$. See remark \ref{rem: canformsarereduced}.   Concretely, the forms $\eta_{\mathcal{E}}$ are products of canonical forms which can be computed explicity from the coproduct of $\omega$.

 We next show, as before,  that the  volume form is non-zero 
in  $H^{\ell_g, \red}( | \LA_g^{\trop, \BB} | )$.   If $\vol_g$ were exact, there would exist forms  $\alpha$ on $|\LA_g^{\trop,\BB}|$ and  $(\beta_{\mathcal{E}})_{\mathcal{E}}$   of degree $\ell_g-1$ such that 
\begin{equation} \label{volasalphabeta} \vol_g= d \alpha \quad \hbox{ and } \quad  \alpha\Big|_{ \mathcal{E}(\R)} = \pi_{\red}^*  \beta  \ . 
\end{equation} 
However, since $\pi_{\red}$ decreases the dimension by at  least $g-1$, such a form $\beta$ is necessarily zero since it has degree $> \ell_{g}-g$. 
Equation \eqref{volasalphabeta} then implies that, since $\alpha$ restricts to zero on every exceptional divisor $\mathcal{E}(\R)$,  
 $\vol_g$ is  zero in relative cohomology $H^{\ell_g}_{dR}\left(          \left|\LA_g^{ \trop,\BB}\right|,   \left|\partial\LA_g^{ \trop,\BB}\right|        \right)$,  which contradicts lemma \ref{lem: volnonzeroinrelative}.  

Now let $\omega \in \Omega(g)$.  There is a $\lambda \neq 0$ such that $\omega \wedge \star \omega = \lambda \vol_g.$ Since \eqref{omegapiKetacompat} is compatible with exterior products of differential forms, the complex  $\mathcal{A}^{\bullet,\red} (  | \LA_g^{\trop, \BB} |)$ inherits a multiplicative structure giving rise to  a product on reduced cohomology: 
\[  H_{dR}^{m,\red}( |\LA^{\trop, \BB}_g|) \otimes  H_{dR}^{n,\red}( |\LA^{\trop, \BB}_g|) \To    H_{dR}^{m+n,\red}( |\LA^{\trop, \BB}_g|) \ . \]
Since $[\star \omega] \wedge [\omega] = \lambda [\vol_g]$ is non-zero, it follows that $[\omega] \in H_{dR}^{\bullet,\red}( |\LA^{\trop, \BB}_g|)$ is non-zero and hence  $\Omega(g)$ indeed injects into $H_{dR}^{\bullet,\red}( |\LA^{\trop, \BB}_g|)$ as claimed. 

By theorem \ref{thm:  CanFormsOnAgBlowup} and  \eqref{blowupfactorsthroughcollapse}, the image of the subspace of compact type  $\Omega_c(g) $   in 
 $H_{dR}^{\bullet,\red}( |\LA^{\trop, \BB}_g|)$ factors through its image in  $H^{\bullet}(   |\LA^{\trop}_g|)$. It follows that the map  
 $\Omega_c(g) \rightarrow  H^{\bullet}(   |\LA^{\trop}_g|)$ given by 
 \eqref{omegacasformonLagdownstairs}  must already be injective.  
 
 The last part follows from the fact that the map  $\Omega_c(g)  \rightarrow  H^{\bullet}(   |\LA^{\trop}_g|)$  factors through $H^{\bullet}(   |\LA^{\trop}_g| ,|\partial \LA_g^{\trop} |)$ by definition.  
\end{proof} 

\subsection{Cohomology of $\GL_g(\Z)$} Theorem \ref{thm: weakerembedding} has the following corollaries. First, some notation. 
Since $\mathcal{P}_g$ is contractible  and an $\R_{>0}^{\times}$-bundle over  the space $L\mathcal{P}_g$, we have
\begin{eqnarray} \label{HnGLg}  H^n(\GL_g(\Z); \R) & = &  H^n(L\mathcal{P}_g/\GL_g(\Z)) \ ,  \\ 
   H_c^{n+1}(\GL_g(\Z); \R)  &= &  H_c^{n}(L\mathcal{P}_g/\GL_g(\Z)) \ .  \nonumber 
   \end{eqnarray} 
   The reason for the shift in indices is that $H^n_c(\R_{>0}^{\times})=0$ for $n\neq 1$ and $H_c^1(\R_{>0}^{\times})=\R$. 
\begin{cor}   Let $g>1$ be odd. The forms of  non-compact  (resp. compact) type embed into the cohomology (resp. compactly supported cohomology) 
of $L \mathcal{P}_g/\GL_g(\Z)$:
\begin{eqnarray}  \label{OmegaInjectsBothLPG} 
 \Omega^{\bullet}_{nc}(g) \otimes_{\Q} \R  &\hookrightarrow& 
  H^{\bullet}(\GL_g(\Z);\R)   \\ 
 \Omega^{\bullet}_c(g) \otimes_{\Q} \R &\hookrightarrow&  
   H_c^{\bullet+1} (\GL_g(\Z);\R)  \nonumber  \ . 
\end{eqnarray} 
 \end{cor} 

\begin{cor}
 Let $g>1$ be odd. Then  for all $h\geq g$, we have:
\begin{equation}  \label{Omegancinjectsallh}  \Omega^{\bullet}_{nc}(g) \quad \hookrightarrow \quad H^{\bullet}(\GL_h(\Z);\R)   \ .   
\end{equation} 
\end{cor}

\begin{proof}Let $k\geq 1$ and consider the maps
\[ \Omega^{\bullet}_{nc}(g) \rightarrow  \Omega^{\bullet}_{nc}(g+2k)   \rightarrow  H^{\bullet}( \GL_{g+2k}(\Z);\R)  \rightarrow   H^{\bullet}( \GL_{g+2k-1}(\Z);\R)   \rightarrow  H^{\bullet}( \GL_{g}(\Z);\R) \]
where the last two maps $H^{\bullet}(  \GL_{g+m}(\Z);\R) \rightarrow H^{\bullet}(  \GL_{g}(\Z);\R)$ are  induced by the stabilisation map  $X \mapsto X \oplus I_m$, where $I_m$ is the identity  matrix of rank $m$.  The composition of all the maps in the above  is the natural map    $\Omega^{\bullet}_{nc}(g) \rightarrow  H^{\bullet}( \GL_{g}(\Z);\R)$, 
since $\omega_{X \oplus I_m} = \omega_X$ for any canonical form $\omega$. We have shown it is  injective, by  \eqref{OmegaInjectsBothLPG}.
Therefore $\Omega^{\bullet}_{nc}(g)$ injects into all intermediate spaces, and hence into $ H^{\bullet}( \GL_{h}(\Z);\R)$ for all $h\geq g$.
\end{proof}
\begin{rem}  A  brief discussion of the history of this result is in the introduction.  Note that the previous  two corollaries are strictly weaker than  theorem \ref{thm: OmegagembedsLAg}.
\end{rem} 

\subsection{Compactly supported cohomology} The following proposition  is essentially a restatement, in a different form,  of  results of \cite{TopWeightAg}.
\begin{prop}  \label{prop: ShortExactAgDRsequence} For all $n, g>1$, there are  short exact sequences:
\[ 0 \To H_{dR}^{n-1} (\left| \LA^{\trop}_{g-1}\right|    ) \To  H_{dR,c}^n (\left| \LA^{\circ, \trop}_g \right|  ) \To H_{dR}^n (\left| \LA^{\trop}_g\right|    )  \To 0 \]
\end{prop} 
\begin{proof}
Consider the long exact relative de Rham cohomology sequence:
\[ \cdots\To H_{dR}^{n-1} (\left| \partial \LA^{\trop}_{g}\right|   ) \To  H_{dR}^n (\left| \LA^{\trop}_g\right| , \left| \partial  \LA^{\trop}_g\right|   ) \To H_{dR}^n (\left| \LA^{\trop}_g\right|     )  \To \cdots \ . \]
The statement  follows on interpreting the relative cohomology group in the middle as compactly supported cohomology \eqref{dRCompactSupportsToRelative};  from the homeomorphism $\left| \partial \LA^{\trop}_{g}\right| \cong   \left| \LA^{\trop}_{g-1}\right|$, which follows from  \cite[Corollary 4.12]{TopWeightAg}; and from the fact that 
the restriction  maps $H_{dR}^n (\left| \LA^{\trop}_g\right| )\rightarrow H_{dR}^{n} (\left| \partial \LA^{\trop}_{g}\right|    ) $ are zero. This last fact is dual 
to the statement that the inclusion of the boundary  $\left| \partial \LA^{\trop}_{g}\right|\hookrightarrow \left|  \LA^{\trop}_{g}\right|$ induces the zero map in homology   $H_{n} (\left| \partial \LA^{\trop}_{g}\right|   ;\R )  \rightarrow H_n (\left| \LA^{\trop}_g\right|    ;\R )$.
 The reason for this is that the image of a chain in the boundary $\left| \partial \LA^{\trop}_{g}\right|$ only  involves cones $\sigma_Q$ which are equivalent to the cone of a  positive definite quadratic form of rank strictly less than $g$, and therefore lies
 in the inflation subcomplex $I^{(g)}$   of the perfect cone complex $P^{(g)}$. Since the inflation complex is acyclic, by  \cite[Theorem 5.15]{TopWeightAg}, any  such  chain is  trivial in homology. 
\end{proof} 

\begin{cor} Let $g>1$ be odd.  Then there is a natural injective map 
\begin{equation} \label{OmegacinjectsCompactSupports}
\Omega_c^{\bullet} (g)  \hookrightarrow   H_{c}^{\bullet+2} ( \GL_{g+1}(\Z)   ;\R )\ .  
\end{equation} 
\end{cor}
\begin{proof} 
By  proposition \ref{prop: ShortExactAgDRsequence}, there is an injective map 
\[   H_{dR}^{\bullet} (\left| \LA^{\trop}_{g}\right|  ) \To    H_{dR, c}^{\bullet+1} ( |\LA_{g+1}^{\circ, \trop}| ) \cong   H_{dR,c}^{\bullet+1} ( L \mathcal{P}_{g+1} /\GL_{g+1}(\Z)    )  \]
The result follows from theorem \ref{thm: OmegagembedsLAg} and \eqref{HnGLg}.
\end{proof} 

\subsection{Unstable cohomology of $\SL_g(\Z)$}  \ The following lemma is standard \cite{ElbazVincentGanglSoule}.
\begin{lem} Let $n\geq 0$. Then 
\[ H^n(\SL_g(\Z); \R) \cong  \begin{cases}  H^n(\GL_g(\Z);\R) \qquad \qquad\qquad\qquad \qquad \quad  \hbox{ if } g \hbox{ is  odd}  \nonumber \\ 
 H^n(\GL_g(\Z);\R) \oplus H_c^{d_g-n}(\GL_g(\Z); \R)^{\vee}  \qquad  \hbox{ if } g \hbox{  is even }  \nonumber \end{cases}
 \]
\end{lem} 
\begin{proof} 
Consider the short exact sequence
\[ 1 \To \SL_g(\Z) \To \GL_g(\Z) \overset{\det}{\To} \GL_1(\Z)  \To 1\]
where $\GL_1(\Z) = \Z^{\times} \cong \Z/2\Z$. 
When $g$ is odd, it splits: in fact, since the diagonal matrix $-I_g$ of rank $g$ is central, $\GL_g(\Z)$ is isomorphic to the direct product  $\SL_g(\Z) \times \Z/2\Z$.   Since $\Z/2 \Z$ is torsion we deduce that $H^n(\SL_g(\Z); \Q) \cong H^n(\GL_g(\Z);\Q)$ when $g$ is odd. 
Now suppose that $g$ is even, and consider the rank one  $\R$-vector space $\R^{\det}$ equipped with the action of $g\in \GL_g(\Z)$ given by $g. v = \det(g) v$.  By Shapiro's lemma one has:
\[ H^n(\SL_g(\Z); \R) \cong H^n(\GL_g(\Z);\mathrm{Ind}_{\SL_g(\Z)}^{\GL_g(\Z)} \, \R) \cong  H^n(\GL_g(\Z);\R \oplus \R^{\det}) \ .\]  
Denote by $\mathcal{O}$ the orientation bundle on $L\mathcal{P}_g$, which is non-trivial when $g$ is even. We have
\[ H^n(\GL_g(\Z); \R^{\det} ) = H^n( \mathcal{P}_g /\GL_g(\Z); \mathcal{O}) \cong H_c^{d_g-n} (  \mathcal{P}_g /\GL_g(\Z); \R)^{\vee} \]
where the first equality  follows from the argument of \cite[Lemma 7.2]{ElbazVincentGanglSoule} and the second for  Poincar\'e duality for non-orientable homology manifolds. 
\end{proof} 
The previous lemma and \eqref{OmegacinjectsCompactSupports} imply the following result on the cohomology of $\SL_g(\Z)$.  Choose a surjective map $s:H_{c}^{\bullet+2} ( \GL_{g+1}(\Z)   ;\R )\rightarrow  \Omega_c^{\bullet} (g)\otimes_{\Q} \R $ which is a section of  the injective map
$ \Omega_c^{\bullet} (g) \otimes_{\Q} \R  \hookrightarrow  H_{c}^{\bullet+2} ( \GL_{g+1}(\Z)   ;\R )$  equivalent to
\eqref{OmegacinjectsCompactSupports}.

\begin{cor} Let $g>1$ be odd and $n\geq 0$. Then 
\begin{eqnarray}   \Omega_{nc}^{n}(g) \otimes_{\Q} \R  &\hookrightarrow  & H^n(\SL_g(\Z);\R)   \nonumber \\ 
\left(\Omega^n_{nc}(g)\otimes_{\Q} \R \right)\oplus  \left(\Omega_c^{d_{g+1}-n-2}(g)\otimes_{\Q} \R \right)^{\vee} & \hookrightarrow &  H^n(\SL_{g+1}(\Z);\R)  \nonumber
\end{eqnarray} 
 The first map is natural. The second map depends on the choice of section $s$: it is given by its dual $s^{\vee}$ on the second component  $(\Omega_c^{d_{g+1}-n-2}(g)\otimes \R)^{\vee}$.
\end{cor} 
If we identify $( \Omega_c^{d_{g}-n-1} (g))^{\vee} $ with $\Omega^{n}_{nc}(g)$ via the Hodge star operator, the second map in the corollary may be written (non-canonically) in the more convenient form
\[ \left( \Omega^n_{nc}(g) \oplus \Omega^{n-g}_{nc}(g)\right)\otimes_{\Q} \R  \hookrightarrow   H^n(\SL_{g+1}(\Z);\R) \ . \] 
\begin{ex} For $g=3$,  we have $\Omega_{nc}(3) =\Q$ and $\Omega_c(3) = \omega^5 \Q$.  The previous corollary produces classes in
$H^0(\SL_3(\Z);\R)$, as well as  $H^0(\SL_4(\Z);\R)$, $H^3(\SL_4(\Z);\R)$.

For $g=5$,   we have  forms of non-compact type $\Omega_{nc}(5) =\Q\oplus \Q \omega^5$.   They give rise, via the previous corollary, to classes in
$H^0(\SL_5(\Z);\R)$ and $H^5(\SL_5(\Z);\R)$,  and also
$ H^0(\SL_6(\Z);\R)$ and $H^5(\SL_6(\Z);\R)$. The  forms of compact type $\Omega_c(5) = \omega^9 \Q\oplus (\omega^5\wedge \omega^9) \Q$ give rise to two further   classes in 
$H^5(\SL_6(\Z);\R)$ and $H^{10} (\SL_6(\Z);\R)$.
\end{ex}
\begin{rem} A natural way to phrase the  corollary is to say that there is a complex
\[   \Omega^n_{nc}(g)\otimes_{\Q} \R \To   H^n(\SL_{g+1}(\Z);\R) \To   \left(\Omega_c^{d_{g+1}-n-2}(g) \otimes_\Q \R\right)^{\vee} \]
where the first map is injective, and the second map is surjective. 
\end{rem}

\subsection{Stable range} Let $g>1$ be odd.
By  Borel \cite{Borel},  rephrased in terms of forms of non-compact type, there exists a $\kappa(g)$ which grows linearly in $g$ such that  
\begin{equation} \label{stablerange}   \Omega_{nc}^n(g)\otimes_{\Q} \R \overset{\sim}{\To} H^n(\GL_g(\Z);\R) \quad \hbox{ for all } \,   n< \kappa(g) \ . 
\end{equation}
Borel showed that $\kappa(g)$ is of the order of $g/4$.
Using the recent work  \cite[Theorem A]{KupersMillerPatzt}, \cite{StableShuBinyong},  and  \eqref{Omegancinjectsallh},  it may be improved to  $\kappa(g) =g$. For $n$ in this range, $  \Omega_{nc}^n(g) = \Omega^n(g)$.
\begin{cor} \label{cor: stablecohomLag} Let  $g>1$ be odd and let $\kappa(g)$ be such that  \eqref{stablerange} holds. Then 
\begin{eqnarray}   H^n_{dR}(\left|\LA_g^{\trop}\right|) & \cong&  \Omega^n_c(g)\otimes_{\Q} \R    \nonumber \\ 
  H^{n-1}_{dR}(\left|\LA_{g-1}^{\trop}\right|) &  = &  0 \ , \nonumber
   \end{eqnarray}
for all $n\geq d_g - \kappa(g)$. 
\end{cor}
\begin{proof} 
Let $g>1$ be odd. Since $L \mathcal{P}_g/\GL_g(\Z)$ is orientable of dimension $\ell_g= d_g-1$, Poincar\'e duality implies that 
$H_{c}^n (L\mathcal{P}_g/ \GL_g(\Z)  ;\R )  \cong H^{\ell_g-n} (L\mathcal{P}_g/ \GL_g(\Z)  ;\R )^{\vee} $ and so we have 
\[ \dim_{\R} H_c^{n}(L\mathcal{P}_g/\GL_g(\Z);\R) =  \dim_{\R} H^{\ell_g-n} (L\mathcal{P}_g/ \GL_g(\Z)  ;\R )  \overset{\eqref{stablerange}}{=} \dim_{\Q} \Omega_{nc}^{\ell_g-n}(g) \]
where the last equality holds for $n > \ell_g - \kappa(g)$. 
In  \eqref{OmegaInjectsBothLPG}   we established an injection 
\begin{equation} \label{recapOmegainjects} \Omega_c^{n} (g) \otimes_{\Q} \R \hookrightarrow  H_c^{n}(L\mathcal{P}_g/\GL_g(\Z)) 
\end{equation}  for all $n$. The Hodge star operator  formalism implies  that 
$  \dim_{\Q} \Omega_c^n(g) = \dim_{\Q} \Omega_{nc}^{\ell_g-n}(g)$ for all $n$, and hence if   $ n > \ell_g - \kappa(g)$,   we have equality of dimensions in \eqref{recapOmegainjects}. In summary,
\[  \Omega^n_c(g) \otimes_{\Q} \R \cong H_c^n (L\mathcal{P}_g/ \GL_g(\Z)  ;\R )   \quad \hbox{ for } n > \ell_g - \kappa(g)\ .\]
Since  $  \left| \LA^{\circ, \trop}_g \right| \cong L\mathcal{P}_g/\GL_g(\Z)$ we deduce that  
$H_{dR,c}^n (\left| \LA^{\circ, \trop}_g \right|  ) \cong   \Omega^n_c(g) \otimes_{\Q} \R $ for $n$ in this range, and the exact sequence of 
proposition  \ref{prop: ShortExactAgDRsequence} becomes 
\[ 0 \To H_{dR}^{n-1} (\left| \LA^{\trop}_{g-1}\right|    ) \To  \Omega^n_c(g) \otimes_{\Q} \R \overset{(*)}{\To} H_{dR}^n (\left| \LA^{\trop}_g\right|    )  \To 0 \ .  \]
By theorem \ref{thm: OmegagembedsLAg}, the  map  $\Omega^n_c(g) \otimes_{\Q} \R \rightarrow H_{dR}^n (\left| \LA^{\trop}_g\right|    ) $ is injective, hence  $(*)$ is an isomorphism, giving the first statement. Its kernel vanishes, which gives the second.
\end{proof}

\subsection{Canonical forms and the tropical Torelli map} \label{sect: VanishUnderTorelli} 
The composition 
\[  \Omega^n_c(g)\otimes_{\Q} \R  \To    H^n_{dR} \left(  |\LA_g^{\trop}| \right)   \overset{\lambda^*}{\To}  H^n_{dR} \left(  |\LM_g^{\trop}| \right) \]
 sends the majority of forms to zero. Indeed, if $\eta \in \Omega^d(g)$ has degree $d$ then
\begin{equation}  \label{lambdastarkills} 
\lambda^*  (\omega^{2g-1} \wedge \eta) = 0  \quad \hbox{ if }\quad   d> g-2  \ . \end{equation} 
This is because the left-hand side restricts to the form  $\omega^{2g-1}_G\wedge \eta_G$ in the notation of \cite{BrSigma} on the simplex $\sigma_G$ associated to a graph $G$. By  \cite[Corollary 6.19]{BrSigma}, it vanishes if its degree is greater than $3g-3$. It might  be possible to reduce the degree $d$ in  \eqref{lambdastarkills} by a more detailed analysis  of vanishing of graph forms \cite[\S6.5]{BrSigma}. Evidence suggests that  $\lambda^*(\omega)$ is non-zero if and only if $\omega$ is primitive, or   more precisely,  that the image of
\[  \lambda^*: H_{dR}^n(L\mathcal{P}_g/\GL_g(\Z)) \To  H^n_{dR}(\left| \LM_g^{\trop} \right|) \] 
is  one-dimensional:   $\mathrm{Im}(\lambda^*) = \Q [\omega^{2g-1}]$.  The results of  \cite{Wheels} imply that the image of $\lambda^*$ contains $ \Q [\omega^{2g-1}]$  and is dual to the class of the wheel with $g$ spokes.

\begin{rem} Nevertheless, a conjecture in \emph{loc. cit.} states that the free Lie algebra  embeds:
\[ \mathrm{Lie}\,(\Omega^{\can}) \hookrightarrow  \bigoplus_{g\geq 1} H^\bullet  \left(  |\LM_g^{\trop}| \right)\ .\]
It is corroborated by all  known evidence  in low degrees, and suggests that all canonical forms  do appear in the cohomology of 
$ H^{\bullet}_{dR} \left(  |\LM_g^{\trop}| \right)$, but for a higher than expected genus. 
\end{rem}
\section{Periods} \label{section: Periods} 
We work in a suitable neutral Tannakian category  $\mathcal{H}_{\Q}$  of motivic realisations over $\Q$.   There are many possible variants \cite{DeligneP1}. We assume that  it has two fiber functors
\[ M \mapsto M_B, M_{dR}  \ : \  \mathcal{H}_{\Q} \To \mathrm{Vec}_{\Q}\]
and that there is a canonical element $\mathrm{comp}_{B,dR} \in \mathrm{Isom}_{\mathcal{H}_{\Q}}^{\otimes}(\omega_{dR}, \omega_B)(\C)$.  A convenient option  is to let  $\mathcal{H}_{\Q}$  be the category of mixed Hodge structures with additional $\Q$-de Rham structure studied in \cite{NotesMot}. 
 The element $\mathrm{comp}_{B,dR}$ defines an isomorphism 
 \[M_{dR} \otimes_{\Q} \C\overset{\sim}{\To} M_B\otimes_{\Q} \C \ ,\]  which is natural in $M$,  which may be interpreted as a  period pairing 
$M_{dR} \otimes_{\Q} M_B^{\vee} \rightarrow \C$.

Consider the affine group scheme:
\[ G_{dR} = \mathrm{Aut}^{\otimes}_{\mathcal{H}_{\Q}}(\omega_{dR})\]
given by  the automorphisms of the fiber functor $\omega_{dR}: M \mapsto M_{dR}$.  By the Tannaka theorem \cite{Saavedra}, the fiber functor $\omega_{dR}$ defines an equivalence of categories:
\[ \omega_{dR}:  \mathcal{H}_{\Q}   \overset{\sim}{\To}  \mathrm{Rep}\, G_{dR} \ .\]  
\subsection{`Motives' of quadratic forms} 

\begin{defn} \label{def: motQ} Let $Q$ be a positive definite quadratic form on $V$ such that $\sigma_Q>0$. Recall that 
$\cone_{Q}^{\BB}=   ( P^{\BB_Q},   L^{\BB_Q}, \sigma^{\BB}_Q)$.  Let $d = \dim \sigma_Q$. Define 
\begin{equation}\label{motQ}  \mot_Q  = H^d \left( P^{\BB_Q} \backslash \widetilde{\Det} \  , \  L^{\BB_Q} \backslash (L^{\BB_Q} \cap  \widetilde{\Det}) \right)\ .
\end{equation}
By corollary \ref{cor: sigmadoesnotmeetDet}, we have $[\sigma^{\BB}_Q] \in (\mot_Q)_B^{\vee}$. Any canonical form  $\omega \in \Omega^d_{\can}$ defines  a class $[\omega] \in (\mot_Q)_{dR}$, and so we may define the `motivic period':
\begin{equation} \label{ImotivicQ}
I^{\mm}_Q(\omega) =  [ \mot_Q, [\sigma_Q] , [\omega] ]^{\mm}   \quad \in \quad \mathcal{O}(\mathrm{Isom}^{\otimes}_{\mathcal{H}_{\Q}}(\omega_{dR}, \omega_B)) \ .
\end{equation} 
\end{defn} 
Note that for cones $\sigma_Q$ which are contained in the determinant locus,  \cite[Lemma 4.9]{TopWeightAg} implies that they are equivalent to the cone of a quadratic form of smaller rank, so definition \ref{def: motQ} in fact covers all cases. 
The period of $I^{\mm}_Q(\omega)$ is the convergent integral
\[ \mathrm{per} \left( I^{\mm}_Q(\omega) \right) = \int_{\sigma_Q} \omega \]
which is finite by corollary \ref{cor: IntQcoverges}. Furthemore, the relation of theorem \ref{thm: Stokes} is motivic: it holds  verbatim for the objects $I^{\mm}_Q(\omega)$. 
In the case when $Q=Q_G$ comes from a graph, the object $I^{\mm}_{Q_G}(\omega)$ is equal to the `canonical  motivic Feynman integral' $I^{\mm}_G(\omega)$ defined in \cite{BrSigma}. 

\begin{rem} 
Every isomorphism $[Q] \cong [Q']$ in the category $\mathcal{D}^{\perf, \BB}$ gives rise to an isomorphism $\mot_Q \cong  \mot_Q'$. 
Similarly, face diagrams give rise to a diagram of morphisms in the category $\mathcal{H}_{\Q}$ (see \cite{BrSigma} for the case of graphs), but extraneous modifications are \emph{not} necessarily isomorphisms (although they do induce equivalences of motivic periods \eqref{ImotivicQ}). 
\end{rem} 

\begin{example} We refer to \cite{BrSigma} for many examples of graphs for which $I_G(\omega)$ are known. 
The most interesting was computed by Borinsky and Schnetz \cite{BorinskySchnetz}: 
\begin{equation} \label{IK6}  I_{K_6} (\omega^5 \wedge \omega^9) = 3\, \frac{10!}{16} \left( 12 \,  \zeta(3,5)  -  29\, \zeta(8) + 23 \, \zeta(3)\zeta(5) \right) \ .
\end{equation} 
It was proven in \cite[lemmas 3.7, 6.8]{BrSigma} that the canonical integrals are stable under duality, which implies that the volume of the graphic cell $\sigma_{K_6^{\vee}}$ where $K_6^{\vee}$ is the  matroid  dual to $K_6$, is given by the same integral. It is the `principal cone' \cite[Def. 2.13]{TopWeightAg} in the fundamental domain $\mathcal{F}_5$ for $\GL_5(\Z)$.  By \eqref{MinkowskiVolume} the volume $\int_{\mathcal{F}_5} \omega^5 \wedge \omega^9$ is  proportional to a product  $\zeta(3)\zeta(5)$. The appearance, therefore, of the  non-trivial multiple zeta value  $\zeta(3,5)$ in \eqref{IK6} suggests that  the principal cone does not span the fundamental domain: 
$\sigma_{K_6^{\vee}} \subsetneq  \mathcal{F}_5$  
and  that the tropical Torelli map is not surjective. 
   Indeed, it is known (e.g. \cite{ElbazVincentGanglSoule}) that the fundamental domain involves two further cones. We expect from \eqref{IK6} that their volumes, which are not yet  known,  will involve  non-trivial multiple zeta values. 
\end{example}
As noted in \cite[\S10.3.1]{BrSigma}, the particular linear combination $12 \,  \zeta(3,5)  -  29\, \zeta(8)$ is the same one which occurs in quantum field theory as the Feynman residue of $K_{3,4}$. It is a very striking  fact that $\zeta(3,5)$ is only ever observed in this particular combination, and never in isolation, amongst the  vast array of periods in  quantum field theory \cite{Schnetz}. This surprising phenomenon is one particular consequence of the `Cosmic' Galois group.  The geometric interpretation \eqref{IK6} of this linear combination points to an unexplored and  deep connection between quantum fields and the reduction theory of quadratic forms. 

\subsection{Locally finite homology `motive' of $\GL_g(\Z)$} 
Let $\mathcal{H}_{\Q}$, as above,  be a  Tannakian category  of realisations which contains every  object $\mot_Q$ \eqref{motQ}.

Consider the cohomology
\[ H_c^d(g) = H^d \left( \ker \left( \Omega^{\bullet} \left(   \LA^{\trop, \BB}_g \ \backslash  \ \widetilde{\Det} \right) \To 
 \Omega^{\bullet} \left( \partial  \LA^{\trop, \BB}_g \ \backslash  \  (  \partial  \LA^{\trop, \BB}_g \cap \widetilde{\Det} )\right) \right)  \right)\]
 of the complex of compatible systems (definition \ref{defn: GlobalForm}) of algebraic forms on $  \LA^{\trop, \BB}_g\backslash \widetilde{\Det}$ whose restriction to the boundary $  \partial  \LA^{\trop, \BB}_g$ vanishes.  There is a natural map
 \[ H_c^d(g) \To H^d_{dR} \left(   \left|   \LA^{\trop, \BB}_g  \right|  \ , \  \left|   \partial \LA^{\trop, \BB}_g  \right| \right) \cong H^d_c(L \mathcal{P}_g/\GL_g(\Z);\R)\] 
 via which elements of $H_c^d(g)$  may be interpreted as smooth  compactly supported differential forms 
 on  $L \mathcal{P}_g/\GL_g(\Z)$. There is a natural map $\Omega^d_c(g) \rightarrow H_c^d(g)$ (which is injective).
 
 \begin{thm} \label{thm: MotiveGLg} For every $g>1$, and $d\geq0 $, there  exists a minimal  object $\MM^d_g$ of $\mathcal{H}_{\Q}$ which is equipped with a pair of canonical  linear maps
\begin{eqnarray}
H^{\lf}_d(\GL_g(\Z);\Q)   & \To &   ( \MM^d_g)_B^{\vee} \nonumber \\ 
H^d_c(g)   & \To &   (\MM^d_g)_{dR} \nonumber 
\end{eqnarray} 
such that the integration pairing
\[  H^{\lf}_d(\GL_g(\Z);\Q)\otimes_{\Q} H^{d}_{c}(g)  \To \C\]
factors through the period pairing:
$\langle \ , \ \rangle: (\MM^d_g)_B^{\vee}  \otimes_{\Q}  (\MM^d_g)_{dR} \rightarrow \C$. 
\end{thm} 
\begin{proof} 
First consider the object of $\mathcal{H}_{\Q}$ defined by:
\[ M^d_g = \bigoplus_{\dim \sigma_Q = d}  \mot_Q^{\mathrm{Aut}(\sigma_Q)}\ ,\]
where the sum is over  isomorphism  classes of cones $\sigma_Q$ of dimension $d$. 
Consider the subspace  $Z_g^d \subset \Omega^d( \LA^{\trop,\BB}_g\backslash \widetilde{\Det})$ consisting of compatible systems of closed forms whose restriction to 
    $ \partial \LA^{\trop, \BB}_g$ vanishes, and let $B_g^d\subset Z^d_g$ be the subspace spanned by  $d\omega$, where $\omega$ are compatible systems 
    of forms vanishing  on   $ \partial \LA^{\trop, \BB}_g$.
      There is a natural map
$Z_g^d \rightarrow  (M^d_g)_{dR}$.  By the Tannaka theorem, there are unique subobjects $\langle B_g^d \rangle \subset \langle Z_g^d \rangle \subset M^d_g$ such that 
\[  \langle Z_g^d\rangle_{dR}  =    G_{dR}\, Z_g^d    \qquad \hbox{and} \qquad   \langle B_g^d\rangle_{dR}  =  G_{dR}\, B_g^d\ .\]
Let $ N_g^d  =  \langle Z_g^d \rangle /  \langle B_g^d \rangle$. 
By definition, there is  a natural map $H^{d}_{c}(g)   \rightarrow   \left( N^d_g \right)_{dR}$.

By theorem \ref{thm: Detatinfinity},  inclusion defines a  natural map
\[ H_d(\sigma_Q, \partial \sigma_{Q})  \To (\mot_Q)^{\vee}_B =  H_d \left( (P^{\BB_Q} \backslash \widetilde{\Det}) (\C) \  , \  L^{\BB_Q} \backslash (L^{\BB_Q} \cap  \widetilde{\Det})(\C);\Q \right) \] 
which induces  a map
\[  \bigoplus_{ \dim \sigma_Q =d}  H_d(\sigma_Q, \partial \sigma_{Q})  /\mathrm{Aut}(\sigma_Q)  \To (M^d_g)_B^{\vee} \ ,  \]
where the direct sum is over all isomorphism classes of cones $\sigma_Q$  of dimension $d$.  Via the cellular interpretation of homology (theorem \ref{thm: CellularHomology}), we deduce a map from closed   
 cellular chains $\gamma$  in $\LA_g^{\trop,\BB}$  
   to    $(M_g^d)^{\vee}_B$ and thence to its quotient  $\langle Z^d_g\rangle_B^{\vee}$. Since Stokes' theorem \eqref{StokesForQ} holds for motivic periods, the period pairing 
$ \langle \gamma, g\omega \rangle $
vanishes 
for all $g\in G^{dR}$ if $\omega \in B_g^d$; and for all $g\in G^{dR}$ and $\omega\in Z_g^d$ if $\gamma$ is  supported on 
$\partial \LA_g^{\trop,\BB}$, or if $\gamma = \partial \alpha$ is  a boundary. 
This proves, at the same time, that the image of a cellular chain $\gamma$  lands in the subspace $(\langle Z^d_g\rangle  / \langle B^d_g\rangle)_B^{\vee}$ of $\langle Z^d_g\rangle_B^{\vee}$, and that it is zero if $\gamma$ is a boundary, or supported on $\partial \LA_g^{\trop, \BB}$. Thus by theorem \ref{thm: RelativeCellular}, there is a  well-defined map 
\[ H_d\left( \left| \LA_g^{\trop,\BB} \right| \ ,  \  \left| \partial \LA_g^{\trop,\BB} \right|  \right)   \To  (N^d_g)_B^{\vee} \ .  \] 
 The left-hand group  is isomorphic to the locally finite homology
$H^{\lf}_d(\GL_g(\Z); \Q),$
and so the object $N^d_g$ satisfies the conditions of the theorem. This proves the existence. For the minimality, one can repeat the argument of \cite[\S2.4]{NotesMot}, to show that  any object $M$  of $\mathcal{H}_{\Q}$ equipped with maps  $V \hookrightarrow  M_B^{\vee}$ 
and $W \hookrightarrow M_{dR}$, where $V,W$ are $\Q$-vector spaces, has a unique smallest subquotient with the this property (the case when $V$, and $W$ have dimension $1$ is proven in  \cite[\S2.4]{NotesMot}, but the general case is similar). 
\end{proof}

It follows from \eqref{OmegaInjectsBothLPG}  and the relative de Rham theorem \ref{thm: RelativeCellular} that the composition
\[\Omega_{c}^d(g)\To H^{d}_c(g)    \To   (\MM^d_g)_{dR}\] 
is injective. Theorem \ref{thm: MotiveGLg} implies that any integral of a canonical form of compact type over a locally finite homology class in $H^{\lf}_d(\GL_g(\Z))$ is motivic.
In particular, we deduce a motivic interpretation of Minkowski's  volume integrals \eqref{MinkowskiVolume}.

\begin{rem}  An interesting question is whether the map  $H^{\lf}_d(\GL_g(\Z);\Q) \rightarrow ( \MM^d_g)_B^{\vee}$ is injective or not: is it the case that the entire  locally finite  homology of $\GL_g(\Z)$ is motivic in this sense, or only a quotient of it?  (One could  capture the entire cohomology by enlarging $H^d_c(g)$, for instance by replacing the determinant locus with  suitable linear subspaces, since every cohomology class can be represented by a compatible family of  algebraic forms  by remark \ref{rem: AllFormsPolynomial}.)   A deeper question is  to describe the action of  $G_{dR}$ on $ ( \MM^d_g)_{dR}$.  See  \cite[\S9.5]{BrSigma} for a related discussion for  motives arising from the moduli space of tropical curves. 
\end{rem} 

\begin{example} In the case $g=3, d=5$, the map 
\[ \Q \cong H_5^{\lf}(\GL_3(\Z);\Q) \To (M^5_3)^{\vee}_B\] 
is injective. In fact, $M_3^5$ is given by the graph `motive' \cite{BEK} of the wheel with 3 spokes $W_3$, which, one can show, is of rank two, and  is a non-trivial  extension of $\Q(-3)$ by $\Q(0)$.  The space $\Q(-3)_{dR}$ is spanned by the class of the canonical form $\omega^5$; the group $ H_5^{\lf}(\GL_3(\Z);\Q)$ is spanned by the  fundamental class   which is  the relative homology  class of   $\sigma^{\BB}_{W_3}$.  The corresponding period is the `wheel integral' and equals  $60 \zeta(3)$.  
\end{example}

\section{Appendix: Blow-ups and Borel-Serre for $\GL_n$} 

In this appendix, we prove that our blow-up construction can be used to retrieve the Borel-Serre compactification. 
In view of our applications, we have focused on the case $\GL_n(\Z)$, but many of the  arguments presented  here are quite general. This section is mostly self-contained, although one or two  proofs require statements in the main text. 
\subsection{Statements} \label{sect: BSstatement}
Let $V$ be a vector space of dimension $n$ over $\Q$. Let $\Quad(V)$ denote the space of quadratic forms on $V$ and $\Pro(\Quad(V))$ the associated projective space.
Let 
\begin{equation} \label{XsubPQ} X  \quad \subset \quad \Pro(\Quad(V))(\R) \end{equation} 
denote the space of positive definite  real quadratic forms.   It is contained in the complement of the determinant hypersurface $\Det \subset \Pro(\Quad(V))$. 
For any  $0\neq K\subset V$, the space $\Pro(\Quad(V/K))\subset \Det$ is the subspace of quadratic forms with null space $K$.

\begin{defn} Let $\pi_{\BBB}: \BBB \rightarrow \Pro(\Quad(V))$ denote the space\footnote{We shall only consider   $\BBB(\R)$, viewed as a  topological space with a stratification by exceptional divisors} obtained by blowing up all (infinitely many) linear  subspaces of $ \Pro(\Quad(V)) $   of the form  
$ \Pro(\Quad(V/K))$ 
for $0\neq K\subset V$, in increasing order of dimension.  Let $\widetilde{\Det} \subset \BBB$ denote the strict transform of $\Det$.
\end{defn}

Let $X^{\mathrm{BS}}$ denote the Borel-Serre compactification of $X$.  It has a stratification 
\[ X^{\mathrm{BS}} = X \cup \bigcup_{P} e(P)\]
where $P$ ranges over rational parabolics (see \S \ref{sect: RemindersBS}). The exceptional divisor  $\mathcal{E}$ for $\pi_{\BBB}$  defines a stratification on $\BBB$ by taking intersections of irreducible components.  Since  $\pi_{\BBB}(\mathcal{E})\subset \Det$,  the inclusion \eqref{XsubPQ} induces an injective map: 
\begin{equation} \label{XsubBopen}  X  \ \To \   \left( \Pro(\Quad(V)) \setminus \Det \right) (\R)   \overset{\pi_{\BBB}^{-1}}{\To}  \left( \BBB   \setminus   \mathcal{E} \right) (\R)  \ .
\end{equation}

\begin{thm}   \label{thm: MainXBStoBlowup} There is a continuous injective map of stratified spaces
\begin{equation} \label{XBSToBlowUp} 
f : X^{\mathrm{BS}} \To  \left(\BBB  \ \setminus \ \widetilde{\Det}\right) (\R) 
\end{equation} 
  whose restriction to the big open stratum is the embedding \eqref{XsubBopen}.
 The image of \eqref{XBSToBlowUp} is the closure of $X$, for the analytic topology, inside $\BBB (\R)$.
 The map $f$ is equivariant with respect to the natural actions of $\GL(V)$ on both $X^{\mathrm{BS}}$ and $\BBB$.
 
Let $P$ denote the rational parabolic associated to  a  nested sequence of  strict subspaces $0 \subset V_d\subset V_{d-1}\subset \ldots \subset V_1 \subset V$. The map  \eqref{XBSToBlowUp}  restricts to a map
\begin{equation} \label{foneP}  f: e(P) \To \mathcal{E}_P \backslash \left( \widetilde{\Det} \cap \mathcal{E}_P\right)  \,(\R) 
\end{equation} 
where $\mathcal{E}_P\subset \BBB$ is the exceptional locus associated to the iterated  blow-up of   \[ \Pro(\Quad(V/V_1)) \subset \ldots \subset \Pro(\Quad(V/V_d))  \ \]
in increasing order of dimension.  In particular, $f(e(P)) = f(X^{\mathrm{BS}}) \cap \mathcal{E}_P(\R)$.
The complement of the strict transform of $\Det$ is    canonically  isomorphic to  a product of hypersurface complements:
\[   \mathcal{E}_P \backslash \left( \widetilde{\Det} \cap \mathcal{E}_P\right) \cong \prod_{k=0}^d   \left(  \Pro\left( \frac{\Quad(V/V_{k+1}) }{   \Quad(V/V_{k})  }   \right) 
\ \setminus \  \Det\big|_{V_{k}/V_{k+1}}   \right) 
\]
where  we write $V_0=V$ and $V_{d+1}=0$, and
 $\Det|_{V_{k}/V_{k+1}}$ is  the vanishing locus of the determinant of the  restriction of a quadratic form on $V/V_{k+1}$ to the subspace $V_k/V_{k+1}$.
\end{thm} 

In short, the space $\BBB$ is an algebraic incarnation of the Borel-Serre compactification with identical combinatorial structure. The space $X^{\mathrm{BS}}$ is  identified with a semi-algebraic subset  of its real points $\BBB(\R)$. It is defined by  algebraic inequalities  of the form $u\geq 0$, where $u$ is a homogeneous polynomial given by the determinant of a matrix minor.
Note that although $\BBB$ is defined by infinitely many blow-ups,  the local structure of $X^{\mathrm{BS}}$ in terms of spaces $X(P)$ (see \S\ref{sect: RemindersBS}) may be studied by embedding each $X(P)$ into a space obtained by performing only  finitely many  linear blow-ups of $\Pro(\Quad(V))$.  This is how we shall prove theorem \ref{thm: MainXBStoBlowup}.

\begin{cor} \label{cor: LAgBBtropisBS}  Let $|\LA_g|^{\trop,\BBB}$ be the topological space defined in \eqref{LAgtropBBasFunctoronInvertedBlowups}. Then the there is  a canonical homeomorphism
\[ f: X_g^{\mathrm{BS}} / \GL_g(\Z)  \overset{\sim}{\To}  |\LA_g|^{\trop,\BBB} \ .\] 
\end{cor} 

\subsection{Reminders on the Borel-Serre compactifcation for $\GL_n$} \label{sect: RemindersBS} 
Let $V$ be a vector space of dimension $n$ over $\R$ equipped with an  inner product, or equivalently, an isomorphism $V \cong  V^{\vee}$. The group of automorphisms $\GL(V)$ is isomorphic to 
$\GL_n(\R)$, and the subgroup preserving the inner product defines a maximal compact subgroup $K= O(V)$ of $\GL(V)$.  It is isomorphic to $O_n(\R)$.
  Given any filtration $F$ of $V$ of length $d$ by subspaces 
\begin{equation} \label{FlagF} F:  \qquad  0= V_{d+1} \subset  V_d \subset \ldots \subset  V_1 \subset V_{0}=  V \ , 
\end{equation} 
where all inclusions are strict, 
 denote by $P_F \leq \mathrm{GL}(V)$
the parabolic subgroup of automorphisms of $V$ which preserve $F$.
A filtration   $F'$ 
\[ F': \qquad 0 \subset  V_{i_k}  \subset \ldots \subset V_{i_1}  \subset V\ \]
where $1 \leq i_1<\ldots <i_k \leq d$ is  obtained by omitting elements from the filtration $F$, is denoted by $F'\leq F$.   The set of filtrations $F'$ such that  $F'\leq F$  forms a finite poset.  Since for $F'\leq F$,  one has  $P_F\leq P_{F'}$, the set $\{P_{F'} : F'\leq F\}$ with respect to inclusion of groups is  isomorphic to the opposite poset.

For any such $P=P_F$, as above, denote by  $K_P = P \cap K$.  The natural  map
\[ P \To \prod^d_{i=0} \GL(V_{i}/V_{i+1})\]
  induces an isomorphism   $K_P \cong \prod_{i=0}^d  O(V_{i}/V_{i+1})$, the product  of orthogonal groups with respect to the induced inner products on $V_i\rightarrow V \cong V^{\vee} \rightarrow V_i^{\vee}$.

  Let   $Z_P \leq P$ denote the subgroup of central elements acting via scalar multiplication by $\R^{\times}_{>0}$    on each quotient $V_{i}/V_{i+1}$ for $0\leq i \leq d$.
 Thus $Z_P\cong (\R^{\times}_{>0})^{d+1}$. 

\begin{example} \label{ex: Matrices}  Let $e_1,\ldots, e_n$ denote a basis  of $V$ compatible with $F$. In this basis, automorphisms of $V$ may be identified with $n\times n$ matrices, and the groups $P$ and $Z_P$ are block-lower triangular, and block diagonal respectively: 
\[ P = \begin{pmatrix}  P_{00} & & \\ P_{01} & P_{11} \\   \vdots & & \ddots \\  P_{0k} & P_{1k} & \ldots & P_{kk} \end{pmatrix} 
\qquad  , \qquad Z_P = \begin{pmatrix}  \mu_0 I_0 & & \\  & \mu_1 I_1  \\    & & \ddots \\   & &  & \mu_k I_k  \end{pmatrix}  \] 
where $k=d$, $I_0,\ldots, I_k$ are block identity matrices and $\mu_0,\ldots, \mu_k >0$.  The $P_{ij}$ are likewise block matrices. The group $K_P$ is isomorphic to the group of block diagonal matrices where the matrix $O_i$ in the $i^\mathrm{th}$ block is orthogonal: $O_i^T O_i = I_i$.
\end{example}

The group $Z_P$ will be made to  act upon  $P$ by   multiplication on the left.
Denote the  group of central elements  by $H = Z_P \cap Z(\mathrm{GL}(V)) \cong \R^{\times}_{>0}$.  In example \ref{ex: Matrices}, $H\leq Z_P$ is the subgroup of diagonal matrices. 
Let $X$ denote the space of projective equivalence classes of real positive definite quadratic forms on $V$.

\subsubsection{Geodesic action}
For any  $P=P_F$  as above there is a homeomorphism:
\begin{eqnarray}  \label{XasquotientofP}   H K_P  \setminus P  & \overset{\sim}{\To}  &   X    \\
g & \mapsto & g^T g \nonumber \ .
\end{eqnarray} 
Since it commutes with $K_P$ and $H$,  the action of $Z_P$ on $P$ passes to a well-defined action on $ H K_P  \setminus P$.  We write $A_P=H \setminus  Z_P$. The   action of $A_P$ on $X$ induced by  \eqref{XasquotientofP}  is called the 
 geodesic action. Borel and Serre consider the following space of   orbits for its action:
\[ e(P)  =   A_P \setminus X \overset{\eqref{XasquotientofP}}\cong Z_P K_P \setminus P\ .\]
It is homeomorphic to a Euclidean space $\R^{\binom{n+1}{2}-d-1}$.

The compactification of $X$ is defined as follows. Consider the  isomorphism
 \begin{eqnarray} \label{muCoordsonAP} 
 H \setminus Z_P   & \cong &  A_P \cong (\R^{\times}_{>0})^d    \\
  (\mu_0,\mu_1,\ldots, \mu_{d}) & \mapsto  & ( \lambda_1,\ldots, \lambda_{d})    \nonumber 
  \end{eqnarray} 
  where $\lambda_i = \mu_{i}/\mu_{i-1}$ for $1\leq i\leq d$. Consider the multiplicative monoid
\[ \overline{A}_P =    (\R^{\times}_{\geq 0})^d\]
with coordinates $\lambda_i \geq 0$.  The subgroup $A_P\leq \overline{A}_P$ has coordinates  $\lambda_i>0$.  
Borel and Serre define 
\[ X(P)  =   A_P \setminus (  \overline{A}_P \times X)   \ . \]
The space $\overline{A}_P$ is  stratified by the closed subspaces $\lambda_i=0$ for $i=1,\ldots, d$. The  corresponding open  stratification  is a union of copies of $(\R_{>0}^{\times})^a$ for $a\leq d$, which may be identified with the subgroups $A_{P_{F'}}$, for all  $F' \leq F$.   Thus 
\[ X(P) = X \sqcup  \coprod_{P'  }  A_{P'} \setminus  X = X \sqcup \coprod_{P' } e(P')\]  
is a disjoint union of spaces $e(P')$, where   $P'=P_{F'}$ for all $F' \leq F$. 

\subsubsection{Borel-Serre compactification}
Now suppose that $V$ has a $\Q$-structure and that its inner product  is defined over $\Q$. A parabolic $P= P_F$ is called rational if the vector spaces $V_i$ in the corresponding sequence \eqref{FlagF} are all defined over $\Q$.

The Borel-Serre compactification is defined to be
\begin{equation} \label{XBSasunionXP} X^{\mathrm{BS}} = \bigcup_{P}  X(P) 
\end{equation} 
where the union is over all rational parabolics $P$.

\begin{example} In the following example we compare and contrast the Borel-Serre approach with our own.

(1) (Borel-Serre).  Let $V= \R^3$, $d=2$, and $V_i \cong \R^{3-i}$, for $0\leq i\leq 2$. 
Write 
\[  P =  \begin{pmatrix}  p_{00} &  &   \\
p_{01} & p_{11} &   \\
p_{02} & p_{12} & p_{22}
\end{pmatrix} 
\quad \ , \ \quad
Z =  \begin{pmatrix}  \mu_0  &   &   \\
    & \mu_1 &   \\
    &     & \mu_2\end{pmatrix} 
\]
where $p_{ij} \in \R$ and $\det(P)= p_{00}p_{11}p_{22}\neq 0$. 
We set 
\[
X = P^T P =  \begin{pmatrix}  p_{00}^2 + p_{01}^2 + p_{02}^2 &  p_{02} p_{12}+p_{01} p_{11}  & p_{02}p_{22}  \\ 
p_{02}p_{12}+p_{01}p_{11} &  p_{12}^2+p_{11}^2 &  p_{12}p_{22} \\
 p_{02}p_{22}  &  p_{1 2}p_{2 2}& p_{22}^2
\end{pmatrix} \]
Let $X_{ij}=x_{ij}$. 
The geodesic action $X\mapsto Z \circ X $ where $Z \circ X =    P^T Z^T Z P $, may, as one can  check, be written intrinsically as a function of $X$
via 
\[Z \circ X =  
\mu_0^2 \begin{pmatrix}  \frac{\det(X)}{\det(X^{11})} & & \\
& & \\
& & 
\end{pmatrix} +\mu_1^2
 \begin{pmatrix}     \frac{\det(X^{21})^2}{\det(X^{11}) x_{33}}    & 
\frac{ \det(X^{21})}{x_{33}}        &    \\
\frac{ \det(X^{21})}{x_{33}}       &   \frac{\det(X^{11})}{x_{33}}    &     \\
&  & \end{pmatrix}
+ 
\mu_2^2
 \begin{pmatrix}  \frac{x_{13}^2}{x_{33}}    & 
 \frac{x_{1 3}x_{2 3}}{x_{33}}   &     x_{1 3}  \\
 \frac{x_{1 3}x_{2 3}}{x_{33}}   &   \frac{x_{2 3}^2}{x_{3 3}}    &   x_{2 3}  \\
  x_{1 3} &    x_{2 3}&  x_{3 3}
\end{pmatrix}
\]
where  $\det(X^{11}) = x_{22} x_{33} -x_{23}^2$
and 
$ \det(X^{21})  = x_{1 2}x_{3 3} -x_{13}x_{23}$ are determinants of minors of $X$.
The action $X \mapsto Z \circ X$  is clearly algebraic in the entries of $X$ and is well-defined on the locus where 
$ \det(X) \det(X^{11}) x_{33} \neq 0$. 
This example illustrates that the coordinates  given by the geodesic action are perhaps not ideally suited to studying the asymptotic behaviour of canonical differential forms at infinity. 
\\

(2). (Blow-ups.)          With notations as above, the coordinates in the blow-up are, in our view, much simpler and  are given by  scalar multiplication on matrix blocks (see \eqref{LambdaBlockForm}). They  correspond to a transformation of the form
\[  \begin{pmatrix} x_{11} & x_{12} & x_{13} \\
x_{12} & x_{22} & x_{23} \\
x_{13} & x_{23} & x_{33} \\ 
\end{pmatrix}  \mapsto  \begin{pmatrix}  \nu_0\nu_1\nu_2  \,   x_{11} & \nu_1 \nu_2 \,     x_{12} &  \nu_2\,   x_{13} \\
\nu_1\nu_2\,   x_{12} & \nu_1\nu_2\,  x_{22} & \nu_2\,  x_{23} \\
\nu_2\, x_{13} &  \nu_2\, x_{23} & \nu_2\,  x_{33} \\ 
\end{pmatrix} 
 \]
for suitable $\nu_0,\nu_1,\nu_2$.

Although the geometric origin of these two points of view are markedly  different, 
we  prove that they are in fact equivalent when restricted to the space of positive definite matrices.

\end{example} 

\begin{rem} The following intuition for the structure of $\BB(\R)$ may be helpful. Consider an  analytic family of (projective classes of) symmetric matrices $P_t \subset X\subset 
\BB(\R)$, as $t\rightarrow 0$, which are positive definite for $t>0$ and   which we think of as quadratic forms on a vector space $V$. This family  is not necessarily  geodesic. The limiting matrix $P_{0} \in \Pro(\Quad(V))(\R)$ may have vanishing determinant, in which case it is positive semi-definite, and has a null space $K\subset V$. By choosing a complementary space $V = K \oplus C$, we may write $P_t$ locally near $t=0$ in block matrix form 
\[ P_t = \begin{pmatrix}  A t &  U t \\ U^T t  & P_0 \end{pmatrix} + \hbox{higher order terms in } t \ .\] 
If $A$ is positive definite, the curve $ t\mapsto P_t$ in $\BBB(\R)$ has a limiting  point for $t=0$ which lies on  the exceptional divisor associated to the blow-up of $\Pro(\Quad(V/K))$  (if $A$ is not positive definite, then repeat the argument for the restriction of $P_t$ to $K$, i.e. the top-left block, to define a point on a higher codimension exceptional locus). To describe this limiting point, consider the projective classes of the pair  of symmetric matrices: 
\[      \begin{pmatrix} A & U \\ U^T & 0 \end{pmatrix} ,     \begin{pmatrix}  0  & 0  \\ 0  & P_0 \end{pmatrix}  \] 
which define  a point  in  the exceptional divisor $\Pro(\Quad(V)/\Quad(V/K))(\R)  \times \Pro(\Quad(V/K))(\R)$. The latter is the projectivised normal bundle of $\Pro(\Quad(V/K))$ inside $\Pro(\Quad(V))$ as evidenced by the fact that the limiting point of $P_t$ in $\BB(\R)$ is  the pair: (normal vector, limiting value). \end{rem}

\subsection{Linear monoidal actions and blow-ups}\label{sect: setup}
The Borel-Serre construction is topological, and relies on the geodesic action which is only defined on the space of positive definite matrices, since it makes crucial use of the surjectivity of  \eqref{XasquotientofP}. In order to reformulate it  in the language of algebraic geometry, we  first need some preliminaries on  monoidal actions on vector space schemes 
and their quotients by multiplicative groups. A possible relation between Borel-Serre's construction and blow-ups is hinted at in \cite[pg 437, line 5]{BorelSerre}, but does not seem to have been developed further as far as we are aware. 

Let $k$ be a field. All schemes will be over $k$.

\subsubsection{Multiplicative monoid}
The multiplicative group  is defined by $\Gm= \mathrm{Spec}\,  k[t,t^{-1}]$, where $\Or(\Gm) = k[t,t^{-1}]$ is the  Hopf algebra whose coproduct $\Delta: \Or(\Gm)\rightarrow \Or(\Gm) \otimes_k \Or(\Gm)$  satisfies $\Delta t = t\otimes t$. The multiplication on $\Gm$  is dual to $\Delta$. 

We denote the multiplicative monoid  scheme  $\Gm\leq \Gmb$ to be
\[ \Gmb = \mathrm{Spec} \, k[t] \]
where $k[t]$ is the coalgebra equipped with the   coproduct $\Delta:k[t]\rightarrow k[t]\otimes_k k[t]$ satisfying the same formula $\Delta t= t\otimes t$.  Thus,  as a scheme, $\Gmb$ is simply the affine line $\mathbb{A}^1$, but has a multiplication morphism $ \Gmb \times \Gmb \rightarrow \Gmb$. It has a distinguished point $0 \in \Gmb$ such that 
$\Gmb \backslash 0 = \Gm$, and has the property that the multiplication satisfies $0 \times \Gmb \rightarrow  0$.

\subsubsection{Linear monoidal action}
A vector space of dimension $n$ over a field $k$ will be viewed as the $k$-points of an affine scheme $\mathbb{A}^n$ in the usual manner. Direct sums correspond to products of vector-space schemes.
Consider a  finite-dimensional graded vector space $W$ over a field $k$  
\[  W= W_n \oplus W_{n-1} \oplus \ldots \oplus W_1 \oplus W_0 \ \]
where  $\dim W_i\geq 1$ for all $i$. 
The multiplicative group $\Gm \times W_i \rightarrow W_i$ acts by scalar multiplication on each component, denoted by $(\lambda, w) \mapsto \lambda w$. 
It extends to an action $\Gmb \times W_i \rightarrow W_i$  which sends $(0,w)$ to the origin of  $W_i$.
Consider the monoid  action 
\begin{equation} \label{MonoidActionMdefn}  m:  \Gmb^n \times W    \To   W   \ , 
\end{equation}
where the $i^\mathrm{th}$ component of $\Gmb$ for $i=1,\ldots, n$,  acts via 
\begin{eqnarray}  \Gmb   \times W   & \To&    W   \label{ithcomponentGmbaction} \\
(\lambda_i,   (w_n,\ldots, w_0) )  & \mapsto &  (\lambda_i w_n, \ldots, \lambda_i w_i, w_{i-1}, \ldots, w_0)  \ .  \nonumber
 \end{eqnarray} 
The action \eqref{MonoidActionMdefn} restricts to an action of $\Gm^n$ on $W$. Furthermore, let $\Gm^n$ act on $\Gmb^n$ componentwise via the map 
$\Gm \times \Gmb \rightarrow \Gmb $ given by $(\mu, \lambda ) \mapsto \lambda \mu^{-1} $, where $\mu^{-1}$ is the inverse in the group $\Gm$. 
We deduce that     \eqref{MonoidActionMdefn}   is equivariant for the diagonal action of  $\Gm^n$  on $\Gmb^n \times W$, in other words, the following diagram commutes
\begin{equation} \label{GmGmbcommute}
\begin{array}{ccc}
 \Gm^n \times (\Gmb^n \times W )  & \To   &   \Gmb^n \times W    \\
  \downarrow &   & \downarrow_m  \\
 \Gmb^n \times W   & \overset{m}{\To}    &  W  
\end{array}
\end{equation} 
where the vertical map on the left is projection onto $ \Gmb^n \times W$, and the horizontal map is the diagonal action of $\Gm^n$ on both $\Gmb^n$ and $W$.
In the case $n=1$ this diagram merely expresses the trivial fact that
$m  (  \mu.  (\lambda, w)) = m (\lambda\mu^{-1}, \mu w) = \lambda w = m (\lambda, w)$.

\subsubsection{Quotients  and blowups}
For any finite dimensional vector space $W$  over $k$,  let $W^{\times} = W \backslash \{0\}$. It is an open subscheme of $W$  stable under the action of $\Gm$, but not under the extended action   of $\Gmb$, since $ \Gmb \times W^{\times} \rightarrow W$. 
Let us denote by:
\[ W^{\star} =  W^{\times}_n \times  W^{\times}_{n-1} \times \ldots \times W^{\times}_1 \times W_0 \ , \]
(component $W_0$ \emph{sic}). 
The morphism \eqref{MonoidActionMdefn} restricts to a morphism $m: \Gmb^n\times W^{\star} \rightarrow W$, which, by \eqref{GmGmbcommute}, is equivariant with respect to the action of $\Gm^n$.

We wish to define its quotient by $\Gm^n$ in the category of schemes.

\begin{example} \label{Ex: ProjAsQuotient} We warm up with the construction of projective space as a quotient.
Let $x_i^1,\ldots, x_i^n$ denote  coordinates on the  $k$-vector space $W_i \cong \mathbb{A}^n$, where $n= \dim W_i$,  and write  
\[ W_i^{\times} = \bigcup_{j=1}^n U_i^j \]
where $U_i^j \subset \mathbb{A}^n$ is the open complement of $x_i^j=0$.  Let    $U_i^{jk} = U_i^j \cap U_i^k$. The opens $U_i^j, U_i^{jk}$ are stable under the action of $\Gm$ on $W$, and the inclusions $U_i^{jk} \hookrightarrow  U_i^j$ are $\Gm$-equivariant. One may define the quotient by $\Gm$: 
\[ \Gm \ql W_i^{\times} : = \bigcup_{i=1}^n   \Gm \ql U_i^j \]
to be the  union of  schemes $ \Gm \ql U_i^i$ glued along the $ \Gm \ql  U_i^{jk}$, 
where $\Gm \ql  U_i^j = \mathrm{Spec} \, ( \Or(U^j_i)^{\Gm})$   (and similarly  $ \Gm \ql  U_i^{jk}$)  is the affine scheme whose coordinate ring consists of  $\Gm$-invariant polynomials. Since 
$ \Or(U_i^j)^{\Gm} = k[ \frac{x^1_i}{x_i^j}, \ldots, \frac{x_i^n}{x_i^j}]$, we recover the standard affine covering of projective space and  thus 
\[   \Gm \ql W_i^{\times} \cong \Pro(W_i) \]
is canonically isomorphic to the projective space of $W_i$. 
\end{example} 

\begin{defn}   Consider   a $\Gm$-equivariant covering $W^{\times}_i= \bigcup_{1\leq j \leq \dim W_i}  U_i^{j}$  of $W_i$ for every $i=1,\ldots, n$ as in the previous example, and define a scheme:
\begin{equation} \label{MonoidalQuotientDefn}  \Gm^n \ql (\Gmb^n\times W^{\star}) :  =   \bigcup_{i_1,\ldots, i_n}    \Gm^n  \ql \left(  \Gmb^n \times U_n^{i_n} \times \ldots \times U_1^{i_1} \times W_0\right) \ ,\end{equation}
where each $1\leq i_k \leq \dim W_k$.  It does not depend on the choice of coordinates (since $\Gm^n$ is central in the product of  general linear groups $\mathrm{GL}(W_i)$).
\end{defn} 
Since it is  $\Gm^n$-equivariant, the map $m$ \eqref{MonoidActionMdefn}  defines a morphism of schemes also denoted by
\begin{equation}  \label{monquotient}   m \ : \    \Gm^n \ql (\Gmb^n\times W^{\star}) \To   W\ . 
\end{equation}

\begin{prop} \label{prop: quotientandblowup}
The scheme   \eqref{MonoidalQuotientDefn}     canonically embeds as an open subscheme of   the  iterated  blow-up $B\rightarrow W$  of  the affine space $W$  along the linear subspaces
\[ W_0 \  \subset \ W_0 \oplus W_1 \  \subset \   \ldots  \ \subset \  W_0 \oplus W_1 \oplus \ldots \oplus W_{n-1} \]
in increasing order of dimension. It is Zariski-dense in $B$.
 \end{prop} 

\begin{proof} The most direct proof is  via explicit coordinates generalising example \ref{Ex: ProjAsQuotient}.
Consider the coordinate ring of a particular open  in \eqref{MonoidalQuotientDefn}
\[    \Or\left(  \Gmb^n \times U_n^{i_n} \times \ldots \times U_1^{i_1} \times W_0\right) \]
where $1 \leq i_r \leq \dim W_r$. With the labelling of coordinates of example  \ref{Ex: ProjAsQuotient}, it is
\[ k  \left[\lambda_1,\ldots, \lambda_n,   \left(x^i_{j}\right)_{\substack{0\leq j \leq n \\ 1\leq i \leq \dim W_i}},  \frac{1}{x_1^{i_1}}, \ldots,  \frac{1}{x_n^{i_n}} \right] \]
where $\lambda_i$ is the coordinate on the $i$th component of $\Gmb$ (and the $x^j_0$ are coordinates on $W_0$). The  action of the $i$th component of $\Gm$ is equivalent to the action
\[  \lambda_j \mapsto \begin{cases}  \lambda_j  \mu_i^{-1} \quad  \hbox{ if } i = j \ , \\
\lambda_j  \quad \qquad \hbox{ if } i \neq j  \end{cases}  
\quad \hbox{ and } \quad   
x_k^j \mapsto \begin{cases} \mu_i x_k^j \quad \hbox{ if } k \geq i  \\ 
x_k^j  \qquad \hbox{ if } k < i 
\end{cases}
\]
for all $j\geq 1$. The ring of  $\Gm^n$-invariants of $\Or\left(  \Gmb^n \times U_n^{i_n} \times \ldots \times U_1^{i_1} \times W_0\right)$ is 
\[ \Or(W_0) \Bigg[  \left(\frac{x^i_1}{x^{i_1}_1}\right)_{1\leq i\leq \dim W_1},   \left(\frac{x^i_2}{x^{i_2}_2}\right)_{1\leq i\leq \dim W_2},     \ldots ,  \left(\frac{x^i_n}{x^{i_n}_n}\right)_{1\leq i\leq \dim W_n}, \lambda_1x_1^{i_1},
\lambda_2 \frac{x_2^{i_2}}{x_1^{i_1}}, \ldots, \lambda_n\frac{x_n^{i_n}}{x_{n-1}^{i_{n-1}}}     \Bigg] \ , \]
since the former is generated over the latter (which is clearly $\Gm^n$-invariant) by $(x_1^{i_1})^{\pm 1}, \ldots, (x_n^{i_n})^{ \pm 1}$. 
By introducing $\Gm^n$-invariant variables 
  \[  \alpha_j^i=  \lambda_1 \lambda_2 \ldots \lambda_j x^i_j     \]
we can rewrite the ring  of  $\Gm^n$-invariants in the slightly different form
\[  \Or(W_0)\Bigg[ \alpha_1^{i_1}  ,   \left(\frac{\alpha_1^{i}}{\alpha_1^{i_1}}\right)_{1\leq i\leq \dim W_1},  \frac{ \alpha_2^{i_2}}{\alpha_1^{i_1}}  ,     \left(\frac{\alpha_2^{i}}{\alpha_2^{i_2}}\right)_{1\leq i\leq \dim W_2} ,   \ldots   , \quad 
 \frac{ \alpha_n^{i_n} }{\alpha_{n-1}^{i_{n-1}}}       , \left( \frac{\alpha_n^{i}}{\alpha_n^{i_n}}\right)_{1\leq i\leq \dim W_n}     \Bigg]  \ . \] 
These are identical to a system of local coordinates on the iterated blow-up described in   \S\ref{sect: LocalBlowUpCoordinates} (after reversing the indexation of the subscripts of the $\alpha$'s)  along
\[ V(\alpha_j^i: j\geq   1) \quad \subset  \quad  V(\alpha_j^i: j\geq   2) \quad \subset  \quad \ldots \quad  \subset \quad  V(\alpha_j^i: j\geq   n)  \ .  \]
 These coordinates define a canonical  open immersion of this affine  chart  of  \eqref{MonoidalQuotientDefn}     into $B$. Since these coordinate charts are compatible they glue together to form an explicit open immersion of schemes from    \eqref{MonoidalQuotientDefn}     to $B$. \end{proof}

Inspection of the proof shows that the  scheme   \eqref{MonoidalQuotientDefn}  is canonically isomorphic to the complement in $B$ of strict transforms of  some linear subspaces  of $W$ which may be written in terms of the $W_i$, although we will not need such an explicit description here.

\begin{rem} The recursive structure of the iterated blow-up may be interpreted using the monoidal action. 
Consider the morphism 
\begin{multline}
\pi_r :   \Gmb^r \times W_n \times  \ldots  \times W_{r+1} \times  W_r^{\times}\times \ldots \times W_1^{\times} \times W_0  
\To\\
    \Gmb^{r-1} \times W_n \times  \ldots  \times W_{r} \times  W_{r-1}^{\times}\times \ldots \times W_1^{\times} \times W_0 \end{multline} 
which sends $((\lambda_1,\ldots, \lambda_r) , ( w_n,\ldots, w_0))$ to $((\lambda_1,\ldots, \lambda_{r-1} ), ( \lambda_r w_n,\ldots, \lambda_r w_r, w_{r-1} ,\ldots, w_0))$.

\begin{lem}  \label{lem:pirinducesblowup}
 The morphism $\pi_r$ induces a map 
 \begin{multline}
 \Gm^r \ql \left(  \Gmb^r \times W_n \times  \ldots  \times W_{r+1} \times  W_r^{\times}\times \ldots \times W_1^{\times} \times W_0  \right) 
\To\\
  \Gm^{r-1} \ql \left( \Gmb^{r-1} \times W_n \times  \ldots  \times W_{r} \times  W_{r-1}^{\times}\times \ldots \times W_1^{\times} \times W_0 \right) \end{multline} 
which sends  $V(\lambda_r)$, the vanishing locus of $\lambda_r$, to 
\[Z_r=  0  \times \ldots \times 0   \times    \Gm^{r-1} \ql  \left( \Gmb^{r-1} \times   (  W_{r-1}^{\times} \times \ldots \times W_1^{\times}\times W_0 )\right)   \ . \] 
Its restriction to the complement of $V(\lambda_r)$ is  an isomorphism  onto the complement of $Z_r$ in its image. (It is not surjective, since the locus  $w_r=0$ is not in its image).
\end{lem}

\begin{proof} The map $\pi_r$ is $\Gm$-invariant for the action of the last component of $\Gm^{r}=\Gm^{r-1}\times \Gm$  (with coordinate $\mu_r$)
and equivariant for the action of $\Gm^{r-1}$ (with coordinates $(\mu_1,\ldots, \mu_{r-1})$). 
Since $w_r \neq 0$, the equation $\lambda_r w_r \neq 0$ is equivalent to  $\lambda_r\neq 0$.  Hence  $\pi^{-1}_r (Z_r) = V(\lambda_r)$. 
Recall that the last copy of $\Gm$ in $\Gm^{r} = \Gm^{r-1} \times \Gm$, with coordinate $\mu_r$, acts via 
\[  \mu_r \times  (\lambda_1,\ldots, \lambda_r, w_n,\ldots, w_0) \mapsto ( \lambda_1,\ldots, \lambda_{r-1}, \lambda_r \mu_r^{-1}  ,  \mu_r w_n ,\ldots, \mu_r w_r, w_{r-1} ,\ldots, w_0)\]
 By passing to the quotient for this action, every point with  $\lambda_r\neq 0$  has a unique representative  with  $\lambda_r=1$. On such points, the map in question  is the identity and hence $\pi_r$ is an isomorphism onto its image outside $V(\lambda_r)$.  
\end{proof} 
Now observe that the  map  $\Gm^n \ql (\Gmb^{n} \times W^{\star} ) \rightarrow W$ is a composition of morphisms
\begin{eqnarray}  \Gm^n \ql \left( \Gmb^{n} \times W^{\star}\right)  & \overset{\pi_n}{\To} &  W_n \times    \Gm^{n-1} \ql  \left(\Gmb^{n-1} \times    (  W_{n-1}^{\times} \times \ldots \times W_1^{\times}\times  W_0   ) \right)       \nonumber \\
&    \overset{\id \times \pi_{n-1}}{\To}   &
W_n \times  W_{n-1} \times  \Gm^{n-2} \ql   \left(\Gmb^{n-2} \times    (  W_{n-2}^{\times} \times \ldots \times W_1^{\times}\times  W_0   ) \right)       \nonumber  \\
& \vdots & \nonumber \\
&  \To &   W_n \times  W_{n-1} \times \ldots W_1 \times W_0\ . \nonumber
\end{eqnarray} 
The divisors $V(\lambda_r)$ can be described by  a computation given in \S\ref{sectExceptional}.
It follows from this and lemma \ref{lem:pirinducesblowup},  that each map is a blow-up along  $ Z_r$, and  one deduces that  $\Gm^n \ql (\Gmb^{n} \times W^{\star} )$  is an iterated blow-up along  the strict transforms of $0 \times \ldots \times 0 \times W_{r-1} \times \ldots \times W_0$.
\end{rem} 

\subsubsection{Variant: projective version} \label{sect: VariantProjective}
With notation as above, let 
\[ W^{\circ} =  W^{\times}_n \times W^{\times}_{n-1} \times \ldots \times W^{\times}_1 \times W_0^{\times}  \]
with an action of $\Gm^{n+1} =\Gm^n \times \Gm$  where $\Gm^n$ acts as above, but with an additional  copy of $\Gm$ (with coordinate $\mu_0$) acting by scalar multiplication on $W^{\circ}$ (i.e. diagonally on each copy $W^{\times}_i$ ).
Then  we may define the  quotient $  \Gm^{n+1} \ql (\Gmb^n\times W^{\circ})$ as before. 
By \eqref{ithcomponentGmbaction}, all components of $\Gm^n$ act trivially on $W_0$ and the morphism \eqref{monquotient}  now corresponds to:
 \[  \Gm^{n+1} \ql (\Gmb^n\times W^{\circ}) \To \Gm \ql( W\backslash \{0\}) = \Pro(W)  \ . \]
Note that  its image is  the  open complement of $\Pro(W_n\oplus \ldots \oplus W_1 \oplus \{0\})$ in $\Pro(W)$.
The projective version of proposition \ref{prop: quotientandblowup} is:

\begin{prop}  \label{prop: projversionQuotientBlowup}The scheme  $  \Gm^{n+1} \ql (\Gmb^n\times W^{\circ}) $, relative to the morphism:
\[   \Gm^{n+1} \ql (\Gmb^n\times W^{\circ}) \To \Pro(W) \]
 canonically embeds into the  iterated  blow-up $B_W\rightarrow \Pro(W)$  along the linear subspaces
\[ \Pro(W_0) \  \subset \ \Pro(W_0 \oplus W_1) \  \subset \   \ldots  \ \subset \  \Pro(W_0 \oplus W_1 \oplus \ldots \oplus W_{n-1}) \]
in increasing order of dimension. It is Zariski-dense in $B_W$.
 \end{prop} 

\subsubsection{Exceptional divisors}  \label{sectExceptional}
The exceptional divisors of $  \Gm^{n+1} \ql (\Gmb^n\times W^{\circ})$  are in one-to-one correspondence with the divisors $V(\lambda_i)$ defined by $\lambda_i=0$ for $i=1,\ldots, n$.

\begin{lem}
The  divisor $V(\lambda_i)$  is canonically isomorphic to  a product
\[  V(\lambda_i)  \ \cong  \  \Gm^{n-i+1} \ql (\Gmb^{n-i}  \times W_n^{\times}\times \ldots \times  W_{i}^{\times})  \ \times  \   \Gm^{i} \ql (\Gmb^{i-1}  \times W_{i-1}^{\times} \times \ldots \times W_{0}^{\times})\ .  \]
\end{lem} 
\begin{proof}
 There is a canonical isomorphism from 
$ \Gmb^{i-1} \times 0 \times \Gmb^{n-i}  \times W^{\times}_n\times \ldots \times W_0^{\times}$ to 
\[ \Gmb^{i-1} \times  W^{\times}_{i-1}\times \ldots \times W_0^{\times}  \times    \Gmb^{n-i} \times  W^{\times}_{n}\times \ldots \times W_i^{\times} \]
which maps  an element  $
(\lambda_1,\ldots, \lambda_{i-1}, 0, \lambda_{i+1}, \ldots, \lambda_n,  w_n,\ldots, w_0) $ to 
\[ (\lambda_1,\ldots, \lambda_{i-1},  w_{i-1},\ldots, w_0)
\times (\lambda_{i+1}, \ldots, \lambda_n,  w_n,\ldots, w_{i+1}, w_i) \ . \] 
One checks that it is equivariant with respect to the respective  actions of 
\begin{eqnarray}   \Gm^{i} \times \Gm^{n-i+1}  &\overset{\sim}{\To}&   \Gm^{n+1}  \nonumber \\
(\mu_0,\ldots, \mu_{i-1}) \times ( \mu_{i}, \ldots, \mu_n) & \mapsto &  (\mu_0,\ldots, \mu_{i-1}, \left(\mu_0\ldots \mu_{i-1}\right)^{-1}  \mu_i , \mu_{i+1}, \ldots, \mu_n) \nonumber \ . 
\end{eqnarray} 
The statement follows by passing to the quotient. 
\end{proof}

If we write $W'= W_{i-1} \oplus  \ldots \oplus W_0$,  it follows that $V(\lambda_i)$ embeds into the product   \begin{equation} \label{VlambdatoProduct}  V(\lambda_i) \  \hookrightarrow  \  B_{W/W'} \,  \Pro( W/W' ) \times    B_{W'}  \, \Pro(W')
\end{equation}
where the blow-ups are with respect to the flags of subspaces defined by the filtration $F_m=\bigoplus_{k=0}^m W_k$ on $W$ and the filtrations it induces on $W'$ and $W/W'$.

The irreducible components $\SE_i$ of the exceptional divisor $\SE$  defines a stratification on the blow-up $B_W$ of $\Pro(W)$. Consider the corresponding  quasi-projective stratification
\begin{equation}   \label{Bopenstrata} B_W =  \bigcup_{I}  \SE^{\circ}_I \end{equation}
where the union is over all subsets $I\subset \{1,\ldots, n\}$, and where we denote   $\SE_I= \bigcap_{i\in I} \SE_i$ , and  $\SE_{\emptyset} = \Pro(W)$,  and define 
\[ \SE^{\circ}_I =  \SE_I  \setminus \left( \bigcup_{j\notin I} \SE_I  \cap \SE_j \right) \ .  \]  
It follows from the description in proposition \ref{prop: BlowUpStructure}  that each  stratum of \eqref{Bopenstrata} (and  each stratum of $  \Gm^{n+1} \ql (\Gmb^n\times W^{\circ})$ via \eqref{VlambdatoProduct}), is embedded in a product of projective spaces:
\begin{equation} \label{SEcircembeds}   \SE^{\circ}_I  \quad \subset \quad \prod^{|I|}_{k=1} \Pro \left( \bigoplus_{i_{k-1} \leq j < i_k} W_j \right)   = \prod^{|I|}_{k=1} \Pro \left( \mathrm{gr}^{F'}_k W \right)  \ ,  \end{equation}
where  $F' W$ is the increasing filtration on $W$ whose $k^\mathrm{th}$ graded quotient is 
$  \bigoplus_{i_{k-1} \leq j < i_k} W_j $ where we set $i_0=0, i_{n+1}=n+1$. The strata of \eqref{Bopenstrata}  are in one-to-one correspondence with filtrations $F' \leq F$.

The restriction of the embedding \eqref{SEcircembeds} to the locus  in  $ \Gm^{n+1} \ql (\Gmb^n\times W^{\circ})$ where $\lambda_i=0$ for all $i\in I$ and 
$\lambda_j\neq 0$ for all $j\neq 0$,   is  given explicitly  by the map
\[ (\lambda_1,\ldots, \lambda_n, w_n,\ldots, w_0) \mapsto (w_0,\ldots, w_{i_1-1}) \times  (w_{i_1},\ldots, w_{i_2-1}) \times \ldots \times 
(w_{i_{n-1}},\ldots, w_n)   \]
where the right-hand side is in the space
$ \prod_{k=1}^{n+1}   \Gm \ql (W^{\times}_{i_{k-1}}  \times \ldots \times W_{i_{k}-1}^{\times}) $ which embeds in a product of projective spaces by example \ref{Ex: ProjAsQuotient}. One may verify that it  is an embedding because, for every $\lambda_j\neq 0$, we may apply a unique element of the $j$th copy of $\Gm$ in $\Gm^{n+1}$ to make the coordinate $\lambda_j$ equal to $1$.

\subsection{Spaces of parabolic and symmetric matrices} \label{sect: SpacesParabolics}
Using the previous construction, we may define  an algebraic version of   $A_P \setminus (\overline{A}_P \times P)$ associated to a rational parabolic $P$  and embed  it in an iterated blow-up.  We then provide an algebraic interpretation of the space $X(P)$ by embedding it into the real points of a finite iterated blow-up of projective space. 
We shall denote algebraic groups with caligraphic letters to distinguish them from the topological groups considered in \S\ref{sect: RemindersBS}.  

Let $V$ be a finite-dimensional vector space over $\Q$ and consider a flag of subspaces  
\[ F \  :  \ 0 \subset V_d \subset V_{d-1} \subset \ldots \subset V_1 \subset V_0=V\]
 which has a  splitting: $V\cong \bigoplus_{i\geq 0} V_{i}/V_{i+1}$.  Let 
$\mathcal{P}_F \subset \mathrm{End}_F(V)$
denote the affine variety over $\Q$ of  endomorphisms of $V$ which preserve $F$.   Using the splitting of $V$ we may write
\[ \mathrm{End}_F(V) = \mathrm{Hom}_F( V  \ , \   \bigoplus_i  V_i/V_{i+1} ) = \bigoplus_i   \mathrm{Hom}_F( V ,  V_i/V_{i+1} )\] 
which, upon setting  $\mathcal{P}_i= \mathrm{Hom}_F( V ,  V_i/V_{i+1} )$,  gives a decomposition (of schemes, corresponding to the direct sum in the category of vector spaces) 
\begin{equation} 
\label{PFdecomp}  \mathcal{P}_F = \mathcal{P}_d\times  \mathcal{P}_{d-1} \times \ldots \times \mathcal{P}_0\ .
\end{equation}
Define a subscheme $\mathcal{P}_F^{\circ}=   \mathcal{P}^{\times}_d\times  \mathcal{P}^{\times}_{d-1} \times \ldots \times \mathcal{P}^{\times}_0$  of $\mathcal{P}_F$ as in  \S\ref{sect: VariantProjective}.
Let 
\[  \mathcal{Z}_P = \Gm^{d+1} \   \hookrightarrow  \ \GL \left(\bigoplus_i V_i/V_{i+1} \right) = \prod_i \GL( V_i/V_{i+1})\]
  be the  product of  multiplicative groups whose $k^{\mathrm{th}}$ component, where $0\leq k\leq d$,  acts by scalar multiplication on  $V_k/V_{k+1}$. Denote its coordinate ring  $\Or(\mathcal{Z}_P) = k[\mu_0,\mu_0^{-1},\ldots, \mu_d,\mu_d^{-1}]$.

 Consider the  affine group scheme $\mathcal{A}_P = \Gm^d$, with coordinates  $\lambda_1,\ldots, \lambda_d$. There is a morphism of affine group schemes 
 $ \mathcal{A}_P \rightarrow \mathcal{Z}_P $,
  which on  coordinate rings is the map
  \[  (\mu_0,\ldots , \mu_d) \mapsto (1,\  \lambda_1, \  \lambda_1\lambda_2, \  \ldots, \ \lambda_1\cdots \lambda_d) : \Or(\mathcal{Z}_P) \To \Or(\mathcal{A}_P)\ . \]
  This  is consistent with  \eqref{muCoordsonAP}. The previous morphism, via the action of $\mathcal{Z}_P$ on $\bigoplus_i V_i/V_{i+1}$, defines an action of $\mathcal{A}_P$ on 
  $\mathcal{P}_F^{\circ}$. This action extends to an action of the  multiplicative monoid 
   $\overline{\mathcal{A}}_P= \Gmb^d$  whose  $k^{\mathrm{th}}$ component has  coordinate $\lambda_k$ where $1\leq k\leq d$. 
   It acts by scalar multiplication on $V_k$
  as in \S \ref{sect: setup}, i.e., for $v_i \in  V_i/V_{i+1}  $, for $0\leq i\leq d$ we have:
    \[ (\lambda_1,\ldots, \lambda_d) \times (v_d,\ldots, v_1, v_0) \mapsto  (\lambda_d\cdots \lambda_1 v_d,\   \ldots , \ \lambda_2\lambda_1 v_2,  \  \lambda_1 v_1, \  v_0) \ .  \]  
    Since the components of  $\mathcal{Z}_P$ act centrally on $V_i/V_{i+1}$, this defines an action  $\overline{\mathcal{A}}_P  \times \mathcal{P}_F^{\circ}\rightarrow   \mathcal{P}_F $ by scalar multiplication on each component.   It is given by the identical formula to the above, upon replacing $v_i$ with $p_i$, for $p_i \in \mathcal{P}^{\times}_i$. On points, an element   $a\in \mathcal{A}_P$ acts upon $(\lambda, p) \in \overline{\mathcal{A}}_P \times \mathcal{P}_F^{\circ}$ by $a. (\lambda, p) = ( \lambda . a^{-1}, a. p)$. 
    Finally let $\mathcal{H}= \Gm$ denote the subgroup of $\mathcal{Z}_P$ with coordinates $\mu_0=\ldots = \mu_d$, corresponding to diagonal matrices. It   acts on    $\overline{\mathcal{A}}_P \times \mathcal{P}_F^{\circ}$ by acting trivially on $ \overline{\mathcal{A}}_P$, and acting via scalar multiplication on
    $ \mathcal{P}_F^{\circ}$.

\begin{thm}  \label{cor: ePasblowup} 
There is a canonical Zariski-dense embedding: 
  \begin{equation} \label{APembeds}  (\mathcal{A}_P\times \mathcal{H} ) \ql \left( \overline{\mathcal{A}}_P \times \mathcal{P}^{\circ}_F \right)  \To B_F 
 \, \Pro(\mathcal{P}_F)  \ , 
 \end{equation} 
where $B_F \Pro(\mathcal{P}_F )$ is the iterated blow-up of $\Pro(\mathcal{P}_F)$ along the subspaces
 \[\Pro(\mathcal{P}_0)  \ \subset  \  \Pro(\mathcal{P}_{0} \oplus \mathcal{P}_{1})  \ \subset \   \ldots  \ \subset \  \Pro(\mathcal{P}_{0} \oplus \ldots \oplus \mathcal{P}_{d-1}) \ .\] 
Let $I=\{i_1,\ldots, i_k\}$ and $1\leq i_1 < \ldots<i_k\leq d$. The image of the  locus $V(\lambda_i, i\in I)$  under the morphism \eqref{APembeds} is contained in  
the codimension $k$ stratum   $\SE_I$ of the exceptional divisor of $B_F \Pro(\mathcal{P}_F )$ which corresponds to the flag $F'\leq F$: 
\begin{equation} \label{Fprimeflag}   F'  \ : \  0=V_{i_{k+1}} \subset  V_{i_k} \subset V_{i_{k-1}} \subset \ldots \subset V_{i_1} \subset  V_{i_0}=V 
\ .
\end{equation}  
Both  $V(\lambda_i, i\in I)$ and its Zariski-closure $\SE_I$  are canonically isomorphic to products,  and the morphism 
\eqref{APembeds} restricts to the embedding:
\begin{equation} \label{ProductsAPFquotients}  V(\lambda_i, i\in I) \cong \prod_{m=0}^k   (\mathcal{A}_{P_{F_m}} \times \mathcal{H} )\ql \left( \overline{\mathcal{A}}_{P_{F_m}}\times \mathcal{P}^{\circ}_{F_m}  \right)  \To  \SE_I  \cong \prod_{m=0}^k  B_{F_m} \Pro(P_{F_m})
\end{equation}
 where  $\mathcal{P}_{F_m} = \prod_{i_m\leq j < i_{m+1}} \mathcal{P}_{j} $,   $\mathcal{P}^{\circ}_{F_m} = \prod_{i_m\leq j < i_{m+1}} \mathcal{P}^{\times}_{j} $  and
 $F_m$ are the filtrations   induced by $F$ on the quotients $V_{i_m}/V_{i_{m+1}}$.

 The restriction of \eqref{ProductsAPFquotients} to the open locus $\{\lambda_j \neq 0 :  j \notin I\}$ gives an embedding 
 \begin{equation}  \label{OpenStratumAndAPquotient} 
(  \mathcal{A}_{P_{F'}}  \times \mathcal{H} ) \ql \mathcal{P}_F^{\circ} 
  \quad  \hookrightarrow  \quad    \SE^{\circ}_I   \subseteq   \prod_{m=0}^k \Pro(P_{F_m})\ .
 \end{equation} 
 \end{thm}
 
 \begin{proof} Apply proposition  \ref{prop: projversionQuotientBlowup} to  $\mathcal{P}_F$. The second part follows by iterating the description of a codimension one exceptional divisor given in \S \ref{sectExceptional}. The only novelty  is   \eqref{OpenStratumAndAPquotient}.
 It follows from the fact that the locus   $\{\lambda_j \neq 0: j\notin I\}$ is isomorphic to
 \[        \left( \prod_{m=0}^k (\mathcal{A}_{P_{F_m}}  \times \mathcal{H}) \ql \left( \overline{\mathcal{A}}_{P_{F_m}}\times \mathcal{P}^{\circ}_{F_m}  \right) \right) \setminus \bigcup_{j\notin I} V(\lambda_j)  \cong (\mathcal{A}_{P_{F'}}  \times \mathcal{H}) \ql \mathcal{P}_F^{\circ}    \ .  \]
 This follows from the isomorphism $\mathcal{P}_F^{\circ}  =  \mathcal{P}^{\circ}_{F_0} \times \ldots \times \mathcal{P}^{\circ}_{F_m} $, and the 
  fact that the locus where $\lambda_j\neq 0$ is a torsor over  the jth copy of $\Gm$ in $\mathcal{A}_{P_{F_m}}$, and so any point with coordinate $\lambda_j\neq 0$ has a unique representative in its  $\Gm$-orbit with $\lambda_j=1$. 
 \end{proof} 
  
  Since the determinant is a homogeneous polynomial in matrix entries, we may define  the determinant  hypersurface
 $\mathrm{Det} \subset \Pro(\mathcal{P}_F)$ to be its  vanishing locus.  Let 
 \[  \widetilde{\mathrm{Det}} \quad \subset\quad  B_F \Pro( \mathcal{P}_F) \]
 denote its strict transform.  
Let $P= \mathrm{GL}_F(V\otimes_{\Q} \R)$ denote the (topological) group of automorphisms  considered in \S\ref{sect: RemindersBS}, and  likewise $A_P\leq P$.
Then we may interpret 
\[ P  \  \cong  \   \left( \mathcal{P}_F \backslash \mathrm{Det} \right) (\R) \] 
and  it follows from theorem \ref{cor: ePasblowup} that there are injections:
\begin{equation} \label{Apinjects}
(A_P  \times H )\setminus \left( \overline{A}_P \times P \right)  \To   \left(( \mathcal{A}_P \times \mathcal{H} )\ql  \left( \overline{\mathcal{A}}_P \times ( \mathcal{P}^{\circ}_F \backslash \mathrm{Det}\right)\right) (\R) \To \left( B_F \Pro(\mathcal{P}_F) \setminus \widetilde{\mathrm{Det}}\right)(\R)  \  .   
\end{equation} 
On the left is a topological group;   in the middle, the real points of a scheme-theoretic group quotient; on the right, the real points of an iterated blow up.

\begin{example}  \label{example: Vflag2} Consider the case when $d=2$ and $F$ is $0\subset V_2 \subset V_1 \subset V$. By choosing a basis of $V$ adapted to this filtration, 
the points of $\mathcal{P}_F$ are given by block lower triangular matrices, and $\overline{\mathcal{A}}_P \cong \Gmb^2$ acts by 
\[  (\lambda_1, \lambda_2)  \times  \begin{pmatrix} P_{0} & & \\ P_{01} & P_1 & \\ P_{02} & P_{12} & P_{2}  \end{pmatrix}  \mapsto      \begin{pmatrix} P_{0} & & \\  
\lambda_1 P_{01} & \lambda_1 P_1 & \\ \lambda_1 \lambda_2 P_{02} &  \lambda_1\lambda_2 P_{12} &  \lambda_1 \lambda_2 P_{2}  \end{pmatrix}   \]
where $\lambda_1, \lambda_2$ are points on $\Gmb^2$. 
 The  points of  $\mathcal{P}_i$, for $i=0,1,2$ are given by the subsets  of matrices which vanish everywhere except in (block) row $i+1$, and the subspaces $\mathcal{P}_0 \subset \mathcal{P}_0 \oplus \mathcal{P}_1$ which are to be  blown up are the spaces of matrices of the form: 
\[    \begin{pmatrix} P_{0} & 0   & 0  \\    0   & 0  & 0  \\   0  &  0  & 0   \end{pmatrix}   \ \subseteq  \  \begin{pmatrix} P_{0} & 0 & 0  \\ P_{01} & P_1 & 0  \\ 0    & 0    & 0    \end{pmatrix}   
\ . \]
The points of the subspace $\mathcal{P}^{\circ}_F$ satisfy $P_0 \neq 0$,   $P_1\neq 0$ and $P_2 \neq 0$. This strictly contains the complement of the locus  $\mathrm{Det}$ 
 defined by $\det(P_0) \det(P_1) \det(P_2) =0$.
 
 The   stratification of   $\Gm \ql (\Gmb^2 \times \mathcal{P}^{\circ}_F)$  is generated by the equations  $  \lambda_1=0$ and $\lambda_2=0$ (below left). The strata themselves are described as quotients  (below, middle), which   embed into the following products of blow-ups (below, right):
\[  \begin{array}{rccc}
 V(\lambda_1)   \quad :  & \left(\Gm \ql  \mathcal{P}^{\circ}_0\right) \times  \Gm^2 \ql\left( \Gmb \times \mathcal{P}^{\circ}_1\times  \mathcal{P}^{\circ}_2   \right)  &   \subseteq &   \Pro(\mathcal{P}_0) \times B_{\,\Pro(\mathcal{P}_1)} \Pro( \mathcal{P}_2 \oplus \mathcal{P}_1) \nonumber \\ 
   V(\lambda_2)    \quad : &  \Gm^2 \ql   \left( \Gmb \times \mathcal{P}^{\circ}_0\times  \mathcal{P}^{\circ}_1   \right)  \times  \left(\Gm \ql  \mathcal{P}^{\circ}_2 \right)  &  \subseteq  &  B_{\,\Pro(\mathcal{P}_0)}  \Pro(\mathcal{P}_0\oplus \mathcal{P}_1) \times\Pro( \mathcal{P}_2 ) \nonumber \\ 
 V(\lambda_1,\lambda_2)   \quad   : &  \left( \Gm \ql  \mathcal{P}^{\circ}_0\right)  \times  \left( \Gm \ql  \mathcal{P}^{\circ}_1\right) \times \left(\Gm \ql  \mathcal{P}^{\circ}_2  \right)    & \subseteq  &    \Pro(\mathcal{P}_0) \times \Pro(\mathcal{P}_1) \times \Pro(\mathcal{P}_2) \nonumber\  . 
 \end{array}  \ .
 \]
 These strata, respectively, correspond to the flags:
 $0 \subset V_1 \subset V$ for $\lambda_1=0$;  $0 \subset V_2 \subset V$ for $\lambda_2=0$; and the full flag $0\subset V_2  \subset V_1 \subset V$ for $\lambda_1=\lambda_2=0$.  
 
 See Example \ref{ex: Vflag2cont} for the continuation of this example.

 \end{example}

\subsection{Spaces of quadratic forms, blow-ups and $X(P)$}
Let $F$, $P$ be as in \S\ref{sect: SpacesParabolics}. There is a natural morphism of algebraic varieties 
\begin{eqnarray} \label{PFtoPQV}  \Pro(\mathcal{P}_F) & \To &  \Pro(\Quad(V))   \\
M  &\mapsto&  M^T M  \nonumber 
\end{eqnarray} 
which sends  the subspace
$\Pro(\mathcal{P}_0 \oplus \ldots \oplus \mathcal{P}_i)$ to $\Pro(\Quad(V/V_{i+1})) \subset \Pro(\Quad(V)) $, the projective space of the subspace of quadratic forms which vanish on $V_{i+1}$.

Consequently,  by  the universal property of blow-ups, \eqref{PFtoPQV} induces a morphism 
\begin{equation} \label{blowupMTM}   B_F  \Pro(\mathcal{P}_F) \To    B_F \Pro(\Quad(V))  
\end{equation} 
where 
$B_F \Pro(\Quad(V))$ is the iterated blow-up of $\Pro(\Quad(V))$ along the linear subspaces
\[   \Pro\left( \Quad(V/V_1)\right)  \  \subset \   \Pro\left(\Quad(V/V_{2})\right) \  \subset \  \ldots \  \subset \  \Pro\left( \Quad(V/V_d) \right)  \]
associated to the flag $F$. 
Combining with  theorem \ref{cor: ePasblowup}   we deduce a natural map:
\begin{equation} \label{APtoQV}    (\mathcal{A}_P  \times \mathcal{H}) \ql \left( \overline{\mathcal{A}}_P \times \mathcal{P}^{\circ}_F \right) \To  B_F  \Pro(\mathcal{P}_F) \To    B_F \Pro(\Quad(V))     \  .  \end{equation}
 Recall from \S\ref{sect: BlowUpandDet} that  the   exceptional divisor $\mathcal{E}$ of  $B_F \Pro(\Quad(V))$  defines a stratification 
 whose strata $\mathcal{E}_{F'}$ are in one-to-one correspondence with flags $F' \leq F$  \eqref{Fprimeflag}, and 
are canonically isomorphic to a product of iterated  blow-ups 
\[  \mathcal{E}_{F'} \cong  \prod_{m=0}^k  B_{F_i}  \Pro\left( \frac{ \Quad(V/V_{i_{m+1}})}{ \Quad(V/V_{i_{m}})} \right) \ ,   \  \]
where $F_i$ is the induced filtration on the successive quotients $V_i/V_{i+1}$ of   \eqref{Fprimeflag}.
 Let $\mathrm{Det} \subset \Pro(\Quad(V))$ denote the vanishing locus of the determinant, and $\widetilde{\mathrm{Det}}$ its strict transform in $B_F \Pro(\Quad(V))$. 
It was proven in proposition \ref{prop: structureoffacesinblownupcone}  that 
\begin{equation} \label{Detfactorizes}  \mathcal{E}_{F'}   \setminus \left(\mathcal{E}_{F'}   \cap \widetilde{\mathrm{Det}}  \right)  \cong \prod_{m=0}^k    \left(  \Pro\left( \frac{ \Quad(V/V_{i_{m+1}})}{ \Quad(V/V_{i_{m}})} \right)  \ \setminus  \ \mathrm{Det}\big|_{V_{i_m}/V_{i_{m+1}}}  \right) \  
\end{equation} 
where  $\mathrm{Det}\big|_{V_{i_m}/V_{i_{m+1}}}$ is the zero locus of the homogeneous map 
\[  \frac{ \Quad(V/V_{i_{m+1}})}{ \Quad(V/V_{i_{m}})}  \To \Quad(V_{i_m}/V_{i_{m+1}}) \overset{\mathrm{det}}{\To} \Q  , \]
where the first map is restriction of quadratic forms (compare \S\ref{sect: CollapseNormal}).

\begin{thm} \label{thm: XPmapstoBlowup}  The map \eqref{APtoQV} induces an injective  continuous map 
\begin{equation} \label{XPtoBlowUp}   X(P)   \To  \left( B_F \, \Pro(Q(V))  \ \setminus \  \widetilde{\mathrm{Det}} \right) (\R)  \ . 
\end{equation} 
For every flag $F'$ of the form \eqref{Fprimeflag} the map \eqref{XPtoBlowUp} restricts to  a  map 
\begin{equation} \label{ePtostratum}  e(P_{F'})   \To      \mathcal{E}_{F'}   \setminus \left(\mathcal{E}_{F'}   \cap \widetilde{\mathrm{Det}}  \right) (\R)   
\end{equation} 
to the  corresponding stratum in the exceptional locus of $ B_F \, \Pro(Q(V))$.
The image of $e(P_{F'})$, via the identification \eqref{Detfactorizes}, is the product $e(P_{F'}) \cong \prod_{m} X_m$ where 
\[    X_m \ \subset \ \Pro\left( \frac{ \Quad(V/V_{i_{m+1}})}{ \Quad(V/V_{i_{m}})}  \right)(\R)   \]
 is the set of projective classes of matrices whose image in 
$\Pro(\Quad(V_{i_m}/V_{i_{m+1}}))(\R)$ is positive definite. 
\end{thm}

  \begin{proof}
 The equation $\det(M^TM) = \det(M)^2$ implies that \eqref{PFtoPQV} preserves  the determinant loci $\mathrm{Det}$ in $\Pro(\mathcal{P}_F)$ and $\Pro(\Quad(V))$, and so the morphism  \eqref{blowupMTM} preserves their strict transforms.   
    It follows that  \eqref{APtoQV} defines a natural  map
\begin{equation} \label{inproof: APtoBF}    (\mathcal{A}_P \times \mathcal{H} )  \ql \left( \overline{\mathcal{A}}_P \times \mathcal{P}^{\circ}_F \right)  \To  B_F \Pro(\Quad(V)) \setminus \  \widetilde{\mathrm{Det}}       \end{equation} 
whose restriction to the open locus $\{\lambda_i\neq 0,  \hbox{ for } 1\leq i \leq d\}$ is  the morphism
\[ \Gm \ql \mathcal{P}^{\circ}_F \To  \Pro(\Quad(V)) \setminus \mathrm{Det}  \]
which, in coordinates given by a basis of $V$, sends a matrix $M$ to the projective class of $M^T M$. This map is invariant by left-multiplication by the subscheme
$\mathcal{K}_{P_F} \leq    \mathcal{P}^{\circ}_F$
whose points are  matrices $O$ such that  $ O^TO = 1$, since one has  $(OM)^T OM =M$.  
Since this equation is algebraic, it remains true on the Zariski closure of $\mathcal{P}^{\circ}_F$, and hence \eqref{inproof: APtoBF} is also invariant under left multiplication by $\mathcal{K}_{P_F}$, since   the action of $\mathcal{K}_{P_F}$ on $\mathcal{P}^{\circ}_F$ commutes with that of $\mathcal{A}_P$ and $\mathcal{H}$.
We deduce that the map  of topological spaces
  \[  ( A_P  \times H) \setminus \left(\overline{A}_P \times P\right)      \To    \left( B_F \Pro(\Quad(V)) \setminus \  \widetilde{\mathrm{Det}}   \right)(\R) \]
which is induced by  taking real points of  the map 
$\eqref{inproof: APtoBF}$  and composing with  the first map in \eqref{Apinjects},  is   $K_P \leq \mathcal{K}_{P_F}(\R)$ invariant and so   passes to the quotient 
\[  X(P) =   Z_P K_P  \setminus  \left(\overline{A}_P \times P\right)  \To  \left(B_F \Pro(\Quad(V)) \setminus \  \widetilde{\mathrm{Det}}  \right)(\R)    \  . \]
 It gives a well-defined continuous map   \eqref{XPtoBlowUp} such that  the following diagram  commutes: 
    \[
\begin{array}{ccc}
Z_P \setminus \left(\overline{A}_P \times P\right)  &  \To     &  \left( B_F \Pro(P_F) \setminus \  \widetilde{\mathrm{Det}}   \right)(\R)  \\
  \downarrow  &   &  \downarrow   \\
X(P) =    A_P \setminus  \left(\overline{A}_P \times X\right)  &  \To   &    \left( B_F \Pro(\Quad(V)) \setminus \  \widetilde{\mathrm{Det}}   \right)(\R)  
\end{array} \ , 
\]
where the horizontal map along the top is the morphism \eqref{APembeds}   and 
  the vertical maps are $M\mapsto M^TM$ on the left, and \eqref{blowupMTM} on the right. 
  
  Now compute the restriction of the map \eqref{XPtoBlowUp} on each stratum indexed by a flag $F'$ of the form \eqref{Fprimeflag}.
 The  map \eqref{inproof: APtoBF} restricts to a map on exceptional divisors 
   \[   ( \mathcal{A}_{P_{F'}} \times \mathcal{H} )  \ql \mathcal{P}^{\circ}_F \  \overset{\eqref{OpenStratumAndAPquotient}}{\To}   \ \mathcal{S}^{\circ}_{F'}  \overset{\eqref{blowupMTM}}{\To} \mathcal{E}_{F'}\ . \]
    Since it is    left-invariant by $\mathcal{K}_{P_F}$,  it descends, after taking real points,  to the quotient 
  \[ e(P_{F'})=  A_{P_{F'}} \setminus X  \cong Z_{P_{F'}} K_{P_{F'}} \setminus P_{F'}  \To   \mathcal{E}_{F'} \backslash \left(\mathcal{E}_{F'} \setminus \widetilde{\mathrm{Det}}\right) (\R)\ .\]
 This morphism is injective on each stratum $e(P_{F'})$ since $ M^T M=N^T N$ holds for two invertible real matrices $M,N$ if and only if $M=ON$, for $O$ an orthogonal matrix.
 The final statement follows from the fact that \eqref{XPtoBlowUp}  identifies the open $X\subset X(P)$  with  the subspace of projective classes of positive-definite matrices inside  $ \Pro(\Quad(V))(\R)$. This proves the result for the trivial flag. In the case of a flag $0\subset V_1 \subset V$ of length $d=1$,  the map from $e(P)$ to the corresponding exceptional locus is given on the level of matrices  by 
 \begin{equation} \label{inproofAUBmatrix}   \begin{pmatrix}  A & 0 \\ U & B \end{pmatrix}    \quad   \mapsto  \quad   \begin{pmatrix}  A^TA  & 0 \\ 0 & 0  \end{pmatrix}  , \begin{pmatrix}  0 & U^TB \\    B^TU  &B^TB  \end{pmatrix} \ ,\end{equation}
where the matrices are on the right denote projective equivalence classes. We can choose $A,B$ such that $A^TA$, and $B^TB$  are arbitrary (projective classes of) positive definite real symmetric matrices, and finally choose $U$ such that $U^TB$ is any other matrix, since $B$ is invertible. This indeed proves that $e(P)$ surjects onto the  space described in the statement of the theorem. The general case, for a flag of arbitrary length, is similar (or follows from the case of length one by induction). 
  \end{proof}

\begin{ex} \label{ex: Vflag2cont}
As in  example \ref{example: Vflag2}, 
use  the splitting $V \cong V/V_1 \oplus   V_1/V_2 \oplus V_2 $ to write endomorphisms of $V$ in block matrix form. 
Consider the deepest stratum which is associated to the full flag $F$,  which corresponds to $\mathcal{E}_{12}$ (denoted  $V(\lambda_1,\lambda_2)$ in  example \ref{example: Vflag2})). 

The morphism \eqref{ePtostratum}  
\[  e(P) \To  \Pro(\Quad(V/V_1))(\R) \times  \Pro\left(\frac{\Quad(V/V_2)}{\Quad(V/V_1)}\right) (\R)\times  \Pro\left( \frac{\Quad(V)}{\Quad(V/V_2)}\right)(\R)\]
is given explicitly on matrix representatives in $P$ of classes in $e(P) =A_P K_P \setminus P$  by:
 \[ \begin{pmatrix} P_{0} & & \\ P_{01} & P_1 & \\ P_{02} & P_{12} & P_{2}  \end{pmatrix}  \mapsto   \begin{pmatrix} P_{0}^T P_{0} &   0&0 \\ 0 & 0  & 0 \\ 0 &0  &0  \end{pmatrix}   \ , \  \begin{pmatrix} 0 &  P_{01}^T P_1 & 0  \\ P_{1}^T P_{01}  & P_{1}^T P_1 & 0  \\ 0 &0 & 0  \end{pmatrix} 
 \ , \  \begin{pmatrix} 0 &  0 & P_{02}^T P_{2}   \\ 0 & 0  & P_{12}^T P_{2}     \\ P_2^T P_{02}  & P_2^T P_{12} & P_2^T P_2  \end{pmatrix}   \]

\noindent  Now consider  the stratum associated to the flag $F':  0\subset V_1 \subset V$, which corresponds to the exceptional divisor $\mathcal{E}_1$ (corresponding to $V(\lambda_1)$ in  example \ref{example: Vflag2}). The morphism  \eqref{ePtostratum}  
\[     e(P')   \To    \Pro(\Quad(V/V_1))(\R) \times   \Pro\left( \frac{\Quad(V)}{\Quad(V/V_1)}\right)(\R) \]
 is given on matrix representatives  in $P$ of classes  in $e(P') =A_{P'} K_{P'} \setminus P' =  A_{P'} K_P \setminus P$  by 
setting  $A =   P_0$,   $U = \left(\begin{smallmatrix} P_{01} \\ P_{02} \end{smallmatrix}\right)$, and    $B = \left( \begin{smallmatrix} P_{1} & 0  \\ P_{12} & P_2 \end{smallmatrix}\right)$  in   \eqref{inproofAUBmatrix}. This  gives the map
\[    \begin{pmatrix} P_{0} & & \\ P_{01} & P_1 & \\ P_{02} & P_{12} & P_{2}  \end{pmatrix}  \mapsto   \begin{pmatrix} P_{0}^T P_{0} &   0&0 \\ 0 & 0  & 0 \\ 0 &0  &0  \end{pmatrix}   
 \ , \  \begin{pmatrix} 0 &  P_{01}^T P_1 +P_{02}^TP_{12}  & P_{02}^T P_{2}   \\ P_{1}^T P_{01} +P_{12}^TP_{02}  & P_{1}^T P_1  +P_{12}^T P_{12}  & P_{12}^T P_{2}     \\ P_2^T P_{02}  & P_2^T P_{12} & P_2^T P_2  \end{pmatrix} 
      \]
\end{ex}
\subsection{Proof of the statements in \S\ref{sect: BSstatement}}

\subsubsection{Proof of theorem  \ref{thm: MainXBStoBlowup}} 
Theorem   \ref{thm: XPmapstoBlowup}  provides an injective map of stratified spaces from $X(P)$ to  the iterated blow-up $B_F \Pro(Q(V))$. 
By the argument of  Proposition \ref{prop: UniversalANDIdeal}  (ii), there is a canonical map $\BBB(\R)\rightarrow   B_F \Pro(Q(V))$ which is an isomorphism away from exceptional divisors. Since the image of the exceptional divisors is contained in the (strict transform) of the determinant locus, which does not meet the image of $X(P)$, it follows that the continuous map $X(P) \rightarrow  B_F \Pro(Q(V)) \setminus \widetilde{\Det}$ canonically lifts to a map
$X(P) \rightarrow  (\BBB\setminus \widetilde{\Det})(\R)$.  Via the description \eqref{XBSasunionXP} of $X^{\mathrm{BS}}$ as a union of $X(P)$, for $P$ rational parabolics, we deduce the canonical continous map $X^{\mathrm{BS}}\rightarrow (\BBB\setminus \widetilde{\Det})(\R)$.  It is injective since, by   Theorem   \ref{thm: XPmapstoBlowup}, the image of each open stratum $e(P)$  of $X^{\mathrm{BS}}$ is contained in the complement of the determinant locus in a unique exceptional stratum. They are disjoint, since the intersections of exceptional strata are contained in the determinant loci. 
   The statement abut the closure of $X$ will follow from the argument in the next paragraph below. 
      The description of the image of $e(P)$ follows from   Theorem   \ref{thm: XPmapstoBlowup}   and   \eqref{Detfactorizes}.
 Finally, the group of rational points  $\GL(V)(\Q)$ acts on $X$ via $P \mapsto Pg$. Correspondingly, the 
 algebraic group $\GL(V)$ acts upon $\Pro(\Quad(V))$ via $(g,M) \mapsto g^T M g$. Since  it permutes the set of  flags \eqref{FlagF} and hence the centers of all the blow-ups which define $\BBB$, it extends to a natural action of $\GL(V)$ upon $\BBB$. Its group of $\Q$-points  $\GL(V)(\Q)$ acts algebraically upon $\BBB$. This action is by compatible, via  $P \mapsto P^TP$,  with the  action of $\GL(V)(\Q)$ upon $X$.

 \subsubsection{Proof of corollary  \ref{cor: LAgBBtropisBS}} 
 Consider a $\GL_g(\Z)$-equivariant admissible decomposition $X= \bigcup_{\sigma} \sigma$, obtained, for example, as the intersection of an admissible decomposition  \S\ref{sect: AdmissDecomp} with the set of positive definite symmetric matrices. 
 Let us identify $X^{\mathrm{BS}}$ with its image in $ \BBB(\R)$ under the map $f$. Thus we may write 
 \[  X \subset X^{\mathrm{BS}}  \subset \BBB(\R)\ . \]
 Consider a polyhedron $\sigma \subset X$, and denote its closure in $X^{\mathrm{BS}}$ by $\overline{\sigma}$. Its closure in $\BBB(\R)$ is the compact polyhedron
 $\sigma^{\BB}=\sigma^{\BB^{\min,\sigma}}$ in the minimal blow-up of $\sigma$, since, by definition, the latter is the topological closure of the inverse image of the interior of $\sigma$ in its minimal blow up.  The equality $\sigma^{\BB}=\sigma^{\BB^{\min,\sigma}}$ follows from invariance under extraneous modifications (proposition \ref{prop: minimalblow-down}). 
 By proposition \ref{prop: structureoffacesinblownupcone}, each open face  of $\sigma^{\BB}$ is a  product of strictly positive polyhedra, and is therefore contained in the image of an $e(P)$, for a certain $P$, by  the last part of theorem    \ref{thm: XPmapstoBlowup}. It follows that $\sigma^{\BB} \subset X^{\mathrm{BS}}$.  Thus 
 \[X \subset  \bigcup_{\sigma} \sigma^{\BB} \subset X^{\mathrm{BS}} \   \]
  and in particular   $\overline{\sigma} = \sigma^{\mathrm{BS}}$.
   By construction, $X(P) = A_P \setminus (\overline{A}_P \times X)$ is contained in the closure of $X = A_P \setminus (A_P \times X)$ (let the coordinates $\lambda_i$ on $A_P$ tend to zero) and so $X^{\mathrm{BS}}$ is contained in the closure of $X$. The previous equation then implies  that
  \begin{equation} \label{XBSdecomp}  X^{\mathrm{BS}} =      \bigcup_{\sigma} \sigma^{\BB} \end{equation}
  It follows that $X^{\mathrm{BS}}$ equals the closure of $X$ in $\BBB(\R)$, which was the remaining statement to prove in theorem \ref{thm: MainXBStoBlowup}.
 The decomposition \eqref{XBSdecomp}  has similar properties to an admissible decomposition  as in \S\ref{sect: AdmissDecomp} (finitely many $\GL_g(\Z)$ orbits, intersections are closures of faces, etc).      Conclude using the fact that, by definition \eqref{LAgtropBBasFunctoronInvertedBlowups}, the topological space $\left| \LA_g^{\trop,\BBB}\right|$ is the quotient of  $\bigcup_{\sigma} \sigma^{\BB}$ by the action of $\GL_g(\Z)$, since the gluing relations in the diagram category $\mathcal{D}_g^{\trop,\BB}$ are  those induced by $\GL_g(\Z)$.

\bibliographystyle{alpha}

\bibliography{biblio}

\end{document}